\documentclass[11pt]{article}
\usepackage[margin=1in]{geometry}

\usepackage{cancel}
\usepackage{tipa}

\usepackage{listings}
\usepackage{graphicx}
\usepackage{amsmath,amsthm,amsfonts,amssymb,amscd}

\usepackage[colorlinks=true, pdfstartview=FitV, linkcolor=blue,citecolor=blue, urlcolor=blue]{hyperref}
\usepackage{times}

\usepackage{wrapfig}
\usepackage{mathtools}
\usepackage{mathabx}
\usepackage{float}
\usepackage{makeidx}
\usepackage{stmaryrd}
\usepackage{mathrsfs}
\usepackage{emptypage}
\usepackage{amsrefs}
\usepackage{xfrac}

\usepackage{enumitem}

\allowdisplaybreaks
 
\newcommand{\XXX}{{\mathcal X}}
\newcommand{\FFF}{{\mathcal F}}
\newcommand{\LLL}{{\mathcal L}}
\renewcommand{\epsilon}{\varepsilon}

\def\les{\lesssim}

\def\eps{\varepsilon}
\renewcommand*{\div}{\ensuremath{\mathrm{div\,}}}

\newcommand{\norm}[1]{\left \| #1 \right\|} 
\newcommand{\snorm}[1]{\bigl\| #1 \bigr\|} 
\newcommand{\abs}[1]{\left|#1\right|}
\newcommand{\sabs}[1]{\bigl|#1\bigr|}

\newcommand{\RR}{\mathbb R}

\newcommand{\OO}{\mathcal O}
\newcommand{\RSZ}{\mathcal R}
\renewcommand*{\tilde}{\widetilde}

\renewcommand*{\bar}{\overline}
\newcommand{\Ncal}{{\mathsf{N}}}
\newcommand{\Tcal}{{\mathsf{T}}}
\newcommand{\Vcal}{{\mathsf{v}}}
\newcommand{\Jcal}{{\mathsf{J}}}

\newcommand{\acal}{{\mathsf{a}}}
\newcommand{\bcal}{{\mathsf{b}}}
\newcommand{\mcal}{{\mathsf{m}}}

\newcommand{\xcal}{{\text{x}}}
\newcommand{\tcal}{{\text{t}}}

\newcommand{\EE}{\mathcal E}
\newcommand{\PP}{\mathcal P}

 \newcommand*{\Id}{\ensuremath{\mathrm{Id\,}}}

\newtheorem{theorem}{Theorem}[section]
\newtheorem{lemma}[theorem]{Lemma}

\newtheorem{proposition}[theorem]{Proposition}
\newtheorem{corollary}[theorem]{Corollary}
\theoremstyle{definition}
\newtheorem{definition}[theorem]{Definition}

\newtheorem{remark}[theorem]{Remark}

\numberwithin{equation}{section}

\def\p{\partial}

\def\p{\partial}

\def\f1r{{\frac{1}{r}}  }

\def\ckg{\check \gamma}
\def\modckg{{ \abs{\ckg}} }
\def\tt{\Tcal}
\def\nn{\Ncal}
\def\mrg{\mathring{\zeta}}
\def\mru{\mathring{u}}

\def\mrs{\mathring{\sigma}}
\def\gui{\p^\gamma U_i}
\def\guj{\p^\gamma U_j}
\def\gs{\p^\gamma S}
\def\comm#1#2{{\llbracket#1,#2\rrbracket}}

\def\tu{\tilde{u} }
\def\trho{\tilde{\rho}}
\def\XX{{\scriptstyle \mathcal{X} }}
\def\pw{\Phi_{\scriptscriptstyle W}}
\def\pz{\Phi_{\scriptscriptstyle Z}}
\def\pa{\Phi_{\scriptscriptstyle U}}
\def\pu{\Phi_{\scriptscriptstyle U}}
\def\pwy{\pw^{\scriptscriptstyle y_0}}
\def\pzy{\pz^{\scriptscriptstyle y_0}}
\def\pay{\pa^{\scriptscriptstyle y_0}}
\def\te{\tilde{e}}
\def\tu{\tilde{u}}
\def\trho{\tilde{\rho}}
\def\ww{\texttoptiebar{w}}

\usepackage{fancyhdr}
\title{{\bf Formation of point shocks  for 3D compressible Euler}}
 
\author{
{ \small {\bf Tristan Buckmaster}}\thanks{\footnotesize Department of Mathematics, 
Princeton University, Princeton, NJ 08544,
 \href{buckmaster@math.princeton.edu}{buckmaster@math.princeton.edu}}
\and  
{\small {\bf Steve  Shkoller}}\thanks{\footnotesize Department of Mathematics, UC Davis, Davis, CA 95616, \href{shkoller@math.ucdavis.edu}{shkoller@math.ucdavis.edu}.}
\and 
{\small {\bf Vlad Vicol}}\thanks{\footnotesize Courant Institute of Mathematical Sciences, New York University, New York, NY 10012, \href{vicol@cims.nyu.edu}{vicol@cims.nyu.edu}.}
}

\date{} %
\begin{document}

\maketitle

\begin{abstract}
We consider the 3D isentropic compressible Euler equations with the ideal gas law. 
We provide a   constructive 
proof of shock formation from  smooth initial datum of finite energy, with no vacuum regions, with {\it nontrivial vorticity} present at the shock, 
and under {\it  no symmetry assumptions}.   We  prove that for 
an open set of Sobolev-class initial data which are a small $L^ \infty $ perturbation of a constant state,  there exist smooth solutions to the Euler equations which form a 
{\em generic} stable shock in finite time. The blow up time and location  can be explicitly computed, and solutions at the blow up time  are smooth except for {\em a single point}, where they are of cusp-type with H\"{o}lder  $C^ {\sfrac{1}{3}}$ regularity. Our proof is based on the use of modulated self-similar variables that are used to
enforce a number of constraints on the blow up profile, necessary  to establish   global existence and asymptotic stability in self-similar variables.
\end{abstract}

\renewcommand{\baselinestretch}{0.75}\normalsize
\setcounter{tocdepth}{1}

\tableofcontents
\renewcommand{\baselinestretch}{1.0}\normalsize

\section{Introduction}
 A fundamental problem in the analysis of nonlinear partial differential equations concerns the finite-time breakdown of smooth solutions and the
nature of the singularity that creates this breakdown.  In the context of gas dynamics and the compressible Euler equations which  model those dynamics,
the classical singularity is a shock.
When the initial disturbance to a constant state is sufficiently strong,  created for example by explosions,  supersonic projectiles, or a
 kingfisher shot out of a cannon,  violent pressure changes  lead to a progressive self-steepening of the wave, which ends in a shock.   

Our main goal is to give a detailed characterization of this shock formation process leading to the first singularity,  for the isentropic compressible Euler
equations in three space
dimensions.   Specifically, we shall give a precise description of the  initial data from which  smooth solutions to the Euler equations evolve, steepen, and
form a stable {\it generic} shock in finite time, in which the gradient of velocity  and gradient of density become infinite at a single point, while
the velocity, density, and vorticity remain bounded.
 In the process, we shall provide the exact blow up time, blow up location, and regularity
of the three-dimensional generic blow up profile.   Away from this single blow up point, the solution remains smooth.

Let us now introduce the mathematical description.
The three-dimensional isentropic compressible Euler equations are written as
\begin{subequations}
\label{eq:Euler}
\begin{align}
\partial_\tcal (\rho u) + \operatorname{div}_\xcal  (\rho u \otimes u) + \nabla_{\!\xcal} p(\rho) &= 0 \,,  \label{eq:momentum} \\
\partial_\tcal \rho  +  \operatorname{div}_\xcal  (\rho u)&=0 \,,  \label{eq:mass}
\end{align}
\end{subequations}
where   $\xcal=(\xcal_1,\xcal_2,\xcal_3) \in \mathbb{R}^3  $ and $\tcal \in \mathbb{R}  $ are the space and time
coordinates, respectively.  The unknowns are the velocity vector field 
 $u :\mathbb{R}^3  \times \mathbb{R}  \to \mathbb{R}^3$, the strictly positive density scalar field $\rho: \mathbb{R}^3 \times \mathbb{R}  \to \mathbb{R}  _+$, 
and the pressure $p: \mathbb{R}^3  \times \mathbb{R}  \to \mathbb{R}_+$, which is defined by the ideal gas law
$$
 p(\rho) = \tfrac{1}{\gamma} \rho^\gamma\,, \qquad \gamma >1 \,.
$$
The sound speed $c(\rho) = \sqrt{ \sfrac{\p p}{\p \rho} }$ is then given by 
$c= \rho^ \alpha $ where  $\alpha = \frac{\gamma-1}{2}$.
The Euler equations \eqref{eq:Euler}  are a system of conservation laws:  \eqref{eq:momentum} is the conservation of momentum and
\eqref{eq:mass}  is conservation of mass.   Defining the scaled sound speed by $\sigma= {\frac{1}{\alpha}} \rho^\alpha$,  
\eqref{eq:Euler} can be equivalently written as the system
\begin{subequations}
\label{eq:Euler2}
\begin{align}
\p_\tcal u + (u \cdot \nabla_{\!\xcal}) u + \alpha \sigma  \nabla_{\!\xcal} \sigma&=0 \,,  \label{eq:momentum2} \\
\partial_\tcal \sigma + (u \cdot \nabla_{\!\xcal}) \sigma  + \alpha \sigma \operatorname{div}_\xcal u&=0 \,.  \label{eq:mass2}
\end{align}
\end{subequations}
We let $\omega= \operatorname{curl}_{\xcal} u$ denote the vorticity vector and we shall refer to the vector $ \zeta=\tfrac{\omega }{\rho} $ as the 
{\it specific vorticity}, which satisfies the vector transport equation
\begin{align} 
\p_\tcal \zeta + (u \cdot \nabla_{\!\xcal }) \zeta - \left( \zeta  \cdot \nabla_{\!\xcal} \right) u =0 \,.
\label{svorticity}
\end{align} 

Our proof of shock formation relies upon a transformation of 
the problem from the original space-time variables $(\xcal,\tcal)$ to modulated self-similar space-time coordinates $(y,s)$, and on a change of unknowns from $(u,\sigma)$ to a set of geometric Riemann-like variables $(W,Z,A)$ in the self-similar coordinates. The singularity model is  characterized by the behavior near $y=0$ of the stable, stationary solution  $\bar W= \bar W(y_1,y_2,y_3)$ (described in Section \ref{sec:wbar} and shown in Figure \ref{fig:wbar}) of the 3D self-similar Burgers equation
\begin{align}
 -\tfrac 12 \bar W + \left( \tfrac{3}{2} y_1+ \bar W \right) \partial_{y_1} \bar W+ \tfrac{1}{2} y_2  \partial_{y_2} \bar W   + \tfrac{1}{2} y_3 \partial_{y_3} \bar W     = 0 \,. 
 \label{eq:Burgers:self:similar}
\end{align}
For a fixed $T$,  the vector $v=(v_1,v_2,v_3)$ given by
$$
v_1(\xcal_1,  \xcal_2,\xcal_3 ,\tcal) =  (T- \tcal)^ {\frac{1}{2}} \bar W\left( \frac{\xcal_1}{(T-\tcal)^ {\frac{3}{2}} }, \frac{ \xcal_2}{(T-\tcal)^ {\frac{1}{2}} } , 
\frac{ \xcal_3}{(T-\tcal)^ {\frac{1}{2}} }\right) \,, \ \
v_2 \equiv0 \,, \ \ v_3 \equiv 0   \,,
$$
 is the solution of the 3D Burgers equation in original variables, $\p_\tcal v +(v \cdot \nabla_\xcal ) v =0$, forming a shock at a single point at time $\tcal=T$.
An explicit computation shows that the Hessian matrix $\p_{y_1} \nabla_y ^2 \bar W|_{y=0}$ is strictly positive definite.   This property ensures that the
blow up profile $\bar W$ is {\it generic} in the sense described by Christodoulou in equation (15.2) of \cite{Ch2007}.   This genericity condition, in turn, provides 
stability of the shock profile for solutions to the Euler equations as we will explain in detail below.    
\begin{wrapfigure}{r}{0.3\textwidth}
\vspace{-0in}
\includegraphics[scale = 0.4]{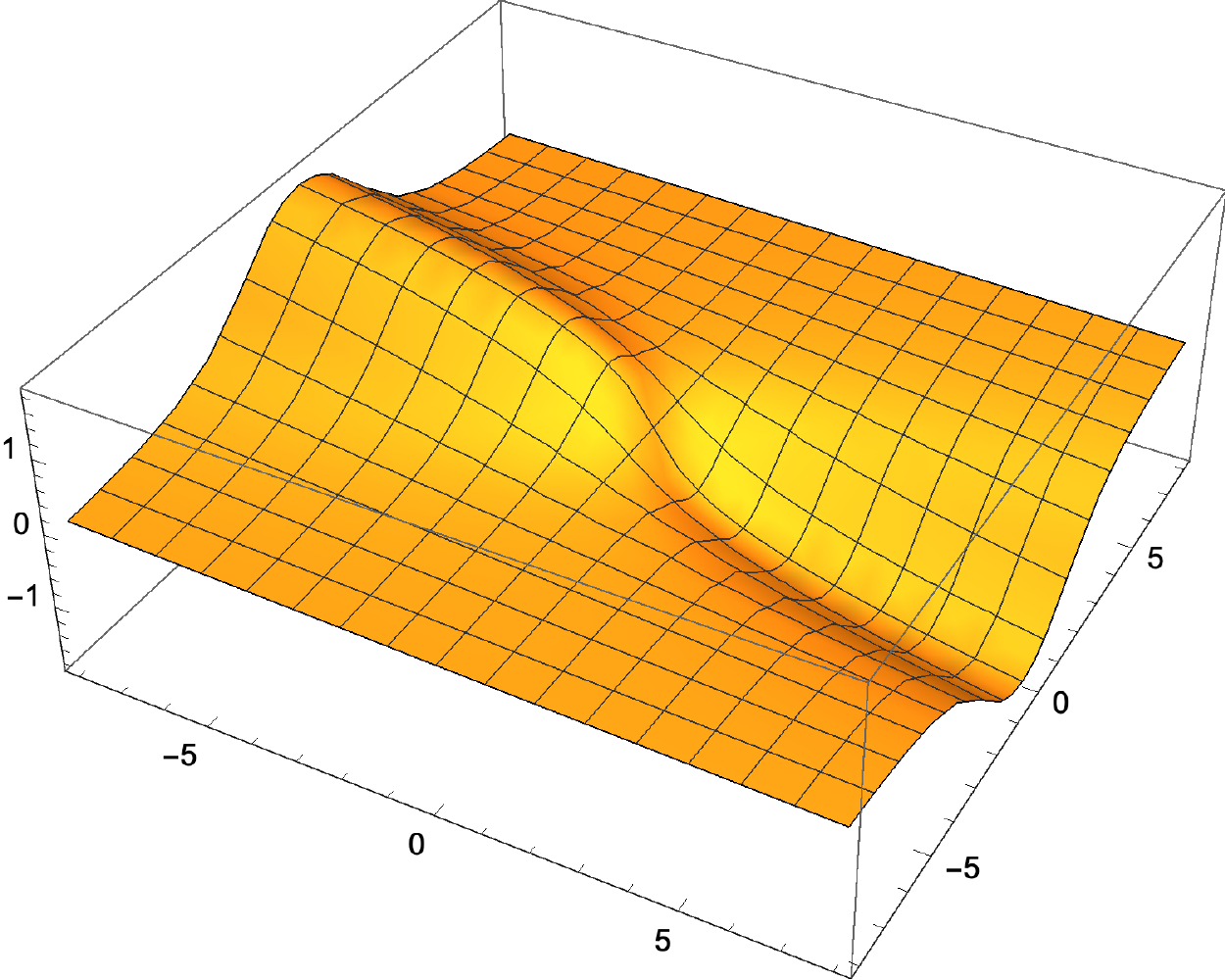}
\vspace{-.2in}
\caption{\footnotesize{The stable generic shock profile (shown in 2D).}}\label{fig:wbar}
\end{wrapfigure}

A precise description of  shock formation necessitates explicitly defining the set of initial data which lead to a finite-time  singularity, or shock. Additionally, from 
the initial datum alone, one has to be able to infer the following properties of the solution at the first shock:  
(a) the geometry of the {\em shock set}, i.e., to classify whether  the first singularity occurs along either a point,  multiple points,  a line, or along a surface;
(b) the  {\em  precise regularity} of the solution at the blow up time;  
(c) the explicitly computable {\em space-time location of the first singularity}; 
(d) the {\em stability} of the shock.
For the last condition (d), by stability, we mean that for any small, smooth, and {\em generic} (meaning outside of any symmetry class) perturbation of the given initial data, 
the  Euler dynamics yields a smooth solution which self-steepens and shocks in finite time with the same shock set geometry, with a  shock location that is a small
perturbation, and with the same shock regularity; that is, properties  (a)--(c) are stable.

As an example, the solution $\bar W$ shown in Figure \ref{fig:wbar} is stable: the shock occurs at a single point, and any small generic perturbation of $\bar W$ (as
we will prove) also develops a shock at only a single point, and with  the same properties as those satisfied by $\bar W$.  On the other hand, a simple plane wave solution of
the Euler equations that travels along the $\xcal_1$ axis and is constant in $(\xcal_2,\xcal_3)$ produces a finite-time shock along an entire plane, but a small perturbation of
this simple plane wave solution can produce a very different shock geometry (any of the sets from condition (a) are possible).
Our main result can be roughly stated as  follows:
\begin{theorem}[\bf Rough statement of  the main theorem]
For an open set of smooth initial data without vacuum, with nontrivial vorticity, and with a maximally negative gradient of size $\OO({\sfrac{1}{\eps}} )$, for $\eps>0$ sufficiently small, there exist smooth solutions of the 3D Euler equations \eqref{eq:Euler} which form a shock singularity within time $\OO(\eps)$. 
The first singularity occurs at a single point in space, whose location can be explicitly computed, along with the precise time at which it occurs.
The blow up profile is shown to be a cusp with $C^ {\sfrac{1}{3}} $ regularity, and the singularity is given by an asymptotically self-similar shock profile which is stable with
respect to the $H^{k}(\RR^3)$ topology for $k\ge 18$.
\end{theorem} 
A precise statement of the main result will be given below as Theorem \ref{thm:main}

\subsection{Prior results on shock formation for the Euler equations}
In one space dimension, the isentropic Euler equations are an example of a  $2 \times 2$ system of conservation laws, which can be written in terms of the
Riemann invariants $z= u - \sfrac{c}{\alpha }$ and $w=u+\sfrac{c}{\alpha } $ introduced in \cite{Ri1860}; the functions $z$ and $w$ are constant along the characteristics of the two wave speeds $ \lambda_1 = u-c$ and $ \lambda _2 = u + c$.
Using Riemann invariants,  Lax~\cite{Lax1964} proved that finite-time shocks can form from smooth data for general  $2 \times 2$ genuinely nonlinear 
hyperbolic systems.   The proof showed that the derivative of $w$ must become infinite in finite time, but the nature of the proof did not permit for any
classification of the type of shock that forms.    Generalizations and improvements of Lax's result were obtained by John~\cite{John1974},  Liu~\cite{Li1979}, and  Majda~\cite{Ma1984}, for the 1D Euler equations.   Again, these proofs showed that either a slope becomes
 infinite in finite time or that (equivalently) the distance between nearby characteristics approaches zero, but we note that a precise description of the
 shock was not given.   See the book of Dafermos \cite{Da2010} for a more extensive bibliography of 1D results.
 
For the 3D Euler equations,  Sideris \cite{Si1985} formulated a proof by contradiction (based on virial identities) that $C^1$  regular solutions to \eqref{eq:Euler} have a finite lifespan; in particular, he showed that $\OO( \exp(1/ \eps ))$  is an upper bound for the lifespan (of 3D flows) for
data of size $ \eps $.  The proof, however, did not reveal the type of singularity that develops, but rather, that some  finite-time breakdown of smooth
solutions must occur.   

The first proof of shock formation for the compressible Euler equations in the multi-dimensional setting  was given by
Christodoulou \cite{Ch2007} for  relativistic fluids and with the restriction of {\it irrotational} flow, and later by Christodoulou-Miao \cite{ChMi2014} for 
non-relativistic, irrotational flow.\footnote{For the restricted shock development problem, in which the Euler solution is continued past the time of first singularity but vorticity production is neglected, see the discussion in Section 1.6 of \cite{Ch2019}.}  This geometric
method uses an eikonal function (see also  \cite{ChKl1993}, \cite{KlRo2003}), whose level sets correspond to characteristic surfaces; 
it is shown that in finite time, the distance between nearby characteristics tends to zero.  For irrotational flows, the isentropic Euler equations can be written as a scalar 
second-order quasilinear wave equation.  The first results  on shock formation for 2D quasilinear wave equations which do not satisfy Klainerman's null condition  
\cite{Kl1984} were established by Alinhac \cite{Al1999a,Al1999b}, wherein a detailed
description of the blow up was provided.   
The first proof of shock formation
for fluid flows with vorticity   was given by Luk-Speck  \cite{LuSp2018}, for the 2D isentropic Euler equations.   Their proof uses Christodoulou's geometric framework 
and develops new methods to study the vorticity transport.   In  \cite{Ch2007,ChMi2014,LuSp2018}, solutions are constructed which are 
small perturbations  of  simple plane waves.   It is shown that there exists at least one point in spacetime where a
shock must form, and a bound is given for this blow up time; however, since the construction of the shock solution is a perturbation
of a simple plane wave, there are numerous possibilities for the type of singularity that actually forms.   In particular, their  method of proof does not   distinguish 
between these different scenarios.
To be precise,  a simple plane wave solution of the 2D isentropic
Euler equations that travels along the $\xcal_1$ axis and is constant in $\xcal_2$ produces a finite-time shock along a line, but a small  perturbation of
this simple plane wave solution can produce a very different singular set, with blow up occurring on different  spatial sets such as one point, 
multiple points,  or a line.  

In our earlier work~\cite{BuShVi2019}, we considered solutions to the  2D isentropic Euler equations with $\OO(1)$ vorticity and with azimuthal symmetry.  
Using modulated self-similar variables,  we provided the first construction of shock solutions that completely classify the shock profile: the shock is an asymptotically 
self-similar,  stable, a generic 1D blow up profile, with  explicitly computable blow up time and location, 
and with a  precise description of the $C^ {\sfrac{1}{3}} $ H\"{o}lder regularity of the shock.  Azimuthal symmetry allowed us to use transport-type $L^ \infty $ bounds
which simplified the technical nature of the estimates,  but the proof already contained some of the fundamental ideas required to study the full 3D Euler equations with no  symmetry assumptions.

\subsection{The variables used in the analysis and strategy of the proof} 
We now introduce the variables  used in the analysis of shock formation.  For convenience we first rescale time $\tcal \mapsto t$, as described in \eqref{eq:time:rescale}. Associated to certain modulation functions (described in Section~\ref{sec:modulation} below), are a succession of transformations for both the independent variables and the dependent variables.   In order to dynamically align
the blow up direction with the $e_1$ direction,  a time-dependent rotation and translation are made in \eqref{eq:tilde:x:def} which maps $\xcal$ to $\tilde x$, with $u$, $\sigma$,
and $\zeta$ transformed to $\tilde u$, $\tilde \sigma$, and $\tilde \zeta$ via \eqref{eq:tilde:u:def} and \eqref{vort00}.   Fundamental to the analysis of stable shock formation,
we make a further coordinate transformation $\tilde x \mapsto x$ given by \eqref{x-sheep}; this mapping  modifies the $\tilde x_1$ variable by
a  function $ f(\tilde x_2, \tilde x_3, t) = \tfrac{1}{2} \phi_{\nu\gamma}(t) \tilde x_\nu\tilde x_\gamma  $ which is quadratic in space and dynamically modulated by $\phi_{\nu\gamma}(t)$.   The parameterized surface $( f(\tilde x_2, \tilde x_3, t), \tilde x_2, \tilde x_3)$ can be viewed as describing the steepening shock front near $x=0$, and provides
a time-dependent orthonormal basis along the surface, given by the vectors the unit normal vector $\Ncal(\check x, t)$ and the two unit tangent vectors $ \Tcal^2(\check x, t)$, and $ \Tcal^3(\check x, t)$ defined in \eqref{normal} and \eqref{tangent}.  Together with the coordinate transformation $\tilde x \mapsto x$, the functions
$\tilde u$, $\tilde \sigma$, and $\tilde \zeta$ are transformed to $\mathring u$, $\mathring \sigma$, and $\mathring \zeta$ using \eqref{usigma-sheep} and 
\eqref{sv-sheep}.    Moreover, the Riemann variables $w= \mathring u \cdot \Ncal + \mathring \sigma$ and  $z= \mathring u \cdot \Ncal - \mathring \sigma$, as well as
the tangential components of velocity $a_\nu= \mathring u \cdot \Tcal^\nu$ are introduced in \eqref{tildeu-dot-T}.

Finally, we map $(x, t)$ to the modulated self-similar coordinates $(y,s)$ using the transformation \eqref{eq:y:s:def}.   The variables $\mathring u$, $\mathring \sigma$, and 
$\mathring \zeta$ are mapped to their self-similar counterparts $U$, $S$, and $\Omega$ via \eqref{U-trammy}, \eqref{S-trammy},  and \eqref{svort-trammy}, 
while $w$, $z$, and $a_\nu$ are mapped to the self-similar variables $e^{-\frac s2} W + \kappa $, $Z$, and $A_\nu$ in \eqref{eq:ss:ansatz}. 

As a consequence of this sequence of coordinate and variable changes, the Euler equations in the original variables~\eqref{eq:Euler2} for the unknowns $(u(\xcal,\tcal),
\sigma(\xcal,\tcal))$ become the self-similar evolution \eqref{US-euler-ss} for the unknowns $(U(y,s),S(y,s))$.  Of crucial importance for our analysis is the evolution of the 
self-similar Riemann type variables $(W(y,s),Z(y,s),A(y,s))$ in \eqref{euler-ss}, which encode the full Euler dynamics in view of \eqref{eq:UdotN:S}. The key insight to our analysis is that the self-similar Lagrangian trajectories associated to the $W$ equation escape exponentially fast towards spatial infinity if their starting label is at a fixed (small) distance away from the blowup location $y=0$, whereas the Lagrangian trajectories for $Z$ and $A$ escape towards infinity independently of their starting label, spending at most an $\OO(1)$ time near $y=0$. This exponential escape towards infinity is what allows us to transfer information about spatial decay of various derivatives of $W$ into integrable temporal decay for several damping and forcing terms, when viewed in Lagrangian coordinates. As opposed to our earlier work~\cite{BuShVi2019}, these pointwise estimates for $(W,Z,A)$ do not close by themselves, as there is a loss of a $\check\nabla $ derivatives when the equations are analyzed in $L^\infty$. This difficulty is overcome by using the energy structure of the 3D compressible Euler system, which translates into a favorable $\dot{H}^k$ estimate for the self-similar variables $(U,S)$, for $k$ sufficiently large (e.g.~$k\geq 18$ is sufficient).

Coupled to the $(W,Z,A)$ evolution we have a nonlinear system of 10 ODEs which describe the evolution of our 10 dynamic modulation variables 
$\kappa, \tau, n_2, n_3, \xi_1,\xi_2,\xi_3, \phi_{22},\phi_{23}, \phi_{33}$, whose role is to dynamically enforce constraints for $W, \nabla W$ and 
$\nabla^2 W$ at $y=0$, cf.~\eqref{eq:constraints}.

For all $s<\infty$, or equivalently, $t<T_*$, the above described transformations are explicitly invertible. Therefore, our main result, Theorem~\ref{thm:main}, is a 
direct consequence of Theorem~\ref{thm:SS},   which establishes the global-in-self-similar-time stability of the solution $(W,Z,A)$, in a suitable topology near the blowup
profile  $(\bar W,0,0)$, along with the  stability of the 10 ODEs for the modulation parameters. In turn,  this is achieved by a standard bootstrap argument: fix an 
initial datum with certain {\em quantitative properties}; 
then postulate that these properties worsen by a factor of at most $K$, for some sufficiently large constant $K$; to conclude the proof, we a-posteriori show that 
in fact the solutions' quantitative properties worsen by a factor of at most $K/2$.  Invoking local well-posedness of smooth solutions~\cite{Ma1984} and 
continuity-in-time, we then close the bootstrap argument, yielding global-in-time solutions bounded by $K/2$.
 
The global existence of solutions $(W,Z,A)$ in self-similar variables, together with the  stability of the $\bar W$, leads  to a precise description of the blow up 
of a certain directional derivative of $w$.  
For the dynamic modulations functions  mentioned above, the function $\tau(t)$ converges to the blow up time $T_*$, the vector $\xi(t)$ converges to the blow up
location $\xi_*$, and the normal vector $\Ncal(t, \cdot )$ converges to $\Ncal_*$ as $t \to T_*$.    Moreover, we will show that
\begin{align}
(\Ncal(t, \xi_2(t), \xi_3(t)) \cdot \nabla_{\! \xcal})w(\xi(t), t) = e^s \p_{y_1}W(0,s)  = -\tfrac{1}{\tau(t) -t}  \to - \infty  \qquad \text{ as } \qquad t\to T_*  \,.
\label{tall:people:can:t:drink}
\end{align}
Thus, it is only the directional derivative of $w$ in the $\Ncal$ direction that blows up as $t\to T_*$, while  the tangential directional derivatives
$(\Tcal^2(t, \xi_2(t), \xi_3(t)) \cdot \nabla_{\! \xcal})w(\xi(t), t) $ and $(\Tcal^3(t, \xi_2(t), \xi_3(t)) \cdot \nabla_{\! \xcal})w(\xi(t), t) $ remain uniformly bounded as $t \to T_*$. Additionally, we prove that the directional derivative  $\Ncal(t, \xi_2(t), \xi_3(t)) \cdot \nabla_{\! \xcal}$ of $z$ and $a$ remain uniformly bounded as $t\to T_*$. Thus, \eqref{tall:people:can:t:drink} shows that the wave profile steepens along the $\Ncal$ direction, leading to a single point shock at the space time location $(\xi_*,T_*)$.

\subsection{Modulation variables and the geometry of shock formation}
\label{sec:modulation}
The symmetries of the 3D Euler equations lead to dynamical instabilities in 
the space-time vicinity of the shock, which are amplified when considering self-similar variables~\cite{EgFo2009}. 
Our analysis relies crucially on the  size of this invariance group. We recall that the 3D Euler equations
are invariant under the 10  dimensional Lie group of  Galilean transformations consisting of  rotations,  translations, and rigid motions of spacetime,  as well as the
2 dimensional group of rescaling symmetries. Explicitly, given a time shift $\tcal_0 \in \RR$, a space shift $\xcal_0 \in \RR^3$, a velocity shift (Galilean boost) $v_0\in \RR^3$, a rotation matrix $R \in SO(3)$, a hyperbolic scaling parameter $\lambda \in \RR_+$, a temporal scaling parameter $\mu \in \RR_+$, and a solution $(u,\sigma)$ of the 3D compressible Euler system \eqref{eq:Euler2}, where as before $\sigma = (\sfrac{1}{\alpha}) \rho^\alpha$, the pair of functions
\begin{align*}
u_{\rm new}(\xcal,\tcal) &=   \frac{1}{\mu} R^T u\left( \frac{R(\xcal- \xcal_0 - \tcal v_0)}{\lambda}, \frac{\tcal-\tcal_0}{\lambda \mu} \right) +  v_0  \\
\sigma_{\rm new}(\xcal,\tcal) &= \frac{1}{\mu} \sigma\left(\frac{R(\xcal- \xcal_0- \tcal v_0)}{\lambda} , \frac{\tcal-\tcal_0}{\lambda \mu} \right)
\end{align*}
also solve the 3D Euler system \eqref{eq:Euler2}, and hence, these transformations define  the  12  dimensional group of symmetries of the 3D Euler equations.
For simplicity we sacrifice $5$ of these $12$  of these degrees of freedom: we fix a temporal rescaling  since we choose to prove that an initial slope of size (negative) 
$\sfrac 1\eps$ causes a blowup in time $\eps + \OO(\eps^2)$ (just as for the 1D Burgers equation);  
we discard the degree of freedom provided by hyperbolic scaling since it is not necessary for our analysis to fix the determinant of $\p_{y_1} \nabla_{\! y}^2 W$ to be 
constant in time; we also only utilize  two of the three degrees of freedom in the rotation matrix $R\in SO(3)$  since we choose a particular basis for the plane orthogonal 
to the shock direction; lastly, we discard two Galilean boosts as we do not need to modulate $A_\nu(0,s)$ to be constant  in time. This leaves us with a 7
dimensional group of symmetries which we use at the precise shock location. Additionally, since in self-similar coordinates  our blow up is modeled by the shear flow in the $x_1$ direction, using a quadratic-in-$\check x$ shift function, we are also able to modulate translational instabilities away from the shock in the directions orthogonal to the shock.

A fundamental aspect of our analysis is to show that  there is a correspondence between the instabilities of the Euler solution  and the symmetries discussed above.  Thus,  in order to develop a theory of stable shock  formation, it is of paramount importance to be able to {\it modulate} away these instabilities.  This idea was successfully used in~\cite{Merle96,MeRa05,MeRaSz2018}  in the context of the Schr\"odinger equation,  and in~\cite{MeZa97} for the nonlinear 
heat equation.  We also note here recent applications of modulated self-similar blowup techniques in fluid dynamics:~\cite{DaMa19,CoGhMa2018,CoGhIbMa18} for the Prandtl equations and \cite{El19,ChHo19,ElGhMa19} for the incompressible 3D Euler equation with axisymmetry.

In the aforementioned works, the role of the modulation variables is to enforce certain orthogonality conditions which prohibit the self-similar dynamics from evolving
toward the unstable directions of a suitably defined weighted energy space. Rather than enforcing orthogonality conditions, we shall instead 
employ a generalization of  the idea that we previously introduced in~\cite{BuShVi2019} in the setting of the 2D Euler equations with azimuthal symmetry, in
which the modulation functions are used to dynamically enforce {\em pointwise constraints} at precisely the blow up location for a Riemann-type function $W$.  For the
2D Euler equations with azimuthal symmetry, we required only three modulation functions to enforce constraints on $W$ and its first two derivatives.
In the 3D case considered herein, for which no symmetry assumptions are imposed, the 7 remaining invariances of 3D Euler correspond to 7  
modulation functions $\kappa,\tau \in \RR$, $\xi  \in \RR^3$, $\check n \in \RR^2$,
 whose role is to enforce 7 pointwise constraints 
for a 3D Riemann-type function $W(y,s)$ and its first-order and second-order partial derivatives at $y=0$.  We describe the one-to-one correspondence between symmetries and  pointwise constraints at $y=0$ as follows:
\begin{itemize}[itemsep=2pt,parsep=2pt,leftmargin=.15in]
\item The amplitude of the Riemann variable $W$  is modulated via the unknown 
$\kappa(\tcal)$  by a Galilean boost of the type $(\kappa(\tcal),0,0)$, whose role is to enforce the constraint $W(0,s) = 0$.
\item The time-shift invariance of the equations is modulated via    the unknown $\tau(\tcal)$, which allows us to precisely compute the time at which the shock occurs. 
This modulation function enforces the constraint $\p_1W(0,s) = -1$.
\item The invariance of the equations under the remaining two dimensional orthogonal rotation symmetry group is modulated via the modulation vector
$\check n(\tcal) = (n_2(\tcal),n_3(\tcal))$, allowing us to precisely compute the direction of the shock and its orthogonal plane. This modulation vector enforces the constraint 
$\check \nabla_{\! y} W(0,s) = 0$.
\item The space-shift invariance of the equations is modulated via  the vector $\xi(t)$, thereby allowing us to precisely compute the location of the shock. Dynamically, 
the modulation vector $\xi$ enforces the constraint $\p_1 \nabla W_{\! y}(0,s) = 0$.
\end{itemize} 
The remaining 3 modulation functions  $\phi_{22}(\tcal), \phi_{23}( \tcal), \phi_{33}(\tcal)\!\in\! \RR$ 
which correspond to $(x_2,x_3)$-dependent spatial shifts, are used to enforce the constraint  $\check \nabla_{\! y}^2 W(0,s) = 0$.
Geometrically, these 3 functions modulate the second fundamental form of the shock profile in the directions orthogonal to the shock
direction.

\subsection{Outline}

The remainder of the paper is structured as follows:
\begin{itemize}[itemsep=2pt,parsep=2pt,leftmargin=.15in]
\item In Section~\ref{sec:variables}, we describe the changes of variables which transform the Euler system from its original form \eqref{eq:Euler} to its modulated 
self-similar version in Riemann-type variables~\eqref{euler-ss}. Certain tedious aspects of this derivation are postponed to Appendix~\ref{sec:derivation}.  Herein,  we 
also introduce the self-similar Lagrangian flows used for the remainder of the paper, we define the self-similar blow up profile $\bar W$ and collect its principal properties, 
and we record the evolution equations for  higher-order derivatives of the $(W,Z,A)$ variables.

\item In Section~\ref{sec:results},  we state the assumptions on the initial datum in the original space-time variables and then state (in full detail) the main result of our paper, Theorem~\ref{thm:main}.   We emphasize that  the set of assumptions on the initial datum stated here is not the most general.  Instead, in Theorem~\ref{thm:open:set:IC},
we show that the set of allowable initial data can be taken from an open neighborhood in the $H^{18}$ topology near that datum described in Theorem~\ref{thm:main}. 
In this section,   we also state the self-similar version of our main result, Theorem~\ref{thm:SS}.

\item In Section~\ref{sec:bootstrap}, we state the pointwise self-similar bootstrap assumptions  which  imply Theorem~\ref{thm:SS}, as discussed above. Note that these 
bootstraps are strictly worse than the initial datum assumptions discussed in Section~\ref{sec:results}.  We also state a few  consequences of our bootstrap assumptions, chief among which is the global in time $\dot{H}^k$ energy estimate of Proposition~\ref{cor:L2}, whose proof is postponed to Section~\ref{sec:energy}.

\item In Section~\ref{sec:constraints}, we show how the dynamic constraints of $W, \nabla W$ and $\nabla^2 W$ at $(0,s)$ translate precisely into a system of 10 coupled 
nonlinear ODEs for the time-dependent modulation parameters $\kappa, \tau, n_\nu, \xi_i, \phi_{\nu \mu}$, given by polynomials and  rational functions with coefficients obtained from the derivatives of the functions $(W,Z,A)$ evaluated at $y=0$, cf.~\eqref{eq:dot:phi:dot:n} and \eqref{eq:dot:kappa:dot:tau}.

\item In Section~\ref{sec:dynamic:closure}, we improve  the bootstrap assumptions \eqref{eq:speed:bound} and \eqref{eq:acceleration:bound}  for our dynamic modulation variables. The analysis in this section crucially uses the explicit formulas derived earlier in Section~\ref{sec:constraints}. 

\item In Section~\ref{sec:preliminary}, we collect a number of technical estimates to be used later in the proof. These include bounds for the $y_1$ velocity components $(g_W, g_Z,g_U)$ defined in \eqref{eq:g:def}, the $y_\nu$ velocity components $(h_W, h_Z,h_U)$ given by \eqref{eq:h:def}, the $(W,Z,A)$ forcing terms from \eqref{eq:F:def}, and also  the forcing terms arising in the evolution of $\tilde W = W- \bar W$.

\item In Section~\ref{sec:Lagrangian}, we close the bootstrap on the spatial support of our solutions, cf.~\eqref{eq:support}. Additionally,  prove a number of  
{\em Lagrangian estimates} which are fundamental to our analysis in $L^\infty$ or weighted $L^\infty$ spaces for the $(W,Z,A)$ system. We single out 
Lemma~\ref{lem:escape} which proves that  trajectories of the (transport velocity of the) $W$ evolution, which start a small distance away from the origin, escape exponentially fast towards infinity.
Additionally,   Lemma~\ref{lem:phiZ}  proves that the flows of the transport velocities in the $Z$ and $U$ equations,  are swept towards infinity independently of
their starting point, and spend very little time near $y=0$.

\item In Section~\ref{sec:vorticity:sound},  we establish pointwise estimates on the self-similar specific vorticity $\mrg$ and the scaled sound speed $S$.
 The  bounds on $\mrg$  rely on the  structure of the equations satisfied by the geometric
components $\mrg \cdot \Ncal$, $\mrg \cdot \Tcal^2$, and $\mrg \cdot \Tcal^3$.

\item In Section~\ref{sec:Z:A}, we improve the bootstrap assumptions for 
$Z $ and $A$ stated in \eqref{eq:Z_bootstrap} and \eqref{eq:A_bootstrap}. The most delicate argument  required is for the bound of $\p_1 A$; we note in Lemma~\ref{lem:remarkable:sheep:structure} that this vector may be computed from the specific vorticity vector, the sound speed, and quantities which were already 
bounded   in view of our bootstrap assumptions. 

\item In Section~\ref{sec:W},  we improve on the bootstrap assumptions for $W$ and $\tilde W$, cf.~\eqref{eq:W_decay} and \eqref{eq:bootstrap:Wtilde}--\eqref{eq:bootstrap:Wtilde3:at:0}.
This analysis takes advantage of the forcing estimates established in Section~\ref{sec:preliminary} and the Lagrangian trajectory estimates of Section~\ref{sec:Lagrangian}.

\item In Section~\ref{sec:energy},  we give the proof of the $\dot{H}^k$ energy estimate stated earlier in 
Proposition~\ref{cor:L2}. As opposed to the analysis  which precedes this section and which relied on pointwise estimates for the $(W,Z,A)$ system, for the energetic 
arguments presented here,  it is convenient to work directly with the self-similar velocity variable $U$ and the scaled sound speed $S$, whose evolution is given 
by \eqref{US-short} and whose derivatives satisfy~\eqref{US-L2}. It is here that the good energy structure of the Euler system is fundamental. In our proof,
we use a weighted Sobolev norm to account for binomial coefficients, and appeal to some interpolation inequalities collected in Appendix~\ref{sec:interpolation}.

\item In Section~\ref{sec:conclusion}, we use the above established bootstrap estimates to conclude the proofs of Theorem~\ref{thm:SS}, and as a consequence of
Theorem~\ref{thm:main}. Herein,  we provide the definition of the blow up time and location, establish the H\"older $\sfrac 13$ regularity of the solution at the first 
singular time, and show that the vorticity is nontrivial at the shock.   Moreover, we establish convergence to an  asymptotic profile,  proving that $\lim_{s \to \infty } W(y,s) = 
\bar W_ \mathcal{A} (y)$ for all fixed $y$, where $\bar W_ \mathcal{A}$ denotes a stable stationary solution of the self-similar 3D Burgers equation.   The 
ten-dimensional  family of such solutions, parameterized by a symmetric $3$-tensor $ \mathcal{A} $,  is constructed in Proposition \ref{prop-stationary-burgers} 
of  Appendix \ref{sec:delay}.
Additionally,  we give a detailed proof of the statement that the set of initial conditions for which Theorem~\ref{thm:main} holds contains an open neighborhood in the $H^{18}$ topology, as claimed in Theorem~\ref{thm:open:set:IC}.
\end{itemize}

\section{Self-similar shock formation} 
\label{sec:variables}
Prior to stating the main theorem (cf.~Theorem~\ref{thm:main} below), we describe how starting from the 3D Euler equations \eqref{eq:Euler} for the unknowns $(u,\rho)$, 
which are functions of the spatial variable $\xcal \in \RR^3$ and of the time variable $\tcal \in I \subset \RR$, we  arrive at the equations for the modulated self-similar 
Riemann variables $(W,Z, A_\nu)$, which are functions of  $y \in \RR^3$ and $s\in [-\log \eps,\infty)$. This change of variables is performed in the following three 
subsections, with some of the computational details provided in  Appendix~\ref{sec:derivation}.

\subsection{A time-dependent  coordinate system} 
\label{sec:time:dependent:coord}
In this section we switch coordinates, from the original space variable $\xcal$ to a new space variable $\tilde x$, which is obtained from a rigid body rotation and 
a translation.   It is  convenient for our subsequent analysis to perform and $\alpha$-dependent rescaling of time, by letting 
\begin{align}
\tcal \mapsto  \tfrac{1+ \alpha }{2} \tcal = t\, .
\label{eq:time:rescale}
\end{align}
Throughout the rest of the paper we abuse notation and denote the  time variable defined in \eqref{eq:time:rescale} still by $t$.

In order to align our coordinate system with the orientation of the developing shock, we introduce a time dependent unit normal vector\footnote{Frequently we will use the notation $\check n$ to denote the last two coordinates of a vector $n=(n_1,n_2,n_3)$, i.e. $\check n = (n_2,n_3)$.}
\[
n =n(t) = (n_1(t), n_2(t),n_3(t)) = (n_1(t),\check n(t)),
\] 
with $\abs{\check n}^2 = \abs{n_2}^2+ \abs{n_3}^2 \ll 1$,  so that $n_1 = \sqrt{1 -n_2^2 - n_3^2} = \sqrt{1 - |\check n|^2}$ is close to $1$. Associated with these parameters we introduce the skew-symmetric matrix $\tilde R$ whose first row is the vector $(0,-n_2,-n_3)$, first column is $(0,n_2,n_3)$, and has $0$ entries otherwise. In terms of $\tilde R$ we define the rotation matrix
\begin{align}
R &=   R(t)  
= \Id + \tilde R(t) + \frac{1 - e_1 \cdot n(t) }{\abs{e_1 \times n(t)}^2} \tilde R^2(t)  
\label{eq:R:def}
\end{align}
whose purpose is to rotate the unit vector $e_1$ onto the vector $n(t)$. Since $R \in SO(3)$, we have that
the vectors $\{R(t) e_1, R(t) e_2, R(t)  e_3\}$ form a time dependent orthonormal basis for $\RR^3$, and for convenience we sometimes write  $\te_i  = R e_i$ for $i\in\{1,2,3\}$. Geometrically, the vectors $\{\te_2,\te_3\}$ span the plane orthogonal to the shock direction $n$, and we will for ease of notation denote $n = \te_1$.

It is convenient at this stage to record the formula for the time derivative of $R(t)$. One may verify that 
\begin{align}
 \dot{R}(t) = \dot{n}_2(t) R^{(2)}(t) + \dot{n}_3(t) R^{(3)}(t)
 \label{eq:R:dot:def}
\end{align}
where the matrices $R^{(2)}$ and $R^{(3)}$ are defined explicitly in \eqref{eq:R:2:def} and \eqref{eq:R:3:def}. For compactness of notation it is convenient to define the {\em skew-symmetric matrix} $\dot Q = \dot R^T R$, written out in components as
\begin{align}
\dot{Q}_{ij} 
= \dot{R}_{ki} R_{k j} = \dot{n}_2  R^{(2)}_{ki} R_{kj} + \dot{n}_3 R^{(3)}_{ki} R_{kj} = \dot{n}_2 Q^{(2)}_{ij} + \dot{n}_3 Q^{(3)}_{ij}
\label{eq:Q:def}
\end{align}
where the skew-symmetric matrices $Q^{(2)}$ and $Q^{(3)}$ are stated explicitly in \eqref{eq:Q2:def} and \eqref{eq:Q3:def}, respectively.

In addition to the vector $\check n(t)$, which determines the rotation matrix $R(t)$, we also define a time dependent shift vector  
\begin{align*}
\xi = \xi(t) = (\xi_1(t),\xi_2(t),\xi_3(t)) = (\xi_1(t), \check \xi(t)) \,.
\end{align*}
The point $\xi(t) \in \RR^3$ dynamically tracks the location of the developing shock.

In terms of $R(t)$ and $\xi(t)$ we introduce the new position variable
\begin{align}
 \tilde x = R^T(t) (\xcal-\xi(t))
 \label{eq:tilde:x:def}
\end{align}
and the rotated velocity and rescaled sound speed as
\begin{align}
\tu ( \tilde x ,t) =   R^T(t)   u(\xcal ,t) \, ,
\qquad 
\tilde \sigma ( \tilde x ,t) =    \sigma(\xcal,t) \, .
 \label{eq:tilde:u:def}
\end{align}
From \eqref{eq:tilde:x:def} and \eqref{eq:tilde:u:def}, after a short computation  detailed in Appendix~\ref{app:time:dependent:coord} below, we obtain that 
the Euler equations \eqref{eq:Euler-coordinates} are written as
\begin{subequations}
\label{eq:new:Euler}
\begin{align}
 \tfrac{1+ \alpha }{2} \p_t \tilde u - \dot Q\tilde u + \Big( ( \tilde v  + \tilde u)\cdot \nabla_{\! \tilde x} \Big) \tilde u 
 + \alpha \tilde \sigma   \nabla_{\! \tilde x}  \tilde \sigma  &=0\\
  \tfrac{1+ \alpha }{2}\partial_t \tilde \sigma + \Big(( \tilde v  + \tilde u) \cdot \nabla_{\! \tilde x}\Big)\tilde \sigma +  \alpha   \tilde \sigma  \div_{\! \tilde x} \tilde u&=0
\end{align}
\end{subequations}
where 
$$
\tilde v(\tilde x,t) := \dot Q  \tilde x  - R^T \dot \xi \,,
$$
the matrix $\dot Q$ is given by \eqref{eq:Q:def}, and the matrix $R(t)$ and vector $\xi(t)$ are yet to be determined.

Similarly, defining the rotated {\em specific vorticity} vector  $\tilde \zeta$ by
\begin{equation}\label{vort00}
\tilde \zeta ( \tilde x ,t) =   R^T(t)   \zeta(\xcal ,t) 
 \,,
\end{equation} 
we have that $\tilde \zeta$ is a solution of  
\begin{align}
 \tfrac{1+ \alpha }{2} \p_t \tilde \zeta - \dot Q \tilde \zeta + \Big( ( \tilde v + \tilde u)\cdot \nabla_{\! \tilde x}\Big)\tilde \zeta - \Big(\tilde \zeta \cdot \nabla_{\! \tilde x}\Big) \tilde u &= 0 \,. \label{tgvorticity}
\end{align}
Deriving \eqref{tgvorticity} from \eqref{svorticity} fundamentally uses that  $\dot Q$ is skew-symmetric.

\begin{remark}[Notation]
It will be convenient to denote the last two components of a three-component vector $v$ simply as $\check v$. For instance, the 
gradient operator may be written as $\nabla = (\p_1, \p_2, \p_3) = (\p_1, \check \nabla)$ and the velocity vector as $\tu = (\tu_1, \tu_2, \tu_3) = (\tu_1, \check \tu)$.
 Moreover, for a $3\times 3$ matrix $R$, we will denote by $\check R$ the matrix whose first column is set to $0$.    We will also use the Einstein summation convention, in
which repeated Latin indices are summed from $1$ to $3$, and repeated Greek indices are summed from $2$ to $3$.   We shall denote
a partial derivative $\p_{\tilde x_j} F$ by $F,_j$  and $\p_{\tilde x_\nu}$  will be denoted simply 
by $F,_\nu$. We note that the $\cdot,_j$ derivative notation shall always denote a derivative with respect to $\tilde x$. 
\end{remark}

\subsection{Coordinates adapted to the shock}
\label{sec:Riemann:sheep}
We shall next introduce one further coordinate transformation 
that will allow us to modulate $\check {\tilde x}$-dependent shifts, and 
simultaneously parameterize the steepening shock front by a quadratic profile.
Specifically, coordinates $\tilde x$ will be transformed to new coordinates $x$,  so that with respect to $x$, the local parabolic geometry near the 
steepening shock is flattened. The new coordinate satisfies  $\check x = \check{\tilde x}$.  

In order to understand the geometry of the shock, we define a time-dependent parameterized surface over the $ \tilde x_2$-$\tilde x_3$ plane by
\begin{equation}\label{surface}
\left( f(\tilde x_2, \tilde x_3, t), \tilde x_2, \tilde x_3\right) 
\end{equation} 
where the function $f \colon \RR^2 \times [-\tfrac{\eps }{2} ,T_*) \to \RR^2$ is a spatially  quadratic {\em modulation function} defined as
\begin{align}
 f(\check{\tilde x},t) = \tfrac{1}{2} \phi_{\nu\gamma}(t) \tilde x_\nu\tilde x_\gamma  \,. \label{def_f}
\end{align}
The coefficients $\phi_{\nu\gamma}(\tcal)$ are symmetric with respect to the indices $\nu$ and $\gamma$, and their time evolution plays a crucial role in our proof.  
A derivative with respect to $\tcal$ is denoted as  as
\begin{align}
 \dot f(\check{\tilde x},t) =\tfrac{1}{2} \dot\phi_{\nu\gamma}(t) \tilde x_\nu\tilde x_\gamma  \,. \label{eq:dt:f}
\end{align}

Associated to the parameterized surface \eqref{surface}, we define the unit-length tangent  vectors 
\begin{align} 
\Tcal^2 =  \left( \tfrac{f,_2}{\Jcal}, 1 - \tfrac{(f,_2)^2}{\Jcal(\Jcal+1)} \,,  \tfrac{- f,_2f,_3}{\Jcal(\Jcal+1)}\right) \, , \qquad 
\Tcal^3 =   \left( \tfrac{f,_3}{\Jcal},  \tfrac{- f,_2f,_3}{\Jcal(\Jcal+1)} \,,  1 - \tfrac{(f,_3)^2}{\Jcal(\Jcal+1)} \right) \,,  
\label{tangent}
\end{align} 
and the unit-length normal vector
\begin{align} 
\Ncal = \Jcal^{-1} (1, -f,_2, -f,_3) \, , \label{normal}
\end{align} 
where
$$
\Jcal = ( 1+  |f,_2|^2+ |f,_3|^2)^ {\frac{1}{2}}  \,.
$$
It is easy to verify that $(\nn, \tt^2,\tt^3)$ form an orthonormal basis and that $\Ncal \times \Tcal^2= \Tcal^3$ and $\Ncal \times \Tcal^3= -\Tcal^2$.   With respect to the
parameterized quadratic surface $(f(\check{\tilde x}), \check{\tilde x})$, the second fundamental form is given by the 2-tensor $\Jcal ^{-1} \phi_{\nu \gamma}(t)$, and hence
the modulation functions  $\phi_{\nu \gamma}(t)$ are dynamically measuring the curvature of the steepening shock front.

Using the function $f(\tilde x_2,\tilde x_3,t)$ we now introduce a   new transformation that we call
the {\it sheep shear transform}.    The new space coordinate $x$ is defined as
\begin{equation}\label{x-sheep}
x_1 = \tilde x_1 - f(\tilde x_2,\tilde x_3,t)\,, \qquad x_2 = \tilde x_2\,, \qquad x_3 = \tilde x_3 \,,
\end{equation}
so that the surface defined in \eqref{surface} is now flattened. Note that we are only modifying the $\tilde x_1$ coordinate, and since $\Ncal, \Jcal, \Tcal$ are independent of $\tilde x_1$, these functions are not affected by the sheep shear transform. We write $f(\check x,t)$ instead of $f(\check{\tilde x},t)$ and the similar notation overload is used for $\Ncal, \Jcal$, and $\Tcal$.

In terms of this new space variable $x$, the velocity field and the rescaled sound speed are redefined as
\begin{subequations} 
\label{usigma-sheep}
\begin{align} 
\mru(x ,t) & =\tilde u (\tilde x, t) = \tilde u (x_1+ f(x_2,x_3,t),x_2,x_3,t)\,, \\
\mrs(x ,t) & =\tilde \sigma (\tilde x, t)  = \tilde \sigma  (x_1+ f(x_2,x_3,t),x_2,x_3,t)\,. \label{sigma-sheep}
\end{align} 
\end{subequations} 
Before stating the equations obeyed by $\mru$ and $\mrs$, which involve many $\alpha$-dependent parameters, for the sake of brevity, we introduce the notation
\begin{align} 
\beta_1 = \beta_1(\alpha)= \tfrac{1}{1+\alpha}, \qquad \beta_2 = \beta_2(\alpha)=\tfrac{1-\alpha}{1+\alpha}, \qquad \beta_3 = \beta_3(\alpha)= \tfrac{\alpha}{1+\alpha}\,,
\label{eq:various:beta}
\end{align} 
where  
$\beta_i = \beta_i(\alpha)$ are fixed parameters of our problem. Note that for $\alpha>0$ (i.e. $\gamma>1$) we have $0\leq \beta_1,\beta_2,\beta_3 < 1$.

With the notation introduced in \eqref{usigma-sheep} and \eqref{eq:time:rescale},  the system \eqref{eq:new:Euler} may be written as
\begin{subequations}
\label{Euler_sheep}
\begin{align}
& \p_t \mru  - 2\beta_1\dot Q \mru + 2\beta_1  (- \tfrac{\dot f}{2\beta_1} +\Jcal  v \cdot \nn + \Jcal \mru \cdot \nn  ) \p_1\mru  
 +2\beta_1(v_\nu + \mru_\nu)\p_\nu \mru   
+ 2\beta_3 \mrs( \Jcal \Ncal   \p_1 \mrs +   \delta^{\cdot \nu}    \p_\nu \mrs )=0 \,, \\
&
 \partial_t \mrs + 2 \beta_1 ( - \tfrac{\dot f}{2\beta_1} + \Jcal v \cdot \nn + \Jcal \mru \cdot \nn  )\p_1\mrs   + 2\beta_1(v_\nu + \mru_\nu)\p_\nu \mrs
 + 2\beta_3 \mrs\left( \p_1 \mru \cdot \nn \Jcal + \p_\nu \mru_\nu \right) =0 \,,
\end{align}
\end{subequations}
where in analogy to \eqref{usigma-sheep} we have denoted
\begin{align}
v(x,t) =  \tilde v (\tilde x, t ) = \tilde v (x_1+ f(x_2,x_3,t),x_2,x_3,t)\,.
\label{eq:v:x:t:def}
\end{align}
In particular, note that $v_i(x,t) = \dot Q_{i1}(x_1 + f(\check x,t)) + \dot Q_{i\nu} x_\nu - R_{ji} \dot \xi_j$. 
Similarly, we define the sheared version of the rotated specific vorticity vector by
\begin{equation}\label{sv-sheep}
\mrg(x,t) = \tilde \zeta(\tilde x,t) = \tilde \zeta (x_1+ f(x_2,x_3,t),x_2,x_3,t)\,, 
\end{equation} 
so that the equation \eqref{tgvorticity} becomes
\begin{align} 
&\p_t \mrg - 2\beta_1\dot Q \mrg 
+ 2 \beta_1 ( - \tfrac{\dot f}{2\beta_1} + \Jcal  v \cdot \Ncal + \Jcal \mru \cdot \Ncal) \p_1 \mrg + 2 \beta_1 ( v_\nu + \mru_\nu) \p_\nu \mrg  -2\beta_1 \Jcal \Ncal \cdot \mrg \p_1 \mru  - 2 \beta_1 \mrg_\nu \p_\nu \mru  =0 \,.
\label{svorticity_sheep}
\end{align} 

\subsection{Riemann variables adapted to the shock geometry}
The Euler system \eqref{Euler_sheep} has a surprising geometric structure which is discovered by introducing Riemann-type variables. For this purpose, we switch from the unknowns $(\mru,\mrs)$ to the Riemann variables $(w,z,a)$ defined by
\begin{align} 
w = \mru \cdot \Ncal + \mrs  \,, \qquad  z =  \mru   \cdot \Ncal - \mrs   \,, \qquad  a_\nu = \mru \cdot \Tcal^\nu
\label{tildeu-dot-T}
\end{align} 
so that
\begin{align} 
\mru \cdot \Ncal  = \tfrac{1}{2} ( w+ z) \,, \qquad  \mrs = \tfrac{1}{2} (  w-   z) \,.  \label{tildeu-dot-N}
\end{align} 
The Euler sytem \eqref{Euler_sheep} can be written in terms of the new variables $(  w,   z,  a_2,   a_3)$ as
\begin{subequations}
\label{eq:Euler-riemann-sheep}
\begin{align}
&\p_t  w +  \left(2 \beta_1 (-\tfrac{\dot f}{2\beta_1}+\Jcal v \cdot \nn) + \Jcal w + \beta_2 \Jcal  z  \right) \p_1  w
+ \left(2 \beta_1 v_\mu + w \Ncal _\mu - \beta_2 z \Ncal _\mu  + 2\beta_1 a_\nu \Tcal^\nu_\mu\right) \p_\mu w \notag \\
&
\quad = - 2\beta_3 \mrs \Tcal^\nu_\mu \p_\mu a_\nu +  2 \beta_1 a_\nu \Tcal^\nu_i   \dot{\Ncal_i} + 2 \beta_1  \dot Q_{ij}  a_\nu \Tcal^\nu_j \Ncal _i  
+2 \beta_1  \left( v_\mu + \mru\cdot \Ncal  \Ncal _\mu +  a_\nu \Tcal^\nu_\mu\right)  a_\gamma \Tcal^\gamma_i 
\Ncal_{i,\mu}  \notag\\
&\quad \qquad 
- 2\beta_3 \mrs  (a_\nu \Tcal^\nu_{\mu,\mu}
+  \mru\cdot \Ncal  \Ncal_{\mu,\mu} )  \,, \label{alrightalrightalright}\\
&\p_t  z +  \left(2\beta_1 (-\tfrac{\dot f}{2\beta_1} + \Jcal v \cdot \nn)   +\beta_2 \Jcal  w + \Jcal z \right) \p_1  z
+ \left(2 \beta_1  v_\mu + \beta_2  w  \Ncal_\mu
+z \Ncal _\mu + 2 \beta_1  a_\nu \Tcal^\nu_\mu\right) \p_\mu z \notag \\
&
\quad = 2\beta_3 \mrs \Tcal^\nu_\mu \p_\mu a_\nu
+2 \beta_1   a_\nu \Tcal^\nu_i   \dot{\Ncal_i} +2 \beta_1  \dot Q_{ij}  a_\nu \Tcal^\nu_j \Ncal _i  
+2 \beta_1  \left( v_\mu +  \mru\cdot \Ncal  \Ncal _\mu +  a_\nu \Tcal^\nu_\mu\right)  a_\gamma \Tcal^\gamma_i \Ncal_{i,\mu}  \notag\\
&\quad  \qquad 
+ 2\beta_3 \mrs  (a_\nu \Tcal^\nu_{\mu,\mu}
+  \mru\cdot \Ncal  \Ncal_{\mu,\mu} )\,, 
\\
&\p_t  a_\nu +  \left(2 \beta_1(-\tfrac{\dot f}{2\beta_1}+\Jcal v\cdot \nn) +\beta_1 \Jcal w + \beta_1 \Jcal z   \right) \p_1  a_\nu 
+ 2 \beta_1  \left(v_\mu + {\tfrac{1}{2}} ( w+ z) \Ncal _\mu + a_\gamma \Tcal^\gamma_\mu \right)  \p_\mu a_\nu \notag\\
&\quad  = -2\beta_3\mrs \Tcal^\nu_\mu  \p_\mu \mrs
+   2\beta_1  \left(   \mru\cdot \Ncal  \Ncal_i +  a_\gamma \Tcal^\gamma_i\right) \dot\Tcal^\nu_i  
 +   2\beta_1  \dot Q_{ij} \left( (\mru\cdot \Ncal  \Ncal _j  +  a_\gamma  \Tcal^\gamma_j \right) \Tcal^\nu _i   \notag \\
 &\quad  \qquad 
  +  \beta_1 \left( v_\mu + \mru\cdot \Ncal \Ncal _\mu 
 +  2a_\gamma \Tcal^\gamma_\mu\right)  \left(  \mru\cdot \Ncal \Ncal_i +  a_\gamma \Tcal^\gamma_i\right) \Tcal^\nu_{i,\mu}\,.
 \end{align}
\end{subequations}
At this stage we comment on the temporal transformation \eqref{eq:time:rescale}: its purpose is to ensure that the coefficient of $w \p_1 w$ in \eqref{alrightalrightalright},
 when evaluated at $\check x=0$, is equal to $1$, in analogy to the 1D Burgers equation.

\subsection{Modulated self-similar variables}
\label{sec:s:s:variables}

In order to study the formation of shocks in the Riemann-form of the Euler equations~\eqref{eq:Euler-riemann-sheep}, we introduce the following (modulated) 
self-similar variables:
\begin{subequations}
\label{eq:y:s:def}
\begin{align}
 s &= s(t)=-\log(\tau(t)-t) \, ,\\
  y_1 &=   y_1(  x_1,t)=\frac{  x_1}{(\tau(t)-t)^{\frac32}} =   x_1 e^{\frac{3s}{2}} \,, \\   
  y_j &=   y_j(  x_j ,t)=\frac{  x_j}{(\tau(t)-t)^{\frac12}} =   x_j e^{\frac s2}\,, \qquad \mbox{for} \qquad j\in \{2,3\}  
 \,.
\end{align}
\end{subequations}
Note the different scaling of the first component $  y_1$ versus the vector of the second and third components $\check {  y}$. 
We have the following useful identities: 
\begin{align*}
\tau-t =e^{-s}\,, \ \tfrac{ds}{dt}= (1-\dot \tau)e^s\,, \ 
\partial_{  x_1}   y_1 = e^{\frac32s}\,, \  \partial_t y_1 =  \tfrac{3(1-\dot\tau)}{2}   y_1 e^s\,,  \ 
\partial_{  x_\gamma}   y_\nu = e^{\frac s2} \delta_{\gamma \nu}\, \  \partial_t   y_\nu  = \tfrac{1-\dot\tau}{2}  y_\nu e^s \,.
\end{align*}

\subsection{Euler equations in modulated self-similar variables}
\label{sec:s:s:equations}
Using the self-similar variables $y$ and $s$, we rewrite the functions $w$, $z$ and $a_\nu$ defined in \eqref{tildeu-dot-T} as 
\begin{subequations}
\label{eq:ss:ansatz}
\begin{align}
w(  x,t)&=e^{-\frac s2}  W(  y,s)+\kappa (t) \,, \label{w_ansatz}  \\
z(  x,t)&=  Z(  y,s) \,,  \label{z_ansatz}   \\
a_\nu(  x ,t)&= A_\nu (  y,s) \, , \label{a_ansatz}
\end{align}
\end{subequations}
where $\kappa(t)$ is a modulation function whose dynamics shall be given below.
We also change the function $v$ defined in \eqref{eq:v:x:t:def} to self-similar coordinates by letting $v(x,t) = V(y,s)$, so that 
\begin{equation}\label{def_V}
V_i(y,s) =  \dot Q_{i1} \left( e^{-\frac{3s}{2}}y_1 +  {\tfrac 12 e^{-s} \phi_{\nu\mu} y_\nu y_\mu} \right)+e^{-\frac s2} \dot Q_{i\nu}y_\nu  - R_{ji} \dot \xi_j \, .
\end{equation} 
 Next, we derive the system of equations obeyed by $W,  Z$, and $A$. We introduce the notation
 \[\beta_\tau = \beta_\tau(t) = \tfrac{1}{1-\dot \tau(t)}\,.\] 
With the self-similar change of coordinates \eqref{eq:y:s:def}--\eqref{eq:ss:ansatz}, the Euler system~\eqref{eq:Euler-riemann-sheep} becomes
\begin{subequations} 
\label{euler-ss}
\begin{align} 
( \p_s- \tfrac{1}{2} )  W 
+  \left( g_{ W}  + \tfrac{3}{2}   y_1  \right)  \p_{1}  W
+  \left(  h_{W}^\mu + \tfrac{1}{2} y_\mu \right) \p_{\mu} W  
&=   {F}_{ W} - e^{-\frac s2} \beta_\tau \dot \kappa  \label{eq:euler:ss:a} \\
\p_s  Z + \left( g_{ Z}+\tfrac{3}{2}  y_1   \right) \p_{1}  Z 
+  \left(  h_{Z}^\mu + \tfrac{1}{2} y_\mu \right) \p_{\mu} Z  &= {F} _{ Z} \label{eq:euler:ss:b}\\
\p_s  A_\nu + \left( g_{U}+ \tfrac{3}{2} y_1 \right) \p_{1}   A_\nu
+  \left(  h_{U}^\mu + \tfrac{1}{2} y_\mu \right) \p_{\mu} A_\nu &= {F}_{A \nu}  \label{eq:euler:ss:c}
\end{align} 
\end{subequations} 
where we have introduced the notation 
\begin{subequations}
\label{eq:g:def}
\begin{align}
g_{ W}&=  \beta_\tau \Jcal W + \beta_\tau e^{\frac s2} \left( -  \dot f +  \Jcal \left(\kappa  + \beta_2  Z + 2\beta_1    V \cdot \nn \right)  \right) = \beta_\tau  \Jcal W +  {G}_{ W} \label{eq:gW}  \\
g_{ Z}&=   \beta_2 \beta_\tau \Jcal  W + \beta_\tau e^{\frac s2} \left(  -  \dot f +   \Jcal \left(\beta_2 \kappa  +   Z  + 2\beta_1  V \cdot \nn \right)       \right)
=\beta_2 \beta_\tau \Jcal W + G_{ Z} \label{eq:gZ}\\
g_{U}&=  \beta_1 \beta_\tau \Jcal W + \beta_\tau e^{\frac s2} \left( -  \dot f +     \Jcal   \left(\beta_1 \kappa  +  \beta_1 Z + 2  \beta_1 V \cdot \nn  \right)   \right)      
= \beta_1 \beta_\tau \Jcal  W  + G_{U}  \label{eq:gA}
\end{align}
\end{subequations}
for the terms in the $y_1$ transport terms,
\begin{subequations}
\label{eq:h:def}
\begin{align}
 h_W^\mu&= 
  \beta_\tau e^{-s} \Ncal_\mu W + \beta_\tau  e^{-\frac s2}  \left( 2    \beta_1  V_\mu + \Ncal_\mu  \kappa  - \beta_2 \Ncal_\mu Z +2 \beta_1 A_\gamma \Tcal^\gamma_\mu \right)  
\,  \label{eq:hW} \\
 h_Z^\mu&= 
 \beta_\tau \beta_2 e^{-s} \Ncal_\mu W + \beta_\tau  e^{-\frac s2}  \left(2    \beta_1  V_\mu + \beta_2 \Ncal_\mu  \kappa + \Ncal_\mu Z + 2 \beta_1 A_\gamma \Tcal^\gamma_\mu \right)  
\,  \label{eq:hZ} \\
 h_{U}^\mu&=   \beta_\tau \beta_1 e^{-s} \Ncal_\mu W +
\beta_\tau  e^{-\frac s2}  \left(2    \beta_1  V_\mu + \beta_1 \Ncal_\mu  \kappa + \beta_1 \Ncal_\mu Z + 2 \beta_1 A_\gamma \Tcal^\gamma_\mu \right)  
\,  \label{eq:hA} 
\end{align}
\end{subequations}
for the terms in the $\check y$ transport terms,
and the forcing terms are written as
\begin{subequations}
\label{eq:F:def}
\begin{align}
 {F}_{ W}&= - 2 \beta_3 \beta_\tau S \Tcal^\nu_\mu \partial_{\mu} A_\nu
 +  2 \beta_1 \beta_\tau  e^{-\frac s2}   A_\nu \Tcal^\nu_i \dot{\Ncal}_i
 +  2 \beta_1 \beta_\tau  e^{-\frac s2} \dot Q_{ij} A_\nu \Tcal^\nu_j \Ncal_i \notag \\
 & \quad
 + 2 \beta_1\beta_\tau e^{-\frac s2}\left(V_\mu +\Ncal_\mu  U \cdot \nn +  A_\nu \Tcal^\nu_\mu \right) A_\gamma \Tcal^\gamma_i \Ncal_{i,\mu}
 - 2 \beta_3 \beta_\tau e^{-\frac s2} S \left( A_\nu \Tcal^\nu_{\mu,\mu} + U\cdot \Ncal \Ncal_{\mu,\mu} \right)
 \label{eq:FW:def}\\
 {F}_{ Z}&=  2   \beta_3 \beta_\tau e^{-\frac s2} S \Tcal^\nu_\mu \partial_{\mu} A_\nu
  +  2 \beta_1 \beta_\tau  e^{-s}   A_\nu \Tcal^\nu_i \dot{\Ncal}_i 
 +  2 \beta_1 \beta_\tau  e^{-s} \dot Q_{ij} A_\nu \Tcal^\nu_j \Ncal_i \notag \\
 & \quad
 + 2\beta_1 \beta_\tau e^{-s}\left(V_\mu +\Ncal_\mu  U \cdot \nn +  A_\nu \Tcal^\nu_\mu \right) A_\gamma \Tcal^\gamma_i \Ncal_{i,\mu} + 2 \beta_3 \beta_\tau e^{-s} S \left( A_\nu \Tcal^\nu_{\mu,\mu} + U\cdot \Ncal \Ncal_{\mu,\mu} \right)  
 \label{eq:FZ:def}\\
 {F}_{ A\nu}&= - 2  \beta_3 \beta_\tau e^{-\frac s2}  S T^\nu_\mu \p_{\mu} S
 + 2 \beta_1 \beta_\tau e^{-s} \left( U\cdot\Ncal \Ncal_i + A_\gamma \Tcal^\gamma_i\right) \dot{\Tcal}^\nu_i 
 + 2\beta_1 \beta_\tau e^{-s} \dot{Q}_{ij} (  U\cdot \Ncal \Ncal_j + A_\gamma \Tcal^\gamma_j ) \Tcal^\nu_i \notag\\
 &\quad 
 + 2\beta_1 \beta_\tau e^{-s} \left(V_\mu + U\cdot \Ncal \Ncal_\mu + A_\gamma \Tcal^\gamma_\mu\right) \left(U\cdot \Ncal \Ncal_i + A_\gamma \Tcal^\gamma_i \right) \Tcal^\nu_{i,\mu}
 \label{eq:A:def}\,.
\end{align}
\end{subequations}
Here and throughout the paper we are using the notation $\varphi_{,\mu}=\p_{x_\mu}\varphi$, and $\p_\mu \varphi = \p_{y_\mu} \varphi$.

In \eqref{eq:F:def} we have also used the self-similar variants of $\mru$ and $\mrs$ defined by
\begin{subequations} 
\label{trannies}
\begin{align}
\mru( x,t)&= U(  y,s) \,, \label{U-trammy}   \\
\mrs( x,t)&=  S(  y,s) \,, \label{S-trammy}
\end{align}
\end{subequations} 
so that 
\begin{align}
U \cdot \Ncal = \tfrac 12 \left( \kappa + e^{-\frac s2} W + Z\right) \qquad \mbox{and}\qquad  S = \tfrac 12 \left( \kappa + e^{-\frac s2} W - Z\right) \,.
\label{eq:UdotN:S}
\end{align}
From \eqref{Euler_sheep}, \eqref{eq:y:s:def}, \eqref{U-trammy}, \eqref{S-trammy} we deduce that $(U,S)$ are solutions of
\begin{subequations} 
\label{US-euler-ss}
\begin{align} 
 &\p_s U_i - 2 \beta_1 \beta_\tau e^{-s} \dot Q_{ij} U_j + (g_{U} + \tfrac{3}{2} y_1) \p_{y_1} U_i + (h^\nu_A + \tfrac{1}{2} y_\nu) \p_{\nu} U_i \notag \\
 & \qquad\qquad\qquad\qquad\qquad\qquad
+ 2\beta_\tau \beta_3 \Jcal \Ncal_i e^{\frac s2}S  \p_{1} S + 2 \beta_\tau\beta_3 \delta^{i\nu} e^{-\frac s2} S \p_{\nu} S=0 \, ,
 \\
 &\p_s S + (g_{U} + \tfrac{3}{2} y_1) \p_{1} S + (h^\nu_A + \tfrac{1}{2} y_\nu) \p_{\nu} S 
 +2\beta_\tau\beta_3 e^{\frac s2}  S \p_{1}U \cdot \nn \Jcal + 2\beta_\tau\beta_3 e^{-\frac s2} S \p_{\nu} U_\nu =0 \,.
\end{align} 
\end{subequations} 

Finally, we defined the self-similar variant of the specific vorticity via  
\begin{align} 
\mrg( x,t)&= \Omega(  y,s) \,.  \label{svort-trammy}
\end{align}

\subsection{Transport velocities, vorticity components, and Lagrangian flows}
Upon writing the 3D transport velocities in \eqref{euler-ss} as the vector fields
\begin{subequations} 
\label{eq:transport_velocities}
\begin{align} 
\mathcal{V} _W &= \left(  g_W  + \tfrac{3}{2} y_1 \,,     h_W^2 + \tfrac{1}{2} y_2\,,     h_W^3 + \tfrac{1}{2} y_3 \right) \,, \\
\mathcal{V} _Z & = \left(  g_Z  + \tfrac{3}{2} y_1  \,,     h_Z^2 + \tfrac{1}{2} y_2\,,     h_Z^3 + \tfrac{1}{2} y_3 \right) \,, \\
\mathcal{V} _U & = \left(  g_U  + \tfrac{3}{2} y_1\,,     h_U^2 + \tfrac{1}{2} y_2\,,     h_U^3 + \tfrac{1}{2} y_3 \right) \label{V_U} \, ,
\end{align} 
\end{subequations} 
the system \eqref{euler-ss} may be written as
\begin{subequations} 
\begin{align*} 
 \p_s W- \tfrac{1}{2}  W  + (\mathcal{V} _W \cdot \nabla) W  &= {F} _W  \,, \\
\p_s Z + (\mathcal{V}_Z \cdot \nabla) Z &={F} _Z \,,   \\
\p_s A_\nu +  (\mathcal{V} _U \cdot \nabla) A_\nu &={F}_{A\nu}  \,,
\end{align*} 
\end{subequations} 
where the gradient is taken with respect to the $y$ variable.
The system  \eqref{US-euler-ss} takes the form
\begin{subequations} 
\label{US-short}
\begin{align} 
\p_s U_i  + (\mathcal{V} _U \cdot \nabla) U_i
 + 2\beta_\tau \beta_3 S \left( \Jcal \Ncal_i e^{\frac s2} \p_{y_1} S + \delta^{i\nu} e^{-\frac s2} \p_{y_\nu} S\right)
 &= 2\beta_\tau \beta_1 e^{-s} \dot Q_{ij} U_j \, ,
 \\
 \p_s S + (\mathcal{V} _U \cdot \nabla) S +2\beta_\tau\beta_3 S \left( e^{\frac s2}    \p_{y_1} U \cdot \Ncal \Jcal +  e^{-\frac s2}   \p_{y_\nu} U_\nu\right)& =0 \, .
 \label{S-ss}
\end{align} 
\end{subequations}

Having defined the transport velocities, we now
 define associated  Lagrangian flows by
\begin{subequations} 
\label{flows}
\begin{gather} 
\p_s \pw (y,s) \!  = \!\mathcal{V}_W(\pw(y,s), s) \,,  \    \p_s \pz (y,s) \!  =\! \mathcal{V}_Z(\pz(y,s), s) \,, \   \p_s \pa (y,s)  \! =\! \mathcal{V}_U(\pa(y,s), s) \,, \\
\pw(y,s_0)=y\,,  \ \ \pz(y,s_0)=y\,, \ \ \pa(y,s_0)=y \,.
\end{gather}
\end{subequations} 
for  $s_0 \geq -\log\eps$.  With $\Phi$ denoting either $\pw$, $\pz$, or $\pu$, we shall denote trajectories emanating from a point $y_0$ at time $s_0$ by
\begin{align} 
\Phi^{y_0}(s)=\Phi (y_0,s) \text{ with }\Phi (y_0,s_0)=y_0 \,. \label{traj}
\end{align}

\subsection{The globally self-similar solution of 3D Burgers}\label{sec:wbar}
We recall (cf.~\cite{CaSmWa1996}) that 
\begin{align}
 W_{\rm 1d}(y_1) = \left(- \frac {y_1}{2} + \left(\frac{1}{27} + \frac{y_1^2}{4}\right)^{\frac 12}\right)^{\frac 13} - \left( \frac {y_1}{2} + \left(\frac{1}{27} + \frac{y_1^2}{4}\right)^{\frac 12} \right)^{\frac 13} \, ,
 \label{eq:barW1d:def}
\end{align} 
is the stable globally self-similar solution of the 1D Burgers equation. 
We define 
 $$
{\mathcal B}(\check y)=\frac{1}{1+|\check y|^2} = \frac{1}{1+ y_2^2 + y_3^2} = {\mathcal B}(y_2,y_3)\,.
$$
Then, as done in two dimensions by Collot, Ghoul, and Masmoudi~\cite{CoGhMa2018}, we have that 
 \begin{align}
\bar W(y) = \frac{1}{{\mathcal B}^{\frac12}(\check y)}   W_{\rm 1d} ({\mathcal B}(\check y)^{\frac32} y_1) = \frac{1}{{\mathcal B}^{\frac 12}(y_2,y_3)}   W_{\rm 1d} ({\mathcal B}(y_2,y_3)^{\frac 32} y_1 ) = \bar W(y_1,y_2,y_3)
 \label{eq:barW:def}
 \end{align} 
is an example of a stable self-similar solution to 3D Burgers equation
\begin{align}
 -\tfrac 12 \bar W + \left( \tfrac{3}{2} y_1+ \bar W \right) \partial_{1} \bar W+ \tfrac{1}{2} y_\mu  \partial_{\mu} \bar W     = 0 \, ,
 \label{eq:barW:dx}
\end{align}
with an explicit representation given by \eqref{eq:barW:def}. As will be explained in Section \ref{s:CSS}, in order to establish the asymptotic profile for $W(y,s)$,
a solution to \eqref{eq:euler:ss:a},  we shall construct the ten-dimensional family of stable self-similar solutions to 3D Burgers of which \eqref{eq:barW:def} is one example.

\subsubsection{Properties of $\bar W$}
We will make use of the fact that the Hessian matrix of $\p_1 \bar W$ at the origin $y=0$ is given by 
\begin{align} 
\label{eq:Hess_py1_barW}
\nabla^2 \partial_{1} \bar W(0) = \left[
\begin{matrix} 
6 & 0 & 0\\
0 & 2 & 0 \\
0 & 0 & 2 \\
\end{matrix}\right] 
\end{align} 
and that the bounds  
\[
-1 \le \partial_{1} \bar W \le 0 \,, \qquad  0 \le \abs{\check \nabla \bar W}  \le  \tfrac{3}{5}  \, ,
\]
hold.
We introduce the weight function 
\begin{align}
\eta(y) = 1 + y_1^2 + \abs{\check y}^6 \, ,
\label{eq:eta:def}
\end{align}
which has the property that $\eta^{\frac 16}$ (and its derivatives) accurately captures the asymptotic growth rate of $\bar W$ (and its derivatives) as $\abs{y}\to \infty$. For the $\p_1 W$ estimate the Taylor series at the origin has to be analyzed more carefully, and for this function we use  the modified weight function
\begin{align}
\tilde \eta(y) = 1+ y_1^2 + \abs{\check{y}}^2  + \abs{\check y}^6 \, .
\label{eq:tilde:eta:def}
\end{align}
With this notation, we note that the function $\bar W$ satisfies the weighted $L^\infty$ estimates
\begin{align}
\label{eq:bar:W:properties}
\snorm{\eta^{-\frac 16} \bar W}_{L^\infty} \leq 1, \ \     
\snorm{{\tilde \eta}^{\frac 13} \p_1 \bar W}_{L^\infty} \leq 1,  \ \      
\snorm{\check \nabla \bar W}_{L^\infty} \leq \tfrac 23,  \ \   
\snorm{\eta^{\frac 13} \p_1   \nabla \bar W}_{L^\infty} \leq \tfrac 34,  \ \   
\snorm{\eta^{\frac 16} \check \nabla^2 \bar W}_{L^\infty} \leq \tfrac 34.
\end{align}

\subsubsection{Genericity condition}
In view of \eqref{eq:Hess_py1_barW}, the matrix $\nabla^2 \partial_{1} \bar W(0)$ is positive definite and satisfies the genericity condition
\begin{align} 
\nabla^2 \partial_{1} \bar W(0)  >0 \,. \label{genericity}
\end{align} 
The condition \eqref{genericity} is equivalent to the  non-degeneracy condition (15.2)  described by Christodoulou in \cite{Ch2007},
and so  $\bar W$ is an example of a generic shock profile.  In particular,   Proposition 12 of Collot-Ghoul-Masmoudi \cite{CoGhMa2018} proves that
the linear operator obtained by linearizing the self-similar  2D Burgers equation about the 2D version of $\bar W$ is spectrally stable.

\subsection{Evolution of higher order derivatives}
\subsubsection{Higher-order derivatives for the $(W,Z,A)$-system}

We now record, for later usage, the equations obeyed by $\partial^\gamma$ applied to $W$, $Z$ and $A$, when $\abs{\gamma} \geq 1$. For a multi-index $\gamma \in {\mathbb N}_0^3$, we write $\gamma = (\gamma_1, \check \gamma) = (\gamma_1, \gamma_2,\gamma_3)$. Then, for $|\gamma|\geq 1$, applying $\p^\gamma$ to \eqref{euler-ss}, we arrive at the differentiated system
\begin{subequations} 
\label{euler_for_Linfinity}
\begin{align}
\left( \p_s + \tfrac{3\gamma_1 + \gamma_2 + \gamma_3-1}{2} + \beta_\tau \left(1 + \gamma_1  {\bf 1}_{\gamma_1\geq 2}   \right) \Jcal \p_1 W \right)\p^\gamma W   +  \left( \mathcal{V}_W \cdot \nabla\right) \p^\gamma W &= F^{(\gamma)}_W \,,
\label{euler_for_Linfinity:a}
\\
\left( \p_s + \tfrac{3\gamma_1 + \gamma_2 + \gamma_3}{2} + \beta_2 \beta_\tau  \gamma_1  \Jcal \p_1 W \right)\p^\gamma Z   +  \left( \mathcal{V}_Z \cdot \nabla\right) \p^\gamma Z &= F^{(\gamma)}_Z  \,,
 \label{euler_for_Linfinity:b}
\\
\left( \p_s + \tfrac{3\gamma_1 + \gamma_2 + \gamma_3}{2} + \beta_1 \beta_\tau  \gamma_1 \Jcal \p_1 W \right)\p^\gamma A_\nu   +  \left( \mathcal{V}_U \cdot \nabla\right) \p^\gamma A_\nu &= F^{(\gamma)}_{A\nu} \,,
 \label{euler_for_Linfinity:c}
\end{align}
\end{subequations}
where the forcing terms are given by
\begin{align}
\label{eq:F:W:def}
F^{(\gamma)}_W 
&= \p^\gamma F_W 
- \sum_{0\leq \beta < \gamma} {\gamma \choose \beta}  \left(\p^{\gamma-\beta}G_W \p_1 \p^\beta W + \p^{\gamma-\beta} h_{W}^\mu \p_\mu \p^\beta W\right) \notag\\
&\quad - \beta_\tau {\bf 1}_{|\gamma|\geq 3}\sum_{\substack{1\leq |\beta| \leq |\gamma|-2 \\ \beta\le\gamma}} {\gamma \choose \beta}  \p^{\gamma-\beta} (\Jcal W)   \p_1\p^\beta W  
- \beta_\tau {\bf 1}_{|\gamma|\geq 2}\sum_{\substack{ |\beta| = |\gamma|-1 \\ \beta\le\gamma, \beta_1 = \gamma_1}} {\gamma \choose \beta}  \p^{\gamma-\beta} (\Jcal W)    \p_1\p^\beta W 
\end{align}
for the $\p^\gamma W$ evolution, and by
\begin{subequations}
\label{eq:F:ZA:def}
\begin{align}
F^{(\gamma)}_Z
&= \p^{\gamma} F_Z 
- \sum_{0\leq \beta < \gamma} {\gamma \choose \beta}  \left(\p^{\gamma-\beta}G_Z \p_1 \p^\beta Z + \p^{\gamma-\beta}h_Z^\mu \p_\mu \p^\beta Z\right) \notag\\
&\quad - \beta_2 \beta_\tau {\bf 1}_{|\gamma|\geq 2}\sum_{\substack{0\leq |\beta| \leq |\gamma|-2 \\ \beta\le\gamma}} {\gamma \choose \beta}  \p^{\gamma-\beta} (\Jcal W)   \p_1\p^\beta Z  
- \beta_2 \beta_\tau  \sum_{\substack{ |\beta| = |\gamma|-1 \\ \beta\le\gamma, \beta_1 = \gamma_1}} {\gamma \choose \beta}  \p^{\gamma-\beta} (\Jcal W)    \p_1\p^\beta Z 
\\
F^{(\gamma)}_{A\nu}
&= \p^\gamma F_{A\nu} 
- \sum_{0\leq \beta < \gamma} { \gamma \choose \beta}   \left(\p^{\gamma-\beta}G_{U} \p_1 \p^\beta A_\nu + \p^{\gamma-\beta} h_{U}^\mu \p_\mu \p^\beta A_\nu\right) \notag\\
&\quad - \beta_1 \beta_\tau  {\bf 1}_{|\gamma|\geq 2} \sum_{\substack{0\leq |\beta| \leq |\gamma|-2 \\ \beta\le\gamma}} {\gamma \choose \beta}  \p^{\gamma-\beta} (\Jcal W)   \p_1\p^\beta A_\nu
- \beta_1 \beta_\tau  \sum_{\substack{ |\beta| = |\gamma|-1 \\ \beta\le\gamma, \beta_1 = \gamma_1}} {\gamma \choose \beta}  \p^{\gamma-\beta} (\Jcal W)    \p_1\p^\beta A_\nu
\end{align}
\end{subequations}
for the $\p^\gamma Z$ and $\p^\gamma A_\nu$ evolutions.
In \eqref{euler_for_Linfinity} we have extracted only the leading order damping terms on the left side of the equations. Indeed, note that the forcing terms defined above contain terms which are proportional to $\p^\gamma(W,Z,A)$. However, because the factors in front of these terms decay exponentially in $s$, we have included them in the force.

\subsubsection{Higher-order derivatives for $\tilde W$}
Additionally, it is useful to consider the evolution of
\begin{align}
\label{eq:tilde:W:def} 
\tilde W(y,s) = W(y,s) - \bar W(y)
\end{align}
and its derivatives. For the case of no derivatives, we have
\begin{align}
  &\p_s \tilde W+(\beta_\tau \Jcal \p_1 \bar W- \tfrac{1}{2})  \tilde W  + (\mathcal{V} _W \cdot \nabla) \tilde W\notag \\
 &\qquad\qquad= F _W - e^{-\frac s2} \beta_\tau \dot \kappa +   ((\beta_\tau \Jcal -1)\bar W-G_W)\p_1 \bar W- h_W^\mu \p_\mu \bar W
 =: \tilde{F}_W \, .\label{eq:tilde:W:evo} 
\end{align}
For $\abs{\gamma} \geq 1$, applying $\p^\gamma$ to \eqref{eq:tilde:W:evo}, we obtain that the function $\p^\gamma \tilde W$ obeys
\begin{align}\label{eq:p:gamma:tilde:W:evo}
\left( \p_s + \tfrac{3\gamma_1 + \gamma_2 + \gamma_3-1}{2} + \beta_\tau \Jcal  \left(\p_1 \bar W + \gamma_1  \p_1 W \right) \right)\p^\gamma \tilde W   +  \left( \mathcal{V}_W \cdot \nabla\right) \p^\gamma \tilde W &= \tilde F^{(\gamma)}_W
\end{align}
where the forcing terms $\tilde F^{(\gamma)}_W$ are given by
\begin{align}
\tilde F^{(\gamma)}_W 
&= \p^\gamma \tilde F_W 
- \sum_{0\leq \beta < \gamma} {\gamma \choose \beta}  \left(\p^{\gamma-\beta}G_W \p_1 \p^\beta \tilde W + \p^{\gamma-\beta} h_{W}^\mu \p_\mu \p^\beta \tilde W+\beta_{\tau}\p^{\gamma-\beta} (\Jcal \partial_1\bar W)   \p^\beta \tilde W  \right) \notag\\
&\quad - \beta_\tau {\bf 1}_{|\gamma|\geq 2}\sum_{\substack{1\leq |\beta| \leq |\gamma|-2 \\ \beta\le\gamma}} {\gamma \choose \beta}  \p^{\gamma-\beta} (\Jcal W)   \p_1\p^\beta \tilde W  
- \beta_\tau \sum_{\substack{ |\beta| = |\gamma|-1 \\ \beta\le\gamma, \beta_1 = \gamma_1}} {\gamma \choose \beta}  \p^{\gamma-\beta} (\Jcal W)    \p_1\p^\beta \tilde W \label{eq:p:gamma:tilde:F}\,.
\end{align}

 \section{Main results}
\label{sec:results}
\subsection{Data in physical variables}
\label{sec:data:real}
We set the initial time to be  $t_0 = -\eps$, which corresponds to $\tcal_0 = - \frac{2}{1+\alpha} \eps$,  and we first define the initial conditions for the modulation variables.  We define
\begin{align} 
\kappa_0 := \kappa(-\eps)  \,, \ \ \ \ \tau_0:=\tau(-\eps)=0\,,  \ \ \ \  \xi_0:=\xi(-\eps)=0\,, \ \ \ \ \check n_0:= \check n(-\eps) =0\,, \ \ \ \  \phi_0:= \phi(-\eps)  \,,
\label{mod-ic}
\end{align} 
where 
\begin{align} 
\kappa_0 >1\,, \quad \abs{\phi_0} \le \eps \,.
\label{ic-kappa0-phi0}
\end{align} 
We note that $\kappa_0$ is a given parameter of the problem, while $\phi_0$ will be chosen suitably in terms of the initial datum via \eqref{eq:phi:0:def}.
Next, we define the initial value for the  function $f$ as
\[
f_0(\check \xcal)= \tfrac 12 {\phi_0}_{\nu\mu} \xcal_\nu \xcal_\mu \, ,
\] 
and according to \eqref{tangent} and \eqref{normal}, we define the orthonormal basis $(\Ncal_0,\Tcal_0^2,\Tcal_0^3)$ by
\begin{subequations}
\label{geom000}
\begin{align}
\Ncal_0 &=  \Jcal_0^{-1} (1, -f_{0,_2}, -f_{0,_3}), \qquad \mbox{where} \qquad \Jcal_0 = ( 1+  |f_{0,_2}|^2+ |f_{0,_3}|^2)^ {\frac{1}{2}}, 
\label{geom00}\\
\Tcal^2_0 &=  \left( \tfrac{f_{0,_2}}{\Jcal_0}, 1 -  \tfrac{(f_{0,_2})^2}{\Jcal_0(\Jcal_0+1)} \,,  \tfrac{- f_{0,_2}f_{0,_3}}{\Jcal_0(\Jcal_0+1)}\right) \, , \qquad \mbox{and} \qquad
\Tcal^3_0 =  \left( \tfrac{f_{0,_3}}{\Jcal_0},  \tfrac{- f_{0,_2}f_{0,_3}}{\Jcal_0(\Jcal_0+1)} \,,  1  -  \tfrac{(f_{0,_3})^2}{\Jcal_0(\Jcal_0+1)} \right).
\label{geom0}
\end{align} 
\end{subequations}
As a consequence of \eqref{ic-kappa0-phi0} and  \eqref{geom000}, we see that
\begin{align} 
\abs{ \Ncal_0 - e_1} \le \eps\,, \ \ \ \  \abs{ \Tcal_0^\nu - e_\nu} \le \eps \,.   \label{gerd}
\end{align} 
From \eqref{mod-ic}, \eqref{eq:tilde:x:def}, and \eqref{x-sheep}, we have  that at $t=-\eps$, the sheared variable $x$ is given by 
\begin{align} 
x_1 = \xcal_1-  f_0(\check \xcal)\,,\quad x_2 = \xcal_2\,, \quad  x_3 = \xcal_3 \,.   \label{ic-cov}
\end{align}

The remaining initial  conditions  are for the velocity field and the density (which yields the rescaled sound speed):
$$
u_0(\xcal) := u(\xcal,-\eps), \quad \rho_0(\xcal) := \rho(\xcal, -\eps) \,, \quad \sigma_0 := \tfrac{\rho_0^ \alpha }{\alpha } \,.
$$
According to \eqref{usigma-sheep} and \eqref{tildeu-dot-T} (see also \eqref{eq:tilde:Riemann}) we introduce the initial datum for our Riemann-type variables in both the $\xcal$ and the $x$ variables:
\begin{subequations}
\label{eq:tilde:wza:0}
\begin{align} 
\tilde w_0(\xcal)& :=   u_0(\xcal) \cdot \Ncal_0(\check \xcal) + \sigma_0(\xcal) =: w_0(x) \,, \\
\tilde z_0(\xcal)& :=   u_0(\xcal) \cdot \Ncal_0(\check \xcal) - \sigma_0(\xcal)  =: z_0(x) \,, \\
\tilde a_{0\nu}(\xcal)& :=  u_0(\xcal) \cdot \Tcal^\nu(\check \xcal) =: a_{0\nu}(x)\,.
\end{align} 
\end{subequations}
It is more convenient (and equivalent in view of \eqref{eq:tilde:wza:0}) to state the initial datum assumptions in terms of the functions $(\tilde w_0, \tilde z_0, \tilde a_0)$, instead of the standard variables $u_0$ and $\sigma_0$.

First, we assume that the support of the initial data  $(\tilde w_0 - \kappa_0, \tilde z_0, \tilde a_0)$, defined in \eqref{eq:tilde:wza:0},  is  contained in the  set $\XX_0$, 
given by
\begin{align}
\XX_0 =\left\{\abs{\xcal_1}\leq  \tfrac{1}{2}  \eps^ {\frac{1}{2}}  ,\abs{\check \xcal}\leq  \eps^ {\frac{1}{6}}  \right\}\,.
\label{eq:support-init-x}
\end{align}
This condition is equivalent to requiring that $u_0 \cdot \Ncal_0 - \frac{\kappa_0}{2}$, $\sigma_0 - \frac{\kappa_0}{2}$, and $u_0 \cdot \Tcal^\nu$ are compactly 
supported in $\XX_0$. In view of the coordinate transformation \eqref{ic-cov} and the bound \eqref{ic-kappa0-phi0}, the functions of $x$ defined in \eqref{eq:tilde:wza:0}, 
namely $(  w_0,   z_0,   a_0)$, have spatial support contained in the set  $\left\{\abs{x_1}\leq   \tfrac{1}{2} \eps^ {\frac{1}{2}}  
+ \eps  ,\abs{\check x}\leq  \eps^ {\frac{1}{6}}\right\} \subset \left\{\abs{x_1}\leq   \eps^ {\frac{1}{2}}   ,\abs{\check x}\leq  \eps^ {\frac{1}{6}} \right\}$.  This larger set 
corresponds to the set $ \XXX(0)$ (defined in \eqref{eq:support}) under the transformation \eqref{eq:y:s:def}.

The function $\tilde w_0(\xcal)$ is chosen such that 
\begin{subequations} 
 \label{IC-setup}
\begin{align} 
&\text{the minimum (negative) slope of $\tilde w_0$ occurs in the $e_1$ direction} \,, \\
&\partial_{\xcal_1} \tilde w_0 \text{ attains its global minimum at } \xcal=0\,, 
\end{align} 
\end{subequations} 
and
\begin{align} 
\nabla_{\!\xcal} \p_{\xcal_1} \tilde w_0 (0)=0\,, \label{grad-d1-w0} 
\end{align} 
and moreover that 
\begin{align}
\tilde w_0(0) = \kappa_0\,,\qquad \partial_{\xcal_1} \tilde w_0(0) = -{\tfrac{1}{\eps}} \,, \qquad  \check\nabla_{\xcal} \tilde w_0(0) =0 \,.
\label{eq:w0:power:series}
\end{align}
Additionally we shall require that $w_0$ satisfies a number of weighted estimates, and that it is close  to a rescaled version of $\bar W$. For this purpose, we introduce the rescaled blow up profile with respect to the coordinate  $x$,  defined by
\begin{align} 
\bar w_\eps(x) := \eps^ {\frac{1}{2}} \bar W \left( \eps^ {-\frac{3}{2}} x_1, \eps^ {-\frac{1}{2}} \check x \right) \,,
\label{eq:bar:w:eps}
\end{align} 
and we set  
$$
 \ww_0(\xcal) := \tilde w_0(\xcal) - \bar w_\eps(\xcal_1 - f_0(\check \xcal),\check \xcal ) = w_0(x) - \bar w_\eps(x) =  \eps^ {\frac{1}{2}}  \tilde W(y, -\log \eps) + \kappa_0 \,.
$$ 
We assume that 
for $\xcal$ such that $\sabs{( \eps^{-\frac 32} \xcal_1, \eps^{-\frac 12} \check \xcal)} \leq 2 \eps^{-\frac{1}{10}}$, the following bounds hold:
\begin{subequations}
\label{love-fosters}
\begin{alignat}{2}
 \abs{\ww_0(\xcal) - \kappa_0} &\leq \eps^{\frac{1}{10}} \left( \eps^3 +   \xcal_1^2 +   \abs{\check \xcal}^6 \right)^ {\frac{1}{6}}   \,,
\label{eq:tilde:w0:zero:derivative}
\\
  \abs{\p_{ \xcal_1} \ww_0(\xcal)} &\leq \eps^{\frac{1}{11}} \left( \eps^3 +  \xcal_1^2 + \abs{\check \xcal}^6 \right)^ {-\frac{1}{3}} \,,
\label{eq:tilde:w0:p1=1}
\\
 \abs{\check \nabla_{\! \xcal} \ww_0(\xcal)} &\leq  \tfrac{1}{2}   \eps^{\frac{1}{12}} \,.
\label{eq:tilde:w0:check=1}
\end{alignat}
\end{subequations}
Furthermore, for   $\xcal$ such that  $\sabs{( \eps^{-\frac 32} \xcal_1, \eps^{-\frac 12} \check \xcal)} \leq 1$, we assume the fourth-derivative estimates
\begin{align} 
  \abs{\p_\xcal^\gamma \ww_0(\xcal)}& \le  \tfrac{1}{2}  \eps^{\frac{5}{8} - {\frac{1}{2}} (3\gamma_1 +\gamma_2+\gamma_3)} \qquad  \mbox{ for } \abs{\gamma}=4 \,,
 \label{eq:tilde:w0:4:derivative}
\end{align} 
while at $\xcal=0$, we assume that
\begin{align} 
    \abs{\p_\xcal^\gamma \ww_0(0)}& \le \tfrac{1}{2}   \eps^{ 1- {\frac{1}{2}} (3\gamma_1+\gamma_2+\gamma_3)- \frac{4}{2k-7}}  \qquad  \mbox{ for } \abs{\gamma}=3 \,.
 \label{eq:tilde:w0:3:derivative}
\end{align} 
For $\xcal \in  \XX_0$  such that  $\sabs{( \eps^{-\frac 32} \xcal_1, \eps^{-\frac 12} \check \xcal)} \geq \frac 12 \eps^{-\frac{1}{10}}$   we assume that  
 \begin{subequations}
 \label{eq:w0:far:out}
\begin{alignat}{2}
 \abs{\tilde w_0(\xcal) - \kappa_0} &\leq ( 1+ \eps^ {\frac{1}{11}} )  \left( \eps^3 +   \xcal_1^2 +   \abs{\check \xcal}^6 \right)^ {\frac{1}{6}}  \,,
\label{eq:w0:zero:derivative}
\\
  \abs{\p_{\xcal_1} \tilde w_0(\xcal)} &\leq ( 1+ \eps^ {\frac{1}{12}} )  \left( \eps^3 +  \xcal_1^2 + \abs{\check \xcal}^6 \right)^ {-\frac{1}{3}} \,,
\label{eq:w0:p1=1}
\\
 \abs{\check \nabla_{\! \xcal} \tilde w_0(\xcal)} &\leq \tfrac{2}{3} + \eps^{\frac{1}{13}}  \,.
\label{eq:w0:check=1}
\end{alignat}
\end{subequations}
Finally, we assume that for all $\xcal \in \XX_0$, the second derivatives of $w_0$ satisfy 
 \begin{subequations}
 \label{eq:w0:pipi}
\begin{align}
 \abs{ \p_{\xcal_1}^2 \tilde w_0(\xcal)} &\leq \eps^{-\frac{3}{2}}  \left( \eps^3 +  \xcal_1^2 + \abs{\check \xcal}^6 \right)^ {-\frac{1}{3}}  \,,
\label{eq:w0:p1p1}
\\
 \abs{ \p_{\xcal_1} \check\nabla_{\xcal} \tilde w_0(\xcal) } &\leq  \tfrac{1}{2} \eps^{-\frac{1}{2}}  \left( \eps^3 +  \xcal_1^2 + \abs{\check \xcal}^6 \right)^ {-\frac{1}{3}}  \,,
\label{eq:w0:p1p2}
\\
 \abs{ \check\nabla_{\xcal}^2 \tilde w_0(\xcal) } &\leq   \tfrac{1}{2} \left( \eps^3 +  \xcal_1^2 + \abs{\check \xcal}^6 \right)^ {-\frac{1}{6}}   \,,
\label{eq:w0:p2p2}
 \end{align}
 \end{subequations}
 and moreover at $\xcal=0$ we assume that
 \begin{align} 
 \abs{\check\nabla_{ \! \xcal}^2 \tilde w_0(0)} \le 1 \,.  \label{check-two-w0}
 \end{align}

For the initial conditions of  $\tilde z_0$ and $\tilde a_0$  we assume that 
\begin{align} 
&\abs{\tilde z_0(\xcal)} \le \eps\,, \qquad  \abs{\p_{\xcal_1}\tilde z_0(\xcal)} \le 1 \,, \qquad \abs{ \check\nabla_{\!\xcal } \tilde z_0(\xcal)} \le  \tfrac{1}{2} \eps^ {\frac{1}{2}} \,, \notag \\
&\abs{ \p_{\xcal_1}^2 \tilde z_0(\xcal)} \le \eps^ {-\frac{3}{2}} \,, \qquad 
\abs{ \p_{\xcal_1} \check\nabla_{\!\xcal}\tilde  z_0(\xcal)} \le \tfrac{1}{2} \eps^ {-\frac{1}{2}} \,, \qquad 
\abs{  \check\nabla_{\!\xcal}^2 \tilde z_0(\xcal)} \le    \tfrac{1}{2}     \,,  \label{eq:z0:ic}
\end{align} 
and\footnote{The bound for  $\p_{\xcal_1} a_0$ in \eqref{eq:a0:ic} can be replaced by a bound that depends on $\kappa_0$, thus permitting arbitrarily large initial
vorticity to be specified.} 
\begin{align} 
&\abs{\tilde a_0(\xcal)} \le \eps\,,\qquad  \abs{\p_{\xcal_1}\tilde a_0(\xcal)} \le 1 \,, \qquad   \abs{ \check\nabla_{\!\xcal} \tilde a_0(\xcal)} \le   \tfrac{1}{2}  \eps^ {\frac{1}{2}} \,, \qquad 
\abs{  \check\nabla_{\!\xcal}^2 \tilde  a_0(\xcal)} \le  \tfrac{1}{2} \,.  \label{eq:a0:ic}
\end{align} 
For the initial specific vorticity, we assume that
\begin{align} 
 \norm{ \tfrac{\operatorname{curl}_{\xcal} u_0(\xcal)}{\rho_0(\xcal)}}_{L^ \infty } \le 1  \,.  \label{eq:svort:IC}
\end{align} 
Lastly, for the Sobolev norm of the initial condition  we assume that for a fixed $k$ with $k\geq 18$ we have
\begin{align}
\sum_{\abs{\gamma}=k} 
\eps^2 \snorm{ \p_{\xcal}^\gamma   \tilde w_0}_{L^2}^2 + \snorm{ \p_{\xcal}^\gamma  \tilde z_0}_{L^2}^2 + \snorm{ \p_{\xcal}^\gamma  \tilde a_0}_{L^2}^2    \le \tfrac 12
\eps^{ {\frac{7}{2}}-(3\gamma_1+\abs{\check\gamma} )} \,. 
  \label{Hk-xland}
\end{align}
We note cf.~\eqref{ic-cov} that the map $x = \xcal - (f_0(\check \xcal),0,0)$ is an $\OO(\eps)$ perturbation of the identity map, and that for any $n\geq 0$, by \eqref{ic-kappa0-phi0} and the support property \eqref{eq:support-init-x} we have $\norm{f_0}_{C^n} \leq \norm{f_0}_{C^2} \leq 2\eps$.  Additionally, from the previous assumptions we have $\norm{\tilde w_0}_{L^2(\XX_0)} + \norm{\tilde z_0}_{L^2(\XX_0)} + \norm{\tilde a_0}_{L^2(\XX_0)} \leq \eps^{\frac 12}$. 
Thus, by appealing to the definition \eqref{eq:tilde:wza:0}, the Fa\'a di Bruno formula, and Sobolev interpolation, we deduce from \eqref{Hk-xland} that
\begin{align}
\sum_{\abs{\gamma}=k} 
\eps^2 \snorm{ \p_x^\gamma   w_0}_{L^2}^2 + \snorm{ \p_x^\gamma   z_0}_{L^2}^2 + \snorm{ \p_x^\gamma   a_0}_{L^2}^2    \le 
\eps^{ {\frac{7}{2}}-(3\gamma_1+\abs{\check\gamma} )} 
  \label{Hk-xland-twist}
\end{align}
holds, upon taking $\eps$ to be sufficiently small in terms of $k$.

At this stage it is convenient to define the coefficients $\phi_{0 \nu\mu}$ from \eqref{mod-ic}. From the change of variables, \eqref{ic-cov} and the fact that $\check \nabla f_0 ( 0) = 0$, we have that
\begin{align} 
\p_{x_\nu} \p_{x_\mu} w_0(0) = \p_{\xcal_\nu} \p_{\xcal_\mu} \tilde w_0(0) + \p_{\xcal_1} w_0(0) \phi_{0 \nu \mu}\,.  \label{phi0-cond}
\end{align} 
In order that our initial data at the blow up location behaves just as the blow up profile $\bar W$ (in self-similar coordinates) at the blow up point, we shall
insist that $ \check\nabla_{\! x}^2 w_0(0) =0$. From the identity \eqref{phi0-cond} and using the second equality in \eqref{eq:w0:power:series},  we achieve this by setting
\begin{align}
\phi_{0\nu\mu} =    \eps \p_{\xcal_\nu}  \p_{\xcal_\mu} \tilde w_0(0) \,.
\label{eq:phi:0:def}
\end{align}
Hence, the condition \eqref{check-two-w0} automatically implies \eqref{ic-kappa0-phi0}.

We note that in view of \eqref{eq:tilde:wza:0}, \eqref{eq:support-init-x}, \eqref{eq:w0:zero:derivative}, \eqref{eq:z0:ic}, the fact that $\abs{\bar W(y)} \leq \eta^{\frac 16}(y)$, which implies $\abs{\bar w_\eps(x)} \leq (\eps^3 + x_1^2 + \abs{\check x}^6)^{\frac 16}$, and the identity $\frac{2}{\alpha} \rho_0^\alpha(\xcal)  = \kappa_0 + (\tilde w_0(\xcal) - \kappa_0) - \tilde z_0(\xcal)$, we have that 
\[
\tfrac{2}{\alpha} \rho_0^\alpha(\xcal) \geq \kappa_0 - (1 +\eps^{\frac{1}{11}}) (\eps^3 + x_1^2 + \abs{\check x}^6)^{\frac 16} - \eps \geq \kappa_0 - (1 +\eps^{\frac{1}{11}}) (3\eps)^{\frac 16} - \eps \geq \kappa_0 - 3 \eps^{\frac 16}
\]
for all $\xcal \in \RR^3$; that is, upon taking $\eps$ to be sufficiently small in terms of $\kappa_0$, we have that the initial density is strictly positive.

\subsection{Statement of the main theorem in physical variables}

\begin{theorem}[\bf Formation of shocks for Euler] \label{thm:main}
Let $\gamma>1$, $ \alpha = {\tfrac{\gamma-1}{2}}$.   There exist a sufficiently large $\kappa_0 = \kappa_0(\alpha) > 1$,  and a sufficiently small 
$\eps = \eps(\alpha,\kappa_0) \in (0,1)$ such that the following holds. 

\underline{Assumptions on the initial data.} Let $u_0(\xcal)$ and $\rho_0(\xcal)$ denote the initial data for the  Euler equations \eqref{eq:Euler}, let $\sigma_0 = \frac{\rho_0^\alpha}{\alpha}$ and $\omega_0 = \operatorname{curl}_\xcal u_0$.  The modulation functions have initial conditions given by \eqref{mod-ic}, where $\phi_0$ is given by \eqref{eq:phi:0:def}. Define $(\Ncal_0, \Tcal_0^2,\Tcal^3_0)$ by \eqref{geom000}  and $(\tilde w_0, \tilde z_0, \tilde a_{0\nu})$ by \eqref{eq:tilde:wza:0}.  Assume that $(\tilde w_0 - \kappa_0, \tilde z_0, \tilde a_0)$
are supported in the set $\XX_0$ defined \eqref{eq:support-init-x}, and   
 that $u_0\in H^k$ and $\rho_0 \in H^k$ for a fixed $k \ge 18$.  
Furthermore 
suppose that the functions $\tilde w_0$, $\tilde z_0$, $\tilde a_0$, and $\omega_0$ satisfy the conditions \eqref{ic-kappa0-phi0}--\eqref{Hk-xland}.

\underline{Shock formation for the 3d Euler equations.}\,  
There exists a time $T_* = \OO(\eps^2)$ and a unique solution $(u,\rho) \in C([-\eps,T_*);H^k)\cap C^1([-\eps,T_*);H^{k-1})$ to  \eqref{eq:Euler} which blows up in an
 asymptotically self-similar fashion  at time $T_*$, at a single point  $\xi_* \in \mathbb{R}^3  $. By letting
$(\Ncal(t), \Tcal^2(t), \Tcal^3(t))$  be defined by \eqref{tangent}  and \eqref{normal}, with  the new space variable $\tilde x = \tilde x(t)$ defined by \eqref{eq:tilde:x:def}, and with $(\tilde u,\tilde \sigma)$ given by \eqref{eq:tilde:u:def}, where $\sigma = \tfrac{\rho^ \alpha }{\alpha }$, we let 
\begin{align} 
\tilde w =\tilde u \cdot \Ncal  + \tilde \sigma\,, \ \ \tilde z= \tilde u \cdot \Ncal  - \tilde \sigma\,, \ \  \tilde{a}_\nu = \tilde u \cdot \Tcal^\nu  \,, \label{laphroaig18}
\end{align} 
as functions of $(\tilde x, t)$. Then, the following results hold:
\begin{itemize}[itemsep=2pt,parsep=2pt,leftmargin=.15in]
\item The blow up time $T_* = \OO(\eps^2)$ and the blow up location  $\xi_*  = \OO( \epsilon )$ are explicitly computable, with $T_*$ defined by the condition
$\int_{-\eps}^{T_*} (1-\dot \tau(t)) dt = \eps $ and with the blow up location given by
$\xi_*  = \lim_{t\to T_*} \xi(t)$.  The amplitude modulation function satisfies $\abs{\kappa_*-\kappa_0} = \OO(\eps^{\frac 32})$ where $\kappa_*  = \lim_{t\to T_*} \kappa(t)$.

\item For each $t \in [-\eps, T_*)$, we have $
\sabs{\Ncal(\check {\tilde x},t) - \Ncal_0(\check \xcal)} +\sabs{\Tcal^\nu(\check {\tilde x},t) - \Tcal_0^\nu(\check \xcal)}  = \OO(\eps) \,.$

\item We have $\sup_{t\in[-\eps, T_*)} \left( \norm{\tilde u \cdot \Ncal - {\tfrac{1}{2}} \kappa_0 }_{ L^\infty }+\norm{\tilde u \cdot \Tcal^\nu}_{ L^\infty }+
\norm{\tilde \sigma - \tfrac{1}{2} \kappa_0}_{ L^\infty } +\| \omega\|_{ L^\infty } \right) \les 1$.
\item There holds $\lim_{t \to T_*} \Ncal \cdot \nabla_{\! \tilde x} \tilde w(\xi(t),t) = -\infty $ and  $\frac{1}{2(T_*-t)} \leq  \norm{\Ncal \cdot \nabla_{\! \tilde x} \tilde w(\cdot,t)}_{L^\infty} \leq \frac{2}{T_*-t}$ as  $t \to T_*$.
\item At the time of blow up, $\tilde w( \cdot , T_*)$ has a cusp-type singularity with  $C^ {\sfrac{1}{3}} $ H\"{o}lder regularity.
\item We have that only the $\p_\Ncal$ derivative of $\tilde u \cdot \Ncal$ and $\tilde \rho$ blow up, while the other first order derivatives remain uniformly bounded:
\begin{subequations} 
\label{lagavulin16}
\begin{align} 
&\quad \lim_{t\to T_*} \Ncal \cdot \nabla_{\! \tilde x} (\tilde u \cdot \Ncal) (\xi(t), t) = \lim_{t\to T_*} \Ncal \cdot \nabla_{\! \tilde x} \tilde \rho ( \xi(t), t) = -\infty \,, \\
&  \sup_{t\in[-\eps,T_*)}  \| \Tcal^\nu \cdot \nabla_{\! \tilde x}\tilde \rho(\cdot , t)\|_{ L^ \infty } +  \| \Tcal^\nu \cdot \nabla_{\! \tilde x} \tilde u(\cdot , t)\|_{ L^ \infty } + 
\|\Ncal \cdot \nabla_{\! \tilde x} (\tilde u \cdot \Tcal^\nu)(\cdot , t)\|_{ L^ \infty}  \les 1 \,.
\label{damn-thats-nice}
\end{align} 
\end{subequations} 
\item Let $\p_t X(\xcal, t) = u(X(\xcal, t), t)$ with $X(\xcal, -\eps)=\xcal$ so that $X(\xcal, t)$ is the Lagrangian flow. Then there exists constants $c_1$, $c_2$ such that
$ c_1 \le \abs{ \nabla _{\! \xcal} X(\xcal, t) } \le c_2$ for all $t\in [-\eps, T_*)$.
\item The density remains uniformly bounded from below and satisfies 
$$
\snorm{ \tilde \rho^ \alpha ( \cdot , t) -   \tfrac{\alpha }{2} \kappa_0}_{ L^ \infty } \le \alpha  \eps^ {\sfrac{1}{8}}  \ \ \text{ for all } \ t\in[-\eps, T_*]  \,.
$$
\item The vorticity satisfies $\snorm{\omega(\cdot, t) }_{L^\infty}  \le C_0 \snorm{\omega(\cdot,-\eps)}_{L^\infty}$   for all 
$t\in[-\eps, T_*] $ for a universal constant $C_0$, and if 
$\abs{\omega(\cdot, -\eps)} \geq c_0 >0$ on the set  $B(0,2 \eps^{\sfrac 34})$ then at  the blow up location $\xi_*$ there is  nontrivial vorticity, and moreover
\begin{align*}
\abs{\omega(\cdot, T_*)} \geq \tfrac{c_0}{C_0} \qquad \mbox{on the set} \qquad    B(0, \eps^{\sfrac 34})  \,.
\end{align*}
\end{itemize} 
\end{theorem}

We note that the  support property \eqref{eq:support-init-x} on the initial data  as well as the  conditions 
\eqref{IC-setup}--\eqref{eq:w0:power:series} preclude the set of  initial data satisfying the hypothesis of Theorem \ref{thm:main} from containing a non-trivial  open set in the
$H^k$ topology.   However, using the symmetries of the Euler equations, these conditions may be relaxed in order to prove the following:
\begin{theorem}[\bf Open set of initial conditions]
\label{thm:open:set:IC} 
Let $\tilde{\mathcal F}$ denote the set of initial data satisfying the hypothesis of Theorem~\ref{thm:main}.
There exists an open neighborhood of $\tilde{\mathcal F}$ in the $H^k$ topology, denoted by ${\mathcal F}$, such that for any initial data to the Euler equations taken from ${\mathcal F}$, the  conclusions of Theorem \ref{thm:main} hold. 
\end{theorem} 
The proofs of  Theorems~\ref{thm:main} and~\ref{thm:open:set:IC} are given in Section~\ref{sec:conclusion}.
We remark that Theorem~\ref{thm:main} is a direct consequence of Theorem \ref{thm:SS}, stated below, which establishes the stability of the self-similar profile $\bar W$ under a suitable open set of perturbations.

\subsection{Data in self-similar variables}
The initial datum assumptions in the $\xcal$ variable  made in Section~\ref{sec:data:real} imply certain properties of the initial datum in the self-similar coordinates $y$. 
In this subsection,  we  provide a list of these properties.

First, we see that at the initial self-similar time, which is given as
$s= -\log \eps$ since by \eqref{mod-ic} we have $\tau_0=0$, the self-similar variable $y$ is defined by \eqref{eq:y:s:def} as
\begin{align}
\label{eq:ss:y:initial}
y_1 = \eps^{-\frac 32} x_1 = \eps^{-\frac 32}\left(\xcal_1 - f_0(\check \xcal)\right)\, ,
\qquad \mbox{and} \qquad
\check y = \eps^{-\frac 12} \check x = \eps^{-\frac 12} \check \xcal\,.
\end{align}
Second, we use \eqref{eq:ss:ansatz},  \eqref{mod-ic}, and \eqref{eq:tilde:wza:0}, to define $W(\cdot,-\log\eps), Z(\cdot,-\log \eps)$, and $A_\nu(\cdot,-\log \eps)$ as
\begin{align}
W(y,-\log \eps) = \eps^{-\frac 12} \left( \tilde w_0(\xcal) - \kappa_0 \right)\,,
\qquad
Z(y,-\log \eps) = \tilde z_0 (\xcal)\, ,
\qquad
A_\nu(y,-\log \eps) = \tilde a_{0\nu} (\xcal) \, .
\label{eq:ss:WZA:initial}
\end{align}

Next, from \eqref{ic-kappa0-phi0}, \eqref{ic-cov} and the fact that $(\tilde w_0 - \kappa_0, \tilde z_0, \tilde a_0)$ are supported in the set $\XX_0$ defined in \eqref{eq:support-init-x}, we deduce that the initial data for $(W, Z,A)$ is supported in the set $\XXX_0$, given by 
\begin{align}
\XXX_0 =\left\{\abs{y_1}\leq   \eps ^{-1} ,\abs{\check y}\leq  \eps^{-\frac 13} \right\}\,.
\label{eq:support-init}
\end{align}
The factor of $\frac 12$ present in \eqref{eq:support-init-x} allows us to absorb the shift of $\xcal_1$ by $f_0(\check \xcal)$. 

Next, let us consider the behavior of $W$ at $y=0$, which corresponds to $\xcal = 0$. By \eqref{grad-d1-w0}, \eqref{eq:w0:power:series}, \eqref{phi0-cond}, \eqref{eq:phi:0:def}, and \eqref{eq:ss:WZA:initial} we deduce that 
\begin{align}
W(0,-\log \eps) = 0\, , 
\quad 
\p_1 W(0,-\log \eps) = -1\,,
\quad 
\check \nabla W(0,-\log \eps) = 0 \,,
\quad
\nabla^2 W(0,-\log \eps) = 0\,.
\label{eq:fat:cat}
\end{align}
These constraints on $W$ at $y=0$ will be shown to persist throughout the self-similar Euler evolution.

At this stage, we introduce a sufficiently large parameter $M = M(\alpha,\kappa_0)\geq 1$. In terms of $M$ and $\eps$, we define a small length scale $\ell$ and a 
large length scale $\LLL$  by 
\begin{subequations}
\begin{align}
\ell &= (\log M)^{-5}\,,
\label{eq:ell:choice}
\\
\LLL &=\eps^{-\frac{1}{10}}  \, .
\label{eq:pounds}
\end{align}
\end{subequations}
Note that  $M$ is  independent of $\eps$.
The   region $\abs{y}\leq \ell$   denotes a  Taylor series region, where $W$ is essentially dominated by its series expansion at $y=0$, while the annular region 
$\ell \leq \abs{y} \leq \LLL$ denotes a region where  $W$ and $\nabla W$ closely resemble $\bar W$ and $\nabla \bar W$.

For the initial datum of $\tilde W = W - \bar W$ given, in view of \eqref{eq:ss:WZA:initial},  by
\[
\tilde W (y,-\log\eps)= W(y,-\log \eps) - \bar W(y) = \eps^{-\frac 12} \left(  \ww_0(\xcal) - \kappa_0 \right)\, ,
\]  
it follows from \eqref{love-fosters}, along with \eqref{ic-kappa0-phi0}, \eqref{ic-cov}, \eqref{eq:support-init-x}, and \eqref{eq:ss:y:initial}  that for $\abs{y} \leq \LLL$ we have
\begin{subequations}
\label{ss-ic-1}
\begin{alignat}{2}
\eta^{-\frac 16}(y) \abs{\tilde W(y,-\log \eps)} &\leq \eps^{\frac{1}{10}} 
\label{eq:tilde:W:zero:derivative}
\\
\eta^{\frac 13}(y) \abs{\p_1 \tilde W(y,-\log \eps)} &\leq \eps^{\frac{1}{11}}
\label{eq:tilde:W:p1=1}
\\
\abs{\check \nabla \tilde W(y,-\log \eps)} &\leq \eps^{\frac{1}{12}}
\label{eq:tilde:W:check=1}
\,,
\end{alignat}
\end{subequations}
where we recall that $\eta(y) = 1 + y_1^2 + \abs{\check y}^6$, and the partial derivatives are taken with respect to the $y$ variable.
Similarly, we have from \eqref{eq:tilde:w0:4:derivative}, the chain rule, and the fact that $\ell \ll 1$,   that for $\abs{y} \leq \ell$,
\begin{align}
 \abs{\p^\gamma \tilde W(y,-\log \eps)} &\leq \eps^{\frac 18} \qquad  \mbox{ for } \abs{\gamma}=4
\label{eq:tilde:W:4:derivative}
\, ,
\end{align}
while from \eqref{eq:tilde:w0:3:derivative} we deduce that at $y=0$, we have
\begin{align}
\abs{\p^\gamma \tilde W(0,-\log \eps)} \leq \eps^{\frac 12- \frac{4}{2k-7}}
\qquad \mbox{ for }  \abs{\gamma}=3 \, .
\label{eq:tilde:W:3:derivative:0}
\end{align}
For $y$ in the region  $\{ \abs{y} \geq \LLL \} \cap \XXX_0$, from \eqref{eq:w0:far:out}, \eqref{eq:ss:y:initial}, and \eqref{eq:ss:WZA:initial},
we deduce that
\begin{subequations}
\begin{align}
 \eta^{-\frac{1}{6}}(y)\abs{W(y,-\log \eps)} &\leq  1+ \eps^ {\frac{1}{11}} 
 \label{eq:rio:de:caca:1}
\\
  \eta^{\frac 13}\left(y\right) \abs{\p_1 W(y,-\log \eps)} &\leq  1 +  \eps^{\frac{1}{12}} 
   \label{eq:rio:de:caca:2}
\\
\abs{\check \nabla W(y,-\log \eps)} &\leq   \tfrac 34 
 \label{eq:rio:de:caca:3}
\end{align}
\end{subequations}
while for the second derivatives of $W$, globally for all $y \in \XXX_0$ we obtain from \eqref{eq:w0:pipi}, \eqref{eq:ss:y:initial}, and \eqref{eq:ss:WZA:initial} that 
\begin{subequations}
\begin{align}
\eta^{\frac 13}(y) \abs{\p^\gamma W(y,-\log \eps)} &\leq 1 \qquad  \mbox{ for }   \gamma_1 \geq 1 \mbox{ and } \abs{\gamma}=2 
\label{eq:W:gamma=2:p1}
\\
\eta^{\frac 16}(y) \abs{\check\nabla^2 W(y,-\log \eps)} &\leq 1
\label{eq:W:gamma=2}
\,.
 \end{align}
 \end{subequations}

\begin{remark}
A comment regarding the introduction of the parameter $\LLL$ is in order. By \eqref{ss-ic-1} we know that $W$ and $\nabla W$ closely track  $\bar W$ and $\nabla \bar W$ for all $y$ such that $\abs{y} \leq \LLL = \eps^{-\frac{1}{10}}$. But the functions $\bar W$ and $\check \nabla \bar W$ do not decay as $\abs{y}\to\infty$ (we only have the bounds \eqref{eq:bar:W:properties} available), and thus neither do $W$ and $\check \nabla W$. At first sight this may seem contradictory with the fact that \eqref{eq:support-init} imposes that $W$ is supported in the set $\XXX(0)$. However, no contradiction ensues: we have chosen $\LLL$ to be a {\em sufficiently small}  power of $\eps^{-1}$ exactly in order to leave enough distance from the boundary of the set $\{y \colon \abs{y} \leq \LLL\}$ to the boundary of the set $\XXX(0)^c$, so that $W$ and $\check \nabla W$ have enough room to attain their compact support.
\end{remark}
 
For the initial conditions of $Z$ and $A$ we   deduce from \eqref{eq:support-init-x}, \eqref{eq:z0:ic}, \eqref{eq:a0:ic}, \eqref{eq:ss:y:initial}, and \eqref{eq:ss:WZA:initial} that
\begin{align}
 \abs{\partial^{\gamma} Z(y,-\log \eps)} &
\leq  \begin{cases}
  \eps^{\frac 32},  &\mbox{if } \gamma_1\geq 1\mbox{ and } \abs{  \gamma}=1,2\\
  \eps , & \mbox{if } \gamma_1=0\mbox{ and } \abs{\check \gamma} =0,1,2
\end{cases} \,, \label{eq:Z_bootstrap:IC}\\
 \abs{\partial^{\gamma} A(y,-\log \eps)}&
\leq  \begin{cases}
  \eps^{\frac 32},  &\mbox{if } \gamma_1= 1\mbox{ and } \abs{\check \gamma}=0\\
  \eps , & \mbox{if } \gamma_1=0\mbox{ and } \abs{\check \gamma}=0,1,2 
\end{cases} \,. \label{eq:A_bootstrap:IC}
\end{align}
For the initial specific vorticity in self-similar variables, we have that  
\begin{align} 
 \| \Omega_0\|_{L^ \infty } \le 1  \,.  \label{eq:svort:IC:SS}
\end{align}
Lastly, for the Sobolev norm of the initial condition,  we deduce from \eqref{Hk-xland-twist}, \eqref{eq:ss:y:initial}, and \eqref{eq:ss:WZA:initial} that
\begin{align}
\eps \snorm{W( \cdot , -\log \eps)}_{\dot H^k}^2 + \snorm{Z( \cdot , -\log \eps)}_{\dot H^k}^2 + \snorm{A( \cdot , -\log \eps)}_{\dot H^k}^2  \leq  \eps 
\label{eq:data:Hk}
\end{align}
for all $k \ge 18$.

\subsection{Statement of the main theorem in self-similar variables and asymptotic stability}

\begin{theorem}[Stability and shock formation in self-similar variables] 
\label{thm:SS}
Let $\gamma>1$, $ \alpha = {\tfrac{\gamma-1}{2}}$.  Let  $\kappa_0 = \kappa_0(\alpha)>1$ be sufficiently large.  Consider the system of equations \eqref{euler-ss} for 
$(W,Z,A)$.
Suppose that at initial (self-similar) time $s=-\log \eps$, the initial data $(W_0,Z_0,A_0) = (W,Z,A)|_{s=-\log\eps}$ are supported in the set $\XXX_0$, defined
in \eqref{eq:support-init}, and satisfy the conditions \eqref{eq:fat:cat}--\eqref{eq:data:Hk}.   In addition, let the modulation functions have initial conditions
which satisfy \eqref{mod-ic}--\eqref{ic-kappa0-phi0}.

Then, there exist  a sufficiently large $M=M(\alpha, \kappa_0)\ge 1$,   and a sufficiently small 
$\eps = \eps(\alpha,\kappa_0,M) \in (0,1)$, and unique global-in-time solutions $(W,Z,A)$ to \eqref{euler-ss}; moreover,     $(W,Z,A)$ are supported in the time-dependent cylinder $\XXX(s)$  defined in \eqref{eq:support},
  $(W,Z,A) \in C([-\log\eps, + \infty ); H^k)\cap C^1([-\log\eps, + \infty ); H^{k-1})$ for $k\ge 18$, and we have
\begin{align*}
 \snorm{W( \cdot , s)}_{\dot H^k}^2 + e^{s}  \snorm{Z( \cdot , s)}_{\dot H^k}^2 + e^{s} \snorm{A( \cdot , s)}_{\dot H^k}^2
 &\leq  \lambda ^{-k} e^{-s-\log\eps} +  (1 -  e^{-s-\log\eps}  ) M^{4k}   \,,
\end{align*}
for a constant $\lambda=\lambda(k) \in (0,1)$.
The Riemann function $W(y,s)$ remains   close to the generic and stable self-similar blow up profile $\bar W$; upon defining the weight function
$\eta(y) = 1 + y_1^2 + \abs{\check y}^6$, we have that 
the perturbation $\tilde W = W - \bar W$ satisfies
$$ \abs{ \tilde W(y,s)} \leq \eps^{\frac{1}{11}} \eta^{\frac 16}(y) \,, \ \ 
\abs{\p_1 \tilde W(y,s)} \leq \eps^{\frac{1}{12}} \eta^{-\frac 13}(y)\,, \ \ 
\abs{\check \nabla \tilde W(y,s)} \leq \eps^{\frac{1}{13}} \,, $$
for all $\abs{y} \le \eps^{-\frac{1}{10}}$ and $s \ge -\log\eps$. Furthermore, $\p^\gamma \tilde W(0,s) =0$ for all $\abs{\gamma} \le 2$, and the bounds \eqref{eq:tildeW_decay2} and \eqref{eq:bootstrap:Wtilde3:at:0} hold. Additionally,   $W(y,s)$ satisfies the bounds given  in  \eqref{eq:W_decay} and \eqref{eq:W:GN}.
  
The limiting function  $\bar W_{\mathcal A}(y) = \lim_{s \to + \infty } W(y,s)$
is  a well-defined {\it blow up profile}, with the following properties:
\begin{itemize}[itemsep=2pt,parsep=2pt,leftmargin=.15in]
\item $\bar W_{\mathcal A}$ is a  $C^\infty$ smooth solution to the self-similar 3D Burgers equation~\eqref{eq:Burgers:self:similar}, which satisfies the bounds \eqref{eq:W_decay} and \eqref{eq:W-L2}.
\item $\bar W_{\mathcal A}(y)$ satisfies the same  genericity condition as $\bar W$ given by \eqref{genericity}.
\item $\bar W_{\mathcal A}$ is uniquely determined by the $10$ parameters: ${\mathcal A}_\alpha = \lim_{s\to\infty} \p^\alpha W(0,s)$ with $\abs{\alpha}=3$.
\end{itemize}

The amplitude of the functions $Z$ and $A$ remains $\OO(\eps)$ for all $s \ge -\log\eps$, while for each $ \abs{\gamma}\le k$, $\p^\gamma Z( \cdot , s) \to 0$ and 
$\p^\gamma A( \cdot , s) \to 0$ as $s \to + \infty $, and $Z$ and $A$ satisfy the bounds \eqref{eq:Z_bootstrap} and \eqref{eq:A_bootstrap}.

The scaled sound speed $S(y,s)$ in self-similar variables satisfies
$$
\snorm{S( \cdot , s) - \tfrac{\kappa_0}{2} }_{ L^ \infty } \le \eps^ {\frac{1}{8}}  \ \ \text{ for all } \ s \ge -\log\eps  \,,
$$
and for a universal constant $C_0$,  the specific vorticity $\Omega(y,s)$ in self-similar variables satisfies
$$
\tfrac{1}{C_0} \norm{\Omega_0(y_0)}^2 \leq \abs{\Omega(\pu^{y_0}(s),s)}^2  \leq C_0\abs{\Omega_0(y_0)}^2 \,,
$$
where $\pay$ is defined in \eqref{traj}.
\end{theorem}

\section{Bootstrap assumptions}
\label{sec:bootstrap}

As discussed above, the proof of Theorem~\ref{thm:SS} consists of a bootstrap argument, which we make precise in this section.  For $M$ sufficiently large, depending on $\kappa_0$ and on $\alpha$, and for $\eps$ sufficiently small, depending on $M$, $\kappa_0$, and $\alpha$, we postulate that the modulation functions are bounded as in \eqref{mod-boot}, that $(W,Z,A)$ are supported in the set given by \eqref{eq:support}, that $W$ satisfies  \eqref{eq:W_decay},  $\tilde W$ obeys \eqref{eq:tildeW_decay}--\eqref{eq:bootstrap:Wtilde3:at:0} and $Z$ and $A$ are bounded as in \eqref{eq:Z_bootstrap} and \eqref{eq:A_bootstrap} respectively. All these bounds have explicit constants in them. Our goal in subsequent sections will be to show that the these estimates in fact hold with strictly better pre-factors, which in view of a continuation argument yields the proof of Theorem~\ref{thm:SS}.

\subsection{Dynamic variables}
For the dynamic modulation variables,  we assume that 
\begin{subequations}
\label{mod-boot}
\begin{align}
&\tfrac 12 \kappa_0 \leq  \kappa(t)  \leq  2 \kappa_0, &&\abs{\tau(t)} \leq M \eps^2, &&&\abs{\xi(t)} \leq  M^{\frac 14}  \eps   , &&&& \abs{\check n(t)} \leq   M^2\eps^{\frac 32},    &&&&& \abs{\phi(t)} \leq M^2 \eps, 
\label{eq:speed:bound} \\
&\abs{\dot \kappa(t)} \leq  M^2  e^{-\frac{s}{2}}, &&\abs{\dot \tau(t)} \leq  M e^{-s}, &&&|\dot \xi(t)| \leq  M^{\frac 14} , &&&& |\dot{\check{n}}(t)| \leq M^2 \eps^{\frac 12}, &&&&& |\dot{\phi}(t)| \leq M^2 ,
\label{eq:acceleration:bound}
\end{align}
\end{subequations}
for all $-\eps \leq t < T_*$.

From \eqref{eq:Q:def}, \eqref{eq:Q2:def}--\eqref{eq:Q3:def}, and the bootstrap assumptions \eqref{mod-boot}, we directly obtain that
\begin{align}
|\dot Q(t)|  \leq 2 M^2 \eps^{\frac 12}
\label{eq:dot:Q}
\end{align}
for all $-\eps \leq t < T_*$. Moreover, we note that as a direct consequence of the $\dot \tau$ estimate in \eqref{eq:acceleration:bound},  we have that 
\begin{align}
\abs{1-\beta_\tau} = \frac{\abs{\dot \tau}}{1-\dot\tau} \leq 2 M  e^{-s} \leq 2 M  \eps
\label{eq:beta:tau} 
\end{align}
since $\eps$ can be made sufficiently small, for all $s\geq -\log \eps$. 

\subsection{Spatial support bootstrap}
We now make the following bootstrap assumption  that  $(W,Z,A)$ have support in the $s$-dependent cylinder defined by
\begin{align}
\XXX(s):=\left\{\abs{y_1}\leq  2\eps^{\frac12}  e^{\frac 32s},\abs{\check y}\leq  2 \eps^ {\frac{1}{6}} e^{\frac{s}{2}}  \right\} \text{ for all } s \ge -\log\eps \,.
\label{eq:support}
\end{align}
Recall from \eqref{eq:eta:def} and \eqref{eq:tilde:eta:def} 
 the definition of the weight functions
\begin{align*}
\eta(y) = 1 + y_1^2 + \abs{\check y}^6 \qquad \mbox{and} \qquad \tilde \eta(y) = \eta(y) +  \abs{\check y}^2  \, .
\end{align*}
Using these, for $y \in \XXX(s)$, we have the estimate
\begin{equation}\label{e:space_time_conv}
\eta(y) \leq 40 \eps e^{3s} \qquad \Leftrightarrow \qquad \eta^{\frac{1}{3}}(y) \leq 4 \eps^{\frac 13}  e^{s}
 \end{equation}
for all $y \in \RR^3$, which allows us to convert temporal decay to spatial decay.

\subsection{$W$ bootstrap}
We postulate the following derivative estimates on $W$ 
\begin{align}
 \abs{\partial^{\gamma} W(y,s)}&
\leq   \begin{cases}
(1+ \eps^{\frac{1}{20}}) \eta^{\frac 16}(y), & \mbox{if } \abs{\gamma}  = 0\,,  \\
\tilde \eta^{-\frac 13}\left(\tfrac y2\right) {\bf 1}_{\abs{y}\leq \LLL} + 2 \eta^{-\frac 13}(y) {\bf 1}_{\abs{y}\geq \LLL}, & \mbox{if }  \gamma_1 = 1 \mbox{ and } \abs{\check \gamma}=0\,, \\
1, & \mbox{if } \gamma_1=0 \mbox{ and }  \abs{\check \gamma} = 1\,,\\
M^{\frac{1+\abs{\check \gamma}}{3}} \eta^{-\frac 13}(y), & \mbox{if }  \gamma_1 \geq 1 \mbox{ and } \abs{\gamma}=2\,, \\
M \eta^{-\frac{1}{6}}(y), & \mbox{if } \gamma_1 = 0 \mbox{ and } \abs{\check \gamma} = 2 \,.
\end{cases}\label{eq:W_decay}
\end{align}
Next, we assume that the solution $W(y,s)$ remains close to the self-similar profile $\bar W(y)$  in the topology defined by the following bounds. For this purpose, it is convenient to state bootstrap assumptions in terms of $\tilde W$, as defined in \eqref{eq:tilde:W:def}. 
For $\abs{y}\leq \LLL$, we assume that 
\begin{subequations}
\label{eq:tildeW_decay}
\begin{align}
\abs{ \tilde W(y,s)} &\leq \eps^{\frac{1}{11}} \eta^{\frac 16}(y) \,,
\label{eq:bootstrap:Wtilde} \\
\abs{\p_1 \tilde W(y,s)} &\leq \eps^{\frac{1}{12}} \eta^{-\frac 13}(y) \,,
\label{eq:bootstrap:Wtilde1} \\
\abs{\check \nabla \tilde W(y,s)} &\leq \eps^{\frac{1}{13}}  \,,
\label{eq:bootstrap:Wtilde2} 
\end{align}
\end{subequations}
where the parameter $\LLL$ is as defined in \eqref{eq:pounds}. 
Furthermore, for $\abs{y} \leq \ell$ we assume that
\begin{subequations}
\label{eq:tildeW_decay2}
\begin{align}
\abs{\p^\gamma \tilde W(y,s)} &\leq (\log M)^4 \eps^{\frac {1}{10}}\abs{y}^{4-\abs{\gamma}}+M\eps^{\frac 14} \abs{y}^{3-\abs{\gamma}} \leq 2 (\log M)^4 \eps^{\frac{1}{10}}\ell^{4-\abs{\gamma}}\,, && \mbox{for all }   \abs{\gamma} \leq 3\,,
\label{eq:bootstrap:Wtilde:near:0} \\
\abs{\p^\gamma \tilde W(y,s)} &\leq \eps^{\frac{1}{10}} (\log M)^{\abs{\check \gamma}}\,, && \mbox{for all }    \abs{\gamma} = 4\,,
\label{eq:bootstrap:Wtilde4}
\end{align}
\end{subequations}
while at $y=0$,  we assume that
\begin{align}
\abs{\p^\gamma \tilde W(0,s)} &\leq \eps^{\frac14} \,, && \mbox{for all} \qquad \abs{\gamma} = 3\,,\label{eq:bootstrap:Wtilde3:at:0}
\end{align} 
for all $s\geq -\log \eps$. 
In \eqref{eq:bootstrap:Wtilde:near:0} and \eqref{eq:bootstrap:Wtilde4}, the parameter $\ell$ is chosen  as in \eqref{eq:ell:choice}.
Note that with this choice of $\ell$,  the bounds \eqref{eq:kingfisher}, \eqref{eq:kingfisher:2}, and \eqref{eq:kingfisher:3} hold.

\begin{remark}\label{goodremark1}
In the region $\abs{y} \leq \LLL$, the first three bounds stated in \eqref{eq:W_decay} follow directly from the   properties of $\bar W$ stated in \eqref{eq:bar:W:properties}, and those of $\tilde W$ in \eqref{eq:tildeW_decay}. The bounds for $W$ and $\check \nabla W$ are immediate. The estimate for $\p_1 W$ is a bit more delicate and uses the explicit bound $ \tilde \eta^{-\frac 13}(y) + \eps^{\frac{1}{12}} \eta^{-\frac 13}(y) \leq \tilde \eta^{-1/3}(y/2)$. 
\end{remark}

\begin{lemma}[Lower bound for $\Jcal \p_1 W$]\label{lem:Jp1W}
\begin{align} 
\Jcal \p_1 W(y,s) \geq -1 \quad \mbox{and} \quad \Jcal \p_1 \bar W(y,s) \geq -1 \ \ \text{ for all } y \in \mathbb{R}^3 \,, s \ge -\log \eps    \,. \label{Jp1W}
\end{align} 
\end{lemma} 
\begin{proof}[Proof of Lemma~\ref{lem:Jp1W}]
By the definition of $\Jcal$ and the bootstrap assumption  \eqref{eq:speed:bound} and \eqref{eq:support}, we have 
 \[
0 \leq \Jcal - 1 = \frac{\Jcal^2 - 1}{\Jcal +1} = \frac{1}{\Jcal+1} \left( (\phi_{2\nu} e^{-\frac s2} y_\nu)^2 + (\phi_{3\nu} e^{-\frac s2} y_\nu)^2 \right) \leq  \eps  e^{-s} \abs{\check y}^2 \leq  \eps .
 \]
 Moreover, using \eqref{eq:bar:W:properties} for the function $\p_1 \bar W$ and \eqref{eq:W_decay} for $\p_1 W$,  we deduce that 
 \[
\min\left\{ 1+ \p_1 \bar W , 1+ \p_1   W \right\} \geq 1 - \tilde \eta^{-\frac 13}\left(\tfrac y2\right)  \geq \frac{\abs{\check y}^2}{20(1+\abs{\check y}^2)}
 \]
 for all $y\in \RR^3$. The last inequality follows from an explicit computation. To conclude, we write 
 \begin{align*}
\min\left\{ 1+ \Jcal \p_1 \bar W , 1+ \Jcal \p_1   W \right\} 
&\geq \min\left\{ 1+ \p_1 \bar W , 1+ \p_1   W \right\} - \abs{\Jcal -1}  \notag\\
&\geq \frac{\abs{\check y}^2}{20(1+\abs{\check y}^2)} - \eps e^{-s} \abs{\check y}^2 \geq 0\,,
 \end{align*}
 thereby finishing the proof.
\end{proof}

\subsection{$Z$ and $A$ bootstrap}
We postulate the following derivative estimates on $Z$ and $A$:
\begin{align}
 \abs{\partial^{\gamma} Z(y,s)} &
\leq  \begin{cases}
 M^{\frac{1+\abs{\check\gamma}}{2}}  e^{-\frac 32 s },  &\mbox{if } \gamma_1\geq 1\mbox{ and } \abs{  \gamma}=1,2\\
M \eps^{ \frac{2-\abs{\check \gamma}}2 } e^{-\frac {\abs{\check \gamma}}2 s}, & \mbox{if } \gamma_1=0\mbox{ and } \abs{\check \gamma} =0,1,2 \, ,
\end{cases}\label{eq:Z_bootstrap}\\
 \abs{\partial^{\gamma} A(y,s)}&
\leq  \begin{cases}
M e^{-\frac 32 s },  &\mbox{if } \gamma_1= 1\mbox{ and } \abs{\check \gamma}=0\\
M \eps^{ \frac{2-\abs{\check \gamma}}2 } e^{-\frac {\abs{\check \gamma}}2 s}, & \mbox{if } \gamma_1=0\mbox{ and } \abs{\check \gamma}=0,1,2 \, .
\end{cases}\label{eq:A_bootstrap}
\end{align}

\subsection{Further consequences of the bootstrap assumptions}
The bootstrap bounds  \eqref{mod-boot}, \eqref{e:space_time_conv}, \eqref{eq:W_decay}--\eqref{eq:bootstrap:Wtilde3:at:0}, \eqref{eq:Z_bootstrap}, and \eqref{eq:A_bootstrap} have a number of consequences, which we collect here for future reference. The first is a global in time $L^2$-based Sobolev estimate:
\begin{proposition}[$\dot H^k$ estimate for $W$, $Z$, and $A$]\label{cor:L2}
For  integers  $k \geq 18$ and for a constant $\lambda = \lambda(k)$, 
\begin{subequations} 
\begin{align}
\snorm{Z( \cdot , s)}_{\dot H^k}^2 + \snorm{A( \cdot , s)}_{\dot H^k}^2
 &\leq 2 \lambda ^{-k}e^{-s} +e^{-s} (1 -  e^{-s}\eps^{-1} ) M^{4k}   \,, \label{eq:AZ-L2} \\
 \snorm{W( \cdot , s)}_{\dot H^k}^2
 &\leq   2 \lambda^{-k} \eps^{-1} e^{-s}  +(1 -  e^{-s}\eps^{-1}) M^{4k} \,, \label{eq:W-L2} 
\end{align}
\end{subequations} 
for all $s\ge  -\log \eps$.
\end{proposition}
The proof of Proposition~\ref{cor:L2}, which will be given  at the end of Section \ref{sec:energy},  relies only upon the initial data assumption \eqref{eq:data:Hk}, on the support bound \eqref{e:space_time_conv},  on  $L^ \infty $ estimates for $\p^\gamma W$ and $\p^\gamma Z$ when $\abs{\gamma} \le 2$, on $\p^\gamma A$ pointwise bounds for $\abs{\gamma}\leq 1$, and on $\check \nabla^2 A$ bounds. That is, Proposition~\ref{cor:L2} follows directly from \eqref{eq:data:Hk} and the bootstrap assumptions \eqref{mod-boot}, \eqref{e:space_time_conv},  \eqref{eq:W_decay}, \eqref{eq:Z_bootstrap},  and \eqref{eq:A_bootstrap}.  

The reason we state Proposition~\ref{cor:L2} at this stage of the analysis  is that the $\dot{H}^k$ estimates and linear interpolation yield useful information for higher order 
derivatives of $(W,Z,A)$, which are  needed in order to close the bootstrap assumptions for high order derivatives. These bounds are summarized in the following
\begin{lemma}\label{eq:higher:order:ests}
For integers $k \ge 18$, we have that
\begin{align}\label{eq:A:higher:order}
 \abs{\partial^{\gamma} A(y,s)}&
\les  \begin{cases}
 e^{-( \frac 32  - \frac{2\abs{\gamma}-1}{2k-5})s} ,  &\mbox{if } \gamma_1\geq 1\mbox{ and } \abs{ \gamma}=2,3\\
  e^{- (1 - \frac{\abs{\gamma}-1}{2k-7})s}    , & \mbox{if } \abs{\gamma}=3,4,5 \, ,
\end{cases} 
\end{align}

\begin{align}\label{eq:Z:higher:order}
 \abs{\partial^{\gamma} Z(y,s)}&
\les  \begin{cases}
 e^{- (\frac32 -\frac{3}{2k-7})s} ,  &\mbox{if } \gamma_1\geq 1\mbox{ and } \abs{ \gamma}=3\\
 e^{- (1 - \frac{\abs{\gamma}-1}{2k-7})s}    , & \mbox{if } \abs{\gamma}=3,4,5 \, ,
\end{cases} 
\end{align}
\begin{align}
 \abs{\partial^{\gamma} W(y,s)}&
\les  \begin{cases}
 e^{\frac{2s}{2k-7}} \eta^{-\frac 13}(y) ,  &\mbox{if } \gamma_1\neq 0\mbox{ and } \abs{ \gamma}=3\\
 e^{\frac{s}{2k-7}} \eta^{-\frac 16}(y) ,  &\mbox{if } \gamma_1=0\mbox{ and } \abs{ \gamma}=3 \,.
\end{cases} \label{eq:W:GN}
\end{align}
\end{lemma}
\begin{proof}[Proof of Lemma~\ref{eq:higher:order:ests}]
First, we consider the case $\gamma_1 \geq 1$ and $\abs{\gamma} \in \{2,3\}$.
By Lemma~\ref{lem:GN} (applied to the function $\p_1 A$), \eqref{eq:A_bootstrap}, and Proposition~\ref{cor:L2},
\begin{align}
\norm{\partial^{\gamma} A}_{L^\infty} 
&\les \norm{A}_{\dot H^{k}}^{\frac{2\abs{\gamma}-2}{2k-5}} \norm{\partial_1 A}_{L^\infty}^{\frac{2k - 3 -2\abs{\gamma}}{2k-5}}
\les  \left(M^{2k} e^{-\frac s2} \right)^{\frac{2\abs{\gamma}-2}{2k-5}}  \left(M e^{-\frac32 s}\right)^{\frac{2k -3 -2\abs{\gamma}}{2k-5}}
\les M^{2k} e^{-( \frac 32  - \frac{2|\gamma|-2}{2k-5})s}\notag  \\
&
\les M^{2k} \eps^ {\frac{1}{2k-5}}   e^{-( \frac 32  - \frac{2|\gamma|-1}{2k-5})s}  \les e^{-( \frac 32  - \frac{2|\gamma|-1}{2k-5})s} \,, \label{Aest00}
\end{align}
where we have taken $\eps$ sufficiently small for the last inequality.
Similarly, for $\abs{\gamma} \in \{3,4,5\}$ we apply Lemma~\ref{lem:GN} to $\nabla^2 A$;  together,  \eqref{eq:A_bootstrap} and \eqref{Aest00} provide
bounds for $ \nabla ^2 A$, and hence we find that
\begin{align*}
\norm{\partial^{\gamma}  A}_{L^\infty}
\les \norm{A}_{\dot H^{k}}^{\frac{2\abs{\gamma}-4}{2k-7}} \norm{\nabla^2 A}_{L^\infty}^{\frac{2k -3 -2\abs{\gamma}}{2k-7}}
\les  \left(M^{2k} e^{-\frac s2} \right)^{\frac{2\abs{\gamma}-4}{2k-7}}   \left( M e^{-s}\right)^{\frac{2k -3 -2\abs{\gamma}}{2k-7}}
\les  M^{2k} e^{- (1 - \frac{\abs{\gamma}-2}{2k-7})s} \, .
\end{align*}
For the estimate of $\p^\gamma Z$,  in the case $\gamma_1 \geq 1$ and $\abs{\check \gamma} =3$, we have that
\begin{align*}
\norm{\partial^{\gamma}  Z}_{L^\infty}
&\les \norm{Z}_{\dot H^{k}}^{\frac{2}{2k-7}} \norm{\partial_1 \nabla Z}_{L^\infty}^{\frac{2k -9}{2k-7}}
\les \left(M^{2k} e^{-\frac s2} \right)^{\frac{2}{2k-7}}     \left( Me^{-\frac32 s}\right)^{\frac{2k -9}{2k-7}}
\les  M^{2k} e^{- (\frac32 -\frac{2}{2k-7})s}  \notag \\
& \les  M^{2k}  \eps^ {\frac{1}{2k-7}} e^{- (\frac32 -\frac{3}{2k-7})s}   \les e^{- (\frac32 -\frac{3}{2k-7})s} \,,
\end{align*}
where we have again absorbed  $M^{2k}$ using $\eps^ {\frac{1}{2k-7}}$.
The second estimate for $\p^\gamma Z$ in \eqref{eq:Z:higher:order} for the case that  $\abs{\gamma} \in \{3,4,5\}$ is completely analogous to the corresponding 
estimate for  $\p^\gamma A$.

We next estimate $\abs{\partial^{\gamma}W}$ for $\abs{\gamma}=3$.   To do so, we decompose $\gamma=\gamma'+\gamma''$ such that $\abs{\gamma'}=1$ 
and $\abs{\gamma''}=2$, and  further assume that $\gamma_1''=\min(\gamma_1,2)$.  In order to apply the Gagliardo-Nirenberg inequality, we rewrite 
\begin{align*}
\eta^{\mu}  \partial^{\gamma}  W 
= \eta^{\mu}  \partial^{\gamma'} \p^{\gamma''} W
=\underbrace{\partial^{\gamma'}\left( \eta^{ \mu}\;  \p^{\gamma''}W\right)}_{=:I}
-\underbrace{\partial^{\gamma'} \eta^{ \mu} \; \p^{\gamma''}W}_{=:II}
\end{align*}
and we set $\mu=\sfrac16$ for the case $\gamma_1=0$ and $\mu=\sfrac 13$ otherwise. 
Since $\abs{\p_1 \eta^\mu} \les \eta^{\mu - \frac 12}$ and $ \abs{\check \nabla \eta^\mu} \les \eta^{\mu - \frac 16}$, it immediately follows from  \eqref{eq:W_decay} that
\begin{align*}
\abs{II}\les M\,.
\end{align*}
Now we apply Lemma~\ref{lem:GN} to the function $\eta^{ \mu}  \p^{\gamma''}W$, appeal to the estimate \eqref{eq:W_decay}, and to the Leibniz rule to obtain
\begin{align*}
\abs{I} 
\les \norm{\eta^{ \mu}\p^{\gamma''}W}_{\dot H^{k-2}}^{\frac{2}{2k-7}}\norm{\eta^{\mu}\p^{\gamma''}W}_{L^{\infty}}^{\frac{2k-9}{2k-7}} 
&\les M \norm{\eta^{\mu}\p^{\gamma''}W}_{\dot H^{k-2}}^{\frac{2}{2k-7}}\, ,
\end{align*}
where we have used that $k\ge 18$ for the last inequality as is required by Proposition~\ref{cor:L2}.   We next estimate the $\dot H^{k-2}$ norm of $\eta^{\mu}\p^{\gamma''}W$.   To do so, we shall use the fact that
$W(\cdot,s)$ has support in the set $\XXX(s)$ defined in \eqref{eq:support}. 
From the Leibniz rule and \eqref{eq:special1}, we obtain
\begin{align*}
\norm{\eta^{\mu}\p^{\gamma''}W}_{\dot H^{k-2}} 
&\les \sum_{m=0}^{k-2}\norm{D^{k-m-2}\left(\eta^\mu\right)D^m\p^{\gamma''}W}_{L^2} \notag\\
&\les \sum_{m=0}^{k-2}\norm{D^{k-m-2}\left(\eta^\mu\right)}_{L^{\frac{2(k-1)}{k-2-m}}(\XXX(s))} \norm{D^m\p^{\gamma''}W}_{L^\frac{2(k-1)}{m+1}} \notag\\
&\les \sum_{m=0}^{k-2}\norm{D^{k-m-2}\left(\eta^\mu\right)}_{L^{\frac{2(k-1)}{k-2-m}}(\XXX(s))} \norm{\nabla W}_{L^\infty}^{1-\frac{m+1}{k-1}} \norm{W}_{\dot{H}^k}^{\frac{m+1}{k-1}} \, .
\end{align*}
Using \eqref{eq:W_decay} and Proposition~\ref{cor:L2}, the $W$ terms are bounded as 
\[
\norm{\nabla W}_{L^\infty}^{1-\frac{m+1}{k-1}} \norm{W}_{\dot{H}^k}^{\frac{m+1}{k-1}}  \les  M^{2k}
\]
for all $m \in \{0,\dots,k-2\}$.
Moreover, applying \eqref{e:space_time_conv}, and using that $k\geq 18$ we have 
\[\norm{D^{k-m-2}( \eta^{\mu} )}_{L^{\frac{2(k-1)}{k-m-2}}(\XXX(s))}\les \eps^{\mu} e^{3\mu s} \]
with the usual abuse of notation $L^{\frac{2(k-1)}{k-m-2}}=L^{\infty}$ for $m=k-2$. Combining the above estimates, we obtain the inequality
\begin{align*}
\abs{I} &\les M^{2k} \left( \eps^{\mu} e^{3\mu s} \right)^{\frac{2}{2k-7}} \les  e^{\frac{6\mu s}{2k-7}} 
\end{align*}
for $\eps$  sufficiently small, since $\mu \geq \frac 16$. From the above estimate the bound  \eqref{eq:W:GN} immediately follows.
\end{proof}

Finally, we note that as a consequence of the definitions \eqref{eq:UdotN:S}, the following estimates on $U\cdot \Ncal $ and $S$.

\begin{lemma}\label{lem:US_est}
For $y \in \XXX(s)$ we have
\begin{equation}
\abs{\p^\gamma U\cdot\Ncal}+\abs{\p^\gamma S}
\les  \begin{cases}
 M^{\frac 14} , & \mbox{if } \abs{\gamma}  = 0  \\
 M^{\frac{1+\abs{\check \gamma}}{3}} e^{-\frac s2}\eta^{-\frac13}(y), & \mbox{if }  \gamma_1 \geq 1 \mbox{ and } \abs{ \gamma}=1,2 \\
e^{-\frac s2}, & \mbox{if } \gamma_1=0 \mbox{ and }  \abs{\check \gamma} = 1\\
 M  e^{-\frac s2}\eta^{-\frac16}(y), & \mbox{if } \gamma_1 = 0 \mbox{ and } \abs{\check \gamma} = 2\\
e^{\left(-\frac12+\frac{3}{2k-7}\right)s} \eta^{-\frac13}(y),  &\mbox{if } \gamma_1\neq 0\mbox{ and } \abs{ \gamma}=3\\
e^{\left(-\frac12+\frac{2}{2k-7}\right)s}\eta^{-\frac16}(y) ,  &\mbox{if } \gamma_1=0\mbox{ and } \abs{ \gamma}=3
\end{cases}
\label{eq:US_est}
\end{equation}
while for $\abs{y} \leq \ell$ and $\abs{\gamma} = 4$ we have
\[
 \abs{\p^\gamma U\cdot\Ncal}+\abs{\p^\gamma S} \les  e^{- \frac{s}{2}}\,. 
\]
\end{lemma}
\begin{proof}[Proof of Lemma~\ref{lem:US_est}]
We consider the estimates on $\p^\gamma U\cdot\Ncal$. The estimates on $\p^\gamma S$ are completely analogous. By definition \eqref{eq:UdotN:S}
\begin{align*}
\abs{\p^\gamma U\cdot\Ncal}\les \abs{\kappa} {\bf 1}_{\abs{\gamma}=0}+e^{-\frac s 2}\abs{\p^\gamma W}+\abs{\p^\gamma Z} \, .
\end{align*}
Here we used $\abs{\kappa} \leq M^{\frac 14}$. Now we simply apply \eqref{eq:W_decay}, \eqref{eq:bootstrap:Wtilde4}, \eqref{eq:Z_bootstrap}, Lemma \ref{eq:higher:order:ests} and \eqref{e:space_time_conv} to conclude.
\end{proof}

\section{Constraints and evolution of modulation variables}
\label{sec:constraints} 

\subsection{Constraints}
The shock is characterized by the following ten constraints on $W$, which we impose throughout the evolution, by suitably choosing our dynamic modulation variables
\begin{align} 
W(0,s)=0 \,, \quad  \p_{1} W(0,s)=-1\,, \quad \check \nabla W(0,s)=0\,,  \quad
\nabla^2 W(0,s)=0 \,. 
\label{eq:constraints}
\end{align} 
These constraints are maintained under the evolution by suitably choosing our ten time-dependent modulation parameters: $n_2, n_3, \xi_1,\xi_2,\xi_3, \kappa, \tau, \phi_{22},\phi_{23}$ and $\phi_{33}$.

\subsection{Evolution of dynamic modulation variables}
The ten modulation parameters at time $t=-\eps$ are defined as
\begin{align}
\kappa(-\eps) = \kappa_0, \qquad \tau(-\eps) = \xi(-\eps) = n_\mu (-\eps) =  0, \qquad  \phi_{\nu\mu}(-\eps) = \phi_{0,\nu \mu}
\label{eq:modulation:IC}
\,,
\end{align}
where $\kappa_0$ is as in \eqref{eq:w0:power:series} and $\phi_0$ is defined by \eqref{eq:phi:0:def}.
In order to determine the definition for the time derivatives of our seven modulation parameters, we will use the explicit form of the evolution equations for $W$, $\nabla W$ and $\nabla^2 W$. These are ten equations, consistent with the fact that we have ten constraints in \eqref{eq:modulation:IC}.
For  convenience, we first state these evolution equations.

\subsubsection{The evolution equations for $\nabla W$ and $\nabla^2 W$}
From \eqref{euler_for_Linfinity:a} we deduce that the evolution equations for $\nabla W$ are
\begin{subequations} 
\label{eq:gradW}
\begin{align}
\left(\partial_s + 1 + \beta_\tau \Jcal \p_1 W  \right)\p_{1}W
+ (\beta_\tau \Jcal W + G_W + \tfrac{3y_1}{2}) \partial_{11} W  + (\tfrac{y_\mu}{2} + h_W^\mu)  \p_\mu \partial_{1} W 
&= F_{W}^{(1,0,0)} \label{eq:grad:W:a}  \\
\left(\partial_s + \beta_\tau \Jcal \partial_{1} W \right) \partial_{2}W 
+ (\beta_\tau \Jcal W + G_W + \tfrac{3y_1}{2}) \partial_{12} W+ (\tfrac{y_\mu}{2} + h_W^\mu)  \p_\mu \partial_{2} W 
&=F_W^{(0,1,0)}   \label{eq:grad:W:b} \\
\left( \p_s  +\beta_\tau \Jcal \partial_{1} W \right)   \partial_{3}W 
+ (\beta_\tau \Jcal W + G_W + \tfrac{3y_1}{2}) \partial_{13} W+ (\tfrac{y_\mu}{2} + h_W^\mu)  \p_\mu \partial_{3} W 
&=F_W^{(0,0,1)}  \label{eq:grad:W:c}
\end{align} 
\end{subequations} 
where we have denoted
\begin{subequations}
\label{eq:FW:gradient}
\begin{align}
F_{W}^{(1,0,0)} &=\partial_{1} F_W -  \p_1 G_W \p_1 W  - \p_1 h_W^\mu \p_\mu W \label{eq:FW:gradient:a} \\
F_{W}^{(0,1,0)} &=\partial_{2} F_W -  \p_2 G_W \p_1 W -  \p_2 h_W^\mu \p_\mu W  \label{eq:FW:gradient:b} \\ 
F_{W}^{(0,0,1)} &=\partial_{3} F_W -  \p_3 G_W \p_1 W -  \p_3 h_W^\mu \p_\mu W  \label{eq:FW:gradient:c} \, .
\end{align}
\end{subequations}
Applying the gradient to \eqref{eq:grad:W:a}, we arrive at the evolution equation  for $\p_1 \nabla W$, given by
\begin{subequations}
\label{eq:grad:2:W}
\begin{align} 
\left(\partial_s + \tfrac52 +3 \beta_\tau \Jcal \p_{1}W \right)\p_{11} W
+ (\beta_\tau \Jcal W + G_W + \tfrac{3y_1}{2}) \partial_{111} W + (\tfrac{y_\mu}{2} + h_W^\mu) \partial_{11\mu} W
&= F_{W}^{(2,0,0)}   \label{eq:grad:2:W:a} \\
\left(\partial_s + \tfrac{3}{2} +  2 \beta_\tau  \Jcal \p_{1} W  \right)
 \p_{12} W   
+ (\beta_\tau \Jcal W + G_W + \tfrac{3y_1}{2}) \partial_{112} W + (\tfrac{y_\mu}{2} + h_W^\mu)   \partial_{12\mu} W
&= F_W^{(1,1,0)}  \label{eq:grad:2:W:b}  \\
\left(\partial_s +  \tfrac{3}{2} +  2 \beta_\tau \Jcal \p_{1} W  \right)
 \p_{13} W   
+ (\beta_\tau \Jcal W + G_W + \tfrac{3y_1}{2}) \partial_{113} W + (\tfrac{y_\mu}{2} + h_W^\mu)  \partial_{13\mu} W
&= F_W^{(1,0,1)}    \label{eq:grad:2:W:c}
\end{align} 
\end{subequations}
where
\begin{subequations}
\label{eq:FW:D1:gradient}
\begin{align}
F_{W}^{(2,0,0)} &=\partial_{11} F_W - \p_{11} G_W \p_1 W - \p_{11} h_W^\mu \p_\mu W  - 2 \p_1 G_W \p_{11} W - 2 \p_1 h_W^\mu \p_{1\mu} W 
\\
F_W^{(1,1,0)}&=   \p_{12} F_W -   \p_{12} G_W \p_1 W  - \p_{12}  h_W^\mu \p_\mu W  -  \p_1 G_W   \p_{12} W - \p_1 h_W^\mu  \p_{2\mu} W \notag\\
&\qquad - \p_2 G_W \p_{11}  W -  \p_2 h_W^\mu  \p_{1\mu}  W   - \beta_\tau \p_2 (\Jcal W)  \p_{11} W\\
F_W^{(1,0,1)}&= \p_{13} F_W - \p_{13} G_W \p_1 W - \p_{13 }h_W^\mu \p_\mu W -  \p_1 G_W \p_{13} W  - \p_1 h_W^\mu   \p_{3\mu} W \notag\\
&\qquad  - \p_3 G_W \p_{11} W -  \p_3 h_W^\mu   \p_{1\mu} W  - \beta_\tau \p_3 (\Jcal W) \p_{11} W \, .
\end{align} 
\end{subequations}
Lastly, differentiating in the $\check \nabla$ direction equations \eqref{eq:grad:2:W:b}--\eqref{eq:grad:2:W:c} we obtain the evolution equation for $\check \nabla^2 W$ 
\begin{subequations}
\label{eq:grad:2:check:W}
\begin{align} 
\left(\partial_s  + \tfrac12 +  \beta_\tau \Jcal \p_{1}W \right)\p_{22} W
+ (\beta_\tau \Jcal  W + G_W + \tfrac{3y_1}{2}) \partial_{122} W + (\tfrac{y_\mu}{2} + h_W^\mu)  \partial_{22\mu} W
&= F_{W}^{(0,2,0)}   \label{eq:grad:2:check:W:a} \\
\left(\p_s   +  \tfrac{1}{2} +   \beta_\tau \Jcal  \p_{1} W  \right)\p_{23} W   
+ (\beta_\tau \Jcal W + G_W + \tfrac{3y_1}{2}) \partial_{123} W + (\tfrac{y_\mu}{2} + h_W^\mu)  \partial_{23\mu} W
&= F_W^{(0,1,1)}  \label{eq:grad:2:check:W:b}  \\
\left( \p_s  +   \tfrac{1}{2} +    \beta_\tau \Jcal  \p_{1} W  \right)\p_{33} W   
+ (\beta_\tau \Jcal  W + G_W + \tfrac{3y_1}{2}) \partial_{133} W + (\tfrac{y_\mu}{2} + h_W^\mu)  \partial_{33\mu} W
&= F_W^{(0,0,2)}    \label{eq:grad:2:check:W:c}
\end{align} 
\end{subequations}
where
\begin{subequations}
\label{eq:FW:gradient:check:squared}
\begin{align}
F_{W}^{(0,2,0)} &=\partial_{22} F_W - \p_{22} G_W \p_1 W - \p_{22} h_W^\mu \p_\mu W  - 2 \p_2 G_W \p_{12} W - 2 \p_2 h_W^\mu \p_{2\mu} W \notag\\
&\qquad - 2 \beta_\tau \p_2 (\Jcal W) \p_{12} W \\
F_W^{(0,1,1)}&=   \p_{23} F_W -   \p_{23} G_W \p_1 W  - \p_{23}  h_W^\mu \p_\mu  W -  \p_3 G_W   \p_{12} W - \p_3 h_W^\mu  \p_{2\mu} W \notag\\
&\qquad - \p_2 G_W \p_{13}  W -  \p_2 h_W^\mu  \p_{3\mu}  W   - \beta_\tau \p_3 (\Jcal W) \p_{12} W - \beta_\tau \p_2 (\Jcal W) \p_{13} W \\
F_W^{(0,0,2)}&= \p_{33} F_W - \p_{33} G_W \p_1 W - \p_{33} h_W^\mu \p_\mu  W-  2 \p_3 G_W \p_{13} W  - \p_3 h_W^\mu  \p_{3\mu} W \notag\\
&\qquad - 2 \beta_\tau \p_3 (\Jcal W) \p_{13} W   \, .
\end{align}
\end{subequations}

\subsubsection{The functions $G_W, h_W, F_W$ and their derivatives, evaluated at $y=0$}
\label{sec:explicit:crap}
Throughout this section, for a function $\varphi(y,s)$ we denote $\varphi(0,s)$ simply as $\varphi^0(s)$.

From \eqref{def_f}--\eqref{eq:dt:f} evaluated at $\tilde x =0$, the definition of $V$ in \eqref{def_V}, the definition of $G_W$ in \eqref{eq:gW}, and the constraints in \eqref{eq:constraints}, we deduce that\footnote{Here we have used the identities: $ \Ncal_{1,\nu}^0 =  0$, and $ \Ncal_{\mu,\nu}^0 =  -\phi_{\mu\nu}$, $\Ncal_{\zeta,\mu\nu}^0 = 0$.}
\begin{subequations}
\label{Birds_are_rad}
\begin{align}
\tfrac{1}{\beta_\tau} G_W^0 &=  e^{\frac s2} \left( \kappa + \beta_2 Z^0 - 2\beta_1 R_{j1} \dot \xi_j \right) \label{eq:GW:0:a} \\
\tfrac{1}{\beta_\tau}\p_1 G_W^0 &=  \beta_2 e^{\frac s2} \p_1 Z^0 \label{eq:GW:0:b}\\
\tfrac{1}{\beta_\tau}\p_\nu G_W^0 &=    \beta_2  e^{\frac s2} \p_\nu Z^0  + 2  \beta_1   \dot Q_{1 \nu} + 2   \beta_1  R_{j \gamma} \dot \xi_j \phi_{\gamma \nu}  \label{eq:GW:0:c}\\
\tfrac{1}{\beta_\tau}\p_{11}G_W^0 &=  \beta_2  e^{\frac s2} \p_{11} Z^0\label{eq:GW:0:d} \\
\tfrac{1}{\beta_\tau}\p_{1 \nu} G_W^0 &=    \beta_2 e^{\frac s2}\p_{1\nu} Z^0   - 2  \beta_1 e^{-\frac{3s}{2}} \dot{Q}_{\gamma 1} \phi_{\gamma \nu}   \label{eq:GW:0:e} \\
\tfrac{1}{\beta_\tau}\p_{\gamma \nu} G_W^0 &=   e^{- \frac s2} \left(- \dot{\phi}_{\gamma \nu} +  \beta_2  e^s \p_{\gamma \nu} Z^0  - 2\beta_1( \dot{Q}_{\zeta  \gamma} \phi_{\zeta \nu} +  \dot{Q}_{\zeta \nu} \phi_{\zeta \gamma} +  R_{j1} \dot{\xi_j} \Ncal_{1,\gamma\nu}^0 ) + e^{-\frac s2} \tfrac{G_W^0}{ \beta_\tau} \Jcal_{,\gamma\nu}^0   \right) 
\, .\label{eq:GW:0:f}
\end{align} 
\end{subequations}
Similarly, using \eqref{def_f}--\eqref{eq:dt:f}, \eqref{eq:hW} and the constraints in \eqref{eq:constraints} we have that\footnote{Here we have used the identities:  $\Ncal_\mu^0 = 0$, $\Tcal_\mu^{\gamma,0} = \delta_{\gamma\mu}$, $\Tcal_{\mu,\nu}^{\gamma,0} = 0$, $\Ncal_{\mu,\nu \gamma}^0 = 0$,  and $\Tcal_{1,\nu \gamma}^{\zeta,0} = 0$.}
\begin{align}
\tfrac{1}{\beta_\tau} h_W^{\mu,0} &= 2 \beta_1  e^{-\frac s2} \left(  A_\mu^0  -  R_{j\mu} \dot \xi_j \right) \,. \label{eq:hj:0:a} 
\end{align} 
Then, using
\eqref{eq:FW:gradient},  \eqref{eq:FW:D1:gradient}, and \eqref{Birds_are_rad}, 
for any $\gamma \in {\mathbb N}_0^3$ with $\abs{\gamma} =1$ or $\abs{\gamma} = 2$ we have that
\begin{align*}
F_W^{(\gamma),0} = \p^\gamma F_W^0 + \p^\gamma G_W^0 \, .
\end{align*}

Lastly, appealling to \eqref{def_f}--\eqref{eq:dt:f}, \eqref{eq:FW:def},  we have the explicit expressions\footnote{Here we have used the identities: $\Ncal_{\mu,\mu}^0 = - \phi_{22} - \phi_{33}$, $\Tcal^{\nu,0}_{\mu,\mu} = 0$, $\dot \Ncal_i^0 = 0$, $\dot \Ncal_{1,\nu}^0 =0$, $\dot \Ncal_{\mu,\nu}^0 = - \dot{\phi}_{\mu\nu}$, $\Tcal^{\gamma,0}_{1,\nu} = \phi_{\gamma\nu}$, $\Tcal^{\gamma,0}_{i,\nu} \Ncal_{i,\mu}^0 = 0$, $\Tcal_i^{\gamma,0} \Ncal_{i,\mu \nu}^0 = 0$, $\Ncal_{\mu,\mu\nu}^0 = 0$, and $\dot{\Ncal}_{\zeta,\nu\gamma} = 0$.} 
\begin{subequations}
\begin{align}
\tfrac{1}{\beta_\tau} F_W^0 &= -\beta_3 \left(\kappa - Z^0\right) \p_\mu A_\mu^0  + 2\beta_1 e^{-\frac s2} \dot Q_{1\mu} A_\mu^0 - \tfrac{1}{\beta_\tau} h_W^{\mu,0} A_\zeta^0 \phi_{\zeta \mu}  \notag\\
&\quad + \tfrac 12 \beta_3  e^{-\frac s2} (\kappa - Z^0) (\kappa + Z^0) (\phi_{22}+ \phi_{33})
 \label{eq:FW:0:a}  \\
\tfrac{1}{\beta_\tau} \p_1 F_W^0 &= \beta_3 \left(e^{-\frac s2} + \p_1 Z^0\right) \p_\mu A_\mu^0 - \beta_3  \left(\kappa - Z^0\right) \p_{1\mu} A_\mu^0   + 2\beta_1  e^{-\frac s2} \dot Q_{1\mu} \p_1 A_\mu^0 \notag\\
&\quad -   \left(\tfrac{1}{\beta_\tau} h_W^{\mu,0}  \p_1 A_\zeta^0
+ 2 \beta_1  e^{-\frac s2} ( \p_1 A_\mu^0 + e^{-\frac{3s}{2}} \dot{Q}_{\mu 1})  A_\zeta^0 \right) \phi_{\zeta \mu}  \notag\\
&\quad  - \tfrac 12 \beta_3  e^{-s} 
\left( (1 + e^{\frac s2} \p_1 Z^0) (\kappa + Z^0)  + (\kappa - Z^0) (1 - e^{\frac s2}\p_1 Z^0)  \right) (\phi_{22}+ \phi_{33}) 
 \label{eq:FW:0:b} \\
\tfrac{1}{\beta_\tau} \p_\nu F_W^0&= - \beta_3 ( (\kappa - Z^0) \p_{\nu \mu} A_\mu^0 - \p_\nu Z^0 \p_\mu A_\mu^0 )
- 2 \beta_1  e^{-s} A_\mu^0 \dot{\phi}_{\mu \nu} + 2 \beta_1  e^{-\frac s2} \dot Q_{1\mu} \p_\nu A_\mu^0 
\notag\\
&\quad - 2 \beta_1 e^{-s} \dot{Q}_{\mu \zeta} A_\zeta^0   \phi_{\mu \nu}  - \beta_3 e^{-\frac s2}  Z^0   \p_\nu Z^0   (\phi_{22} + \phi_{33}) - \beta_3  e^{-s} \left(\kappa - Z^0 \right)  A_\zeta^0 \Tcal^{\zeta,0}_{\mu, \mu\nu}  
\notag\\
&\quad  - 2 \beta_1 e^{-\frac s2}  \left( (e^{-\frac s2} \dot{Q}_{\mu \nu} + \p_\nu A_\mu^0 - \tfrac 12 e^{-\frac s2} ( \kappa + Z^0) \phi_{\mu\nu}) A_\gamma^0 \right) \phi_{\gamma \mu} 
- \tfrac{1}{\beta_\tau} h_W^{\mu,0} \p_\nu A_\gamma^0  \phi_{\gamma \mu} 
\label{eq:FW:0:c}\\
\tfrac{1}{\beta_\tau} \p_{11} F_W^0 &= \beta_3  \left(e^{-\frac s2} + \p_1 Z^0\right) \p_\mu A_\mu^0 - \beta_3 \left(\kappa - Z^0\right) \p_{1\mu} A_\mu^0   + 2\beta_1  e^{-\frac s2} \dot Q_{1\mu} \p_{11} A_\mu^0 \notag\\
&\quad -  \left(2 \beta_1   e^{-\frac s2} + \tfrac{1}{\beta_\tau} h_W^{\mu,0} \right) \p_{11} A_\zeta^0 \phi_{\zeta \mu}  - 4 \beta_1   e^{-\frac s2}   ( \p_1 A_\mu^0 + e^{-\frac{3s}{2}} \dot{Q}_{\mu 1})  \p_1 A_\zeta^0  \phi_{\zeta \mu} \notag\\
&\quad  - \beta_3 e^{- \frac{s}{2}}  \left( Z^0 \p_{11} Z^0  - e^{-s} (1 - e^{s} (\p_1 Z^0)^2)  \right) (\phi_{22}+ \phi_{33}) 
 \label{eq:FW:0:d}\\
\tfrac{1}{\beta_\tau} \p_{1\nu} F_W^0 &= - \beta_3  \left( (\kappa - Z^0) \p_{1 \nu \mu} A_\mu^0 - \p_{1\nu} Z^0 \p_\mu A_\mu^0 - \p_\nu Z^0 \p_{1\mu} A_\mu^0 - (e^{-\frac s2} + \p_1 Z^0) \p_{\nu \mu} A_\mu^0 \right) \notag\\
&\quad - 2 \beta_1  e^{-s} \p_1 A_\mu^0 \dot{\phi}_{\mu \nu} + 2 \beta_1  e^{-\frac s2} \dot Q_{1\mu} \p_{1\nu} A_\mu^0  - 2 \beta_1  e^{-s} \dot{Q}_{\mu \zeta} \p_1 A_\zeta^0 \phi_{\mu \nu} 
\notag\\
&\quad   - \beta_3 e^{-\frac s2}  ( \p_1 Z^0   \p_\nu Z^0  + Z^0 \p_{1\nu} Z^0) (\phi_{22} + \phi_{33}) \notag\\
&\quad - \beta_3  e^{-s} \left( (\kappa - Z^0) \p_1 A_\zeta^0  - (e^{-\frac s2} + \p_1 Z^0) A_\zeta^0 \right) \Tcal^{\zeta,0}_{\mu, \mu\nu}  
\notag\\
&\quad  - 2 \beta_1  e^{-\frac s2} \left( (e^{-\frac s2} \dot{Q}_{\mu \nu} + \p_\nu A_\mu^0) \p_1 A_\gamma^0 + (e^{-\frac{3s}{2}} \dot{Q}_{\mu 1} + \p_1 A_\mu^0) \p_\nu A_\gamma^0 + A_\mu^0  \p_{1 \nu} A_\gamma^0 \right) \phi_{\gamma \mu} \notag\\
&\quad  - \tfrac{1}{\beta_\tau} h_W^{\mu,0} \p_{1 \nu} A_\gamma^0  \phi_{\gamma \mu} +   \beta_1  e^{-s} \left(   ( \kappa + Z^0) \p_1 A_\gamma^0 - (e^{-\frac s2} - \p_1 Z^0) A_\gamma^0   \right) \phi_{\mu\nu}    \phi_{\gamma \mu}  
\label{eq:FW:0:e}\\
\tfrac{1}{\beta_\tau} \p_{\gamma\nu} F_W^0 
&=  - 2 \beta_3   ( \p_{\nu \gamma} (S \p_\mu A_\mu))^0 - \beta_3  e^{-s} (\kappa - Z^0) \p_{\mu} A_\zeta^0 \Tcal^{\zeta,0}_{\mu,\nu\gamma} \notag\\
&\quad - 2 \beta_1  e^{-s} \p_\nu A_\mu^0 \dot{\phi}_{\mu \gamma}  - 2 \beta_1  e^{-s} \p_\gamma A_\mu^0 \dot{\phi}_{\mu \nu}  -  \beta_3 e^{-\frac s2}  \p_\gamma Z^0 \p_\nu Z^0   (\phi_{22} + \phi_{33}) \notag\\
&\quad + 2 \beta_1  e^{-\frac s2} \dot{Q}_{1\mu} \p_{\gamma \nu} A_\mu^0 - 2 \beta_1 e^{-s} \dot{Q}_{\zeta \mu} \p_\nu A_\mu^0   \phi_{\zeta\gamma}- 2 \beta_1  e^{-s} \dot{Q}_{\zeta \mu } \p_\gamma A_\mu^0   \phi_{\zeta \nu}  \notag\\
&\quad +2 \beta_1   e^{-\frac{3s}{2}} A_\mu^0 \left( \dot{Q}_{1\zeta} (\phi_{\nu \mu} \phi_{\zeta\gamma} + \phi_{\mu\gamma} \phi_{\zeta\nu} + \phi_{\nu \gamma} \phi_{\mu \zeta} +  \Tcal^{\mu,0}_{\zeta,\nu\gamma}) + \dot{Q}_{1\mu} \Ncal_{1,\nu\gamma}^0 \right) \notag\\
&\quad - \beta_3 e^{-s} \left( (\kappa - Z^0)  \p_\nu A_\zeta^0 - \p_\nu Z^0 A_\zeta^0 \right)\Tcal^{\zeta,0}_{\mu, \mu\gamma}  - \tfrac 12 \beta_3   e^{-\frac{3s}{2}} (\kappa - Z^0) (\kappa + Z^0) \Ncal^0_{\mu,\mu\nu\gamma}
\notag\\
&\quad  - 2 \beta_1 e^{-\frac s2} \left(  e^{-\frac s2} \dot{Q}_{\mu \nu} \p_\gamma A_\zeta^0 + e^{-\frac s2} \dot{Q}_{\mu \gamma} \p_\nu A_\zeta^0 + \p_{\nu\mu} A_\mu^0 A_\zeta^0 + \p_\mu A_\mu^0 \p_\nu A_\nu^0 + \p_\nu A_\mu^0 \p_\mu A_\nu^0\right) \phi_{\zeta \mu} \notag\\
&\quad + 2 \beta_1 e^{-s} \left( \p_\nu ( (U\cdot \Ncal) A_\zeta)^0 \phi_{\mu \gamma} \phi_{\zeta\mu} + \p_\gamma ( (U\cdot \Ncal) A_\zeta)^0 \phi_{\mu \nu} \phi_{\zeta\mu}\right)   - 2\beta_1 e^{-\frac{3s}{2}} A_\iota^0 A_\zeta^0 \Tcal^{\zeta,0}_{\mu, \nu \gamma} \phi_{\iota\mu}
\notag\\
&\quad - \tfrac{1}{\beta_\tau} h_{W}^{\mu,0} \p_{\nu\gamma} A_\zeta^0  \phi_{\zeta \mu} +  e^{-s} \tfrac{1}{\beta_\tau} h_W^{\mu,0} A_\iota^0 \left( \phi_{\iota\nu} \Ncal^0_{1,\mu\gamma} + \phi_{\iota \gamma} \Ncal^0_{1,\mu\nu} + \Ncal^0_{\alpha,\mu\nu\gamma} \right)
 \label{eq:FW:0:f}
\end{align} 
\end{subequations}

\subsubsection{The equations for the constraints}
\label{sec:implicit:crap}
The evolution equations for $W$, $\nabla W$ and $\nabla^2 W$ at $y=0$ yield the equations from which we will deduce the definitions of our constraints 
$\tau, \kappa, \check n, \xi$ and $\phi$. In this subsection, we collect these equations. Then we untangle their coupled nature to actually define the constraints.

At this stage is it convenient to introduce the notation 
$$
\PP_{\diamondsuit}({\mathsf{b}}_1,\ldots,{\mathsf{b}}_n  \big|  {\mathsf{c}}_1,\ldots, {\mathsf{c}}_n) \qquad  \text{ and } \qquad
\mathcal{R}_{\diamondsuit}({\mathsf{b}}_1,\ldots,{\mathsf{b}}_n  \big|  {\mathsf{c}}_1,\ldots, {\mathsf{c}}_n)
$$
to denote a linear function in the parameters $ {\mathsf{c}}_1,\ldots, {\mathsf{c}}_n$ with (bounded in $s$) coefficients  which depend on 
${\mathsf{b}}_1,\ldots,{\mathsf{b}}_n$ through smooth polynomial (for $\PP_\diamondsuit$), respectively,  rational functions (for $\mathcal{R} _\diamondsuit$), 
and on the derivatives of $Z$ and $A$ evaluated at $y=0$.  In particular, these bounds can depend on the constant $M$.
Throughout this section, we will implicitly use the bootstrap estimates \eqref{eq:Z_bootstrap} 
and \eqref{eq:A_bootstrap} to establish these uniform bounds   on the  coefficients, which in turn,  yields local well-posedness of the coupled system of  
ODE for the modulation variables.  

The subscript 
$\diamondsuit$ denotes a label, used to 
distinguish the various functions $\PP_\diamondsuit$ and  $\mathcal{R} _\diamondsuit$.  We note that all of the denominators in $ \mathcal{R} _\diamondsuit$ are
bounded from below by a universal constant.   It is important to note that the notation $\PP_\diamondsuit$ and  $\mathcal{R} _\diamondsuit$ is never
used when explicit bounds are required.

First, we evaluate the equation for $W$ at $y=0$ to obtain a definition for $\dot \kappa$. Using \eqref{eq:euler:ss:a} and \eqref{eq:constraints} we obtain that 
\begin{align}
 - G_W^0 = F_W^0 - e^{-\frac s2} \beta_\tau \dot \kappa \qquad \Rightarrow \qquad \dot \kappa = \tfrac{1}{\beta_\tau} e^{\frac s2} \left( F_W^0 + G_W^0 \right) \, .
  \label{eq:dot:kappa:1}
\end{align}
Using the above introduced notation, upon recalling the definition \eqref{eq:FW:0:a} we deduce that \eqref{eq:dot:kappa:1} may be written schematically as
\begin{align}
\dot \kappa = \PP_\kappa \left(\kappa, \phi \, \big| \,  \dot{Q}, \tfrac{1}{\beta_\tau} e^{\frac s2} h_W^{,0}, \tfrac{1}{\beta_\tau} e^{\frac s2} G_W^{0}  \right)  \, .
  \label{eq:dot:kappa:2}
\end{align}
Once we compute $h_W^{,0}$ and $G_W^{0}$ (cf.~\eqref{eq:GW:def:1}--\eqref{eq:hj:def:1} below) we will return to the formula \eqref{eq:dot:kappa:2}.

Next, we evaluate the equation for $\p_1 W$ at $y=0$ and obtain a formula for $\dot \tau$.
From \eqref{eq:grad:W:a}, \eqref{eq:FW:gradient:a}, and using that $-1+\beta_\tau = \frac{\dot\tau}{1-\dot\tau} = \dot \tau \beta_\tau$, we obtain that 
\begin{align}
 - (1-\beta_\tau) = \p_1 F_W^0 + \p_1 G_W^0  \qquad \Rightarrow \qquad \dot \tau = \tfrac{1}{\beta_\tau} \left(\p_1 F_W^0 + \p_1 G_W^0 \right) \, .
  \label{eq:dot:tau:1}
\end{align}
Using the above introduced notation, upon recalling the explicit functions \eqref{eq:GW:0:b} and \eqref{eq:FW:0:b} we deduce that \eqref{eq:dot:tau:1} may be written schematically as
\begin{align}
\dot \tau = \PP_{  \tau} \left(\kappa,\phi  \, \big| \,  e^{-2s} \dot{Q},  \tfrac{1}{\beta_\tau} h_W^{,0} \right)   \, .
  \label{eq:dot:tau:2}
\end{align}
Once we compute $h_W^{,0}$ and $G_W^{0}$ (cf.~\eqref{eq:GW:def:1}--\eqref{eq:hj:def:1} below) we will return to  \eqref{eq:dot:tau:2}.

We turn to the evolution equation for $\check \nabla W$ at $y=0$, which gives that $\dot Q_{1j}$. Note that once $\dot Q_{1j}$ is known, we can determine $\dot{\check n}$ thorough an algebraic computation; this will be done later. Evaluating \eqref{eq:grad:W:b}--\eqref{eq:grad:W:c} at $y=0$ and using \eqref{eq:FW:gradient:b}--\eqref{eq:FW:gradient:c} we obtain for $\nu \in \{2,3\}$ that
\begin{align}
F_W^{0,(0,1,0)} = F_W^{0,(0,0,1)} = 0 \qquad \Rightarrow \qquad \p_\nu F_W^0 + \p_\nu G_W^0 = 0 \, .
\label{eq:check_derivative_GW}
\end{align}
By appealing to \eqref{eq:GW:0:c} and \eqref{eq:FW:0:c}, and placing the leading order term in  $\dot Q$ on one side, we obtain
\begin{align}
\dot Q_{1 \nu} 
&= - e^{-\frac s2} \dot Q_{1\mu} \p_\nu A_\mu^0 +   e^{-s} \dot{Q}_{\mu \zeta} A_\zeta^0   \phi_{\mu \nu}  +  e^{-s} \dot{Q}_{\mu \nu} A_\zeta^0 \phi_{\zeta \mu} - \tfrac{\beta_2}{2 \beta_1} e^{\frac s2} \p_\nu Z^0 
+    e^{-s} A_\mu^0 \dot{\phi}_{\mu \nu} \notag\\
&  + \tfrac{\beta_3}{2 \beta_1} \left( (\kappa - Z^0) \p_{\nu \mu} A_\mu^0 - \p_\nu Z^0 \p_\mu A_\mu^0 \right)
 + \tfrac{\beta_3}{\beta_1} e^{-\frac s2}  Z^0   \p_\nu Z^0   (\phi_{22} + \phi_{33}) + \tfrac{\beta_3}{2 \beta_1}  e^{-s} \left(\kappa - Z^0 \right)  A_\zeta^0 \Tcal^{\zeta,0}_{\mu, \mu\nu}  
\notag\\
&   +  e^{-\frac s2}  \left( ( \p_\nu A_\mu^0 - \tfrac 12 e^{-\frac s2} ( \kappa + Z^0) \phi_{\mu\nu}) A_\gamma^0 \right) \phi_{\gamma \mu}  +  \tfrac{1}{2\beta_1\beta_\tau} h_W^{\mu,0} \p_\nu A_\gamma^0   \phi_{\gamma \mu}  -  \left(\tfrac{1}{2\beta_1\beta_\tau} e^{\frac s2} h_W^{\gamma,0} - A_\gamma^0 \right) \phi_{\gamma \nu}
\label{eq:dot:Q:1} \, .
\end{align}

We schematically write \eqref{eq:dot:Q:1} as
\begin{align}
\dot{Q}_{1\nu} = \PP_{Q,\nu} \left(\kappa, \phi  \, \big| \,  \tfrac{1}{\beta_\tau} e^{\frac s2} h_W^{,0}, e^{-s} \dot \phi, e^{-s} \dot Q \right)
\label{eq:dot:Q:2} \, . 
\end{align}
Note that once $\dot Q_{1\nu}$ is known, we can determine $\dot{n}_2$ and $\dot{n}_3$ by recalling from \eqref{eq:Q:def}, \eqref{eq:Q2:def}, \eqref{eq:Q3:def}  that
\begin{align}
\left[
\begin{matrix} 
1 + \tfrac{n_2^2}{n_1(1+n_1)} & \tfrac{n_2 n_3}{n_1(1+n_1)} \\
\tfrac{n_2 n_3}{n_1(1+n_1)} & 1 + \tfrac{n_3^2}{n_1(1+n_1)}   \\
\end{matrix}\right]  
\left[
\begin{matrix} 
\dot{n}_2  \\
\dot{n}_3 \\
\end{matrix}\right]  
=
\left(\Id + \tfrac{\check n \otimes \check n}{n_1(1+n_1)} \right) \dot{\check n}
= 
\left[
\begin{matrix} 
\dot{Q}_{12} \\
\dot{Q}_{13} \\
\end{matrix}\right]  \,,
\label{eq:dot:n:def}
\end{align}
where $n_1 = \sqrt{1- n_2^2-n_3^2}$. Since the vector $\check n$ is small (see~\eqref{eq:speed:bound} below), and the matrix on the left side is an $\OO(|\check n|^2)$ 
perturbation of the identity matrix, we obtain from \eqref{eq:dot:n:def} a definition of $\dot n$, as desired.

Next, we turn to the evolution of $\p_1 \nabla W$ at $y=0$. This constraint allows us to compute $G_W^0$ and $h_W^{\mu,0}$, which in turn allows us to express $\dot \xi$. First we focus on computing $G_W^0$ and $h_W^{\mu,0}$. Evaluating \eqref{eq:grad:2:W} at $y=0$ and using \eqref{eq:FW:D1:gradient}, for $i\in \{1,2,3\}$ we obtain
\begin{align}
G_W^0 \p_{1i1}W^0 + h_W^{\mu,0} \p_{1i\mu} W^0 &= \p_{1i} F_W^0 + \p_{1i} G_W^0 \, .
\label{eq:Hessian:emerges}
\end{align}
On the left side of the above identity we recognize the matrix  
\begin{align}
{\mathcal H}^0(s) := (\p_1\nabla^2 W)^0(s) 
\label{eq:d1W:Hessian}
\end{align}
acting on the vector with components $G_W^0$, $h_{W}^{2,0}$, and $h_{W}^{3,0}$. We will show  that the matrix ${\mathcal H}^0$  remains very close to the matrix ${\rm diag}(6,2,2)$, for all $s\geq -\log \eps$, and thus it is invertible (see~\eqref{eq:inverse:Hessian} below).
Therefore, we can express
\begin{subequations}
\begin{align}
G_W^0 &= ({\mathcal H}^0)^{-1}_{1i}  (\p_{1i} F_W^0 + \p_{1i} G_W^0) \label{eq:GW:def:1} \\
h_W^{\mu,0} &= ({\mathcal H}^0)^{-1}_{\mu i}  (\p_{1i} F_W^0 + \p_{1i} G_W^0) \label{eq:hj:def:1}\, .
\end{align}
\end{subequations}
Inspecting \eqref{eq:GW:0:d}--\eqref{eq:GW:0:e} and \eqref{eq:FW:0:d}--\eqref{eq:FW:0:e} and inserting them into \eqref{eq:hj:def:1}, we initially obtain the   dependence
$$
\tfrac{1}{\beta_\tau}  h_W^{\mu,0}  = e^{-\frac s2} \mathcal{R} _{h,\mu}\left(\kappa, \phi  \, \big| \,  e^{-s} \dot Q, e^{-2s} \dot \phi \right) - \tfrac{1}{\beta_{\tau}} h_W^{\gamma,0} ({\mathcal H}^0)^{-1}_{\mu i}  \phi_{\zeta \gamma} \p_{1i} A_\zeta^0 \, .
$$
Note that although $h_W^{,0}$ appears on both sides of the above, the dependence on the right side is paired with a factor of $e^{-s} \leq \eps$, and the functions $\phi_{\zeta\gamma}$ are themselves expected to be $\leq \eps$ for all $s\geq - \log \eps$ (cf.~\eqref{eq:speed:bound} below). This allows us to schematically write 
\begin{align}
\tfrac{1}{\beta_\tau}  h_W^{\mu,0}  = e^{-\frac s2} \mathcal{R} _{h,\mu}\left(\kappa, \phi  \, \big| \,  e^{-s} \dot Q, e^{-2s} \dot \phi \right) \,.
\label{eq:hj:def:2}
\end{align}
Returning to \eqref{eq:GW:def:1},
inspecting \eqref{eq:GW:0:d}--\eqref{eq:GW:0:e} and \eqref{eq:FW:0:d}--\eqref{eq:FW:0:e}, and using \eqref{eq:hj:def:2} we also obtain the dependence
\begin{align}
\tfrac{1}{\beta_\tau} G_W^0 =  e^{-\frac s2} \mathcal{R} _{h,\mu}\left(\kappa, \phi  \, \big| \,  e^{-s} \dot Q, e^{-2s} \dot \phi \right) \label{eq:GW:def:2}  \, .
\end{align}
 
Upon inspecting \eqref{eq:GW:0:a} and \eqref{eq:hj:0:a}, and noting the invertibility of the matrix $R$ in \eqref{eq:R:def} it is clear why \eqref{eq:GW:def:1}--\eqref{eq:hj:def:1} allow us to compute $\xi_j$. Indeed, from \eqref{eq:GW:0:a}, \eqref{eq:hj:0:a},  \eqref{eq:GW:def:1}--\eqref{eq:hj:def:1}, and the fact that $R R^T = \Id$ we deduce that 
\begin{align}
 \dot{\xi}_j 
 = R_{ji}  ( R^T \dot \xi)_i 
 = R_{j1} \left( \tfrac{1}{2\beta_1} (\kappa + \beta_2 Z^0) - \tfrac{1}{2\beta_1\beta_\tau} e^{-\frac s2} G_W^0 \right) + R_{j\mu} \left(A_\mu^0 - \tfrac{1}{2\beta_1\beta_\tau} e^{\frac s2} h_W^{\mu,0} \right) 
\label{eq:dot:xi:def:1}
\end{align}
for $j \in \{1,2,3\}$. Using \eqref{eq:hj:def:2} and \eqref{eq:GW:def:2}, we may then schematically write
\begin{align}
 \dot{\xi}_j  =\mathcal{R} _{\xi,j} \left(\kappa, \phi  \, \big| \,  e^{-s} \dot Q,  e^{-2s} \dot \phi \right) \, .
\label{eq:dot:xi:def:2} 
\end{align}

Lastly, we record the evolution of $\check \nabla^2 W$ at $y=0$. From this constraint we will deduce the evolution equations for $\phi_{jk}$. Evaluating \eqref{eq:grad:2:check:W} at $y=0$, using the definitions \eqref{eq:FW:gradient:check:squared}, we obtain 
\begin{align*}
G_W^0 \p_{1\nu\gamma} W^0  + h_W^{\mu,0} \p_{\mu \nu \gamma}W^0  = \p_{\nu\gamma} F_W^0 + \p_{\nu\gamma} G_W^0
\end{align*}
for $\nu ,\gamma \in \{2,3\}$.
Using \eqref{eq:GW:def:1} and \eqref{eq:hj:def:1} we rewrite the above identity as
\begin{align}
\p_{\nu\gamma} G_W^0 = ({\mathcal H}^0)^{-1}_{1i}  (\p_{1i} F_W^0 + \p_{1i} G_W^0)  \p_{1\nu\gamma} W^0  + ({\mathcal H}^0)^{-1}_{\mu i}  (\p_{1i} F_W^0 + \p_{1i} G_W^0)  \p_{\mu \nu \gamma}W^0  - \p_{\nu\gamma} F_W^0 \, .
\label{eq:D2:GW:def:1}
\end{align}
Note that $\dot \phi_{\nu\gamma}$ is determined  in terms of $e^{\frac s2} \p_{\nu\gamma} G_W^0$ through the first term on the right side of~\eqref{eq:GW:0:f}
\begin{align}
\dot{\phi}_{\gamma \nu} &=  -  \tfrac{1}{\beta_\tau}  e^{\frac s2}  \left(G_W^0 \p_{1\nu\gamma} W^0 + h_W^{\mu,0} \p_{\mu\nu\gamma} W^0 - \p_{\nu\gamma} F_W^0 \right)+  \beta_2   e^s \p_{\gamma \nu} Z^0 -  2\beta_1( \dot{Q}_{\zeta  \gamma} \phi_{\zeta \nu} +  \dot{Q}_{\zeta \nu} \phi_{\zeta \gamma}) \notag\\
&\quad + \left(\tfrac{1}{ \beta_\tau} e^{-\frac s2} G_W^0 - \kappa - \beta_2 Z^0\right)  \Ncal_{1,\gamma\nu}^0   + \Jcal_{,\gamma\nu}^0 \tfrac{1}{ \beta_\tau} e^{-\frac s2} G_W^0  \,,
\label{eq:dot:phi:def:1}
\end{align}
and \eqref{eq:GW:def:1} is used to determine $G_W^0$. In light of \eqref{eq:FW:0:f}, \eqref{eq:GW:def:2} and of \eqref{eq:dot:phi:def:1}, we may schematically write
\begin{align*}
\dot{\phi}_{\gamma \nu} =\mathcal{R} _{\phi,\gamma\nu} \left(\kappa, \phi  \, \big| \,  e^{-s}\dot{Q}, e^{-s} \dot \phi  \right)   - \dot Q_{\zeta\gamma} \phi_{\zeta\nu} -  \dot Q_{\zeta\nu} \phi_{\zeta\gamma}\,,
\end{align*}
which may be then combined with \eqref{eq:dot:Q:2} and \eqref{eq:hj:def:2}
to yield
\begin{align}
\dot{\phi}_{\gamma \nu} = \mathcal{R} _{\phi,\gamma\nu} \left(\kappa, \phi  \, \big| \,  e^{-s}\dot{Q}, e^{-s} \dot \phi  \right)  \,,
\label{eq:dot:phi:def:2}
\end{align}
thus spelling out the dependences of $\dot \phi$ on the other dynamic variables.
 
\subsubsection{Solving for the dynamic modulation parameters} 
The computations of the previous subsection derive implicit definitions for the time derivatives of our ten modulation parameters, in terms of these parameters themselves and of the derivatives of $Z$ and $A$ at the origin.  The goal of this subsection is to show that this system of ten coupled nonlinear ODEs has a local existence of solutions, with initial datum as given by \eqref{eq:modulation:IC}. In Section~\ref{sec:dynamic:closure} it will be then shown that the system of ODEs for the modulation parameters is in fact solvable globally in time, for all $s\geq -\log \eps$.

By combining \eqref{eq:dot:Q:2} and \eqref{eq:hj:def:2} with \eqref{eq:dot:n:def}, and recalling \eqref{eq:dot:phi:def:2} we obtain that 
\[
\dot{\phi}_{\gamma \nu} = \mathcal{R} _{\phi,\gamma\nu} \left(\kappa, \phi, \check n  \, \big| \,  e^{-s}\dot{\check n}, e^{-s} \dot \phi  \right) 
\quad \mbox{and} \quad 
\dot{n}_\nu = \mathcal{R} _{n,\nu} \left(\kappa, \phi, \check n  \, \big| \,  e^{-s}\dot{\check n}, e^{-s} \dot \phi  \right) \,.
\]
Therefore, since $e^{-s} \leq \eps$, and the functions $\PP_{\phi,\gamma\nu}$ and $\PP_{n,\nu}$ are linear in $e^{-s} \dot{\check n}$ and $e^{-s} \dot \phi$, then as long as $\kappa$, $\phi$, and $\check n$ remain bounded,  and $\eps$ is taken to be sufficiently small (in particular, for short time after $t= -\log \eps$), we may analytically solve for $\dot \phi$ and $\dot n$ as rational functions (with bounded denominators) of $\kappa, \phi$, and $\check n$, with coefficients which only depend on the derivatives of $Z$ and $A$ at $y=0$. We write this schematically as 
\begin{align}
\dot{\phi}_{\gamma \nu} = \EE_{\phi,\gamma\nu} \left(\kappa, \phi, \check n  \right) 
\quad \mbox{and} \quad 
\dot{n}_\nu = \EE_{n,\nu} \left(\kappa, \phi, \check n    \right) \,.
\label{eq:dot:phi:dot:n}
\end{align}
Here the $\EE_{\phi,\gamma\nu}(\kappa,\phi,\check n)$ and $\EE_{n,\nu}(\kappa,\phi,\check n)$ are suitable smooth functions of their arguments, as described above.
With \eqref{eq:dot:phi:dot:n} in hand, we return to \eqref{eq:dot:kappa:2} and \eqref{eq:dot:tau:2}, which are to be combined with \eqref{eq:hj:def:2}, and with \eqref{eq:dot:xi:def:2} to obtain that
\begin{align}
\dot{\kappa} = \EE_{\kappa} \left(\kappa, \phi, \check n  \right) 
\,, \qquad
\dot{\tau} = \EE_{\tau} \left(\kappa, \phi, \check n    \right) \,
\quad \mbox{and} \quad 
\dot{\xi}_j = \EE_{\xi,j} \left(\kappa, \phi, \check n    \right)\,.
\label{eq:dot:kappa:dot:tau}
\end{align}
for suitable smooth functions $\EE_{\kappa}, \EE_{\tau},$ and $\EE_{\xi,j}$ of $(\kappa,\phi,\check n)$,  with coefficients  which depend on the derivatives of $Z$ and $A$ at $y=0$. 

\begin{remark}[Local solvability]
\label{rem:local:constraints}
The system of ten nonlinear ODEs described in \eqref{eq:dot:phi:dot:n} and \eqref{eq:dot:kappa:dot:tau} are used to determine the time evolutions of our ten dynamic modulation variables. The local in time solvability of this system is ensured by the fact that $\EE_{\phi,\gamma\nu}, \EE_{n,\nu}, \EE_{\kappa},  \EE_{\tau}, \EE_{\xi,j}$ are rational functions of $\kappa, \phi, n_2$, and $n_3$, with coefficients that only depend on $\p^\gamma Z^0$ and $\p^\gamma A^0$ with $\abs{\gamma} \leq 3$, and moreover that these functions are smooth in the neighborhood of the initial values given by \eqref{eq:modulation:IC}; hence, unique $C^1$ solutions exist for a sufficiently small time.
We emphasize that these functions are explicit, once one traces back the identities in Sections~\ref{sec:explicit:crap} and~\ref{sec:implicit:crap}, which will play a crucial role in Section~\ref{sec:dynamic:closure}, when we prove the bootstrap \eqref{mod-boot}.
\end{remark}

\section{Closure of bootstrap estimates for the dynamic variables}
\label{sec:dynamic:closure}
In this section, we close the bootstrap assumptions on our dynamic modulation parameters, meaning that we establish \eqref{eq:speed:bound} and 
\eqref{eq:acceleration:bound} with constants that are better by at least a factor of $2$.

The starting point is to obtain bounds for $G_W^0$ and $h_W^{\mu,0}$, by appealing to \eqref{eq:GW:def:1}--\eqref{eq:hj:def:1}. The matrix ${\mathcal H}^0$ defined in \eqref{eq:d1W:Hessian} can be rewritten as
\begin{align*}
{\mathcal H}^0(s) = (\p_1\nabla^2 W)^0(s)  =  (\p_1\nabla^2 \bar W)^0 + (\p_1\nabla^2 \tilde W)^0(s) = {\rm diag}(6,2,2) + (\p_1\nabla^2 \tilde W)^0(s).
\end{align*}
From the bootstrap assumption \eqref{eq:bootstrap:Wtilde3:at:0}  we have that 
$\abs{(\p_1\nabla^2 \tilde W)^0(s)} \leq \eps^{\frac 14}$ for all $s\geq - \log \eps$, and thus 
\begin{align}
\abs{({\mathcal H}^0)^{-1}(s)} \leq 1
\label{eq:inverse:Hessian}
\end{align}
for all $s\geq -\log \eps$. Next, we estimate $\p_1 \nabla F_W^0$. Using \eqref{eq:FW:0:d}, \eqref{eq:FW:0:e}, the bootstrap assumptions  \eqref{eq:speed:bound}--\eqref{eq:beta:tau}, the bound \eqref{eq:Z_bootstrap}--\eqref{eq:Z:higher:order}, and the fact that $\sabs{\Tcal^{\zeta,0}_{\mu,\mu\nu} }\leq \abs{\phi}^2$, after a computation we arrive at 
\begin{align}
\abs{\p_1 \nabla F_W^0 }
&\les M \eps^{\frac 12} e^{-s} + M^2 e^{- \frac{3}{2} (1- \frac{4}{2k-5})s} + \abs{h_W^{\cdot,0}} M^3 \eps e^{- \frac{3}{2} (1- \frac{4}{2k-5})s} \, .
\label{eq:F_WZ_deriv_est:replace:1}
\end{align}
Moreover,  from \eqref{eq:GW:0:d}, \eqref{eq:GW:0:e}, \eqref{eq:speed:bound}, \eqref{eq:acceleration:bound},  the first line in \eqref{eq:Z_bootstrap}, the previously established bound \eqref{eq:F_WZ_deriv_est:replace:1}, and the fact that $k \geq 10$, that 
\begin{align}
\abs{\p_{1}\nabla G_W^{0}} + \abs{\p_{1} \nabla F_W^{0}} 
&\les e^{\frac s2} \abs{\p_1 \nabla Z^0} +  M^4 \eps^{\frac 32} e^{-\frac{3s}{2}} + e^{-s} + \eps^2  \abs{h_W^{\cdot,0}} \notag\\
& \les M e^{-s}+ \eps^2  \abs{h_W^{\cdot,0}} \,.
\label{eq:rainy:Halloween}
\end{align}
The bounds  \eqref{eq:inverse:Hessian} and \eqref{eq:rainy:Halloween}, are then inserted  into \eqref{eq:GW:def:1}--\eqref{eq:hj:def:1}. After absorbing the $\eps^2  \abs{h_W^{\cdot,0}}$ term into the left side, we obtain  to estimate
\begin{align}
 \abs{G_W^{0}(s)} + \abs{h_W^{\mu,0}(s)} \les M e^{-s}  \,.
\label{eq:GW:hW:0} 
\end{align}
The bound \eqref{eq:GW:hW:0} plays a crucial role in the following subsections.

\subsection{The $\dot \tau$ estimate}
From \eqref{eq:dot:tau:1}, the definition of $\p_1 G_W^0$ in \eqref{eq:GW:0:b},   the definition of $\p_1 F_W^0$ in \eqref{eq:FW:0:b} , the bootstrap estimates \eqref{eq:speed:bound}--\eqref{eq:beta:tau}, \eqref{eq:Z_bootstrap}, \eqref{eq:A_bootstrap}, and the previously established bound \eqref{eq:GW:hW:0}, we  obtain that 
\begin{align}
\abs{\dot \tau} &\les \abs{\p_1 G_W^0} + \abs{\p_1 F_W^0} \notag\\
&\les e^{\frac s2} \abs{\p_1 Z^0}  +  e^{-\frac s2} \abs{\check \nabla A^0} + M \abs{\check \nabla \p_1 A^0} + M^2 \eps^{\frac 12} e^{-\frac s2}\abs{\p_1 A^0} + M^2 \eps  e^{-2s} \abs{A^0} + M^3 \eps e^{-s} \notag\\ 
&\les M^{\frac 12} e^{-s} + M \eps^{\frac 12} e^{-s} + M e^{- \frac 32 (1 - \frac{2}{2k-5})s} + M^3 \eps e^{s} \notag\\
&\leq \tfrac{M}{4} e^{-s}\,, \label{tau-final}
\end{align}
where we have that $k\geq 10$, and have used a power of $M$ to absorb the implicit constant in the first inequality above. This improves the bootstrap bound for $\dot \tau$ in \eqref{eq:acceleration:bound} by a factor of $4$. Integrating in time from $-\eps$ to $T_*$, where $\abs{T_*} \leq \eps$, we also improve the $\tau$ bound in \eqref{eq:speed:bound} by a factor of $2$, thereby closing the $\tau$ boostrap.

\subsection{The $\dot \kappa$ estimate} 
From \eqref{eq:dot:kappa:1}--\eqref{eq:beta:tau}, the bound \eqref{eq:GW:hW:0}, the definition of $F_W^0$ in \eqref{eq:FW:0:a}, and the estimates \eqref{eq:Z_bootstrap} and \eqref{eq:A_bootstrap}, we deduce that 
\begin{align*}
\abs{\dot \kappa} 
&\les  e^{\frac s2} \abs{G_W^0} + e^{\frac s2} \abs{F_W^0}  
\notag\\
&\les M e^{-\frac s2} + (\kappa_0 + M\eps)  M \eps^{\frac 12} e^{-\frac s2} + M^3 \eps^{\frac 32} e^{-\frac s2} + M^4 \eps^2 e^{-\frac s2}  + e^{-\frac s2} (\kappa_0^2 + M^2 \eps^2) M^2 \eps \notag\\
&\les M e^{-\frac s2}\,.
\end{align*}
Upon using a factor of $M/2$ to absorb the implicit constant in the above estimate, we improve the $\dot \kappa$ bootstrap bound in \eqref{eq:acceleration:bound} by a factor of $2$. Integrating in time, we furthermore deduce that 
\begin{equation}\label{kappa-kappa0}
\abs{\kappa(t) - \kappa_0} \leq M^2 \eps^{\frac 32}
\end{equation} 
for all $t \in [-\eps, T_*)$, since $\abs{T_*} \leq \eps$. Upon taking $\eps$ to be sufficiently small in terms of $M$ and $\kappa_0$, we improve the $\kappa$ bound in \eqref{eq:speed:bound}.

\subsection{The $\dot \xi$ estimate}
In order to bound the $\dot \xi$ vector, we appeal to \eqref{eq:dot:xi:def:1}, to \eqref{eq:GW:hW:0}, to the $\abs{\gamma}=0$ cases in \eqref{eq:Z_bootstrap} and \eqref{eq:A_bootstrap}, and to the bound $\abs{R-\Id} \leq \eps$ which follows from \eqref{eq:R:def} and the $\abs{\check n}$ estimate in \eqref{eq:speed:bound}, to deduce that 
\begin{align}
\abs{\dot{\xi}_j} 
\les \kappa_0 + \abs{Z^0}  +  e^{-\frac s2} \abs{G_W^0} +  \abs{A_\mu^0} + e^{\frac s2} \abs{h_W^{\mu,0}} 
\les \kappa_0 + M \eps + M e^{-\frac s2} \les \kappa_0 \,, \label{xi-final}
\end{align}
upon taking $\eps$ to be sufficiently small in terms of $M$ and $\kappa_0$. The bootstrap estimate for $\dot \xi$ in \eqref{eq:acceleration:bound} is then improved by taking $M$ sufficiently large, in terms of $\kappa$, while the bound on $\xi$ in \eqref{eq:speed:bound} follows by integration in time.

\subsection{The $\dot{\phi}$ estimate}
Using \eqref{eq:dot:phi:def:1}, the fact that  $\abs{\Ncal_{1,\mu\nu}^0} + \abs{\Jcal_{,\mu\nu}^0} \les |\phi|^2$, the bootstrap assumptions \eqref{eq:speed:bound}, \eqref{eq:acceleration:bound},  \eqref{eq:bootstrap:Wtilde3:at:0}, the bounds \eqref{eq:dot:Q}, and the previously established estimate \eqref{eq:GW:hW:0}, we obtain
\[
\abs{\dot{\phi}_{\gamma \nu}}  \les   e^{\frac s2}  \left(M \eps^{\frac 14} e^{-s} + \abs{\p_{\nu\gamma} F_W^0} \right)+ e^s \abs{\p_{\gamma \nu} Z^0}  + M^4 \eps^{\frac 32} + \left(M e^{-\frac{3s}{2}} + \kappa_0 + \abs{Z^0}\right)  M^4 \eps^2  + M^5 \eps^2 e^{-\frac{3s}{2}}  \,. 
\]
Using the definition of $\check \nabla^2 F_W^0$ in \eqref{eq:FW:0:f}, appealing to the bootstrap assumptions (and their consequences) from Section~\ref{sec:bootstrap}, the previously established estimate \eqref{eq:GW:hW:0}, and the fact that $\abs{\Tcal^{\zeta,0}_{\mu, \gamma\nu}} +\abs{\Ncal_{1,\mu\nu}^0} + \abs{\Jcal_{,\mu\nu}^0} + \abs{\Ncal^0_{\zeta,\mu\nu\gamma}} \les |\phi|^2$, it is not hard to show that 
\[
 \abs{\p_{\nu\gamma} F_W^0} \les e^{-\frac s2} \,.
\]
In fact, a stronger estimate holds (cf.~\eqref{eq:F_WZ_deriv_est} below), but we shall not use this fact here. 
Combining the above two estimates with the  $Z$ bounds in \eqref{eq:Z_bootstrap}, 
we derive
\begin{align}
 \abs{\dot{\phi}_{\gamma \nu}}  \les   e^{\frac s2}  \left(M \eps^{\frac 14} e^{-s} + e^{-\frac s2} \right)+ M + M^4 \eps^{\frac 32} + \left(M e^{-\frac{3s}{2}} + \kappa_0 + \eps M \right)  M^4 \eps^2  + M^5 \eps^2 e^{-\frac{3s}{2}} \les M \,.  \label{phi-final}
\end{align}
Upon taking $M$ sufficiently large to absorb the implicit constant in the above estimate, we deduce $|\dot \phi| \leq M^2/4$, which improves the $\dot \phi$ bootstrap in \eqref{eq:acceleration:bound} by a factor of $4$. Integrating in time on $[-\eps,T_*)$, an interval of length $\leq 2 \eps$,  and using that by \eqref{check-two-w0} and \eqref{eq:phi:0:def} we have $\abs{\phi(-\log \eps)}\leq \eps$  thus improving the $\phi$ bootstrap in \eqref{eq:speed:bound} by a factor of $2$.

\subsection{The $\dot{n}$ estimate}
First we obtain estimates on $|\dot Q_{1\nu}|$, by appealing to the identity \eqref{eq:dot:Q:1}. 
Using the bootstrap assumptions \eqref{eq:speed:bound}, \eqref{eq:acceleration:bound}, \eqref{eq:Z_bootstrap}, \eqref{eq:A_bootstrap}, the estimates \eqref{eq:dot:Q} and \eqref{eq:GW:hW:0}, and the fact that $\abs{\Tcal^{\zeta,0}_{\mu, \mu\nu}} \les |\phi|^2$, we obtain 
\begin{align}
\abs{\dot Q_{1 \nu} }
&\les  M^2 \eps^{\frac 12} e^{-\frac s2}  \abs{\p_\nu A_\mu^0} +   M^4 \eps^{\frac 32} e^{-s}  \abs{A^0} + e^{\frac s2} \abs{\check \nabla Z^0} 
+ M^2   e^{-s} \abs{A^0}  
 \notag\\
& \quad +   \left(M \abs{\check \nabla^2A^0} +\abs{\check \nabla Z^0} \abs{\check \nabla A^0} \right)
 +  M^2 \eps e^{-\frac s2}  \abs{Z^0}  \abs{\check \nabla Z^0}   + M^5 \eps^2 e^{-s}  \abs{A^0} \notag\\
& \quad  +  e^{-\frac s2}  \left( ( \abs{\check \nabla A^0} + M^3 \eps e^{-\frac s2} ) \abs{A^0} \right) M^2 \eps  +  M^3 \eps e^{-s} \abs{\check \nabla A^0}
+ M^2 \eps  \left(M e^{-\frac s2} + \abs{A^0} \right) 
\notag\\
&\les  M \eps^{\frac 12} 
\label{eq:dot:Q:bootstrap:close}
\,,
\end{align}
upon taking $\eps$ sufficiently small, in terms of $M$.
Moreover, using the bootstrap assumption $\abs{\check n} \leq M \eps^{\frac 32}$, we deduce that the matrix on the left side of \eqref{eq:dot:n:def} is within $\eps$ of the identity matrix, and thus so is its inverse. We deduce from \eqref{eq:dot:n:def} and \eqref{eq:dot:Q:bootstrap:close} that 
\begin{align} 
\abs{\dot{\check n}} \leq \tfrac{M^2 \eps^{\frac 12}}{4}.   \label{n-final}
\end{align} 
upon taking $M$ to be sufficiently large to absorb the implicit constant. The closure of the $\check n$ boostrap is then achieved by integrating in time on $[-\eps, T_*)$.

\section{Preliminary lemmas} 
\label{sec:preliminary}
We begin by recording some useful bounds that will be used repetitively throughout the section.
\begin{lemma}
\label{lem:BBS}
For $y\in\XXX(s)$ and for $m\geq 0$ we have
\begin{align}
\label{e:bounds_on_garbage}
\abs{\check \nabla^m f} &+ \abs{\check \nabla^m (\Ncal- \Ncal_0 )}+\abs{\check \nabla^m (\tt^\nu -  \Tcal_0^\nu )} \notag\\
& +\abs{\check \nabla^m (\Jcal-1)} +\abs{\check \nabla^m (\Jcal^{-1}-1)}  
\les \eps M^2 e^{-\frac{m+2}{2}s } \abs{\check y}^2 \les \eps e^{-\frac{m}{2}s } \,,
\\
\abs{\check \nabla^m \dot f}&+\abs{\check \nabla^m \dot \Ncal}  \les M^2 e^{-\frac{m+2}{2}s } \abs{\check y}^2 \les \eps^{\frac 14} e^{-\frac{m}{2}s }\,.
\label{e:bounds_on_garbage_2}
\end{align}
Moreover, we have the following estimates on $V$  
\begin{equation}\label{eq:V:bnd}
\abs{\p^\gamma V} \les \begin{cases}
 M^{\frac 14}  & \mbox{if } \abs{\gamma}=0\\
 M^2 \eps^{\frac 12}  e^{-\frac32 s} & \mbox{if }\abs{ \gamma}=1\mbox{ and } \gamma_1=1\\
 M^2 \eps^{\frac 12}  e^{-\frac s2 } & \mbox{if } \abs{ \gamma}=1\mbox{ and } \gamma_1=0\\
 M^4 \eps^{\frac 32}  e^{-s}  &  \mbox{if } \abs{ \gamma}=2\mbox{ and } \gamma_1=0\\
0&  \mbox{else}
\end{cases}
\end{equation}
for all $y \in \XXX(s)$. 
\end{lemma}
\begin{proof}[Proof of Lemma~\ref{lem:BBS}]
The estimates \eqref{e:bounds_on_garbage} follow directly from the definitions of $f$,  $\Ncal$, $\tt$ and $\Jcal$, together with the bounds on $\phi$ given in \eqref{eq:speed:bound} and the inequality \eqref{e:space_time_conv}. Similarly, \eqref{e:bounds_on_garbage_2} follows by using the $\dot \phi$ estimate in \eqref{eq:acceleration:bound}. To obtain the bound \eqref{eq:V:bnd}, we recall that $V$ is defined in \eqref{def_V}, employ the bounds on $\dot \xi$ and $\dot Q$  given by \eqref{eq:acceleration:bound} and \eqref{eq:dot:Q}, and the fact that $\abs{R-\Id} \leq 1$ which follows from \eqref{eq:speed:bound} and the definition of $R$ in \eqref{eq:R:def}.
\end{proof}

\subsection{Transport estimates}

\begin{lemma}[Estimates for $G_W$, $G_Z$, $G_{U}$, $h_W$, $h_Z$ and $h_{U}$]
\label{lemma_g} 
For $\eps>0$ sufficiently small, and $y\in \XXX(s)$, we have
\begin{align} 
&\abs{\partial^{\gamma}G_W}  \les 
\begin{cases}
M e^{-\frac s2}+ M^{\frac 12}   \abs{y_1}e^{-s}+\eps^{\frac 13} \abs{ \check y},  &\mbox{if } \abs{\gamma}=0\\
 M^2  \eps^{\frac12}, & \mbox{if } \gamma_1=0\mbox{ and } \abs{\check \gamma}=1\\
 M  e^{-\frac s2}, & \mbox{if } \gamma=(1,0,0)\mbox{ or } \abs{\gamma}=2
\end{cases}\label{e:G_W_estimates}\,, \\
&\abs{\partial^{\gamma}(G_Z+(1-\beta_2)e^{\frac s2}\kappa_0)}+\abs{\partial^{\gamma}(G_{U}+(1-\beta_1) e^{\frac s2}\kappa_0)} \les 
\begin{cases}
\eps^{\frac 12} e^{\frac s2},  &\mbox{if } \abs{ \gamma}=0\\
 M^2 \eps^{\frac12}, & \mbox{if } \gamma_1=0\mbox{ and } \abs{\check \gamma}=1\\
 M  e^{-\frac s2}, &\mbox{if } \gamma=(1,0,0)\mbox{ or } \abs{\gamma}=2
\end{cases}\label{e:G_ZA_estimates}\,, \\
& \abs{\partial^{\gamma}h_W}+\abs{\partial^{\gamma}h_Z}+\abs{\partial^{\gamma}h_{U}} \les 
\begin{cases}
 e^{-\frac s2},  &\mbox{if } \abs{ \gamma}=0\\
e^{-s}, & \mbox{if } \gamma_1=0\mbox{ and } \abs{\check \gamma}=1\\
e^{- s} \eta^{-\frac 16}(y), & \mbox{if } \gamma=(1,0,0)\mbox{ or } \abs{\gamma}=2
\end{cases}\label{e:h_estimates}\,. 
\end{align}
 Furthermore, for $\abs{\gamma} \in \{3, 4\}$ we have the lossy global estimates  
 \begin{align}
\abs{\partial^{\gamma}G_W}&\les  e^{- (\frac12 - \frac{ \abs{\gamma}-1}{2k-7})s}\,,
\label{eq:GW:lossy}\\
\abs{\partial^{\gamma}h_W}&\les   e^{-s}\,,
\label{eq:h:lossy}
 \end{align}
for all $y \in \XXX(s)$.
 \end{lemma}
\begin{proof}[Proof of Lemma~\ref{lemma_g}]
Recalling the definition of $G_W$ in \eqref{eq:gW}, and applying \eqref{eq:beta:tau}, \eqref{e:bounds_on_garbage}, \eqref{eq:V:bnd}  the inequality $\kappa\leq M$, and the fundamental theorem of calculus, we obtain that
\begin{align*}
\abs{G_W}
&\les M e^{-\frac s2}\abs{\check y}^2 +e^{\frac s2}\abs{ \kappa  + \beta_2 Z + 2\beta_1V\cdot \nn  }
\\
&\les  M \eps^{\frac 12}  \abs{\check y}+
e^{\frac s2}\sabs{\kappa  + \beta_2 Z^0 - 2\beta_1 (R^T \dot \xi)_1  }  +\abs{y_1}e^{\frac s2} \norm{\partial_1 Z}_{L^{\infty}} + \abs{\check y}e^{\frac s2}\norm{\check \nabla Z}_{\infty} \\
& \qquad \qquad + M^2 \eps^ {\frac{1}{2}}( e^{-s} \abs{y_1} +   \abs{\check y})
\\
&\les M e^{-\frac s2}+  M^{\frac 12}  \abs{y_1}e^{-s}+\eps^{\frac 13} \abs{ \check y}
\end{align*}
where in the second and third inequalities,  we have used \eqref{eq:dot:Q}, \eqref{e:space_time_conv},  \eqref{eq:Z_bootstrap}, 
and \eqref{eq:GW:hW:0}. Thus we obtain \eqref{e:G_W_estimates} for the case $\gamma=0$. Similarly, for the case $\gamma\neq 0$, we have
\begin{align}
\abs{\partial^\gamma G_W}
&\les e^{\frac s2}\left(\sabs{\partial^{\gamma}\dot f}+ M\abs{\partial^\gamma \Jcal}+ \abs{\partial^\gamma (\Jcal Z)}+\abs{\partial^{\gamma}(\Jcal V\cdot \Ncal)}\right) \,.\notag\\
&\les e^{\frac s2} \eps e^{-\frac{\abs{\gamma}}{2} s}{\bf 1}_{\gamma_1=0} 
+e^{\frac s2} \sum_{\beta\leq \gamma,~\beta_1=0} ({\bf 1}_{\abs{\beta}=0} + \eps)e^{-\frac{\abs{\beta}}{2}s} \left(  \sabs{\partial^{\gamma -\beta}Z}  +  \sabs{\partial^{\gamma-\beta}V} \right) \,. \label{eq:GW:bound}
\end{align}
where in the last line we invoked \eqref{e:bounds_on_garbage}.
Hence \eqref{e:G_W_estimates} is concluded by invoking \eqref{eq:Z_bootstrap} and \eqref{eq:V:bnd}.

Now consider the estimates on $G_Z$ and $G_{U}$ as defined in \eqref{eq:gZ} and \eqref{eq:gA}. We note that 
\begin{align*}
G_Z + (1-\beta_2) e^{\frac s2} \kappa_0 
&= G_W +  (1-\beta_2) e^{\frac s2} \left( (\kappa_0-\kappa) + (1-\beta_\tau \Jcal) \kappa + \beta_\tau \Jcal Z\right)
\, , \\
G_{U} + (1-\beta_1) e^{\frac s2} \kappa_0 
&= G_W + (1-\beta_1) e^{\frac s2} \left( (\kappa_0-\kappa) + (1-\beta_\tau \Jcal) \kappa  \right) + (\beta_2-\beta_1)  \beta_\tau e^{\frac s2} \Jcal Z
\,.
\end{align*}
The bounds in \eqref{e:G_ZA_estimates} now follow directly from \eqref{e:G_W_estimates}, the $\dot \kappa$ bound in \eqref{eq:acceleration:bound}, the $\beta_\tau$ estimate \eqref{eq:beta:tau}, the support estimate \eqref{e:space_time_conv}, the $\Jcal$ bounds in \eqref{e:bounds_on_garbage}, and the $Z$ bootstrap assumptions \eqref{eq:Z_bootstrap}\,.

Now consider $h_W$, which is defined in \eqref{eq:hW}.
For the case $\gamma=0$, applying \eqref{eq:acceleration:bound}, \eqref{eq:beta:tau},  and \eqref{e:bounds_on_garbage},  we obtain that
\begin{align*}
\abs{h_W}
\les e^{-s} \abs{W}+  e^{-\frac s2} (\abs{\check V}+\abs{Z}+\abs{A}) 
\les \eps^{\frac 16} e^{-\frac s2} + e^{-\frac s2} (M \eps^ {\frac{1}{2}} + M \eps )  
\les e^{-\frac s2}
\end{align*}
where in the second inequality we have also appealed to \eqref{e:space_time_conv}, \eqref{eq:W_decay}, \eqref{eq:Z_bootstrap}, and \eqref{eq:A_bootstrap},
and where we have used the fact that $\abs{\check V} \les M \eps^ {\frac{1}{2}}$.    This last inequality is obtained using the fact that we need only bound
$\abs{\check V}$.   Using definition \eqref{def_V},  because of the bounds  \eqref{eq:speed:bound} and \eqref{eq:dot:Q}, it remains to bound
$\abs{R_{j\mu} \dot \xi_j}$.
Restricting  \eqref{eq:gW} and \eqref{eq:hW} to $y=0$, and with $f$ given by \eqref{def_f} and using \eqref{eq:constraints}, we find that 
\begin{align*}
2\beta_1  (R^T \dot \xi)_\mu = 2\beta_1 A_\mu^0 - \tfrac{1}{\beta_\tau} e^{\frac s2} h_W^{\mu,0} \,.
\end{align*}
Hence, by \eqref{eq:A_bootstrap} and \eqref{eq:GW:hW:0}, we see that $\abs{(R^T \dot \xi)_\mu} \les  M \eps^ {\frac{1}{2}} $.

Similarly, invoking the same set of inequalities together with \eqref{eq:V:bnd}, for the case that $\gamma\neq 0$, we obtain 
\begin{align}
\abs{\partial^\gamma h_W^{\mu}}
&\les e^{-s}\abs{\partial^\gamma(\Ncal_\mu W )}+e^{-\frac s2} \left(
\abs{\p^\gamma V} +  M\abs{\partial^\gamma \Ncal_\mu}+\abs{\partial^\gamma (\Ncal_{\mu}Z )}+\abs{\partial^\gamma (A_\gamma \Tcal^\gamma_\mu)}\right) \notag \\
&
\les   \sum_{\beta\leq \gamma,~\beta_1=0} 
  e^{-\frac{\abs{\beta}+1}{2}s} \left(\eps e^{-\frac s2}\abs{\partial^{\gamma-\beta} W }+ \eps  
 \abs{\partial^{\gamma-\beta}Z }+   ( {\bf 1}_{|\beta|=0} + \eps)   \abs{\partial^{\gamma-\beta} A_\gamma } \right)\notag\\
&\qquad + M \eps  e^{-\frac{\abs{\gamma}+1}{2}s} {\bf 1}_{\gamma_1=0}+  M^2 \eps^{\frac 12} e^{-\frac{\abs{\gamma}+1}{2}s}{\bf 1}_{\gamma_1=0} + M^2 \eps^{\frac 12} e^{-2 s} {\bf 1}_{\gamma_1 \geq 1}  \,.\label{eq:hW:bound}
\end{align}
Finally, applying \eqref{e:space_time_conv},  \eqref{eq:W_decay}, \eqref{eq:Z_bootstrap}, \eqref{eq:A_bootstrap} and \eqref{eq:A:higher:order} we obtain the estimate on $h_W$. The estimates on $h_Z$ and $h_{U}$ are completely analogous since the only difference between these functions and $h_W$ lies in the different combinations of $\beta_1, \beta_2$ parameters.

 The estimates \eqref{eq:GW:lossy} and \eqref{eq:h:lossy}, follow as a consequence of  \eqref{eq:GW:bound}, \eqref{eq:hW:bound}, \eqref{eq:W_decay}, \eqref{eq:Z_bootstrap}--\eqref{eq:Z:higher:order}, and the estimate $\norm{\p^\gamma W}_{L^\infty} \les \norm{D^2 W}_{L^\infty}^{1-\frac{2\abs{\gamma}-4}{2k-7}} \norm{W}_{\dot H^k}^{\frac{2\abs{\gamma}-4}{2k-7}} \les M^{2k}$ which holds for $\abs{\gamma} \in \{3,4\}$ in view of  Lemma~\ref{lem:GN}, Proposition~\ref{cor:L2}, and of \eqref{eq:W_decay}. 
\end{proof} 

\subsection{Forcing estimates}

\begin{lemma}[Estimates on $\partial^\gamma F_W$, $\partial^\gamma F_Z$ and $\partial^\gamma F_A$] 
\label{lem:forcing}
For $y \in \XXX(s)$ we have the force bounds 
\begin{align}
\abs{\partial^\gamma F_W}+
e^{\frac s2}\abs{\partial^\gamma F_Z} &\les 
\begin{cases}
e^{-\frac s 2 },  &\mbox{if } \abs{ \gamma}=0\\
e^{-s} \eta^{- \frac16 + \frac{2 \abs{\gamma} +1}{3(2k-5)}}(y) ,&\mbox{if } \gamma_1 \geq 1 \mbox{ and } \abs{ \gamma}=1,2 \\
 M^2 e^{-s}, & \mbox{if } \gamma_1=0\mbox{ and } \abs{\check \gamma}=1\\
 e^{- (1 -  \frac{3}{2k-7} )s}, & \mbox{if } \gamma_1=0\mbox{ and } \abs{\check \gamma}=2
\end{cases} \,,\label{eq:F_WZ_deriv_est}
\\
\abs{\partial^\gamma F_{A \nu}} &\les 
\begin{cases}
 M^{\frac 12}  e^{-s},  &\mbox{if }  \abs{ \gamma}=0\\
 (M^{\frac 12} + M^2 \eta^{-\frac 16})  e^{-s}, & \mbox{if } \gamma_1=0\mbox{ and } \abs{\check \gamma}=1\\
e^{\left(-1+\frac{3}{2k-7}\right)s} \eta^{-\frac16}(y), & \mbox{if } \gamma_1=0\mbox{ and } \abs{\check \gamma}=2
\end{cases} \,.
\label{eq:F_A_deriv_est}
\end{align}
Moreover, we have the following higher order estimate at $y=0$
\begin{align}
\abs{(\partial^\gamma \tilde F_W)^0}
\les  e^{- (\frac12 - \frac{4}{2k-7})s}\quad\mbox{for}\quad \abs{\gamma}=3\label{e:FW3}
\end{align}
and the bound on $\tilde{F}_W$
\begin{equation}\label{eq:Ftilde_est}
\abs{\partial^\gamma \tilde F_W}
\les  M \eps^{\frac 16} 
\begin{cases}
\eta^{-\frac{1}{6}}(y),&\mbox{if}\quad \abs{\gamma}=0 \\
\eta^{-   \frac 12 + \frac{ 3 }{2k-5}}(y),&\mbox{if}\quad \gamma_1=1 \mbox{ and }\abs{\check \gamma}=0 \\
\eta^{-\frac{1}{3}}(y) ,&\mbox{if}\quad \gamma_1=0 \mbox{ and }\abs{\check \gamma}=1 \\
1 ,&\mbox{if}\quad \abs{\gamma}=4\quad\mbox{and}\quad \abs{y}\leq \ell
\end{cases}
\end{equation}
holds for all $ \abs{y} \leq \LLL $.
\end{lemma}
\begin{proof}[Proof of Lemma~\ref{lem:forcing}]
By the definition \eqref{eq:FW:def} we have
\begin{align*}
\abs{\partial^\gamma F_W}&\les \abs{\partial^{\gamma}(S \Tcal^\nu_\mu \partial_{\mu} A_\nu)}+e^{-\frac s2} \abs{\partial^{\gamma}(A_\nu \Tcal^\nu_i \dot{\Ncal}_i)}+ e^{-\frac s2} \abs{\partial^{\gamma}(A_\nu \Tcal^\nu_j \Ncal_i)}
\notag \\
&+ e^{-\frac s2} \abs{\partial^{\gamma}\left(\left(V_\mu +\Ncal_\mu  U \cdot \nn +  A_\nu \Tcal^\nu_\mu \right)A_\gamma \Tcal^\gamma_i \Ncal_{i,\mu}\right)}
+ e^{-\frac s2} \abs{\partial^{\gamma}\left(S \left( A_\nu \Tcal^\nu_{\mu,\mu} + U\cdot \Ncal \Ncal_{\mu,\mu} \right)\right)}
\notag \\
&\les \sum_{\beta\leq \gamma,~\beta_1=0}e^{-\frac{\abs{\beta}+1}{2}s}\bigg(e^{\frac s2}  \abs{\partial^{\gamma-\beta}(S  \check\nabla A)}  + \eps^{\frac 14} \abs{\partial^{\gamma-\beta}A}+ \eps \abs{\partial^{\gamma-\beta}\left(V \otimes A\right)}
\notag\\
&\qquad +\eps \abs{\partial^{\gamma-\beta}\left(U\cdot \Ncal A\right)}
+\eps \abs{\partial^{\gamma-\beta}\left(A\otimes A\right)}+\eps \abs{\partial^{\gamma-\beta}\left(S A\right)}+\eps \abs{\partial^{\gamma-\beta}\left(S U\cdot \Ncal \right)}
\bigg)
\end{align*}
where we invoked \eqref{eq:dot:Q}, \eqref{e:bounds_on_garbage}, and \eqref{e:bounds_on_garbage_2}. Combining the above estimate with \eqref{eq:A_bootstrap}, \eqref{eq:A:higher:order}, \eqref{eq:V:bnd} and Lemma \ref{lem:US_est} we obtain the bounds claimed in \eqref{eq:F_WZ_deriv_est} for $\p^\gamma F_W$.  Using the same set of estimates we also obtain 
\begin{equation}\abs{\partial^\gamma  F_W }\les e^{-\frac{s}{2}}\label{e:FW3A}
\end{equation}
for $\abs{\gamma}=3$,
which we shall need later in order to prove \eqref{e:FW3},
and 
\begin{equation}
 \abs{\partial^\gamma F_W}\les \eps^{\frac 16}  \label{e:FW4A}
\end{equation}
for $\abs{\gamma}=4$ and $\abs{y}\leq \ell$,
which we shall need later in order to prove the last case of \eqref{eq:Ftilde_est}. Comparing \eqref{eq:FZ:def} and \eqref{eq:FW:def}, we note that the estimates on $\p^\gamma F_Z$ claimed in \eqref{eq:F_WZ_deriv_est} are completely analogous to the estimates ones $\p^\gamma F_W$ up to a factor of $e^{-\frac s2}$.

Now we consider the estimates on $F_A$. By definition \eqref{eq:A:def}, we have
\begin{align*}
\abs{\partial^{\gamma} {F}_{A_\nu}}&\les e^{-\frac s2}  \abs{\partial^{\gamma}( S T^\nu_\mu \p_{\mu} S)} 
 +  e^{-s}  \abs{ \partial^{\gamma}\left(\left( U\cdot\Ncal \Ncal_i + A_\gamma \Tcal^\gamma_i\right) \dot{\Tcal}^\nu_i \right)} 
 +  e^{-s}  \abs{\partial^{\gamma }\left(\left(  U\cdot \Ncal \Ncal_j + A_\gamma \Tcal^\gamma_j\right)\Tcal^\nu_i \right)}\\
 &\qquad
 + e^{-s} \abs{\partial^{\gamma}\left( \left(V_\mu + U\cdot \Ncal \Ncal_\mu + A_\gamma \Tcal^\gamma_\mu\right) \left(U\cdot \Ncal \Ncal_i 
 +    A_\gamma \Tcal^\gamma_i \right) \Tcal^\nu_{i,\mu} \right)}  \\
 &\les \sum_{\beta\leq \gamma,~\beta_1=0}e^{-\frac{\abs{\beta}+2}{2}s}\bigg(
  e^{\frac s2}  \abs{\partial^{\gamma-\beta}( S \check\nabla S)}+ \abs{\partial^{\gamma-\beta}( U\cdot \Ncal )} + \abs{\partial^{\gamma-\beta}A} \\
  &\qquad\qquad\qquad\qquad  + \sum_{\alpha\leq \gamma-\beta}\left(\abs{\partial^{\alpha}V}+
  \abs{\partial^{\alpha}(U\cdot \Ncal)}+\abs{\partial^{\alpha}A}\right)\left(\abs{\partial^{\gamma-\beta-\alpha}(U\cdot\Ncal)}+\abs{\partial^{\gamma-\beta-\alpha}A}\right)\bigg)
  \end{align*}
where we again invoked \eqref{eq:dot:Q}, \eqref{e:bounds_on_garbage}, and \eqref{e:bounds_on_garbage_2}. Combining the above bound with the estimates  \eqref{eq:A_bootstrap}, \eqref{eq:V:bnd} and with Lemma \ref{lem:US_est}, we obtain our claim \eqref{eq:F_A_deriv_est}. 

By definition \eqref{eq:tilde:W:evo} and \eqref{eq:acceleration:bound}
\begin{align}
\abs{ \partial^\gamma\tilde{F}_W}
 &\les \abs{\p^\gamma F _W} +  M^2  e^{-s}{\bf 1}_{\abs{\gamma}=0 } +\abs{\partial^{\gamma}((1-\beta_\tau \Jcal) \bar W\p_1 \bar W)}  + M^2  \abs{\partial^{\gamma}(G_W\p_1\bar W)}+ \abs{\partial^{\gamma}(h_W^\mu \p_\mu \bar W)} \notag\\
  &\les \abs{\partial^\gamma F _W} +  M^2  e^{-s}{\bf 1}_{\abs{\gamma}=0 } + M  \eps\sum_{\beta\leq \gamma,~\beta_1=0} e^{-\frac {|\beta|}2 s} \abs{ \p^{\gamma-\beta}\p_1(\bar W^2)} \notag\\
  &\qquad+ \sum_{\beta\leq \gamma} \abs{\partial^{\beta} G_W\, \p^{\gamma-\beta}\p_1\bar W}+\abs{\partial^{\beta} h_W^\mu  \p^{\gamma-\beta} \p_\mu \bar W } \notag\\
  &\les \abs{\partial^\gamma F _W} + M^2  e^{-s}{\bf 1}_{\abs{\gamma}=0 } + M  \eps\sum_{\beta\leq \gamma,~\beta_1=0} e^{-\frac {|\beta|}2 s} \eta^{-\frac 16 - \frac{\gamma_1}{2}  - \frac{\abs{\check \gamma - \check \beta}}{6}}(y)  \notag\\
  &\qquad+ \sum_{\beta\leq \gamma} \left( \abs{\partial^{\beta} G_W}\eta^{-\frac 13}(y)  +\abs{\partial^{\beta} h_W}  \right) \eta^{  - \frac{\gamma_1}{2}  - \frac{\abs{\check \gamma - \check \beta}}{6}}(y) 
  \label{eq:spumante}
\end{align}
where we  used \eqref{eq:beta:tau} and \eqref{e:bounds_on_garbage} to bound
\[\abs{\partial^{\beta}(1-\beta_\tau \Jcal)} \les (1-\beta_\tau )\abs{\partial^{\beta}\Jcal}+
\abs{\partial^{\beta}(1-\Jcal)}\les  M  \eps e^{-\frac {\abs{\beta}}2 s}\,. \]
Finally, applying \eqref{e:space_time_conv}, \eqref{e:G_W_estimates}, \eqref{e:h_estimates}--\eqref{eq:h:lossy}, \eqref{eq:F_WZ_deriv_est}, and \eqref{e:FW4A}, we can bound all the remaining terms in \eqref{eq:spumante} to obtain \eqref{eq:Ftilde_est}.   Note that in the $G_W$ estimate \eqref{e:G_W_estimates} we have used that $\abs{y} \leq \LLL = \eps^{-\frac{1}{10}}$, while in bounding $\p_1\tilde F_W$, we have used \eqref{e:space_time_conv} in order convert the temporal decay of $\p_1 F_W$ to spatial decay, as well as absorbing the $M$ and gaining the extra factor of $\eps^{\frac 16}$.

Now let us consider the estimate \eqref{e:FW3}. By definition \eqref{eq:tilde:W:evo} and the explicit formula for $\bar W$ (in particular, even derivatives of $\bar W$ vanish at $0$ as well as $\check\nabla \bar W$) and the explicit formula for $\Jcal$, we obtain
\begin{align*}
\abs{(\nabla^3 \tilde F_W)^0}&\les \abs{(\nabla^3 F _W)^0} +   \abs{(\nabla^3((\beta_\tau \Jcal -1)\bar W-G_W))^0}+ \abs{(\nabla((\beta_\tau \Jcal -1)\bar W-G_W))^0}+ \abs{(\nabla h_W)^0}\notag\\
&\les \abs{(\nabla^3 F _W)^0} +   \abs{(\check\nabla^2J)^0}+\abs{1-\beta_\tau}+ \abs{(\nabla^3 G_W)^0}+\abs{(\nabla G_W)^0}+ \abs{(\nabla h_W)^0}\notag\\
&\les e^{-\frac s2} +   e^{-s}+ Me^{-s}+  e^{- (\frac12 - \frac{4}{2k-7})s}+\abs{(\nabla G_W)^0}+e^{-s}\notag\\&
\les  e^{- (\frac12 - \frac{4}{2k-7})s}+\abs{(\nabla G_W)^0}
\end{align*}
where we used \eqref{eq:beta:tau},  \eqref{e:bounds_on_garbage}, \eqref{eq:GW:lossy}, \eqref{e:h_estimates} and  \eqref{e:FW3A}. Using the indentity \eqref{eq:check_derivative_GW}, and applying   \eqref{e:G_W_estimates}  and \eqref{eq:F_WZ_deriv_est} we obtain
\begin{equation*}
\abs{(\nabla G_W)^0}
\les  M  e^{-\frac s2}+\abs{(\check\nabla G_W)^0}
\les  M  e^{-\frac s2}+\abs{(\check\nabla F_W)^0} 
\les  M  e^{-\frac s2}\,.
\end{equation*}
Combining the two estimates above we obtain \eqref{e:FW3}.
\end{proof}

\begin{corollary}[Estimates on the forcing terms] 
\label{cor:forcing}
Assume that $k\geq 18$. Then, we have
\begin{align}
\abs{F^{(\gamma)}_W} &\les 
\begin{cases}
e^{-\frac s 2 },  &\mbox{if } \abs{ \gamma}=0\\
\eps^{\frac18} \eta^{- \frac 12 + \frac{3}{2k-5} }(y),&\mbox{if }  \gamma=(1,0,0)\\
 \eta^{-\frac 13}(y) ,&\mbox{if }  \gamma=(2,0,0) \\
 M^{\frac13} \eta^{-\frac 13}(y)  ,&\mbox{if } \gamma_1=1 \mbox{ and } \abs{\check \gamma}=1 \\
 M^2 \eps^{\frac 13} \eta^{-\frac 13}(y) , & \mbox{if } \gamma_1=0\mbox{ and } \abs{\check \gamma}=1\\
M^{\frac23} \eta^{-\frac 13 + \frac{1}{2k-7} }(y), & \mbox{if } \gamma_1=0\mbox{ and } \abs{\check \gamma}=2
\end{cases}\label{eq:forcing_W}\\
\abs{F^{(\gamma)}_Z} &\les \begin{cases}
e^{-s  },  &\mbox{if } \abs{ \gamma}=0\\
 e^{-\frac 32s} \eta^{- \frac{2}{2k-5}},& \mbox{if } \gamma_1= 1 \mbox{ and } \abs{\gamma}= 1  \\
  e^{-\frac 32s} ( M^{\frac{\abs{\check \gamma}}{2}}  +  M^{2} \eta^{-\frac 16}) ,& \mbox{if } \gamma_1\geq 1 \mbox{ and } \abs{\gamma}= 2  \\
M^2 e^{-\frac32 s}, & \mbox{if } \gamma_1=0\mbox{ and } \abs{\check \gamma}=1\\
e^{- (\frac32 - \frac{3}{2k-7})s}, & \mbox{if } \gamma_1=0\mbox{ and } \abs{\check \gamma}=2
\end{cases} \label{eq:forcing_Z}\\
\abs{F^{(\gamma)}_{A\nu}} &\les
\begin{cases}
M^{\frac 12} e^{-s},  &\mbox{if }  \abs{ \gamma}=0\\
(M^{\frac 12} + M^2 \eta^{- \frac 16}) e^{-s}, & \mbox{if } \gamma_1=0\mbox{ and } \abs{\check \gamma}=1\\
e^{\left(-1+ \frac{3}{2k-7} \right)s}\eta^{-\frac16}(y), & \mbox{if } \gamma_1=0\mbox{ and } \abs{\check \gamma}=2
\end{cases}\label{eq:forcing_U}
\,.
\end{align}
Moreover, we have the following higher order estimate
\begin{align}
\abs{ \tilde F_W^{(\gamma),0}}&\les  e^{- (\frac12 - \frac{4}{2k-7})s}\quad\mbox{for}\quad \abs{\gamma}=3\label{e:forcing:W3}
\end{align} 
and the following estimates on $\tilde F^{(\gamma)}_W$
\begin{alignat}{2}
\abs{\tilde F^{(\gamma)}_W} &\les  \eps^{\frac 1{11}}
\eta^{- \frac 12}(y)
\quad &&\mbox{for }\gamma=(1,0,0) \mbox{ and }  \abs{y} \leq \LLL \label{eq:Ftilde_d1_est}\\
\abs{\tilde F^{(\gamma)}_W} &\les  \eps^{\frac 1{12}} \eta^{-\frac 13}(y)
\quad &&\mbox{for }\gamma_1=0,  \abs{\check \gamma}=1 \mbox{ and } \abs{y} \leq \LLL\label{eq:Ftilde_dcheck_est}\\
\abs{\tilde F^{(\gamma)}_W} &\les \eps^{\frac18}
+ \eps^{\frac{1}{10}} (\log M)^{\abs{\check \gamma}-1}
\quad &&\mbox{for }\abs{\gamma}=4 \mbox{ and }\abs{y}\leq \ell\label{eq:Ftilde_4th_est}\,.
 \end{alignat}
\end{corollary}
\begin{proof}[Proof of Corollary~\ref{cor:forcing}]
First we establish \eqref{eq:forcing_W}. Note that in this estimate $\abs{\gamma}\leq 2$, and thus by definition \eqref{eq:F:W:def} we have
\begin{align*}
\abs{F_{W}^{(\gamma)}}&\les \abs{\p^\gamma F_W }
+ \!\!\! \sum_{0\leq \beta < \gamma}  \left( \sabs{\p^{\gamma-\beta}G_W \p_1 \p^\beta W} + \sabs{\p^{\gamma-\beta}h_W^\mu \p_\mu \p^\beta W}  \right) + {\bf 1}_{|\gamma|= 2}  \!\!\! \sum_{\substack{ |\beta| = |\gamma|-1 \\ \beta\le\gamma, \beta_1 = \gamma_1}}\sabs{ \p^{\gamma-\beta} (\Jcal W)   \p_1\p^\beta W  }\\
&=:\abs{\partial^{\gamma}F_{W}}+\mathcal I_1+\mathcal I_2 \,.
\end{align*}
In order to estimate $\mathcal I_1$, we utilize \eqref{eq:W_decay}, \eqref{e:G_W_estimates}, \eqref{e:h_estimates}, and for $\abs{\gamma} \leq 2$  obtain
\begin{align*}
\mathcal I_1&\les
M \eta^{-\frac 13}  \left(e^{-\frac s2}+  M^2  \eps^{\frac12} ({\bf 1}_{\abs{\gamma}=2 } + {\bf 1}_{\abs{\gamma}=\abs{\check \gamma}=1})  \right)+Me^{-s}\left({\bf 1}_{\abs{\gamma}=\abs{\check\gamma}=1}+ \eta^{-\frac 16} \right)\\
&\les M \eta^{-\frac 13}\left(e^{-\frac s2}+  \eps^{\frac13}  ({\bf 1}_{\abs{\gamma}=2 } + {\bf 1}_{\abs{\gamma}=\abs{\check \gamma}=1})  \right) \, ,
\end{align*}
where in the last inequality we invoked \eqref{e:space_time_conv}.
Next, we consider the   $\mathcal I_2$ term. We first note that  $\mathcal I_2=0$ when $\gamma_1=2$. From \eqref{eq:W_decay} and \eqref{e:bounds_on_garbage}, using that $\abs{\gamma-\beta}=1$, and that $\abs{\check\beta} = \abs{\check \gamma} -1$, we have
\[\mathcal I_2\les  {\bf 1}_{|\gamma|= 2} \sum_{\substack{ |\beta| = |\gamma|-1 \\ \beta\le\gamma, \beta_1 = \gamma_1}}\abs{    \p_1\p^\beta W  }\les  M^{\frac{\abs{\check \gamma}}{3}}  \eta^{-\frac 13}\,. \] 
Combining the above three estimates with \eqref{eq:F_WZ_deriv_est} and  \eqref{e:space_time_conv}, we obtain \eqref{eq:forcing_W}.  Here we have used that for the $\gamma_1\geq 1$ and $\abs{\gamma} \in \{1,2\}$ case of \eqref{eq:F_WZ_deriv_est}, $\frac{2\abs{\gamma}+1}{2k-5} \leq \frac 16$, which is where the assumption $k\geq 18$ arises from.  

Similarly, for $\abs{\gamma}\leq 2$, from \eqref{eq:F:ZA:def} we have
\begin{align*}
\abs{F_{Z}^{(\gamma)}}&\les \abs{\p^\gamma F_Z }
+ \sum_{0\leq \beta < \gamma}  \left(\sabs{\p^{\gamma-\beta}G_Z \p_1 \p^\beta Z} + \sabs{\p^{\gamma-\beta}h_Z^\mu \p_\mu\p^\beta Z}\right) \notag\\
&\qquad +  {\bf 1}_{\abs{\gamma}=2} \sabs{\p_1 Z \p^\gamma(\Jcal W)}   +\sum_{\substack{ |\beta| = |\gamma|-1 \\ \beta\le\gamma, \beta_1 = \gamma_1}}\sabs{ \p^{\gamma-\beta} (\Jcal W)   \p_1\p^\beta Z  }\\
&=\abs{\partial^{\gamma}F_{Z}} +\mathcal I_1 +  {\bf 1}_{\abs{\gamma}=2  \sabs{\p_1 Z \p^\gamma(\Jcal W)}} +\mathcal I_2 \, .
\end{align*}
First, we note that by \eqref{eq:F_WZ_deriv_est}  the available estimates for $\p^\gamma F_Z$ are consistent with \eqref{eq:forcing_Z} since  $k\geq 18$ and thus $- \frac 16 + \frac{5}{2k-5} \leq 0$. 
Second, we note that for $\abs{\gamma}=2$, by   \eqref{eq:W_decay}, \eqref{eq:Z_bootstrap} and \eqref{e:bounds_on_garbage}, we have
\[
 \sabs{\p_1 Z \p^\gamma(\Jcal W)} \les M^{\frac 12} e^{-\frac 32 s} \left( M \eta^{-\frac 16} {\bf 1}_{\gamma_1=0} + M^{\frac 23} \eta^{-\frac 13} {\bf 1}_{\gamma_1 \geq 1}  + \eps e^{-\frac s2} \right) \, ,
\]
a bound which is consistent with \eqref{eq:forcing_Z}.
Next, in order to estimate $\mathcal I_1$ we utilize \eqref{eq:Z_bootstrap}, \eqref{e:G_ZA_estimates}, \eqref{e:h_estimates}, and \eqref{e:space_time_conv}, we obtain 
\begin{align*}
\mathcal I_1
&\les
e^{-\frac32 s}\left(M^2  e^{-\frac s2} + M^3 \eps^{\frac 12} {\bf 1}_{ \abs{\check\gamma}\geq 1}  + M \eps^{\frac12} \eta^{-\frac 16} \right) \,.
\end{align*}
Lastly, we consider $\mathcal I_2$. We first note that for $\abs{\gamma}\leq 2$, we have $\mathcal I_2=0$ whenever $\abs{\gamma}= \gamma_1$. For $\abs{\gamma} > \gamma_1$, from  \eqref{eq:W_decay}, \eqref{eq:Z_bootstrap} and \eqref{e:bounds_on_garbage}, we have
\[
\mathcal I_2\les\sum_{\substack{ |\beta| = |\gamma|-1 \\ \beta\le\gamma, \beta_1 = \gamma_1}}\abs{    \p_1\p^\beta Z  }\les \left({\bf 1}_{\abs{\check\gamma}=1} M^{\frac 12}+{\bf 1}_{\abs{\check\gamma}=2 }M\right) e^{-\frac32 s} 
\,.
\] 
Upon inspection, we note that the bounds for ${\mathcal I}_1$ and ${\mathcal I}_2$ obtained above are consistent with \eqref{eq:forcing_Z}, thereby concluding the proof of this bound.

In order to prove the $\sabs{F_{A}^{(\gamma)}}$ estimate, we use the definition \eqref{eq:F:ZA:def}, with $\gamma_1 =0$ and $\abs{\check \gamma} \leq 2$, and ignore the subindex $\nu$ to arrive at
\begin{align*}
\sabs{F_{A}^{(\gamma)}}&\les \abs{\p^\gamma F_A }
+ \sum_{0\leq \beta < \gamma}  \left(\sabs{\p^{\gamma-\beta}G_{U} \p_1 \p^\beta A} + \sabs{\p^{\gamma-\beta}h_{U}^\mu \p_\mu \p^\beta A}\right)  \notag\\
&\qquad +  {\bf 1}_{\abs{\gamma}=2} \p_1 A \p^\gamma (\Jcal W) +\sum_{\substack{ |\beta| = |\gamma|-1 \\ \beta\le\gamma, \beta_1 = \gamma_1=0}}\sabs{ \p^{\gamma-\beta} (\Jcal W)   \p_1\p^\beta A  }\\
&=\abs{\partial^{\gamma}F_{A}}+\mathcal I_1+  {\bf 1}_{\abs{\gamma}=2} \p_1 A \p^\gamma (\Jcal W)  + \mathcal I_2 \,.
\end{align*}
The bounds for $\p^\gamma F_A$ previously established in \eqref{eq:forcing_U} are the same as the desired bound in \eqref{eq:F_A_deriv_est}. Moreover, for $\abs{\gamma}=2$, by \eqref{eq:W_decay}, \eqref{eq:A_bootstrap} and \eqref{e:bounds_on_garbage},
\[
 \abs{\p_1 A \p^\gamma (\Jcal W)} \les M e^{-\frac 32 s} \left( M \eta^{-\frac 16} + \eps e^{-\frac s2} \right)
\]
which is consistent with the last bound in \eqref{eq:forcing_U}. In order to bound ${\mathcal I}_1$, we appeal to \eqref{eq:A_bootstrap}, \eqref{eq:A:higher:order}, \eqref{e:G_ZA_estimates}, and \eqref{e:h_estimates} to deduce
\[
{\mathcal I}_1 \les M^3 \eps^{\frac 12} e^{-\frac 32 s} + {\bf 1}_{\abs{\gamma}=2} M^2 \eps^{\frac 12} e^{-(\frac 32 - \frac{3}{2k-5})s}
\]
which is consistent with \eqref{eq:forcing_U} in view of \eqref{e:bounds_on_garbage}. 
Lastly, from the same bounds and using \eqref{eq:W_decay}, we arrive at
\[
{\mathcal I}_2 \les \abs{\check \nabla (\Jcal W)} \left( {\bf 1}_{\abs{\gamma}=1} \abs{\p_1 A} + {\bf 1}_{\abs{\gamma}=2} \abs{\p_1 \check \nabla A} \right) \les {\bf 1}_{\abs{\gamma}=1} M e^{-\frac 32 s} + {\bf 1}_{\abs{\gamma}=2} e^{-(\frac 32 - \frac{3}{2k-5})s}
\]
which combined with \eqref{e:bounds_on_garbage} completes the proof of \eqref{eq:forcing_U}.

Next, we turn to the proof of the $\tilde F^{(\gamma)}_W$ in \eqref{e:forcing:W3}--\eqref{eq:Ftilde_4th_est}. 
For $\abs{\gamma}=1$ and $\abs{y}\leq \LLL$, we consider the forcing term $\tilde F_W^{(\gamma)}$ defined in \eqref{eq:p:gamma:tilde:F}, and estimate it as
\begin{align*}
\abs{\tilde F^{(\gamma)}_W}
&\les \abs{\p^\gamma \tilde F_W }
+ \abs{\p^{\gamma}G_W} \abs{\p_1  \tilde W} + \abs{\p^{\gamma} h_{W}} \abs{\check\nabla \tilde W}+ \abs{\p^{\gamma} (\Jcal \partial_1\bar W)} \abs{  \tilde W} 
+ {\bf 1}_{\abs{\check\gamma}=1 }
\abs{ \p^{\gamma} (\Jcal W)}  \abs{  \p_1\tilde W} \,.
\end{align*}
If $\abs{\gamma}=\gamma_1=1$, utilizing \eqref{eq:bootstrap:Wtilde}, \eqref{eq:bootstrap:Wtilde1}, \eqref{eq:bootstrap:Wtilde2},  \eqref{e:bounds_on_garbage}, \eqref{e:G_W_estimates}, \eqref{e:h_estimates}, the explicit bounds on $\bar W$, and the previously  established  estimate \eqref{eq:Ftilde_est}, we obtain
\begin{align*}
\abs{\tilde F^{(\gamma)}_W}
 \les  M \eps^{\frac 16}  \eta^{-\frac 12 + \frac{3}{2k-5}}  + M \eps^{\frac{1}{12}} e^{-\frac s2} \eta^{-\frac 13}+ \eps^{\frac{1}{13}} e^{-s}\eta^{-\frac 16} +\eps^{\frac{1}{11}} \eta^{-\frac 23} 
 \les \eps^{\frac{1}{11}} \eta^{-\frac 12}
 \end{align*}
 where in the last inequality we invoked \eqref{e:space_time_conv}  and the fact that $\abs{y} \leq \LLL = \eps^{-\frac{1}{10}}$, which yields $M \eps^{\frac 16} \eta^{\frac{3}{2k-5}} \les M \eps^{\frac 16} \LLL^{\frac{18}{2k-5}} \les \eps^{\frac{1}{11}}$ for $k \geq 18$, by taking $\eps$ to be sufficiently small in terms of $k$ and $M$. Similarly for $\abs{\gamma}=\abs{\check\gamma}=1$, applying the same set of bounds yields \begin{align*}
\abs{\tilde F^{(\gamma)}_W}
\les  M \eps^{\frac 16} \eta^{-\frac 13} +  \eps^{\frac 12} \eta^{-\frac 13} + \eps^{\frac{1}{13}} e^{-s}  +\eps^{\frac{1}{11}} \eta^{\frac 16} \left(e^{-\frac s2} \eta^{-\frac 13} +\eta^{-\frac 12}  \right) 
+\eps^{\frac{1}{12}} \eta^{-\frac 13} \left(e^{-\frac s2} \eta^{\frac 16} + 1  \right)
\les \eps^{\frac{1}{12}}\eta^{-\frac 13}  \,.
 \end{align*}
Here we have use that $\norm{\eta^{\frac 12} \p_1 \check \nabla \bar W}_{L^\infty} \les 1$, which is a sharper estimate than what we have written earlier in~\eqref{eq:bar:W:properties}.
This concludes the proof of \eqref{eq:Ftilde_d1_est} and of \eqref{eq:Ftilde_dcheck_est}.

Consider now the estimate \eqref{e:forcing:W3}. Evaluating  \eqref{eq:p:gamma:tilde:F} at $y=0$, applying the constraints \eqref{eq:constraints},  the identity \eqref{eq:check_derivative_GW}, and using properties of the function $\bar W$ at $0$, we obtain for $\abs{\gamma}=3$ that
\begin{align*}
\abs{F^{(\gamma),0}_W}
&\les \abs{ \p^\gamma \tilde F_W^0}+
\abs{\nabla G_W^0} \abs{\p_1 \nabla^2 \tilde W^0} + \abs{\nabla h_{W}^0} \abs{\check\nabla \nabla^2 \tilde W^0} \notag\\
&\les \abs{\p^\gamma \tilde F_W^0} +
\left(\abs{\p_1 G_W^0} +\abs{\check\nabla  F_W^0} + \abs{\nabla h_{W}^0}\right) \left( \abs{\nabla^3W^0}   + \abs{\nabla^3 \bar W^0} \right)
\,.
 \end{align*}
 Then apply \eqref{e:G_W_estimates}, \eqref{e:h_estimates}, \eqref{eq:F_WZ_deriv_est}, \eqref{e:FW3}, and \eqref{eq:bootstrap:Wtilde3:at:0}, we obtain
 \begin{align*}
F^{(\gamma),0}_W
&\les  e^{- (\frac12 - \frac{4}{2k-7})s}+ Me^{-\frac s2} + M^2
  e^{-s} + e^{-s}\les e^{- (\frac12 - \frac{4}{2k-7})s} 
\end{align*}
thereby concluding the proof of \eqref{e:forcing:W3}.
 
Lastly, we consider the bound \eqref{eq:Ftilde_4th_est}, which needs to be established only for $\abs{y}\leq \ell$. 
For $\abs{\gamma}=4$ we consider the forcing term defined in \eqref{eq:p:gamma:tilde:F} and bound it using \eqref{eq:bootstrap:Wtilde:near:0}, \eqref{e:bounds_on_garbage}, \eqref{e:G_W_estimates}, \eqref{e:h_estimates}, \eqref{eq:GW:lossy}, \eqref{eq:h:lossy}, \eqref{eq:Ftilde_est}, and the explicit bounds of $\bar W$ as
\begin{align}
\abs{\tilde F^{(\gamma)}_W}
&\les \abs{\p^\gamma \tilde F_W }
+ \sum_{0\leq \beta < \gamma} \left( \abs{\p^{\gamma-\beta}G_W} \abs{\p_1 \p^\beta \tilde W} + \abs{\p^{\gamma-\beta} h_{W}^\mu} \abs{\p_\mu \p^\beta \tilde W}+ \abs{\p^{\gamma-\beta} (\Jcal \partial_1\bar W)} \abs{  \p^\beta \tilde W} \right) \notag\\
&\quad+ \sum_{\substack{ 0 \leq |\beta| \leq |\gamma|-2 \\ \beta\le\gamma}} \abs{\p^{\gamma-\beta} (\Jcal W)}  \abs{  \p_1\p^\beta \tilde W  }
+ \sum_{\substack{ |\beta| = |\gamma|-1 \\ \beta\le\gamma, \beta_1 = \gamma_1}} 
\abs{ \p^{\gamma-\beta} (\Jcal W)}  \abs{  \p_1\p^\beta \tilde W} \notag\\
&\les  M \eps^{\frac16} 
+ \sum_{0\leq \beta < \gamma}  \left(\eps^{\frac13} \abs{\nabla \p^\beta \tilde W}+ \abs{  \p^\beta \tilde W}\right)  + \sum_{\substack{0 \leq |\beta| \leq |\gamma|-2 \\ \beta\le\gamma}}  \abs{  \p_1\p^\beta \tilde W  }
+ \sum_{\substack{ |\beta| = |\gamma|-1 \\ \beta\le\gamma, \beta_1 = \gamma_1}} 
\abs{  \p_1\p^\beta \tilde W} \notag
\end{align}
where we used $W=\bar W+\tilde W$ to bound the terms on the second line of the first inequality, and the exponent bound $\frac 12 - \frac{3}{2k-7} \leq \frac 13$ for $k\geq 18$ for the $G_W$ term. Finally, using \eqref{eq:bootstrap:Wtilde:near:0}, and \eqref{eq:bootstrap:Wtilde4}, we obtain
\begin{align*}
\abs{\tilde F^{(\gamma)}_W}
&\les  M \eps^{\frac16}  +  M \eps^{\frac 13} 
+ (\log M)^4\eps^{\frac {1}{10}}\ell 
+ {\bf 1}_{\abs{\check\gamma} \neq 0} \eps^{\frac{1}{10}} (\log M)^{\abs{\check \gamma}-1}
\les  \eps^{\frac18}
+  \eps^{\frac{1}{10}} (\log M)^{\abs{\check \gamma}-1} 
\, ,
\end{align*} 
where we have used that by the definition of $\ell$ in \eqref{eq:ell:choice} we have
\begin{equation}\label{eq:kingfisher}
\ell \leq (\log M)^{-5}\,.
\end{equation}
This concludes the proof of the corollary.
\end{proof}

\section{Bounds on Lagrangian trajectories}
\label{sec:Lagrangian}

\subsection{Upper bound on  the support}
We now close the bootstrap assumption \eqref{eq:support} on the size of the support.
\begin{lemma}[Estimates on the support]\label{lem:support}
 Let  $\Phi$ denote either $\pw^{y_0}$, $\pz^{y_0}$ or $\pa^{y_0}$.  For any
$y_0 \in \mathcal{X} _0$ defined in \eqref{eq:support-init}, we have that
\begin{subequations} 
\label{eq:upper:bound:traj}
\begin{align} 
  \abs{\Phi_1(s)} & \le  \tfrac{3}{2}  \eps^{\frac12}  e^{\frac 32s} \,,  \label{trajectory1}\\
 \abs{\check \Phi(s)} & \le  \tfrac{3}{2}   \eps^{\frac{1}{6}}  e^{\frac{s}{2}} \,.    \label{trajectory2}
\end{align} 
\end{subequations} 
for all  $s \ge -\log\eps$.
\end{lemma} 
\begin{proof}[Proof of Lemma \ref{lem:support}]
We begin by considering  the case that $\Phi= \Phi^{y_0}_W$, and write $\Phi = (\Phi_1, \check \Phi)$.
Note that by the definitions of \eqref{eq:transport_velocities} and \eqref{flows},
\begin{subequations} 
\begin{align} 
\tfrac{d}{ds}(e^{-\frac32 s}\Phi_1(s)) &= e^{-\frac32 s}(\beta_\tau \Jcal W+G_W)\circ \Phi \,, \label{phi1ode} \\
\tfrac{d}{ds}(e^{-\frac12 s} \Phi_\nu (s)) & =  e^{-\frac s2 }h_W^\nu \circ \Phi \,,  \label{phi2ode} \\
\Phi(-\log \eps) & =y_0 \,. \label{phi3ode}
\end{align} 
\end{subequations} 
Applying the estimates \eqref{eq:beta:tau}, \eqref{eq:W_decay}, \eqref{e:bounds_on_garbage} and \eqref{e:G_W_estimates}, we have that
\begin{align}
\abs{\beta_\tau \Jcal W}+\abs{G_W}
&\les \eta^{\frac 16}(y) +M e^{-\frac s2}+M^{\frac 12} \abs{y_1}e^{-s}+\eps^{\frac 13} \abs{ \check y}\notag\\
&\les \eps^{\frac16} e^{\frac s2}+ Me^{-\frac s2}+ \eps^{\frac13}e^{\frac s2}+\eps^{\frac 12} e^{\frac s2} \notag\\
&\leq e^{\frac s2}\,,\label{eq:yellow:tail}
\end{align}
where in the  penultimate inequality we have invoked \eqref{e:space_time_conv}, and for the last inequality and have taken $\eps$ sufficiently small to absorb the implicit constant. 
Thus,  integrating \eqref{phi1ode} and using the initial condition \eqref{phi3ode} and the bound \eqref{eq:yellow:tail}, we find that
$$
\abs{e^{-\frac32 s}\Phi_1(s)-   \eps^ {\frac{3}{2}}   {y_0}_1 } \le  \int_{-\log \eps}^s e^{-s'}\,ds'\leq \eps\,.
$$
Therefore, for $y_0 \in \mathcal{X} _0$ and for $\eps$ taken sufficiently small,
$$
 e^{-\frac32 s}\abs{\Phi_1(s)}\le   \tfrac{3}{2} \eps^ {\frac{1}{2}}   \,,
$$
so that \eqref{trajectory1} is proved.
 
Similarly, using \eqref{phi2ode} and \eqref{e:h_estimates}, we conclude that
$$
\abs{ e^{-\frac s2  }\check \Phi(s) - \eps^ {\frac{1}{2}} \check y_0}
\le  \int_{-\log \eps}^s e^{-\frac{s'}2}\abs{h_W\circ\Phi(s') }\,ds'\les  \int_{-\log \eps}^s e^{-s'}\,ds'\les
 \eps\,,
$$
and hence for $y_0 \in \mathcal{X} _0$ and for $\eps$ taken sufficiently small,
$$
 e^{-\frac{s}{2}} \abs{\check\Phi(s)}\le  \tfrac{3}{2} \eps^ {\frac{1}{ 6 }}   \,,
$$
which establishes \eqref{trajectory2}.

The estimates for the cases $\Phi= \pz^{y_0},\pu^{y_0}$ are completely analogous, once the estimate \eqref{e:G_W_estimates}  is 
replaced by the estimate \eqref{e:G_ZA_estimates}  in the argument above. 
\end{proof} 

\subsection{Lower bound for $\pw$}
\begin{lemma} \label{lem:escape}
Let $y_0 \in \RR^3$ be such that $\abs{y_0} \geq \ell$. Let $s_0 \geq - \log \eps$. 
Then, the trajectory $\pw^{y_0}$ moves away from the origin at an exponential rate, and we have the lower bound
\begin{equation}
\abs{\pw^{y_0}(s)}\geq \abs{y_0}e^{\frac{s-s_0}{5}}\label{eq:escape_from_LA} 
\end{equation}
for all  $s \geq s_0$.
\end{lemma} 
\begin{proof}[Proof of Lemma~\ref{lem:escape}]
First, we claim that
\begin{equation}
\label{eq:escape_from_NY}
y\cdot \mathcal V_W(y)\geq  \tfrac{1}{5}  \abs{y}^2 \,, \qquad \mbox{for} \qquad \abs{y}\geq \ell \, .
\end{equation}
From the bootstrap $\abs{\partial_1 W}\leq 1$, the explicit formula for $\bar W$ which yields $\bar W(0,\check y)=0$, the fundamental theorem of calculus, and the bound \eqref{eq:bootstrap:Wtilde2} we obtain 
\[
\abs{W(y)}\leq\abs{W(y_1,\check y)-W(0,\check y)} + \abs{\tilde W(0,\check y)}\leq \abs{y_1} + \eps^{\frac{1}{13}}  \abs{\check y}  
\]
for all $y$ such that $\abs{y}\leq \LLL$. 
Together with Lemma \ref{lemma_g}, in which we use an extra factor of $M$ to absorb the implicit constant in the $\les$ symbol, and \eqref{eq:beta:tau}, the above estimate implies that
\begin{align*}
y\cdot \mathcal V_W&= y\cdot \left(  \beta_\tau W+G_W  + \tfrac{3}{2} y_1 \,,     h_2 + \tfrac{1}{2} y_2\,,     h_3 + \tfrac{1}{2} y_3 \right)\\
&\geq y_1^2 + \tfrac 12 \abs{y}^2 - (1+2  M^2 \eps)\abs{y_1}(\abs{y_1}+ \eps^{\frac{1}{13}}  \abs{\check y}) 
 -  \abs{y_1} M^2 ( \eps^{\frac 12} +  \eps \abs{y_1} + \eps^{\frac 13} \abs{\check y}) -  M^2 \eps^{\frac12}\abs{\check y} \\
&\geq  \tfrac{1}{5} \abs{y}^2
\end{align*} 
for all $\ell \leq \abs{y} \leq  \LLL $, upon taking $\eps$ sufficiently small, depending on $M$ and $\ell$.  Similarly, directly from the first bound in \eqref{eq:W_decay} we have that
\[
\abs{W(y)} \leq  (1+\eps^{\frac{1}{20}})  \eta^{\frac 16}(y) \leq  (1+\eps^{\frac{1}{20}})^2 \abs{y}   
\]
for all $\abs{y}\geq \LLL= \eps^{-\frac{1}{10}}$,
and thus 
\begin{align*}
y\cdot \mathcal V_W
&\geq y_1^2 + \tfrac 12 \abs{y}^2 - (1+2  M^2 \eps)\abs{y_1}  (1+\eps^{\frac{1}{20}})^2 \abs{y}   
 -  M^3 \eps^{\frac12} \abs{y}^2-  M^3 \eps^{\frac12}\abs{ y} \\
&\geq \tfrac{1}{2}\abs{y}^2 - \tfrac 14 (1+2  M^2 \eps)^2  (1+\eps^{\frac{1}{20}})^4 \abs{y}^2  -  M^3 \eps^{\frac12} \abs{y}^2-  M^3 \eps^{\frac12} \LLL^{-1} \abs{ y}^2 \\
&\geq \tfrac 15 \abs{y}^2 
\end{align*} 
for all $ \abs{y} \geq  \LLL  = \eps^{-\frac{1}{10}}$ such that $y\in \XXX(s)$, by taking $\eps$ to be sufficiently small. 

We now let $y= \pw^{y_0}(s)$ and use the
fact that $\p_s \pw^{y_0}(s) =  \mathcal{V} _W \circ \pw^{y_0}(s)$, so that 
\eqref{eq:escape_from_NY} implies that
$$
\tfrac{1}{2} \tfrac{d}{ds} \abs{ \pw^{y_0}}^2 \ge \tfrac{1}{5}  \abs{\pw^{y_0}}^2 \,,
$$
which upon integration from $s_0$ to $s$ yields  \eqref{eq:escape_from_LA}.
\end{proof}

\subsection{Lower bounds for $\pz$, $\pa$,  and $\pu$}
We now establish important lower-bounds for $\pz^{y_0}(s)$ or $\pa^{y_0}(s)=\pu^{y_0}(s)$.  
\begin{lemma}\label{lem:phiZ} 
Let $\Phi(s)$ denote either $\pz^{y_0}(s)$ or  $\ \pu^{y_0}(s)$.  If
\begin{align} 
\kappa_0\geq \frac{3}{1-\max(\beta_1,\beta_2)}\,, \label{k0-lb}
\end{align} 
then for any $y_0 \in \XXX_0$ defined in \eqref{eq:support-init}, there exists an $s_*\geq -\log \eps$ such that
\begin{equation}\label{phi-lowerbound}
\abs{\Phi_1(s)}\geq \min\left(\abs{e^{\frac s2}-e^{\frac {s_*}2}}, e^{\frac s2}\right)\,.
\end{equation}
In particular, we have the following inequality:
\begin{equation}\label{phi-lowerbound_conseq}
\int_{-\log \eps }^{\infty} e^{\sigma_1 s'}(1+\abs{\Phi_1(s')})^{-\sigma_2}\,ds'\leq C\,,
\end{equation}
for $0\leq \sigma_1 < \sfrac12$ and $2\sigma_1 < \sigma_2$, where the constant $C$ depends only on the choice of $\sigma_1$ and $\sigma_2$.
\end{lemma}
\begin{proof}[Proof of Lemma~\ref{lem:phiZ}]
We first show that if $\Phi(s)= \pz^{y_0}(s)$ or $\pa^{y_0}(s)$, we have the inequality
\begin{equation}\label{eq:move:left}
\frac{d}{ds}\Phi_1(s)\leq  -\tfrac{1}{2} e^{\frac s2} \ \ \text{ if } \ \ \Phi_1(s) \leq e^{\frac s2} \ \ \text{ for any } \ \ s \in [-\log \eps, \infty ) \,.
\end{equation}
If we set $(j,G)=(2,G_Z)$ for the case $\Phi(s)= \pz^{y_0}(s)$, and $(j,G)=(1,G_{U})$ for the case $\Phi(s)= \pa^{y_0}(s)$, then by definition we  have
that
$$
\frac{d}{ds} \Phi_1 = \tfrac{3}{2} \Phi_1 + \beta_j \beta_\tau \Jcal  W  \circ \Phi   +G \circ \Phi\,.
$$
Since $\beta_1,\beta_2<1$,  by taking $\eps$ sufficiently small, by \eqref{eq:beta:tau} and \eqref{e:bounds_on_garbage}, we have 
that $\abs{\beta_j \beta_\tau \Jcal }\leq 1$ for $j=1,2$; therefore, 
applying \eqref{eq:W_decay} and \eqref{e:G_ZA_estimates}, if $\Phi_1(s) \leq e^{\frac s2}$ then
\begin{align*} 
\frac{d}{ds} \Phi_1 &
\le \tfrac{3}{2} e^{\frac s2}+  2 \eta^{\frac 16}(\Phi)  - (1-\beta_j) \kappa _0 e^{\frac{s}{2}}  +\eps^{\frac 1{2}}  e^{\frac{s}{2}}  \\
&
\le \tfrac{3}{2} e^{\frac s2} - (1-\beta_j) \kappa _0 e^{\frac{s}{2}}  +\eps^{\frac 1{8}} e^{\frac{s}{2}}  \,,
\end{align*} 
where in the last inequality, we have used \eqref{e:space_time_conv} and  taken $\eps$ is sufficiently small. 
Since $1-\beta_j>0$ for  $j=1,2$,  then using the lower bound on $\kappa_0$ given by \eqref{k0-lb}, 
the inequality  \eqref{eq:move:left} holds.

To prove \eqref{phi-lowerbound}, we consider the following two scenarios for $y_0$: 
\begin{enumerate}
\item\label{case:Tokaji1} Either $\Phi(s)> e^{\frac s2}$ for all $s\in[-\log\eps,\infty)$, or $y_{0_1}\leq 0$.
\item\label{case:Tokaji2} There exists a smallest $s_0\in[-\log\eps,\infty)$ such that $0<\Phi(s_0)\leq e^{\frac s2}$ and $y_{0_1}> 0$.
\end{enumerate}
We first consider Case \ref{case:Tokaji1}. If $\Phi_1(s)> e^{\frac s2}$ for all $s\in[-\log\eps,\infty)$, then we trivially obtain \eqref{phi-lowerbound}. Otherwise, if $\Phi_1(-\log \eps)\leq 0$, then as a consequence of \eqref{eq:move:left}, we have that
$$
\Phi_1(s)\leq y_{0_1}-e^{\frac s 2}+\eps^{-\frac12}\le -e^{\frac s 2}+\eps^{-\frac12} \,
$$
 for all $s\in[-\log\eps,\infty)$. Thus \eqref{phi-lowerbound} holds with $s_{*}=-\log\eps$.

We next consider Case \ref{case:Tokaji2}. As a consequence of \eqref{eq:move:left} we have that
\[
\frac{d}{ds}\Phi_1(s)\leq  -e^{\frac s2} \,, \qquad \mbox{for all} \qquad s\geq s_0\, .
\]
Thus by continuity,  there exists a unique $s_*>s_0$ such that $\Phi_1(s_*)=0$. Applying \eqref{eq:move:left} and then by tracing the trajectories either
forwards or backwards from the time $s_*$, we find that  for $s\in[s_0,\infty)$,
$$
\abs{\Phi(s)}\geq \abs{e^{\frac s2}-e^{\frac {s_*}2}}\,.
$$
Hence, \eqref{phi-lowerbound} holds for $s\in[s_0,\infty)$.  Suppose that $s_0\neq -\log\eps$;  then,  by definition,  if 
$s\in[-\log\eps,s_0]$, then  $\Phi_1(s)\geq e^{\frac s2}$, and hence we conclude \eqref{phi-lowerbound}.

In order to prove \eqref{phi-lowerbound_conseq}, we first note that since $\int_{-\log \eps }^{\infty} e^{(\sigma_1-\frac{\sigma_2} 2) s'}\,ds'\les 1$, in order to prove \eqref{phi-lowerbound_conseq}, by  \eqref{phi-lowerbound}, it suffices to prove that
\begin{equation*}
\mathcal I:=\int_{-\log \eps }^{\infty} e^{\sigma_1 s'}\left(1+\sabs{e^{\frac{s'}{2}}-e^{\frac {s_*}2}}\right)^{-\sigma_2}\,ds'\leq C\,.
\end{equation*}
Applying the change of variables $r=e^{\frac{s'}{2}}$,  we have that
\begin{align*}
\mathcal I & =2 \int_{\eps^{-\frac 12}}^{\infty}
r^{2\sigma_1 -1} \left(1+\sabs{r-e^{\frac{s_*}{2}}}\right)^{-\sigma_2}\,dr \\
& \les \int_{\eps^{-\frac 12}}^{\infty}
\left(r^{2\sigma_1 -1-\sigma_2}+ \left(1+\sabs{r-e^{\frac{s_*}{2}}}\right)^{2\sigma_1 -1-\sigma_2}\right)\,dr
\les 1\,,
\end{align*}
where we have used
Young's inequality for the second to last inequality. The implicit constant only depends on $\sigma_1$ and $\sigma_2$.
\end{proof}

 \begin{corollary}\label{cor:p1W} Let $\Phi^{y_0}(s)$ denote either $\pz^{y_0}(s)$ or $\pu^{y_0}(s)$.  Then, 
for all $s \ge -\log \eps$,
\begin{align} 
\sup_{y_0 \in \XXX_0} \int_{ -\log\eps}^s \abs{\p_1 \tilde W } \circ \Phi^{y_0}(s') ds' \les \eps^{\frac{1}{11}} \,\label{eq:p1Wtilde:PhiZAU}.\\
\sup_{y_0 \in \XXX_0} \int_{ -\log\eps}^s \abs{\p_1 W } \circ \Phi^{y_0}(s') ds' \les 1 \label{eq:p1W:PhiZAU} \,.
\end{align} 
\end{corollary} 
\begin{proof}[Proof of Corollary \ref{cor:p1W}]
Due to the estimates in \eqref{eq:bootstrap:Wtilde}, and \eqref{phi-lowerbound_conseq} (with $\sigma_1=0$ and $\sigma_2=\sfrac23$), we obtain \eqref{eq:p1Wtilde:PhiZAU}. The estimate \eqref{eq:p1W:PhiZAU} similarly holds with the help of the second estimate in \eqref{eq:W_decay}.
\end{proof}

\section{$L^ \infty $ bounds for  $\mrg$ and $S$}
\label{sec:vorticity:sound}
We now establish bounds to solutions  $ \mrg $  of the specific vorticity equation \eqref{sv-ss} and solutions $S$ to the sound speed equation
\eqref{S-ss}. We set 
$S_0 (y) = S(y, -\log\eps) $.

\subsection{Sound speed}
\begin{proposition}[Bounds on the sound speed]\label{prop:sound}
We have that
\begin{align} \label{S-bound}
\snorm{S( \cdot , s) - \tfrac{\kappa_0}{2} }_{ L^ \infty } \le \eps^ {\frac{1}{8}}  \ \ \text{ for all } \ s \ge -\log\eps  \,.
\end{align} 
\end{proposition} 
\begin{proof}[Proof of Proposition \ref{prop:sound}]
By \eqref{eq:UdotN:S}, we have that
$$
S( \cdot , s) - \tfrac{\kappa_0}{2}= \tfrac{\kappa-\kappa_0}{2} + {\tfrac{1}{2}} (e^ {-\frac{s}{2}} W-Z) \,.
$$
By \eqref{mod-boot}, \eqref{e:space_time_conv},  \eqref{eq:W_decay},  and \eqref{eq:Z_bootstrap}, and the triangle inequality,
$$\snorm{S( \cdot , s) - \tfrac{\kappa_0}{2} }_{ L^ \infty } \les \eps^ {\frac{1}{6}} $$
which concludes the proof.
\end{proof}

\subsection{Specific vorticity}
From \eqref{svorticity_sheep}, we deduce that the normal and tangential components of the vorticity  satisfy  the system
\begin{subequations}
\label{sv-ss}
\begin{align}
\p_t (\mrg \cdot \Tcal^2)  +    \Vcal \cdot \nabla_{\! x} (\mrg \cdot \Tcal^2)  &= 
\FFF_{21} (\mrg \cdot \Ncal) + \FFF_{2\mu}(\mrg \cdot \Tcal^\mu)  \\
\p_t (\mrg \cdot \Tcal^3)   +    \Vcal \cdot \nabla_{\! x} (\mrg \cdot \Tcal^3)  &=
\FFF_{31} (\mrg \cdot \Ncal) + \FFF_{3\mu}(\mrg \cdot \Tcal^\mu) 
\end{align}
\end{subequations}
where  
\begin{align*}
\Vcal = (\Vcal_1,\Vcal_2,\Vcal_3) =  2\beta_1  \left(- \tfrac{\dot f}{2\beta_1} + \Jcal  v \cdot \Ncal + \Jcal \mru \cdot \Ncal ,  v_2 + \mru_2, v_3 + \mru_3\right)
\end{align*}
and
\begin{subequations}
\label{ssvort-force}
\begin{align}
\FFF_{21} &=  \Ncal \cdot \p_t \Tcal^2      + 2\beta_1 \dot{Q}_{ij} \Tcal^2_i \Ncal_j     + \Vcal_\nu ( \Ncal\cdot \Tcal^2_{,\nu})
  + 2\beta_1  \Ncal_{\nu} \p_{x_\nu} a_2
- 2\beta_1  \Ncal_\nu \mru \cdot  \Tcal^2_{,\nu}   
\\
\FFF_{22} &= 2\beta_1  \Tcal^2_{\nu} \p_{x_\nu} a_2   - 2\beta_1  \Tcal^2_\nu \mru \cdot  \Tcal^2_{,\nu} 
\\
\FFF_{23} &=   \Tcal^3 \cdot \p_t \Tcal^2   + 2\beta_1 \dot{Q}_{ij} \Tcal^2_i \Tcal_j^3    
 \Vcal_\nu (  \Tcal^3\cdot \Tcal^2_{,\nu})
  + 2\beta_1\Tcal^3_{\nu} \p_{x_\nu} a_2
  - 2\beta_1  \Tcal^3_\nu \mru \cdot  \Tcal^2_{,\nu} 
\\
\FFF_{31} &=  \Ncal \cdot \p_t \Tcal^3       + 2\beta_1 \dot{Q}_{ij} \Tcal^3_i \Ncal_j  + \Vcal_\nu ( \Ncal\cdot \Tcal^3_{,\nu}) 
  + 2\beta_1 \Ncal_{\nu} \p_{x_\nu} a_3
   - 2\beta_1  \Ncal_\nu \mru \cdot  \Tcal^3_{,\nu}  
   \\
\FFF_{32} &=    \Tcal^2 \cdot \p_t \Tcal^3   + 2\beta_1 \dot{Q}_{ij} \Tcal^3_i \Tcal_j^2   
 + \Vcal_\nu ( \Tcal^2\cdot \Tcal^3_{,\nu})
  + 2\beta_1 \Tcal^2_{\nu} \p_{x_\nu} a_3
   - 2\beta_1  \Tcal^2_\nu \mru \cdot  \Tcal^3_{,\nu}  
\\
\FFF_{33} &=  2\beta_1  \Tcal^3_{\nu} \p_{x_\nu} a_3 - 2\beta_1 \Tcal^3_\nu \mru \cdot  \Tcal^3_{,\nu}  \,.
 \end{align} 
\end{subequations}

\begin{proposition}[Bounds on specific vorticity]\label{prop:vorticity}
We have the estimate
\begin{align} \label{svort-bound}
\snorm{\mrg(\cdot,t)}_{L^\infty} = \norm{\Omega(\cdot,s)}_{L^\infty} \leq 2
\, .
\end{align} 
\end{proposition} 
\begin{proof}[Proof of Proposition \ref{prop:vorticity}]
By Lemma \ref{lem:BBS}, 
\begin{align} 
\abs{\p_t \Ncal } + \abs{\p_t \Tcal^\mu  }+ \abs{\check\nabla_{\!x}   \Ncal } + \abs{\check\nabla_{\!x}   \Tcal^ \mu  } 
\les \eps^ {\frac{1}{4}}  \,.
\label{page48}
\end{align} 
The transformations \eqref{tildeu-dot-T}, \eqref{a_ansatz}, and \eqref{U-trammy} together with the
bootstrap bounds   \eqref{eq:A_bootstrap}, \eqref{eq:US_est}, Lemma \ref{lem:BBS} and \eqref{eq:V:bnd} we have that
\begin{align*} 
\norm{\mru}_{L^ \infty } \les M^ {\frac{1}{4}} \,, \ \   \norm{\p_{x_\nu}(\mru \cdot \Ncal)}_{L^ \infty } \les 1 \,, \ \ 
\norm{\p_{x_\nu} a}_{L^ \infty }  \le  M \eps^ {\frac{1}{2}}\, , \ \ \norm{\Vcal}_{L^\infty} \les M^{\frac 14}\,.
\end{align*} 
Together with \eqref{eq:dot:Q}, 
it follows that the forcing functions defined in \eqref{ssvort-force} satisfy
\begin{align} 
 \snorm{ \FFF _{ij}}_{L^\infty} \les 1  \ \ \text{ for } \ i,j \in \{1,2,3\} \,. \label{force-bound}
\end{align}

Now, from the definitions \eqref{eq:tilde:u:def}, \eqref{vort00}, \eqref{usigma-sheep},
\eqref{sv-sheep},  we have that
\begin{align} 
  (\alpha \mathring \sigma( x, t))^ {\sfrac{1}{\alpha }}  \mathring \zeta(x, t) = \tilde \rho( \tilde x, t) \tilde \zeta( \tilde x, t) = \tilde \omega ( \tilde x, t) 
= \operatorname{curl} _{ \tilde x}  \tilde u ( \tilde x, t)  = \operatorname{curl} _{ \tilde x} \mathring u (x, t) \,,
\notag
\end{align} 
and
\begin{align} 
\operatorname{curl}_{\tilde x} \mru \cdot \Ncal & = \Tcal^2_j  \p_{\tilde x_j} \mru     \cdot \Tcal^3 -    \Tcal^3_j \p_{\tilde x_j} \mru  \cdot \Tcal ^2
\notag \\
&   = \Tcal^2_\nu  \p_{ x_\nu} \mru  \cdot \Tcal^3 -    \Tcal^3_\nu \p_{ x_\nu} \mru \cdot \Tcal ^2 
\notag  \\
&   = \Tcal^2_\nu  \p_{ x_\nu} a_3 - \Tcal^2_\nu  \mru      \cdot  \Tcal^3_{,\nu}
-  \Tcal^3_\nu  \p_{ x_\nu} a_2 + \Tcal^3_\nu  \mru     \cdot  \Tcal^2_{,\nu} \,. \label{boom3}
\end{align} 
from which it follows that
\begin{align} 
 \mathring \zeta \cdot \Ncal = 
  \frac{ \Tcal^2_\nu  \p_{ x_\nu} a_3 - \Tcal^2_\nu  \mru     \cdot  \Tcal^3_{,\nu}
 -  \Tcal^3_\nu  \p_{ x_\nu} a_2 +  \Tcal^3_\nu  \mru     \cdot  \Tcal^2_{,\nu}}{  (\alpha \mathring \sigma( x, t))^ {\sfrac{1}{\alpha }} }
   \,. \label{boom2}
\end{align} 
By \eqref{S-trammy} and \eqref{S-bound}, we have that
\begin{align} \label{sigma-bound}
\snorm{\mrs( \cdot , t) - \tfrac{\kappa_0}{2} }_{ L^ \infty } \le \eps^ {\frac{1}{8}}  \,.
\end{align} 
Hence, from \eqref{gerd}, \eqref{boom3} and \eqref{sigma-bound}, we have that
\begin{align} 
\sabs{ \mathring \zeta \cdot \Ncal} \les \eps^ {\frac{1}{5}}  \,.  \label{boom4}
\end{align} 

We let $\phi(x,t)$ denote the flow of $\Vcal$ so that
$$
\p_t \phi(x,t) = \Vcal(\phi(x,t),t) \quad \text{ for } \quad t> -\eps\,,\quad  \text{ and }\quad  \phi(x,-\eps)=x\,,
$$
and denote by $\phi^{x_0}(t)$ the trajectory emanating from $x_0$.
We define
$$
\overline {\mathcal{F}} _{ij} = \mathcal{F} _{ij} \circ \phi^{x_0}\,, \ \ 
 \mathcal{Q} _1 = (  \mrg \cdot \Ncal) \circ  \phi^{x_0}    \,, \ \ 
\mathcal{Q} _2 = (  \mrg \cdot \Tcal^2) \circ  \phi^{x_0}\,, \ \ 
\mathcal{Q} _3 = (  \mrg \cdot \Tcal^3) \circ  \phi^{x_0}\,,
$$
Then,  \eqref{sv-ss} is written as the following system of ODEs:
 \begin{align*} 
  \p_t \mathcal{Q} _2   = \overline {\mathcal{F}}_{2 j } \mathcal{Q}_{j } \,, \ \
  \p_t \mathcal{Q} _3   = \overline {\mathcal{F}}_{3 j } \mathcal{Q}_{j }  \,.
 \end{align*} 
Hence,
\begin{align} 
\tfrac{1}{2} \tfrac{d}{dt} \left(  \mathcal{Q}_2^2 +  \mathcal{Q}_3^2 \right) 
& =     \overline {\mathcal{F}} _{\nu \mu} \mathcal{Q}_\nu  \mathcal{Q}_\mu +  \overline {\mathcal{F}} _{\mu 1}  \mathcal{Q}_\mu \mathcal{Q}_1   \,.
 \label{D-bound1}
\end{align} 
By Gr\"onwall's inequality on $[-\eps,t)$, with $t < T_* \leq \eps$, we deduce from \eqref{force-bound} and \eqref{boom4} that there  exists a universal constant $C_0 \geq 1$ such that 
\begin{align*}
\abs{\mathcal{Q}_2(t)} +  \abs{\mathcal{Q}_3(t)}  \leq C_0 \left( \abs{\mathcal{Q}_2(-\eps)} +  \abs{\mathcal{Q}_3(-\eps)} \right) + \eps 
\end{align*}
uniformly for all labels $x_0$, for a constant $C_0 \in (1, e^{\eps^{\frac 12}})$. Since $\Ncal, \Tcal^2,\Tcal^3$ form an orthonormal basis, the above estimate and \eqref{boom4}, together with  the initial datum assumption~\eqref{eq:svort:IC} implies that \eqref{svort-bound} holds. The self-similar specific vorticity bound follows directly from its definition in \eqref{svort-trammy}. 
\end{proof}

\section{Closure of $L^\infty$ based bootstrap for $Z$ and $A$}
\label{sec:Z:A}
Having established bounds on trajectories as well as on the vorticity, we now improve the bootstrap assumptions for 
$\partial^\gamma Z $ and $\partial^\gamma A$ stated in \eqref{eq:Z_bootstrap} and \eqref{eq:A_bootstrap}.  We shall obtain estimates for $\partial^\gamma Z \circ \pzy$ and 
$\partial^\gamma A \circ \pay$ which are weighted by an appropriate exponential factor $e^{\mu s}$.

From \eqref{euler_for_Linfinity:b} we obtain that $e^{\mu s}\p^\gamma Z$ is a solution of
\begin{align*} 
 \p_s (e^{\mu s}\p^\gamma Z) + D_Z^{(\gamma,\mu)} ( e^{\mu s}\p^\gamma Z)
  +  \left( \mathcal{V}_Z \cdot \nabla\right) ( e^{\mu s}\p^\gamma Z) &=e^{\mu s} F^{(\gamma)}_Z  \,,
\end{align*} 
where the damping function is given by
$$
D_Z^{(\gamma,\mu)}:=-\mu + \tfrac{3\gamma_1 + \gamma_2 + \gamma_3}{2} + \beta_2 \beta_\tau  \gamma_1  \Jcal \p_1 W \,.
$$
Upon composing with the flow of ${\mathcal V}_Z$, from Gr\"onwall's inequality  it follows that
\begin{align}
e^{\mu s}\abs{\partial^\gamma Z\circ \pz^{y_0}(s)}
&\leq   \eps^{-\mu} \abs{\partial^\gamma Z(y_0,-\log\eps)} \exp\left(- \int_{-\log\eps}^s  D_{Z}^{(\gamma,\mu)}\circ\pz^{y_0}(s')  \,ds'\right) \notag\\
&\qquad +\int_{-\log\eps}^s e^{\mu s'}\abs{F_Z^{(\gamma)}\circ \pz^{y_0}(s')}\exp\left(- \int_{s'}^s D_{Z}^{(\gamma,\mu)}\circ\pz^{y_0}(s'') \,ds''\right) \,ds' \,.\label{eq:weighted:Z:bnd}
\end{align}
Similarly, from \eqref{euler_for_Linfinity:c} we have that  $e^{\mu s}\p^\gamma A$ is a solution of
\begin{align*} 
 \p_s (e^{\mu s}\p^\gamma A) + D_A^{(\gamma,\mu)} ( e^{\mu s}\p^\gamma A)
  +  \left( \mathcal{V}_U \cdot \nabla\right) ( e^{\mu s}\p^\gamma A) &=e^{\mu s} F^{(\gamma)}_A  \,,
\end{align*} 
where 
$$
D_A^{(\gamma,\mu)}:=-\mu + \tfrac{3\gamma_1 + \gamma_2 + \gamma_3}{2} + \beta_1 \beta_\tau  \gamma_1 \Jcal \p_1 W   \,,
$$
and hence, again by Gronwall's inequality, we have that
\begin{align}
e^{\mu s}\abs{\partial^\gamma A\circ \pa^{y_0}(s)}
&\leq   \eps^{-\mu} \abs{\partial^\gamma A(y_0,-\log\eps)} \exp\left(- \int_{-\log\eps}^s  D_A^{(\gamma,\mu)}\circ\pa^{y_0}(s')  \,ds'\right) \notag\\
&\qquad +\int_{-\log\eps}^s e^{\mu s'}\abs{F_A^{(\gamma)}\circ \pa^{y_0}(s')}\exp\left(- \int_{s'}^s D_A^{(\gamma,\mu)}\circ\pa^{y_0}(s'') \,ds''\right) \,ds' \,.\label{eq:weighted:A:bnd}
\end{align}
For each choice of $\gamma \in {\mathbb N}_0^3$ present in \eqref{eq:Z_bootstrap} and \eqref{eq:A_bootstrap}, we shall require that the exponential factor $\mu$ satisfies
\begin{equation}\label{eq:mu_cond}
 \mu  \leq  \tfrac{3\gamma_1 + \gamma_2 + \gamma_3}{2} \,,
\end{equation}
which, in turn, shows that
\begin{equation}\label{eq:DZ_lower_bnd}
D_Z^{(\gamma,\mu)}\leq 2 \beta_2 \gamma_1 \abs{\partial_1 W}\,.
\end{equation}
For the last inequality, we have used the bound $\abs{\beta_\tau \Jcal}\leq 2$, which follows from \eqref{eq:beta:tau} and \eqref{e:bounds_on_garbage}. Combining \eqref{eq:mu_cond}, \eqref{eq:DZ_lower_bnd}, and \eqref{eq:p1W:PhiZAU},  for $s\geq s' \ge -\log \eps$ we obtain
\begin{equation}\label{eq:damping_Z}
 \exp\left(- \int_{s'}^s  D_{Z}^{(\gamma,\mu)}\circ\pz^{y_0}(s')  \,ds'\right)\les  \exp\left(  \left( \mu -  \tfrac{3\gamma_1 + \gamma_2 + \gamma_3}{2}\right) (s-s') \right)\les 1\,.
 \end{equation}
Replacing $\beta_2$  with $\beta_1$ in \eqref{eq:DZ_lower_bnd}, we  similarly obtain that  for $s\geq s'\geq -\log\eps$,
\begin{equation}\label{eq:damping_U}
 \exp\left(- \int_{s'}^s  D_{A}^{(\gamma,\mu)}\circ\pa^{y_0}(s')  \,ds'\right)\les 1\,.
 \end{equation}
Then as a consequence of \eqref{eq:weighted:Z:bnd}, \eqref{eq:weighted:A:bnd}, \eqref{eq:mu_cond}, \eqref{eq:damping_Z} and \eqref{eq:damping_U}, we obtain
\begin{align}
e^{\mu s}\abs{\partial^\gamma Z\circ \pz^{y_0}(s)}&\les    \eps^{-\mu} \abs{\partial^\gamma Z(y_0,-\log\eps)}\notag\\
&\qquad+\int_{-\log\eps}^s e^{\mu s'}\abs{F_Z^{(\gamma)}\circ \pz^{y_0}(s')} \exp\left(\left( \mu -  \tfrac{3\gamma_1 + \gamma_2 + \gamma_3}{2}\right) (s-s')\right) \,ds' \label{eq:weighted:Z:bnd2:alt}
\\
&\les   \eps^{-\mu} \abs{\partial^\gamma Z(y_0,-\log\eps)}
+\int_{-\log\eps}^s e^{\mu s'}\abs{F_Z^{(\gamma)}\circ \pz^{y_0}(s')} \,ds' \,,
\label{eq:weighted:Z:bnd2}
\end{align}
and
\begin{align}\label{eq:weighted:A:bnd2}
e^{\mu s}\abs{\partial^\gamma A\circ \pa^{y_0}(s)}
&\les   \eps^{-\mu} \abs{\partial^\gamma A(y_0,-\log\eps)}
+\int_{-\log\eps}^s e^{\mu s'}\abs{F_A^{(\gamma)}\circ \pa^{y_0}(s')} \,ds' \,.
\end{align}

\subsection{Estimates on $Z$}

For convenience of notation, in this section we  set $\Phi=\pz^{y_0}$.
We start with the case $\gamma=0$, for which we set $\mu=0$. Then, the first line of \eqref{eq:forcing_Z} combined with \eqref{eq:weighted:Z:bnd2} and our initial datum assumption \eqref{eq:Z_bootstrap:IC} show that
\begin{align*}
\abs{ Z\circ \Phi(s)}
&\les   \abs{ Z(y_0,-\log\eps)}
+\int_{-\log\eps}^s e^{-s'} \,ds' \les \eps \,.
\end{align*}
This improves the bootstrap assumption \eqref{eq:Z_bootstrap} for $\gamma=0$, upon taking $M$ to be sufficiently large to absorb the implicit universal constant in the above inequality.

For the case $\gamma=(1,0,0)$, we set $\mu=\frac 32$ so that \eqref{eq:mu_cond} is verified, and hence from \eqref{eq:Z_bootstrap:IC}, the second case in \eqref{eq:forcing_Z},  and \eqref{eq:weighted:Z:bnd2}, we find that
\begin{align*}
e^{\frac32 s}\abs{\partial_1 Z\circ \Phi(s)}
&\les   \eps^{-\frac32} \abs{\partial_1 Z(y_0,-\log\eps)}
 +\int_{-\log\eps}^s e^{\frac{3}{2}  s'}\abs{F_Z^{(\gamma)}\circ \pz^{y_0}(s')}  \,ds' \\
& \les 1 + \int_{-\log\eps}^s \left(1+\abs{\Phi_1(s')}^{2}\right)^ {-\frac{2}{2k-5}}   \,ds' \, .
\end{align*}
Now, applying 
\eqref{phi-lowerbound_conseq} with $\sigma_1=0$ and $\sigma_2 = \frac{1}{2k-5} $ for $k \ge 18$, we deduce that
\begin{align} 
e^{\frac32 s}\abs{\partial_1 Z\circ \Phi(s)} &\les 1\,, \label{zztop1}
\end{align} 
which improves the bootstrap assumption \eqref{eq:Z_bootstrap} for  $M$ taken sufficiently large.

We next consider  the case that $\gamma_1 \geq 1$ and $\abs{  \gamma}=2$. For such $\gamma$ we let $\mu= \tfrac{3}{2} $, so that  
$$
\mu - \tfrac{3\gamma_1 + \gamma_2 + \gamma_3}{2} = \tfrac 12 - \gamma_1 \leq - \tfrac 12\,.
$$
We deduce from \eqref{eq:weighted:Z:bnd2:alt}, the third case in \eqref{eq:forcing_Z}, the initial datum assumption \eqref{eq:Z_bootstrap:IC}, and Lemma~\ref{lem:phiZ} with $\sigma_1 = \frac 18 $ and $\sigma_2 = \frac 13$,   that
\begin{align}
e^{\frac32 s}\abs{\p^\gamma  Z\circ \Phi(s)}
&\les   \eps^{-\frac32} \abs{\p^\gamma  Z(y_0,-\log\eps)}
+\int_{-\log\eps}^s \left( M^{\frac{\abs{\check \gamma}}{2}} + M^2 \left(1+\abs{\Phi_1(s')}^{2}\right)^ {-\frac{1}{6}} \right)   e^{-\frac{1}2 (s-s')} \,ds' \notag\\
& \les 1+ M^{\frac{\abs{\check \gamma}}{2}} +  \int_{-\log\eps}^s  \eps^{\frac 18 } e^{\frac s8}  M^2 \left(1+\abs{\Phi_1(s')} \right)^ {-\frac{1}{3}}   \,ds' \notag\\
& \les 1+ M^{\frac{\abs{\check \gamma}}{2}} + \eps^{\frac 18 } M^2  \les  M^{\frac{\abs{\check \gamma}}{2}}   \label{zztop2}
\end{align}
for $s\geq   -\log \eps$ and $\gamma_1 \geq 1$ and $\abs{  \gamma}=2$. This improves the bootstrap stated in \eqref{eq:Z_bootstrap} by using the factor $M^{\frac 12}$ to absorb the implicit constant in the above inequality.

We are left to consider $\gamma$ for which $\gamma_1=0$ and $1\leq \abs{\check \gamma} \leq 2$.
For  $\abs{\gamma}=\abs{\check \gamma}=1$, setting $\mu=\frac 12$ (which satisfies
\eqref{eq:mu_cond}) we obtain from 
\eqref{eq:weighted:Z:bnd2}, the forcing bound  \eqref{eq:forcing_Z}, and the initial datum assumption \eqref{eq:Z_bootstrap:IC} that 
\begin{align}
e^{\frac s2 }\abs{\check\nabla Z\circ \Phi(s)}
&\les   \eps^{-\frac12} \abs{\check\nabla Z(y_0,-\log\eps)}
+M^2 \int_{-\log\eps}^s e^{- s'} \,ds' \les \eps^{\frac 12}\,.  \label{zztop3}
\end{align}
Finally, for $\abs{\gamma}=\abs{\check \gamma}=2$  we set $\mu=1$. As a consequence of \eqref{eq:forcing_Z}, \eqref{eq:Z_bootstrap:IC}, and \eqref{eq:weighted:Z:bnd2}, we obtain
\begin{align}
e^{s }\abs{\check\nabla^2 Z\circ \Phi(s)}
&\les   \eps^{-1} \abs{\check\nabla^2  Z(y_0,-\log\eps)}
+\int_{-\log\eps}^s e^{- (\frac12 - \frac{3}{2k-7})s'} \,ds' \les 1\,,  \label{zztop4}
\end{align}
for $k\geq 18$. Together, the estimates \eqref{zztop1}--\eqref{zztop4} improve the bootstrap bound \eqref{eq:Z_bootstrap} by taking  $M$ sufficiently large.

\subsection{Estimates on $A$}
The goal of this section is to improve on the bootstrap bounds \eqref{eq:A_bootstrap}. The $\p_1 A$ estimate is more delicate, and is obtained by considering the vorticity equation; we postpone this estimate for the end of this subsection. In contrast, the $\check \nabla^m A$ estimates with $0\leq m \leq 2$ are very similar to the estimates of $Z$, by setting $\Phi=\pa^{y_0}$ and utilizing \eqref{eq:A_bootstrap:IC}, \eqref{eq:forcing_U} and \eqref{eq:weighted:A:bnd2} in place of \eqref{eq:Z_bootstrap:IC}, \eqref{eq:forcing_Z} and \eqref{eq:weighted:Z:bnd2}. We summarize these as follows: 
\begin{subequations} 
\label{A-estimates-1}
\begin{align}
\abs{ A\circ \Phi(s)}
&\les   \abs{ A(y_0,-\log\eps)} +  M^{\frac 12}  \int_{-\log\eps}^s e^{-s'} \,ds' \les  M^{\frac 12}  \eps  \\
e^{\frac s2 }\abs{\check\nabla A\circ \Phi(s)}
&\les   \eps^{-\frac12} \abs{\check\nabla A(y_0,-\log\eps)}
+ \int_{-\log\eps}^s  \left(M^{\frac 12} + M^2  \left(1+\abs{\Phi_1(s')} \right)^ {-\frac{1}{3}} \right) e^{- \frac{s'}{2}} \,ds' \notag \\
&\les \eps^{\frac 12} + M^{\frac 12} \eps^{\frac 12} + M^2 \eps^{\frac 12 + \frac 18} \int_{-\log\eps}^s  e^{\frac{s'}{8}   \left(1+\abs{\Phi_1(s')} \right)^{-\frac{1}{3}}  }  \,ds'   \les M^{\frac 12} \eps^{\frac 12} \\
e^{ s }\abs{\check\nabla^2 A\circ \Phi(s)}
&\les   \eps^{-1} \abs{\check\nabla^2 A(y_0,-\log\eps)}
+\int_{-\log\eps}^s e^{\frac{3 s'}{2k-7}}\left(1+\abs{\Phi_{1}}^{2}\right)^{-\frac16} \,ds' \les 1
\end{align}
\end{subequations} 
where we applied \eqref{phi-lowerbound_conseq} first with $\sigma_1 = \frac18$ and $\sigma_2 = \frac 13$, and then  with $\sigma_1=\frac{4}{2k-7}$ and $\sigma_2=\frac13$. Taking $M$ sufficiently large, the bounds \eqref{A-estimates-1} close  the bootstrap assumption for $\p^\gamma A$ when $\gamma_1=0$.

It remains to close the bootstrap assumption on $\p_1A_\nu$ for $\nu=2,3$. For this purpose we use the vorticity estimate given in Proposition \ref{prop:vorticity} and the following representation:
\begin{lemma}[Relating $A$ and $\Omega$]
\label{lem:remarkable:sheep:structure}
The following identities hold:
\begin{subequations}
\begin{align}
e^{\frac{3s}{2}} \Jcal \p_1 A_2  
&= (\alpha S)^{\frac 1\alpha} \Omega \cdot \Tcal^3 + \tfrac 12 \Tcal^2_\mu \left(\p_\mu W + e^{\frac s2} \p_\mu Z\right) - e^{\frac s2} \Ncal_\mu \p_\mu A_2 \notag\\
&\qquad - \tfrac 12 \left( \kappa + e^{-\frac s2} W + Z\right) (\operatorname{curl}_{\tilde x} \Ncal) \cdot \Tcal^3 - A_2 (\operatorname{curl}_{\tilde x} \Tcal^2) \cdot \Tcal^3 \\
e^{\frac{3s}{2}} \Jcal \p_1 A_3  
&= - (\alpha S)^{\frac 1\alpha} \Omega \cdot \Tcal^2 + \tfrac 12 \Tcal^3_\mu \left(\p_\mu W + e^{\frac s2} \p_\mu Z\right) - e^{\frac s2} \Ncal_\mu \p_\mu A_3 \notag\\
&\qquad + \tfrac 12 \left( \kappa + e^{-\frac s2} W + Z\right) (\operatorname{curl}_{\tilde x} \Ncal) \cdot \Tcal^2 - A_3 (\operatorname{curl}_{\tilde x} \Tcal^3) \cdot \Tcal^2 
\,.
\end{align} 
\end{subequations}
\end{lemma}
Assuming for the moment that Lemma~\ref{lem:remarkable:sheep:structure} holds, by combining 
Propositions~\ref{prop:sound} and~\ref{prop:vorticity} with estimates \eqref{eq:W_decay}, \eqref{eq:Z_bootstrap}, \eqref{eq:A_bootstrap}, \eqref{e:space_time_conv} and \eqref{e:bounds_on_garbage} we deduce that
\begin{align} 
  e^{\frac{3}{2}s} \abs{ \p_1 A_\nu } \les \kappa_0^{\frac{1}{\alpha}}   +  (1 + \eps^{\frac 12} M^{\frac 12} ) + (\kappa_0 + \eps^{\frac 16} + M \eps)  + M \eps  \,.   \label{A-estimates-2}
\end{align} 
The above estimate thus improves on the bootstrap assumption for $\p_1 A_\nu$, by taking $M$ to be sufficiently large in terms of $\kappa_0$, and then $\eps$ sufficiently small in terms of $M$. The estimates \eqref{A-estimates-1} and \eqref{A-estimates-2}  thus improve the bootstrap assumptions on $A$, and it remains to prove Lemma~\ref{lem:remarkable:sheep:structure}.

\begin{proof}[Proof of Lemma~\ref{lem:remarkable:sheep:structure}]

   We note that for the
velocity $\mathring u$ and with respect to the orthonormal basis $(\Ncal, \Tcal^2,\Tcal^3)$ we have that 
$$
\operatorname{curl}_{\tilde x} \mru = \left( \p_{\Tcal^3} \mru \cdot \Ncal - \p_{\Ncal} \mru \cdot \Tcal ^3\right)\Tcal^2
- \left( \p_{\Tcal^2} \mru \cdot \Ncal - \p_{\Ncal} \mru \cdot \Tcal ^2\right)\Tcal^3
+  \left( \p_{\Tcal^2} \mru \cdot \Tcal^3 - \p_{\Tcal^3} \mru \cdot \Tcal ^2\right)\Ncal \,.
$$
Now, from the definitions \eqref{eq:tilde:u:def}, \eqref{vort00}, \eqref{usigma-sheep},
\eqref{sv-sheep}, \eqref{S-trammy}, and  \eqref{svort-trammy}, we have that
\begin{align*} 
( \alpha S)^ {\sfrac{1}{\alpha }} (y,s) \Omega(y,s) =  (\alpha \mathring \sigma( x, t))^ {\sfrac{1}{\alpha }}  \mathring \zeta(x, t) = \tilde \rho( \tilde x, t) \tilde \zeta( \tilde x, t) = \tilde \omega ( \tilde x, t) 
= \operatorname{curl} _{ \tilde x}  \tilde u ( \tilde x, t)  = \operatorname{curl} _{ \tilde x} \mathring u (x, t) \,.
\end{align*} 
In particular, 
\begin{align} 
( \alpha S)^ {\sfrac{1}{\alpha }} (y,s) \Omega(y,s)= \operatorname{curl} _{\tilde x} \mathring u(x,t)
= \operatorname{curl} _{ \tilde x} \left( \mathring u (\tilde x_1 - f(\check{\tilde x}, t), \tilde x_2, \tilde x_3, t) \right)\,. \label{vort-correct1}
\end{align} 
We only establish the formula for $\p_1 A_3$, as the one for $\p_1 A_2$ is obtained identically. 
 To this end, we write
\begin{align*} 
\operatorname{curl}_{\tilde x} \mru \cdot \Tcal^2 & = \Tcal^3_j  \p_{\tilde x_j} \mru(x,t)     \cdot \Ncal -    \Ncal_j \p_{\tilde x_j} \mru(x,t)  \cdot \Tcal ^3  \,.
\end{align*} 
By the chain-rule and the fact that $\Ncal$ is orthogonal to $\Tcal^3$, we have that
\begin{align*} 
 \p_{\tilde x_j} \mru(x,t) \Tcal^3_j & = \p_{ x_1} \mru \Tcal^3_1 - f,_\nu \p_{x_1} \mru \Tcal^3_\nu + \p_{x_\nu} \mru \Tcal^3_\nu   =  \Jcal \Ncal \cdot \Tcal^3 \p_{x_1}\mru  +  \p_{x_\nu} \mru \Tcal^3_\nu   =  \p_{x_\nu} \mru(x,t) \Tcal^3_\nu \,.
\end{align*} 
The important fact to notice here is that no $x_1$ derivatives of $\mru$ remain.
Similarly, 
\begin{align*} 
 \p_{\tilde x_j} \mru(x,t) \Ncal_j & = \p_{ x_1} \mru \Ncal_1 - f,_\nu \p_{x_1} \mru \Ncal_\nu + \p_{x_\nu} \mru \Ncal_\nu  =  \Jcal \Ncal \cdot \Ncal \p_{x_1}\mru  +  \p_{x_\nu} \mru \Ncal_\nu   =  \Jcal \p_{x_1}\mru  + \p_{x_\nu} \mru(x,t) \Ncal_\nu \,.
\end{align*} 
Hence, it follows that
\begin{align} 
&\operatorname{curl}_{\tilde x} \mru \cdot \Tcal^2 \notag\\
& = \Tcal^3_\nu\p_{x_\nu} \mru(x,t)     \cdot \Ncal -   \Jcal \p_{x_1}(\mru\cdot \Tcal ^3) -   \Ncal_\nu\p_{x_\nu} \mru(x,t)  \cdot \Tcal ^3 \notag \\
&=\Tcal^3_\nu \p_{x_\nu} (\mru(x,t) \cdot \Ncal )  -   \Jcal \p_{x_1}a_3   - \Ncal_\nu    \p_{x_\nu}( \mru(x,t) \cdot \Tcal ^3)  - \mru(x,t)\cdot   \p_{x_\nu}  \Ncal\,  \Tcal^3_\nu  +   \mru(x,t)  \cdot \p_{x_\nu} \Tcal ^3\,  \Ncal_\nu \notag\\
&= \tfrac{1}{2} \Tcal^3_\nu \p_{x_\nu} (w+z ) -   \Jcal \p_{x_1}a_3   - \Ncal_\nu    \p_{x_\nu}a_3 
+\left( \tfrac{1}{2} (w+z) \Ncal + a_\nu \Tcal^\nu\right) \cdot (\p_{\Ncal}\Tcal ^3 - \p_{\Tcal^3} \Ncal) \label{vort-correct2}
\end{align}
where we have used \eqref{tildeu-dot-N}, \eqref{tildeu-dot-T}, and \eqref{wza-sheep}.   The identities \eqref{vort-correct1} and \eqref{vort-correct2} 
and the definition of the self-similar transformation in \eqref{eq:y:s:def} and \eqref{eq:ss:ansatz} yield the desired formula for $\p_1 A_3$.
\end{proof}

\section{Closure of $L^\infty$ based bootstrap for $W$}
\label{sec:W}

The goal of this section is to close the bootstrap assumptions which involve $W$, $\tilde W$ and their derivatives, stated in \eqref{eq:W_decay} and \eqref{eq:bootstrap:Wtilde}--\eqref{eq:bootstrap:Wtilde3:at:0}.

\subsection{Estimates for $\partial^{\gamma}\tilde W(y,s)$ for $\abs{y}\leq \ell$}

\subsubsection{The fourth derivative}
We note that the damping term in \eqref{eq:p:gamma:tilde:W:evo} is  strictly positive if $\abs{\gamma}=4$. Indeed, for  $\abs{\gamma} = 4$, we have that
\begin{align}
 D_{\tilde W}^{(\gamma)} := \tfrac{3\gamma_1 + \gamma_2 + \gamma_3-1}{2} 
+ \beta_\tau \Jcal  \left(\p_1  \bar W  + \gamma_1   \p_1 W \right)
 &= \tfrac 32 + \gamma_1 + \beta_\tau \Jcal\left(\p_1  \bar W  + \gamma_1   \p_1 W \right) \notag\\
  &\geq \tfrac 32 + \gamma_1 - (1+ 2 M\eps)\left(1 + \gamma_1   \right) \notag\\
 &\geq \tfrac 13  \, ,
 \label{eq:p:gamma:tilde:W:damping}
\end{align}
where we have used \eqref{eq:beta:tau} and \eqref{Jp1W}.

Using \eqref{eq:p:gamma:tilde:W:damping} and composing with the flow $\pw^{y_0}(s)$ induced by ${\mathcal V}_W$ whose initial datum is given at $s=-\log \eps$ as $\pw^{y_0}(-\log \eps) = y_0$, we obtain from \eqref{eq:p:gamma:tilde:W:evo} that
\begin{align*}
\tfrac{d}{ds} \left(\p^\gamma \tilde W \circ \pw^{y_0}\right) + \left( D_{\tilde W}^{(\gamma)}\circ \pw^{y_0} \right) \left(  \p^\gamma \tilde W \circ \pw^{y_0}\right)  =  \tilde F_W^{(\gamma)} \circ \pw^{y_0} 
\, .
\end{align*}
Appealing to \eqref{eq:Ftilde_4th_est}, the Gr\"onwall inequality, the damping lower bound \eqref{eq:p:gamma:tilde:W:damping}, and our assumption~\eqref{eq:tilde:W:4:derivative} on the initial datum, we obtain
\begin{align}
\abs{\p^\gamma \tilde W \circ \pw^{y_0}} 
 \les \eps^{\frac18}
+   \eps^{\frac{1}{10}} (\log M)^{\abs{\check \gamma}-1}+\abs{\p^\gamma \tilde W(y_0,-\log\eps)}    \les \eps^{\frac18}
+ \eps^{\frac{1}{10}} (\log M)^{\abs{\check \gamma}-1}
\label{eq:end:of:november}
\end{align}
for all $\abs{y_0} \leq \ell$  and  all $s\geq -\log \eps$ such that $\abs{\pw^{y_0}(s)}\leq \ell$. 
Using a power of $\eps$ or the extra $\log M$ factor to absorb the implicit constants, we have thus closed the bootstrap assumption \eqref{eq:bootstrap:Wtilde4}: indeed, by  Lemma~\ref{lem:escape} we have that given any $\abs{y} \leq \ell$ and $s>-\log \eps$, we may write $y = \pw^{y_0}(s)$, for some $y_0$ with $\abs{y_0} < \ell$, and that $\abs{\pw^{y_0}(s')} \leq \ell$ for all $-\log \eps < s' \leq s$.

\subsubsection{Estimates for $\p^\gamma \tilde W$ with $\abs{\gamma}\leq 3$ and $\abs{y}\leq \ell$}
In this subsection we improve on the bootstrap assumptions \eqref{eq:bootstrap:Wtilde:near:0} and \eqref{eq:bootstrap:Wtilde3:at:0}. 
First we recall that $W$ satisfies the constraints \eqref{eq:constraints}, and that the power series for $\bar W$ near $y=0$ is given by 
\begin{align}
\bar W(y) =-y_1 +  y_1^3 + y_1 y_2^2 + y_1 y_3^2  - 3 y_1^5 - y_1 y_2^4 - y_1 y_3^4 - 4 y_1^3 y_2^2 - 4 y_1^3 y_3^2  - 2 y_1 y_2^2 y_3^2 +  \OO(|y|^6) \, .
\label{eq:bar:W:power:series:0} 
\end{align}
Based on this information,  we have that 
\begin{align}
\tilde W(0,s) =  \nabla \tilde W(0,s) = \nabla^2 \tilde W(0,s) = 0 \, .
\label{eq:tilde:W:at:0}
\end{align}
Consider now the bound on $\p^\gamma$ derivatives with $\abs{\gamma}=3$ at $y=0$, with the goal of improving \eqref{eq:bootstrap:Wtilde3:at:0}.  Evaluating \eqref{eq:p:gamma:tilde:W:evo} at $y=0$ yields
\begin{align*} 
\p_s (\p^\gamma \tilde W )^0   
&=  \tilde F_W^{(\gamma),0}  - G_W^0 (\p_1\p^\gamma \tilde W )^0 
- h_W^{\mu,0}    (\p_\mu\p^\gamma \tilde W )^0  - (1+\gamma_1)(1-\beta_\tau) ( \p^\gamma \tilde W)^0 \,.
\end{align*} 
Using \eqref{eq:bootstrap:Wtilde4}, \eqref{eq:bootstrap:Wtilde3:at:0},  \eqref{eq:GW:hW:0},  \eqref{e:forcing:W3}, and \eqref{eq:beta:tau} we obtain that 
\begin{align}
\abs{ \p_s (\p^\gamma \tilde W )^0   }
&\les e^{- (\frac 12- \frac{4}{2k-7})s}  + M (\log M)^{4} \eps^{\frac{1}{10}}  e^{-s} +  M  \eps^{\frac 14} e^{-s}  \les e^{- (\frac 12- \frac{4}{2k-7})s} \,. \label{cool-k}
\end{align}
Therefore, upon integrating in time, using that $\bar W$ is independent of $s$, and appealing to our initial datum assumption \eqref{eq:tilde:W:3:derivative:0} we have that
\begin{align}
 \abs{\p^\gamma \tilde W(0,s)} 
 &\leq   \abs{\p^\gamma \tilde W(0,-\log \eps)} + \int_{-\log \eps}^s \abs{ \p_s (\p^\gamma W )^0 (s') } ds' \leq \tfrac{1}{10} \eps^{\frac 14}\,,
 \label{eq:Schweinsteiger:1}
\end{align}
where we have used the bound \eqref{cool-k} with $k\geq 18$.
In summary, we have shown that 
\begin{align}
\abs{\p^\gamma \tilde W(0,s)} \leq \tfrac{1}{10} \eps^{\frac 14}
 \label{eq:Schweinsteiger:2}
\end{align}
for all $\abs{\gamma}\leq 3$, and all $s\geq-\log\eps$. This closes the bootstrap bound \eqref{eq:bootstrap:Wtilde3:at:0}.

The estimates for $0 \leq \abs{y} \le \ell$ stated in \eqref{eq:bootstrap:Wtilde:near:0} now follow directly from \eqref{eq:bootstrap:Wtilde4},  \eqref{eq:Schweinsteiger:2}, \eqref{eq:tilde:W:at:0}, and the fundamental theorem of calculus, by integrating from $y=0$.

To close the bootstrap bound \eqref{eq:bootstrap:Wtilde} for $\abs{y}\leq \ell$, we note that the bound  follows by setting $\gamma=0$ in \eqref{eq:bootstrap:Wtilde:near:0}, and using that $\eps$ is sufficiently small.  For  \eqref{eq:bootstrap:Wtilde1},  the bound in the case $\abs{y}\leq \ell$ follows by setting $\gamma=(1,0,0)$ in \eqref{eq:bootstrap:Wtilde:near:0}, and using that  $M \ell^3 \eps^{\frac{1}{10}} \ll \eps^{\frac{1}{11}}$.
For  \eqref{eq:bootstrap:Wtilde2}, in the case $\abs{y}\leq \ell$, the desired bound holds by setting $\abs{\gamma}=1$ in \eqref{eq:bootstrap:Wtilde:near:0}, and using that $(\log M)^4 \ell^3 \eps^{\frac{1}{10}} \ll \eps^{\frac{1}{13}}$. 

\subsection{A framework for weighted estimates}
In order to close the bootstrap estimates \eqref{eq:W_decay} and \eqref{eq:tildeW_decay}, for $\abs{y}\geq \ell$, we will need to employ carefully weighted estimates. If $\RSZ$ is the quantity we wish to estimate (either $\partial^\gamma W$ or $\partial^\gamma \tilde W$), we will write the evolution equation for $\RSZ$ in the form
\begin{equation}\label{eq:q_eq}
\p_s \RSZ + D_\RSZ \; \RSZ   + \mathcal{V}_W\cdot \nabla \RSZ= F_\RSZ \,,
\end{equation}
where $D_\RSZ$ denotes the damping of the $\RSZ$ equation, and $F_\RSZ$ is the forcing term.
If we let
\[q :=\eta^{\mu}\, \RSZ\]
denote the weighted version of $\RSZ$ (we will use exponents $\mu$ with $\abs{\mu} \leq \frac 12$), then $q$ satisfies the evolution equation
\begin{align}
  \p_sq  + \underbrace{\left( D_{\RSZ} - \eta^{-\mu} \mathcal{V}_W \cdot \nabla \eta^\mu \right)}_{=: D_q} q + \mathcal{V}_W \cdot \nabla q &= \underbrace{\eta^{\mu}F_{\RSZ}}_{:=F_q} \label{eq:tildeq}
\end{align}
and we can expand the definition of $\mathcal D_{q}$ as 
\begin{align}
 D_{q}=D_\RSZ -3\mu+3\mu\eta^{-1}- 2 \mu
 \underbrace{\eta^{-1} \left(y_1 (\beta_\tau \Jcal W+G_W ) + 3 h_{W}^{\nu} y_\nu \abs{\check y}^4\right)\,.}_{=: \mathcal D_\eta}
 \label{eq:Dq:def}
\end{align}
Note that $D_\eta$ is independent of $\mu$. 
By Gr\"onwall's inequality, and composing with the trajectories $\pw^{y_0}(s)$ such that $\pw^{y_0}(s_0) = y_0$ for some $s_0\geq - \log \eps$ with $\abs{y_0} \geq \ell$, we deduce from \eqref{eq:tildeq} that
\begin{align}
\label{eq:q:est}
\abs{q\circ \pw^{y_0}(s)}
&\leq   \abs{q(y_0)} \exp\left(- \int_{s_0}^s  D_{q}\circ\pw^{y_0}(s')  \,ds'\right) \notag\\
&\qquad +\int_{s_0}^s\abs{F_q\circ \pw^{y_0}(s')}\exp\left(- \int_{s'}^s D_{q}\circ\pw^{y_0}(s'') \,ds''\right) \,ds' \,.
\end{align}
We first note that the $3\mu\eta^{-1}$ term in the definition of $D_q$ in \eqref{eq:Dq:def} satisfies $- 3 \mu \eta^{-1} \circ \pw^{y_0}(s) \leq 0$ whenever $\mu \geq 0$, and thus this term does not contribute  to the right side of \eqref{eq:q:est}. Next, we estimate the $D_\eta$ contribution to the exponential term on the right side of \eqref{eq:q:est}, as this contribution is independent of $\mu$ and is a-priori not sign-definite.
Using \eqref{eq:W_decay} to bound $W$, \eqref{e:bounds_on_garbage_2} to estimate $\Jcal$, \eqref{eq:beta:tau} to bound $\beta_\tau$, \eqref{e:G_W_estimates} for $G_W$, and \eqref{e:h_estimates} to estimate $h_W$ we deduce
\begin{align}
\abs{\mathcal{D}_\eta}
&\leq
\eta^{-1} \left(4 \abs{y_1}  \eta^{\frac 16} + \abs{y_1} \abs{G_W} + 3 \abs{h_W^\nu} \abs{y_\nu} \abs{\check y}^4 \right)
\notag\\
&\leq 4 \eta^{-\frac 13} +  M  \eta^{-\frac 12} \left( M e^{-\frac s2} +  M^{\frac 12} \abs{y_1} e^{-s} + \eps^{\frac 13} \abs{\check y} \right) + 6  M^2  \eta^{-\frac 16} e^{-\frac s2}
\notag\\
& \leq 5 \eta^{-\frac 13} +   e^{-\frac s3} 
\label{eq:Deta:upper:bound}
\end{align}
for all $s\geq -\log \eps$, upon using \eqref{e:space_time_conv} and taking $\eps$ to be sufficiently small in terms of $M$.

\subsubsection{The case $\ell \leq \abs{y_0} \leq \LLL$}

Composing the upper bound for $D_\eta$ in \eqref{eq:Deta:upper:bound} with a trajectory $\pw^{y_0}(s)$ with $\abs{y_0} \geq \ell$, using \eqref{eq:escape_from_LA},  and the bound $2 \eta(y) \geq 1 + \abs{y}^2 $, we obtain from \eqref{eq:Deta:upper:bound} that  
\begin{equation}
2 \mu \int_{s_0}^s  \abs{\mathcal D_\eta\circ\pw^{y_0}(s')}\,ds'
\leq \int_{s_0}^{\infty} 10 \left( 1+ \ell^2 e^{\frac 25 (s'-s_0)}\right)^{-\frac 13}  + e^{-\frac{s'}{3}} ds'
\leq 65 \log \tfrac{1}{\ell} + \eps^{\frac 13}  \leq 70 \log \tfrac{1}{\ell} \,,
 \label{eq:Deta:est}
\end{equation}
since $s_0\geq -\log\eps$,  $\ell \in (0,1/100]$,   for all $\abs{\mu} \leq \frac 12$. Combining \eqref{eq:q:est} with \eqref{eq:Deta:est}, we deduce that 
\begin{align}
\abs{q\circ \pw^{y_0}(s)}
&\leq \ell^{-70}    \abs{q(y_0)} \exp\left( \int_{s_0}^s \left( 3\mu - D_{\RSZ}- 3\mu\eta^{-1}  \right)\circ\pw^{y_0}(s')  ds'\right) \notag\\
& \qquad + \ell^{-70} \int_{s_0}^s\abs{F_q\circ \pw^{y_0}(s')}\exp\left(  \int_{s'}^s \left( 3\mu - D_{\RSZ}- 3\mu\eta^{-1}  \right)\circ\pw^{y_0}(s'')  ds''\right) \,ds' \, .
\label{eq:q:est:new}
\end{align}
To conclude our weighted estimate, we need information on the size of $q(y_0)$. We recall
that for any $s> -\log \eps$ and any $\ell \leq \abs{y} \leq \LLL$, there exists $s_0 \in [-\log \eps,s)$ and $y_0$ with $\ell \leq \abs{y_0}\leq \LLL$ such that $y = \pw^{y_0}(s)$. 
This follows from Lemma~\ref{lem:escape} by following the trajectory ending at $(y,s)$ backwards in time. We also note that in the situation where $s_0 > -\log \eps$, 
we have  $\abs{y_0} = \ell$. 
Therefore, $q(y_0)$ is bounded using information on the initial datum if $s_0 = -\log \eps$, and appealing to bootstrap bounds which hold for all $s\geq -\log \eps$, and $\abs{y_0} = \ell$.  The bound \eqref{eq:q:est:new} will be applied in the following subsections for various values of $\mu$, with $\abs{\mu}\leq \frac 12$, and with $\RSZ$ being either equal to $W$ or $\tilde W$.

\subsubsection{The case $\abs{y_0} \geq \LLL$}
The only difference from the previously considered case comes in the upper bound \eqref{eq:Deta:est}. In this case, we have that  for $\abs{y_0} \geq \LLL \geq 4 $
\begin{equation}
2 \mu \int_{s_0}^s  \abs{\mathcal D_\eta\circ\pw^{y_0}(s')}\,ds'
\leq \int_{s_0}^{\infty} 10 \left( 1+ \LLL^2 e^{\frac 25 (s'-s_0)}\right)^{-\frac 13}  + e^{-\frac{s'}{3}} ds'
\leq 80 \LLL^{-\frac 23} + \eps^{\frac 13}  \leq \eps^{\frac{1}{16}} \,,
 \label{eq:Deta:est:new}
\end{equation}
for $s_0\geq -\log\eps$, and $\abs{\mu} \leq \frac 12$. Combining \eqref{eq:q:est} with \eqref{eq:Deta:est:new}, we deduce that 
\begin{align}
\abs{q\circ \pw^{y_0}(s)}
&\leq e^{\eps^{\frac{1}{16}}}    \abs{q(y_0)} \exp\left( \int_{s_0}^s \left( 3\mu - D_{\RSZ} - 3\mu\eta^{-1}  \right) \circ\pw^{y_0}(s') ds'\right) \notag\\
&\qquad  + e^{\eps^{\frac{1}{16}}} \int_{s_0}^s\abs{F_q\circ \pw^{y_0}(s')}\exp\left(  \int_{s'}^s \left( 3\mu - D_{\RSZ}  - 3\mu\eta^{-1}  \right) \circ \pw^{y_0} (s'')ds''\right) \,ds' \, .
\label{eq:q:est:new:new}
\end{align}
The bound on $\abs{q(y_0)}$ will now be obtained from the the previous estimate \eqref{eq:q:est:new} when $s_0>-\log \eps$ (since in this case $\abs{y_0} = \LLL$), or from the initial datum assumption when $s_0 = -\log \eps$ (since in this case $\abs{y_0} > \LLL$).

\subsection{Estimate for $\tilde W(y,s)$ for $\ell \leq \abs{y}\leq \LLL$}
\label{sec:tilde:W:middle}
We now close the bootstrap bound \eqref{eq:bootstrap:Wtilde} for $\ell \leq \abs{y}\leq \mathcal{L} $.
We let $\RSZ = \tilde W$, $\mu = - \frac 16$, so that the weighted quantity $q$ is given as $q:=\eta^{-\frac16}\tilde W$. We use the evolution equation \eqref{eq:tilde:W:evo}, so that in this case the quantity $3 \mu - D_{\RSZ} - 3\mu \eta^{-1}$ present in \eqref{eq:q:est:new} equals to $- \beta_\tau \Jcal \p_1 W + \frac 12 \eta^{-1}$, while the forcing term $F_q$ equals to $\eta^{-\frac 16} \tilde F_W$. 

First we estimate the contribution of the damping term. Since $\abs{\beta_\tau \Jcal} \leq 1 + \eps^{\frac 12}$ holds due to \eqref{eq:beta:tau} and \eqref{e:bounds_on_garbage}, and since for $\abs{y_0} \geq \ell$ we may apply to the trajectory estimate \eqref{eq:escape_from_LA}, by also appealing to the bootstrap assumption for $\p_1 W$ in \eqref{eq:W_decay}, and the bound $\tilde \eta^{-\frac 13}(y/2) \leq 4 \eta^{-\frac 13}$, we conclude
\begin{align}
\label{eq:Wtilde_damping}
\int_{s_0}^{s} \beta_\tau  \abs{\Jcal \p_1 W} \circ \pw^{y_0}(s') + \tfrac 12 \eta^{-1}\circ \pw^{y_0}(s') \,ds' 
\leq 5 \int_{s_0}^{s}  \eta^{-\frac 13} \circ \pw^{y_0}(s')\,ds'  
\leq 40 \log \tfrac{1}{\ell}
\end{align}
as in \eqref{eq:Deta:est}, for all $s\geq s_0 \geq -\log \eps$.  Second, we estimate the forcing term in \eqref{eq:q:est:new}. Using the $\gamma=0$ case in \eqref{eq:Ftilde_est} we arrive at
\begin{align} 
\int_{s_0}^{s}\abs{\eta^{-\frac16} \tilde F_W}\circ \pw^{y_0}(s')\,ds'
\les 
\eps^{\frac 18} \int_{s_0}^{s} \eta^{-\frac13} \circ \pw^{y_0}(s')\,ds'
\les 
\eps^{\frac 18} \log \tfrac{1}{\ell} 
\label{eq:Wtilde_forcing}
\end{align}
for all $s\geq s_0 \geq -\log \eps$, and $\ell \in (0,1/10]$. 

Inserting the the bounds \eqref{eq:Wtilde_damping} and \eqref{eq:Wtilde_forcing} into \eqref{eq:q:est:new},  we deduce that
\begin{align}
\abs{ \left(\eta^{-\frac 16} \tilde W\right) \circ \pw^{y_0}(s)} \leq \ell^{-110}  \eta^{-\frac 16}(y_0)  \abs{\tilde W(y_0,s_0)} + M \eps^{\frac 18} \ell^{-110} \log \ell^{-1}
\label{eq:bootstrap:Wtilde:*}
\end{align}
where $M$ absorbs the implicit constant in \eqref{eq:Wtilde_forcing}. 
Using the initial data assumption \eqref{eq:tilde:W:zero:derivative} if $s_0=-\log \eps$, and  \eqref{eq:bootstrap:Wtilde:near:0} if $s_0>-\log \eps$, we deduce from \eqref{eq:bootstrap:Wtilde:*}
that 
\begin{align} 
\eta^{-\frac 16}(y) \abs{\tilde W(y,s)} \leq \ell^{-110} \max\left\{ M  \eps^{\frac{1}{10}} \ell^4, \eps^{\frac{1}{10}} \right\}  + M \eps^{\frac 18} \ell^{-110} \log \ell^{-1}  \leq \tfrac{1}{10} \eps^{\frac{1}{11}}
\label{tildeWfinal}
\end{align} 
for all $\ell \leq \abs{y} \leq \LLL$ and all $s\geq -\log \eps$. Here we have used a small power of $\eps$ to absorb all the $\ell$ and $M$ factors. The above estimate shows that \eqref{eq:bootstrap:Wtilde} may be improved by a factor larger than two, as desired.

\subsection{Estimate for $\partial_1\tilde W(y,s)$ for $\ell \leq \abs{y}\leq \LLL$}
\label{sec:p1:tilde:W:middle}

Our goal is to close the bootstrap bound \eqref{eq:bootstrap:Wtilde1} for $\ell \leq \abs{y}\leq \LLL$.
We let $\RSZ = \p_1 \tilde W$, $\mu = \frac 13$, so that the weighted quantity $q$ is given as $q:=\eta^{\frac13} \p_1 \tilde W$. We use the evolution equation \eqref{eq:p:gamma:tilde:W:evo} with $\gamma = (1,0,0)$, so that  the quantity $3 \mu - D_{\RSZ}$   in \eqref{eq:q:est:new} equals to $- \beta_\tau \Jcal (\p_1 W + \p_1 \bar W)$, while the forcing term $F_q =\eta^{\frac 13} \tilde F_W^{(1,0,0)}$. 

As in the previous subsection (see estimate \eqref{eq:Wtilde_damping}), we have that the contributions to \eqref{eq:q:est:new} due to the damping term $3\mu -D_\RSZ$ are bounded as
\begin{align}
\label{eq:Wtilde_damping:d1}
\int_{s_0}^{s} \beta_\tau  \abs{\Jcal (\p_1 W + \p_1 \bar W)} \circ \pw^{y_0}(s')\,ds' 
\leq 80 \log \tfrac{1}{\ell} \,.
\end{align}
On the other hand, the forcing term $F_q=\eta^{\frac 13} \tilde F_W^{(1,0,0)}$ is estimated using \eqref{eq:Ftilde_d1_est} pointwise in space as 
\[
\abs{F_q} \les \eta^{\frac 13} \eps^{\frac{1}{11}} \eta^{-\frac 12} \les   \eps^{\frac{1}{11}} \eta^{-\frac 16}
\, ,
\] and thus, similarly to \eqref{eq:Wtilde_forcing} we obtain
\begin{align} 
\int_{s_0}^{s}\abs{F_q}\circ \pw^{y_0}(s')\,ds'
 \les \eps^{\frac{1}{11}}  \log \tfrac{1}{\ell} \,.
\label{eq:Wtilde_forcing:d1}
\end{align}
Combining \eqref{eq:Wtilde_damping:d1} and \eqref{eq:Wtilde_forcing:d1} with \eqref{eq:q:est:new}, and using our initial datum assumption \eqref{eq:tilde:W:p1=1} when $s_0 = -\log \eps$, respectively \eqref{eq:bootstrap:Wtilde4} for $s_0 > -\log \eps$, we deduce that 
\begin{align} 
\eta^{\frac 16}(y) \abs{\p_1 \tilde W(y,s)} \leq \ell^{-150} \max\left\{ M  \eps^{\frac{1}{10}} \ell^3, \eps^{\frac{1}{11}}  \right\}  + M \eps^{\frac{1}{11}} \ell^{-150} \log \ell^{-1} \leq \tfrac{1}{10} \eps^{\frac{1}{12}}
\label{p1tildeWfinal}
\end{align} 
for all $\ell \leq \abs{y} \leq \LLL$ and all $s\geq -\log \eps$. Here we have used a small power of $\eps$ to absorb all the $\ell$ and $M$ factors. The above estimate shows that \eqref{eq:bootstrap:Wtilde1} may be improved by a factor larger than two, as desired.

\subsection{Estimate for $\check \nabla \tilde W(y,s)$ for $\ell \leq \abs{y}\leq \LLL$}
\label{sec:pcheck:tilde:W:middle}
The proof of the bootsrap \eqref{eq:bootstrap:Wtilde2} for $\abs{y}\geq \ell$  is nearly identical to the one in the previous subsection, so we only present here the necessary changes.  We let $\RSZ = \check\nabla    W$ and $\mu=0$, so that $q= \check \nabla \tilde W$. Using \eqref{eq:p:gamma:tilde:W:evo} with $\gamma \in \{(0,1,0),(0,0,1)\}$, we obtain that in this case  $3 \mu - D_{\RSZ} = - \beta_\tau \Jcal  \p_1 \bar W$, while the forcing term is $F_q = \tilde F_W^{(\gamma)}$. The integral of the damping term arising in \eqref{eq:q:est:new} is bounded using \eqref{eq:Wtilde_damping} by $40 \log \ell^{-1}$. On the other hand, the forcing term is bounded using \eqref{eq:Ftilde_dcheck_est} by $\eps^{\frac{1}{12}} \eta^{-\frac 13}$. Therefore, as in \eqref{eq:Wtilde_forcing:d1}, the integral of the forcing term composed with the flow $\pw^{y_0}(s)$ is bounded as $\les  \eps^{\frac{1}{12}} \log \ell^{-1}$. Combining these two estimates, with our assumptions on the initial datum~\eqref{eq:tilde:W:check=1}  and \eqref{eq:bootstrap:Wtilde4}, we arrive at
\begin{align} 
\abs{\check \nabla \tilde W(y,s)} \leq \ell^{-110} \max\left\{ M  \eps^{\frac{1}{10}} \ell^3, \eps^{\frac{1}{12}}  \right\}  + M \eps^{\frac{1}{12}} \ell^{-110} \log \ell^{-1} \leq \tfrac{1}{10} \eps^{\frac{1}{13}}
\label{p2tildeWfinal}
\end{align} 
for all $\ell \leq \abs{y} \leq \LLL$ and all $s\geq -\log \eps$, thereby improving the bootstrap bound \eqref{eq:bootstrap:Wtilde2}.

\subsection{Estimate for $\partial^{\gamma} W(y,s)$ with $\abs{\gamma}=2$ for $\abs{y}\geq \ell$}
Our last remaining $W$ bootstrap bound is \eqref{eq:W_decay}. Recall that $W = \bar W + \tilde W$, and thus, the $\abs{\gamma}=0$ and $\abs{\gamma}=1$ cases of  \eqref{eq:W_decay} follow directly from the properties \eqref{eq:bar:W:properties} of the function $\bar W$, and the previously established estimates \eqref{eq:bootstrap:Wtilde}--\eqref{eq:bootstrap:Wtilde2}. Thus, it remains to treat the cases for which $\abs{\gamma} = 2$, which are the third and respectively the fifth bounds stated in \eqref{eq:W_decay}. 

For $\abs{\gamma} = 2$, we let $\RSZ = \p^\gamma W$, and we define $\mu$ as
\begin{align*}
\mu 
=
\begin{cases} 
\tfrac 13\, ,\qquad \mbox{for} \qquad \abs{\gamma}=2 \; \mbox{  and  } \; \gamma_1 \geq 1 \, ,\\
 \tfrac 16\, ,\qquad \mbox{for} \qquad \abs{\gamma}=2 \; \mbox{  and  } \; \gamma_1 = 0 \, .
\end{cases}
\end{align*}
According to these choices we define $q = \eta^\mu \p^\gamma W$, and appeal to the evolution equation \eqref{euler_for_Linfinity:a}, to deduce that the quantity $3 \mu - D_{\RSZ}$ present in \eqref{eq:q:est:new} equals to
\begin{align}
3 \mu - D_{\RSZ} 
= 
\begin{cases}
-\tfrac{2\gamma_1-1}{2} - (2\gamma_1-1) \beta_\tau \Jcal \p_1 W\, ,\qquad &\mbox{for} \qquad \abs{\gamma}=2 \; \mbox{  and  } \; \gamma_1 \geq 1 \, ,\\
-  \beta_\tau \Jcal \p_1 W \, ,\qquad &\mbox{for} \qquad \abs{\gamma}=2 \; \mbox{  and  } \; \gamma_1 = 0 \, .
\end{cases}
\label{eq:Magic:Johnson}
\end{align}
We next consider these two cases separately.

The case $\gamma_1=0$ and $\abs{\check \gamma} =2$ is similar to the cases treated earlier: as in \eqref{eq:Wtilde_damping}  we have 
\begin{align}
\int_{s_0}^{s} \beta_\tau  \abs{\Jcal \p_1 W} \circ \pw^{y_0}(s')\,ds' 
\leq 40 \log \tfrac{1}{\ell}
\label{eq:Magic:Johnson:1}
\end{align}
and similarly to \eqref{eq:Wtilde_forcing}, by appealing to \eqref{eq:forcing_W}, using that $-\frac 16 -\frac{1}{2k-7} \geq - \frac{1}{12}$ for $k\geq 10$, we have
\begin{align} 
\int_{s_0}^{s}\abs{\eta^{\frac16} F_W^{(\gamma)}}\circ \pw^{y_0}(s')\,ds'
\leq  
M^{\frac 23} \int_{s_0}^{s} \eta^{-\frac{1}{12}} \circ \pw^{y_0}(s')\,ds'
\leq 
M^{\frac56} \log \tfrac{1}{\ell} .
\label{eq:Magic:Johnson:2}
\end{align}
By inserting the bounds \eqref{eq:Magic:Johnson:1} and \eqref{eq:Magic:Johnson:2} into \eqref{eq:q:est:new}, we arrive at
\begin{align*}
\eta^{\frac 16}(y) \abs{\check\nabla^2 W(y,s)} 
&\leq \ell^{-110} \eta^{\frac 16}(y_0) \abs{\check\nabla^2 W(y_0,s_0)} +  M^{\frac 56} \ell^{-110}  \log \tfrac{1}{\ell} \notag\\
&\leq \ell^{-110} \max\left \{ M^{\frac 56} , 2 M \eps^{\frac{1}{10}} \ell^2 \right\} +  M^{\frac 56} \ell^{-112}   
\end{align*}
for $\ell \in (0,1/100]$,
where we have also appealed to our initial datum assumption \eqref{eq:W:gamma=2}  when $s_0 = - \log \eps$, and to \eqref{eq:bootstrap:Wtilde:near:0} when $s> -\log \eps$. Since by \eqref{eq:ell:choice} we have $\ell = (\log M)^{-5} $ we have that 
\begin{align}
\ell^{-112} \leq \tfrac{1}{10} M^{\frac 16} \, ,
\label{eq:kingfisher:2}
\end{align}
by taking $M$ to be sufficiently large, and so we obtain an improvement over the $\check \nabla^2 W$ bootstrap assumption in \eqref{eq:W_decay}.

To conclude, we consider the cases when $\abs{\gamma}=2$, with $\gamma_1\geq 1$. In this case, by appealing to \eqref{eq:Magic:Johnson} and \eqref{eq:Magic:Johnson:1}, we obtain that 
\begin{align}
\exp\left( \int_{s'}^s \left( 3\mu - D_{\RSZ}\circ\pw^{y_0}(s'') \right) ds'' \right)\leq \ell^{-120} e^{\frac{s'-s}{2}} 
\label{eq:Magic:Johnson:3}
\end{align}
for any $s> s' > s_0 \geq -\log \eps$. On the other hand, from \eqref{eq:forcing_W} we deduce that 
\begin{align}
\abs{F_q} \leq \eta^{\frac 13} \abs{F_W^{(\gamma)}}  \leq M^{\frac{\abs{\check \gamma}}{3} + \frac 16}
\,.
\label{eq:Magic:Johnson:4} 
\end{align}
Combining \eqref{eq:Magic:Johnson:3} and \eqref{eq:Magic:Johnson:4} with \eqref{eq:q:est:new}, for $\abs{\gamma}=2$ with $\gamma_1\geq 1$ we arrive at
\begin{align}
\eta^{\frac 13}(y) \abs{\p^\gamma W(y,s)} 
&\leq \ell^{-190} \eta^{\frac 13}(y_0) \abs{\p^\gamma W(y_0,s_0)} + M^{\frac{\abs{\check \gamma}}{3} + \frac 16} \ell^{-190}  \int_{s_0}^s e^{\frac{s'-s}{2}} ds' \notag\\
&\leq \ell^{-190} \max\left \{ M^{\frac 16}  , 2 M \eps^{\frac{1}{10}} \ell^2 \right\} + 2 M^{\frac{\abs{\check \gamma}}{3} + \frac 16} \ell^{-190}  
\label{D2Wfinal}
\end{align}
by appealing to our assumptions on the initial datum assumption~\eqref{eq:W:gamma=2:p1}  if $s_0 = -\log \eps$, and to \eqref{eq:bootstrap:Wtilde:near:0} when $s> -\log \eps$.
Since by \eqref{eq:ell:choice} we have $\ell = (\log M)^{-5}$, for  $M$ sufficiently large the bound
\begin{align}
\label{eq:kingfisher:3}
\ell^{-190}  \leq \tfrac{1}{10} M^{\frac 16}
\end{align}
holds, and we obtain an improvement over the $\p^\gamma W$ bootstrap assumption in \eqref{eq:W_decay}.

\subsection{Estimate for $W(y,s)$ for $ \abs{y}\geq \LLL$}
The bounds in this section are similar to those in Section~\ref{sec:tilde:W:middle}. We use $\mu = -\frac 16$ and $\RSZ =W$, so that $q = \eta^{-\frac 16} \tilde W$. From \eqref{eq:euler:ss:a}, we obtain that $3 \mu - D_{\RSZ} - 3\mu \eta^{-1}$  equals to $\frac 12 \eta^{-1}$, while the forcing term $F_q$ equals to $\eta^{-\frac 16} ( F_W - e^{-\frac s2} \beta_\tau \dot \kappa)$. In order to apply \eqref{eq:q:est:new:new}
similarly to \eqref{eq:Deta:est:new} we use  Lemma~\ref{lem:escape} to estimate
\begin{align*}
\int_{s_0}^{s}  \tfrac 12 \eta^{-1}\circ \pw^{y_0}(s') \,ds' 
\leq  \int_{s_0}^{\infty}  \left(1 + \LLL^2 e^{\frac 25 (s'-s_0)}\right)^{-1}   \,ds' 
\leq  \LLL^{-\frac 23} = \eps^{\frac{1}{16}}
\end{align*}
while using \eqref{eq:F_WZ_deriv_est} and \eqref{eq:acceleration:bound} we derive
\begin{align*}
 \int_{s_0}^s \abs{F_q \circ \pw^{y_0}}(s') ds' \les \int_{s_0}^s e^{-\frac{s'}{2}} \les \eps^{\frac 12}.
\end{align*}
Inserting the above two estimates into \eqref{eq:q:est:new:new}, we obtain
\begin{align*}
\abs{\eta^{-\frac 16}   W\circ \pw^{y_0}(s)} \leq   e^{2 \eps^{\frac{1}{16}}} \left( \abs{q(y_0)}  + \eps^{\frac 13} \right)\,.
\end{align*}
In the case $s_0> -\log \eps$, we have $\abs{y_0} = \LLL$, and so from \eqref{eq:bootstrap:Wtilde} and the first inequality in \eqref{eq:bar:W:properties} we have that $\abs{q(y_0)} \leq 1 + \eps^{\frac{1}{11}}$.
On the other hand, when $s_0 = -\log \eps$ we use the  initial data assumption \eqref{eq:rio:de:caca:1}, so that $\abs{q(y_0)} \leq 1 + \eps^{\frac{1}{11}}$. In summary, from the above bound we deduce that for any $\abs{y}\geq \LLL$ we have
\begin{align} 
\abs{\eta^{-\frac 16}   W(y,s)} \leq e^{2 \eps^{\frac{1}{16}}} \left( 1 + \eps^{\frac{1}{11}} + \eps^{\frac 13} \right) \leq 1 + \eps^{\frac{1}{19}}
\label{Wlargey-final}
\end{align} 
for $\eps$ sufficiently small, 
which improves the bootstrap bound in the first line of \eqref{eq:W_decay}.

\subsection{Estimate for $\p_1 W(y,s)$ for $ \abs{y}\geq \LLL$}
In order to close the bootstrap for the second bound in \eqref{eq:W_decay}, we proceed  similarly to Section~\ref{sec:p1:tilde:W:middle}. Letting $q = \eta^{\frac 13} \p_1 W$, from the evolution equation~\eqref{eq:grad:W:a} we deduce that the damping term at the exponential in \eqref{eq:q:est:new:new} obeys $3 \mu - D_{\RSZ} - 3\mu \eta^{-1} \leq - \beta_\tau \Jcal \p_1 W$, while the forcing term $F_q$ equals to $\eta^{\frac 13}  F_W^{(1,0,0)}$. Using the $\p_1 W$ bound in \eqref{eq:W_decay} for $\abs{y}\geq \LLL$,  and Lemma~\ref{lem:escape} with $\abs{y_0} \geq \LLL$, similarly to \eqref{eq:Deta:est:new} we obtain that 
\[
\int_{s_0}^s \left( 3\mu - D_{\RSZ} - 3\mu\eta^{-1}  \right) \circ\pw^{y_0}(s') ds'
\leq 3 \int_{s_0}^s \eta^{-\frac 13} \circ\pw^{y_0}(s') ds'
\leq \eps^{\frac{1}{16}}\,.
\]
On the other hand, from the second bound in \eqref{eq:forcing_W} and the fact that $k\geq 18$ we similarly deduce that 
\[
\int_{s_0}^s \abs{F_q \circ\pw^{y_0}(s')}  ds' \les \eps^{\frac 18} \int_{s_0}^s \eta^{\frac 13 - \frac 12 + \frac{3}{2k-5}} \circ\pw^{y_0}(s') ds'
\les 
\eps^{\frac 18} \int_{s_0}^s \eta^{- \frac{1}{15}} \circ\pw^{y_0}(s') ds'
\les \eps^{\frac 18}
\]
since $\abs{y_0} \geq \LLL$. Combining the above two estimates with \eqref{eq:q:est:new:new} we deduce that
\begin{align*}
\abs{\eta^{\frac 13}   \p_1W\circ \pw^{y_0}(s)} \leq   e^{2 \eps^{\frac{1}{16}}} \left( \abs{q(y_0)}  + \eps^{\frac 17} \right)\,.
\end{align*}
When $s_0>-\log \eps$ we have $\abs{y_0} = \LLL$ and   $q(y_0)$ is may be estimated using the second estimate in \eqref{eq:bar:W:properties}, the fact that $\tilde \eta^{-\frac 13} \leq \eta^{-\frac 13}$, and the bootstrap assumption \eqref{eq:bootstrap:Wtilde1} as $\abs{q(y_0)} \leq \eta^{\frac 13} \abs{\p_1 \bar W} + \eta^{\frac 13} \abs{\p_1 \tilde W} \leq 1 + \eps^{\frac{1}{12}}$. On the other hand, when $s_0=-\log \eps$ we have $\abs{y_0} > \LLL$ and from the initial datum assumption \eqref{eq:rio:de:caca:2} we also deduce $\abs{q_0(y_0)} \leq 1+ \eps^{\frac{1}{12}}$. Combining these bounds with the above estimate along trajectories, we deduce that
\begin{align}
\abs{\eta^{\frac 13}   \p_1W(y,s)} \leq   e^{2 \eps^{\frac{1}{16}}} \left( 1 + \eps^{\frac{1}{12}} + \eps^{\frac 17} \right) \leq \tfrac 32 
\label{p1Wlargey-final}
\end{align} 
for all $\abs{y} \geq \LLL$ and $s\geq -\log \eps$, thereby clsoing the bootstrap bound on the second line of \eqref{eq:W_decay}, in this $y$-region.

\subsection{Estimate for $\check \nabla W(y,s)$ for $ \abs{y}\geq \LLL$}
Closing the third bootstrap in \eqref{eq:W_decay}, for $\abs{y} \geq \LLL$, is done similarly to Section~\ref{sec:pcheck:tilde:W:middle}. In this region we have that $\mu = 0$ and $q = \check \nabla W$. From \eqref{eq:grad:W:b} and \eqref{eq:grad:W:c} we deduce that that damping term is given by 
$3 \mu - D_{\RSZ} - 3\mu \eta^{-1} = - \beta_\tau \Jcal \p_1 W$ so that we may use the same estimate for it as in the previous subsection. For the forcing term we appeal to the fifth case in \eqref{eq:forcing_W} which bounds $\abs{F_q}$ from above by $M^2 \eps^{\frac 13} \eta^{-\frac 13}$, so that
\[
\int_{s_0}^s \abs{F_q \circ \pw^{y_0}(s')} ds' \leq  \eps^{\frac 14} 
\]
for $\abs{y_0} \geq \LLL$. We deduce from \eqref{eq:q:est:new:new} that 
\begin{align*}
\abs{\check \nabla W\circ \pw^{y_0}(s)} \leq   e^{2 \eps^{\frac{1}{16}}} \left( \abs{\check \nabla W (y_0)}  + \eps^{\frac 14} \right)\,.
\end{align*}
For $s_0>-\log \eps$ we combine the third bound in \eqref{eq:bar:W:properties} with \eqref{eq:bootstrap:Wtilde2}, while for $s_0 = -\log \eps$ we appeal to the initial datum assumption \eqref{eq:rio:de:caca:3} to deduce that $\abs{\check \nabla W (y_0)}  \leq \frac 34$.
We deduce that 
\begin{align}
\abs{\check \nabla W (y,s)} \leq   e^{2 \eps^{\frac{1}{16}}} \left( \tfrac 34 + \eps^{\frac 14} \right) \leq \tfrac 56
\label{p2Wlargey-final}
\end{align}
holds for all $\abs{y} \geq \LLL$ and all $s\geq -\log \eps$, which closes the bootstrap from the third line of \eqref{eq:W_decay}.

\section{$\dot{H}^k$ bounds}\label{sec:energy}
\begin{definition}[Modified $\dot H^k$-norm]
For $k \geq 18$ we introduce the semi-norm
\begin{align}
E_k^2(s) = E_k^2[U,S](s) :=
\sum_{|\gamma|=k} \lambda^\modckg  \left( \norm{\p^\gamma U(\cdot,s)}_{L^2}^2 +\norm{\p^\gamma S(\cdot,s)}_{L^2}^2 \right)
\label{eq:Ek:def}
\end{align}
where $ \lambda = \lambda(k) \in (0,1)$ is to be made precise below (cf.~Lemma~\ref{lem:forcing:1}). 
\end{definition} 
Clearly, $E_k^2$ is equivalent to the homogenous Sobolev norm $\dot{H}^k$, and we have the inequalities
\begin{align}
 \lambda^k \left( \snorm{ U}_{\dot H^k}^2 + \snorm{ S}_{\dot H^k}^2\right) \le E_k^2 \le  \snorm{ U}_{\dot H^k}^2 + \snorm{ S}_{\dot H^k}^2 \,.
 \label{norm_compare}
 \end{align} 
 
 \subsection{Higher-order derivatives for the $(U,S)$-system}
In order to estimate $E_k(s)$ we need the differentiated form of the $(U,S)$-system \eqref{US-short}. For this purpose, fix $\gamma \in {\mathbb N}_0^3$ with $\abs{\gamma} = k$, and apply $\p^\gamma$ to \eqref{US-short}, to obtain
\begin{subequations} 
\label{US-L2}
\begin{align}
&\p_s (\gui) 
- 2 \beta_1 \beta_\tau e^{-s} \dot Q_{ij} ( \guj) 
+ ( \mathcal{V} _U \cdot \nabla ) \gui + \mathcal{D}_\gamma \gui  
+  \beta_\tau   (\beta_3  + \beta_3  \gamma_1) \Jcal \Ncal_i  \p_1 W \gs  \notag \\
& \qquad \qquad
+ 2 \beta_\tau \beta_3 S \left( \Jcal \Ncal_i e^{\frac{s}{2}}   \p_1 (\gs) + e^{-\frac{s}{2}} \delta^{i\nu}  \p_\nu ( \gs) \right)
= \mathcal{F} _{U_i}^{(\gamma)} \,, \label{US-L2-U}\\
&\p_s (\gs)  
+ ( \mathcal{V} _U \cdot \nabla ) \gs 
+ \mathcal{D}_\gamma \gs  
+   \beta_\tau (\beta_1 + \beta_3   \gamma_1)  \Jcal \Ncal_j \guj \p_1 W 
\notag\\
& \qquad \qquad
+ 2 \beta_\tau \beta_3 S \left( e^{\frac{s}{2}} \Jcal \Ncal_j \p_1( \guj) + e^{-\frac{s}{2}} \p_\nu( \p^\gamma U_\nu)\right)
= \mathcal{F} _{S}^{(\gamma)} \,, \label{US-L2-S}
\end{align} 
\end{subequations} 
where the damping function $ \mathcal{D} _\gamma$ is defined as 
\begin{align}
\label{Dgamma}
\mathcal{D} _\gamma =  \gamma_1 ( 1 + \p_1 g_U) 
 + \tfrac{1}{2} \abs{ \gamma}   \,,
\end{align} 
the transport velocity $ \mathcal{V} _U$ is given in \eqref{V_U}, and since $\abs{\gamma} \ge 3$ the forcing functions in \eqref{US-L2} are given by 
\begin{subequations} 
\label{U-gamma-forcing}
\begin{align} 
\mathcal{F} _{ U_i}^{(\gamma)}
&= F_{U_i}^{(\gamma,U)} + F_{U_i}^{(\gamma-1,U)} + F_{U_i}^{(\gamma,S)} + F_{U_i}^{(\gamma-1,S)}    \,, \label{FU-sum}\\
F_{U_i}^{(\gamma,U)} 
&= -2 \beta_\tau \beta_1 \left( e^{\frac{s}{2}}  \Jcal \Ncal_j \guj \p_1 U_i  + e^{-\frac{s}{2}} \p^\gamma U_\nu \p_\nu U_i\right)  \notag\\
& \qquad
-   \sum_{\substack{ |\beta| = |\gamma|-1 \\ \beta \leq \gamma, \beta_1 = \gamma_1}} {\gamma \choose \beta} 
\p^{\gamma -\beta} g_U \p^\beta \p_1 U_i
-  \sum_{\substack{ |\beta| = |\gamma|-1 \\ \beta \leq \gamma}} {\gamma \choose \beta} 
\p^{\gamma -\beta} h_U^\nu \p^\beta \p_\nu U_i \,, \notag \\
&=: F_{U_i,(1)}^{(\gamma,U)} + F_{U_i,(2)}^{(\gamma,U)} + F_{U_i,(3)}^{(\gamma,U)}
\label{FU-gamma-U}  \\
F_{U_i}^{(\gamma-1,U)} 
&=  -   \sum_{\substack{1 \le |\beta| \le |\gamma|-2 \\ \beta\leq\gamma}} {\gamma \choose \beta} 
\left( \p^{\gamma -\beta} g_U \p^\beta \p_1 U_i + \p^{ \alpha - \beta } h_U^\nu \p^\beta \p_\nu U_i 
 \right) - 2 \beta_\tau \beta_1 e^{\frac{s}{2}} \comm{\p^\gamma}{\Jcal \Ncal_j} U_j    \p_1 U_i \notag\\
 &=: F_{U_i,(1)}^{(\gamma-1,U)} +F_{U_i,(2)}^{(\gamma-1,U)}  \,, \label{FU-gamma-1-U}\\
F_{U_i}^{(\gamma,S)} 
&=  -2 \beta_\tau \beta_3  e^{-\frac{s}{2}} \delta^{i\nu} \p_\nu S  \gs  
- 2 \beta_\tau \beta_3  \sum_{\substack{ |\beta| = |\gamma|-1 \\ \beta\leq \gamma}} {\gamma \choose \beta}  e^{-\frac{s}{2}} \delta^{i\nu} \p^{\gamma -\beta} S \p^\beta \p_\nu S   \notag \\
& \qquad
+ \beta_\tau \beta_3  e^{\frac{s}{2}} \Jcal \Ncal_i (1 +  \gamma_1) \p_1 Z \gs   
- 2 \beta_\tau \beta_3  \sum_{\substack{ |\beta| = |\gamma|-1 \\ \beta\leq \gamma,\beta_1 = \gamma_1}} {\gamma \choose \beta} e^{\frac{s}{2}}  \p^{\gamma -\beta}(S \Jcal \Ncal_i) \p^\beta \p_1 S 
\notag\\
&=: F_{U_i(1)}^{(\gamma,S)} +F_{U_i(2)}^{(\gamma,S)} +F_{U_i(3)}^{(\gamma,S)} +F_{U_i(4)}^{(\gamma,S)} 
 \,,  \label{FU-gamma-S}\\
F_{U_i}^{(\gamma-1,S)} 
&= -2 \beta_\tau \beta_3   \sum_{\substack{ 1 \le |\beta| \le  |\gamma|-2 \\ \beta\leq \gamma}} {\gamma \choose \beta} 
 \left( e^{\frac{s}{2}}  \p^{\gamma -\beta}(S \Jcal \Ncal_i) \p^\beta \p_1 S + e^{-\frac{s}{2}} \delta^{i\nu} \p^{\gamma -\beta} S \p^\beta \p_\nu S\right) 
  \notag \\
& \qquad
 -2 \beta_\tau \beta_3 e^{\frac{s}{2}} \comm{ \p^\gamma}{\Jcal \Ncal_i}S \, \p_1 S
\, ,  \label{FU-gamma-1-S}
\end{align} 
\end{subequations} 
and
\begin{subequations} 
\label{S-gamma-forcing}
\begin{align} 
\mathcal{F} _{ S}^{(\gamma)}
&= F_{S}^{(\gamma,S)} + F_{S}^{(\gamma,U)} + F_{S}^{(\gamma-1,S)}+ F_{S}^{(\gamma-1,U)}   \,,\\
F_{S}^{(\gamma,S)} 
&= - 2 \beta_\tau \beta_3 \left(e^{\frac s2} \p^\gamma S \Jcal \Ncal_j \p_1 U_j + e^{-\frac s2} \p^\gamma S \p_\nu U_\nu \right) \notag\\
&\qquad  - \sum_{\substack{ |\beta| = \abs{\gamma}-1 \\ \beta\leq\gamma, \beta_1 = \gamma_1}} {\gamma \choose \beta}  \p^{\gamma -\beta} g_U \p^\beta \p_1 S
-  \sum_{\substack{ |\beta| = |\gamma|-1 \\ \beta\leq\gamma}} {\gamma \choose \beta} 
\p^{\gamma -\beta} h_U^\nu \p^\beta \p_\nu S \,,
 \\
F_{S}^{(\gamma,U)} &=   
 - 2 \beta_\tau \beta_1 e^{-\frac{s}{2}} \p_\nu S \p^\gamma U_\nu   - 2 \beta_\tau \beta_3   \sum_{\substack{ |\beta| = |\gamma|-1 \\ \beta\leq\gamma}} 
{\gamma \choose \beta}   e^{-\frac{s}{2}}  \p^{\gamma -\beta} S \p^\beta \p_\nu U^\nu \notag \\
&  \qquad
+  \beta_\tau (\beta_1 + \beta_3   \gamma_1) e^{\frac{s}{2}} \Jcal \Ncal_j \p_1 Z \guj    - 2 \beta_\tau \beta_3  \sum_{\substack{ |\beta| = |\gamma|-1 \\ \beta\leq\gamma, \beta_1=\gamma_1}} 
{\gamma \choose \beta}   e^{\frac s2} \p^{\gamma -\beta} (S \Jcal \Ncal_j ) \p^\beta \p_1 U_j    \,,\\
F_{S}^{(\gamma-1,S)} &=  
-  \sum_{\substack{1\le |\beta| \le |\gamma|-2 \\ \beta\leq\gamma}} 
{\gamma \choose \beta} \left(  \p^{\gamma -\beta} g_U \p^\beta \p_1 S+  \p^{\gamma-\beta}h_U^\nu  \p^\beta \p_\nu S \right) - 2 \beta_\tau \beta_3 e^{\frac{s}{2}} \comm{ \p^\gamma}{\Jcal \nn_j} S \p_1 U_j   \notag \\
& \qquad 
 - 2 \beta_\tau \beta_3  \sum_{\substack{1\le |\beta| \le |\gamma|-2 \\ \beta\leq\gamma}} 
{\gamma \choose \beta} \left( e^{\frac{s}{2}} \p^{\gamma-\beta}(S \Jcal \Ncal_j ) \p^\beta \p_1 U_j  +  e^{-\frac{s}{2}}   \p^{\gamma -\beta} S \p^\beta \p_\nu U^\nu 
\right) \,,
 \\
 F_{S}^{(\gamma-1,U)} &= - 2 \beta_\tau \beta_1 e^{\frac{s}{2}} \comm{\p^\gamma}{\Jcal \nn_j} U_j \p_1 S 
     \,.
\end{align} 
\end{subequations} 
In \eqref{U-gamma-forcing} and \eqref{S-gamma-forcing} we have used the notation $\comm{a}{b}$ to denote the commutator $a b- b a$. Here we have also appealed to the fact that $f$ and $V$ are quadratic functions of $\check y$, whereas $\Jcal \Ncal$ is an affine function of $\check y$; therefore $\p^\gamma$ annihilates these terms.

\subsection{Forcing estimates}
In order to analyze \eqref{US-L2} we first estimate the forcing terms defined in \eqref{U-gamma-forcing} and \eqref{S-gamma-forcing}.
We shall sometimes denote a partial derivative $\p^\gamma$ with $\abs{\gamma}=k$ as $D^k$,  when there is no need to keep track of the binomial coefficients from the product rule.

\begin{lemma}\label{lem:forcing:1} Consider the forcing functions $\mathcal{F} _{ U^i}^{(\gamma)} $ and $\mathcal{F} _{S}^{(\gamma)}$ defined in
\eqref{U-gamma-forcing} and \eqref{S-gamma-forcing}, respectively.
Let $k \ge 18$, fix $0 < \delta\le \tfrac{1}{32}$, and define the parameter $\lambda$ from \eqref{eq:Ek:def} as $\lambda= \tfrac{\delta ^2}{12k^2} $. Then, for $ \eps $ taken sufficiently small we have
\begin{subequations} 
\begin{align} 
2 \sum_{\abs{\gamma}=k}
 \lambda^\modckg \int_{ \mathbb{R}^3  } \abs{ \mathcal{F} _{ U^i}^{(\gamma)} \,  \gui } & 
\le   (2+ 8\delta )  E_k^2  + e^{-s} M^{4k-1}\,,   \label{Hk_est_FU}\\
2 \sum_{\abs{\gamma}=k}
 \lambda^\modckg \int_{ \mathbb{R}^3  }\abs{ \mathcal{F} _{S}^{(\gamma)}\, \gs} &
 \le  (2+ 8\delta )  E_k^2  +   e^{-s} M^{4k-1} \,.  \label{Hk_est_FS} 
\end{align} 
\end{subequations} 
\end{lemma} 
\begin{proof}[Proof of Lemma~\ref{lem:forcing:1}]
We shall first prove \eqref{Hk_est_FU}, and to do so, we estimate each term in the  sum \eqref{FU-sum}.  
We first recall the decomposition of the forcing function $F_{U_i}^{(\gamma,U)}$ in \eqref{FU-gamma-U} as the sum 
$F_{U_i}^{(\gamma,U)} =F_{U_i,(1)}^{(\gamma,U)}  + F_{U_i,(2)}^{(\gamma,U)} + F_{U_i,(3)}^{(\gamma,U)}$, and we recall that by definition we have
\begin{align} 
U_i = U\cdot \Ncal \Ncal_i +  A_\nu \Tcal^\nu_i = \tfrac{1}{2} ( e^{-\frac{s}{2}} W + \kappa+ Z )\Ncal_i + A_\nu \Tcal^\nu_i \,.  \label{Udef00}
\end{align} 
From \eqref{e:bounds_on_garbage}, $\abs{\Jcal} \le 1+ \eps^ {\frac{3}{4}} $, and using \eqref{eq:beta:tau}
\begin{align} 
\beta_\tau \beta_1 \le (1+ \eps^{\frac 34})  \tfrac{1}{1+\alpha}  \le 1 \label{beta_bound}
\end{align} 
for $\eps$ taken sufficiently small.  Hence, for the first term in \eqref{FU-gamma-U}  we have that 
\begin{align}   
2 \sum_{\abs{\gamma}=k}
 \lambda^\modckg \int_{ \mathbb{R}^3  }\abs{ F _{U_i,(1)}^{(\gamma,U)}  \,  \gui }
&\le   4E_k^2 \left( (1+ \eps^ {\frac{3}{4}} ) e^{\frac{s}{2}} \snorm{\p_1 U}_{L^ \infty } + e^{-\frac{s}{2}} \|\check\nabla  U\|_{L^ \infty } \right)   \notag\\
&\le 2 E_k^2(1+ \eps^ {\frac{3}{4}} )   \Big( \snorm{\p_1 W}_{L^ \infty }  
+ e^{\frac{s}{2}} \snorm{\p_1 Z}_{L^ \infty } + 2e^{\frac{s}{2}} \snorm{\p_1 A}_{L^ \infty }
 +  e^{-s} \snorm{\check\nabla   W}_{L^ \infty }  \notag\\
 &\qquad \qquad  + e^{-\frac{s}{2}} \snorm{\check\nabla   Z}_{L^ \infty }+
  2e^{-\frac{s}{2}}\snorm{\check\nabla   A}_{L^ \infty } + e^{-s} \norm{Z}_{L^\infty} +  e^{-s} \norm{A}_{L^\infty}   \Big) \notag\\
& \le (2 + \eps^ {\frac{1}{2}})  E_k^2  \,, \label{FU1}
\end{align} 
where we have used \eqref{e:bounds_on_garbage} on the second inequality, and
 \eqref{eq:W_decay},  \eqref{eq:Z_bootstrap},  \eqref{eq:A_bootstrap} for the last inequality.

Next,  for the second term in \eqref{FU-gamma-U} we have
\begin{align*}   
2 \sum_{\abs{\gamma}=k}
 \lambda^\modckg \int_{ \mathbb{R}^3  }\abs{ F _{U_i,(2)}^{(\gamma,U)}  \,  \gui  }
&\le  \sum_{\abs{\gamma}=k}
 \int_{ \mathbb{R}^3  }    2\lambda^\modckg \abs{ \gui }\sum_{\substack{ |\beta| = |\gamma|-1 \\ \beta\leq\gamma, \beta_1 = \gamma_1}} {\gamma \choose \beta} 
\abs{\p^{\gamma -\beta} g_U} \ \abs{ \p^\beta \p_1 U^i } \\
&\le  \sum_{\abs{\gamma}=k}
  2 \abs{\ckg} \lambda^{\frac{\modckg}{2} + \frac 12} \snorm{ \gui } _{L^2} \snorm{\check\nabla g_U}_{L^ \infty } 
    \sum_{\substack{ |\beta| = |\gamma|-1 \\ \beta \leq \gamma, \beta_1 = \gamma_1}}  \lambda^{\frac{\abs{\check \beta}}{2}}
 \snorm{ \p^{\beta} \p_1 U } _{L^2} \,,
\end{align*} 
where we have used that $\modckg - \tfrac{1}{2} \abs{\check \beta} =  \tfrac{1}{2} ( \modckg +1)$. By Young's inequality, for $ \delta >0$,
\begin{align*}   
2 \sum_{\abs{\gamma}=k}
 \lambda^\modckg \int_{ \mathbb{R}^3  }\abs{ F _{U_i,(2)}^{(\gamma,U)}  \,  \gui  } 
 \le
 \sum_{\abs{\gamma}=k} \Bigg(
  \tfrac{4 \abs{\ckg}^2}{\delta}   \snorm{\check\nabla g_U}_{L^ \infty }^2  \lambda^{\modckg +1} \snorm{ \gui }^2 _{L^2} +
    \sum_{\substack{ |\beta| = |\gamma|-1 \\ \beta<\gamma, \beta_1 = \gamma_1}}  \delta  \lambda^{\abs{\check \beta}}
 \snorm{ \p^{\beta}\p_1 U }^2_{L^2} \Bigg)\,.
\end{align*} 
Note that for each $\gamma$ with $\abs{\gamma}=k$, and for each $\beta$ with $\abs{\beta} = k-1$ and $\beta_1 = \gamma_1$, the term $\lambda^{\abs{\check \beta}}
 \snorm{ \p^{\beta+e_1} U }^2_{L^2} $ defines a different summand  of $E_k^2$. Moreover, from   the definition \eqref{eq:gA}, the bounds \eqref{eq:W_decay} and \eqref{e:G_ZA_estimates}  we obtain that
$\snorm{\check\nabla g_U}_{L^ \infty }\le 1$.\footnote{While here for simplicity we appeal to second bound in \eqref{e:G_ZA_estimates}, we note that this bound just directly follows from the definitions \eqref{eq:gA} and \eqref{def_V}, and the bootstrap assumptions \eqref{eq:speed:bound}, \eqref{eq:acceleration:bound}, \eqref{e:space_time_conv}, and \eqref{eq:Z_bootstrap}. In particular, none of these bounds rely on Proposition~\ref{cor:L2}, which is proven in this section. The same comment applies for the bound $\snorm{\check\nabla h_U}_{L^ \infty }\les \eps$.} Hence, 
\begin{align}   
2 \sum_{\abs{\gamma}=k}
 \lambda^\modckg \int_{ \mathbb{R}^3  }\abs{F_{U_i,(2)}^{(\gamma,U)}  \,  \gui  } 
 \le
(\lambda \tfrac{4 k^2}{\delta}  + \delta )E_k^2 \,. \label{FU2}
\end{align} 

Similarly, from \eqref{e:h_estimates} (or alternatively, the definition \eqref{eq:hA} and the bootstrap assumptions \eqref{mod-boot}--\eqref{eq:A_bootstrap}), we have $\snorm{ \nabla h_U}_{L^ \infty }\les \eps$; hence,  it immediately follows
that for $\eps$ taken  sufficiently small the contribution from the third term in \eqref{FU-gamma-U} is estimated as
\begin{align}   
2 \sum_{\abs{\gamma}=k}
 \lambda^\modckg \int_{ \mathbb{R}^3  }\abs{F _{U_i,(3)}^{(\gamma,U)}  \,  \gui  } 
 \le \eps^ {\frac{1}{2}}  E_k^2 \,. \label{FU3}
\end{align} 
Combining \eqref{FU1}--\eqref{FU3}, and using the definition of $\lambda$ in the statement of Lemma~\ref{lem:forcing:1}, we obtain
\begin{align} 
2 \sum_{\abs{\gamma}=k}
 \lambda^\modckg \int_{ \mathbb{R}^3  }\abs{ F_{U_i}^{(\gamma,U)}  \,  \gui  } 
 \le 2 \left(1+    \delta +  \eps^ {\frac{1}{2}} \right)  E_k^2 \, , \label{FU4}
\end{align} 
where $\delta$ is a small universal constant. We emphasize that our choice of $\lambda$ only enters the proof in the transition from \eqref{FU2} to \eqref{FU4}.

We now estimate the next forcing term $F_{U_i}^{(\gamma-1,U)}$ in \eqref{FU-gamma-1-U} which we have decomposed as $F_{U_i}^{(\gamma-1,U)} = F_{U_i,(1)}^{(\gamma-1,U)} +F_{U_i,(2)}^{(\gamma-1,U)}$.  Our goal is to split off the $A$ from the $W$ and $Z$ contributions to these terms, since the bootstrap assumption  for $A$ in \eqref{eq:A_bootstrap} does not include bounds on the full Hessian $\nabla^2 A$.  Using \eqref{Udef00}   we write   $F_{U_i,(1)}^{(\gamma-1,U)}$ as 
\begin{align} 
F_{U_i,(1)}^{(\gamma-1,U)}
&=  \mathcal{I}_1 + \mathcal{I}_2+ \mathcal{I} _3\,, 
\label{eq:basic:shitstorm}
\end{align}
where 
\begin{align*}
\mathcal{I} _1 
&=
 -   \sum_{\substack{1 \le |\beta| \le |\gamma|-2 \\ \beta \leq \gamma}} {\gamma \choose \beta} 
\p^{\gamma -\beta} g_U \p^\beta \p_1(U \cdot \Ncal \, \Ncal_i )
 \,, \ \ \ \\
 \mathcal{I} _2 
&= -   \sum_{\substack{1 \le |\beta| \le |\gamma|-2 \\ \beta\leq \gamma}} {\gamma \choose \beta} 
\p^{\gamma -\beta} g_U \p^\beta (\p_1 A_\nu \Tcal^\nu_i )
\,, \\
 \mathcal{I} _3
 &=
 -   \sum_{\substack{1 \le |\beta| \le |\gamma|-2 \\ \beta\leq \gamma}} {\gamma \choose \beta} 
 \p^{ \alpha - \beta } h_U^\nu \p^\beta \p_\nu U_i
 \,.
\end{align*} 
 
First, for the ${\mathcal I}_1$ term in \eqref{eq:basic:shitstorm}, by  Lemma \ref{lem:tailored:interpolation}  for $q=\frac{6(2k-3)}{2k-1}$, we have that
\begin{align} 
2 \sum_{\abs{\gamma}=k}
 \lambda^\modckg \int_{ \mathbb{R}^3  }\abs{ \mathcal{I} _1 \,  \gui  } 
 &
  \les \snorm{D^kg_U }_{L^2}^\acal        \snorm{D^k U}_{L^2}^\bcal 
   \snorm{ D^2 g_U}_{L^q}^{1- \acal } \snorm{ D^2 ( U\cdot \Ncal \Ncal)}_{L^q}^{1- \bcal }  \snorm{D^k U}_{L^2} \,.
   \label{decomp1}
 \end{align} 
where $ \acal $ and $ \bcal $ are given by \eqref{alpha_beta}, and they obey $\acal + \bcal = 1 - \frac{1}{2k-4}$. Note by \eqref{eq:gA} that $g_U$ does not include any $A$ term. Thus, using the bootstrap bounds \eqref{mod-boot}--\eqref{eq:Z_bootstrap}, or alternatively by appealing directly to \eqref{eq:W_decay}, \eqref{e:bounds_on_garbage} and the last bound in \eqref{e:G_ZA_estimates}, and the definition of $\XXX(s)$ in \eqref{eq:support} we 
deduce that 
\begin{align}
  \snorm{ D^2 g_U}_{L^q(\XXX(s))} \les M \snorm{\eta^{-\frac 16}}_{L^q(\XXX(s))}  + M e^{-\frac s2} \abs{\XXX(s)}^{\frac 1q} \les M 
  \label{eq:D2:gU}
\end{align}
since $q \in [\frac{11}{2},6)$ for $k\geq 18$.  Similarly, from the first four bounds in \eqref{eq:US_est} (bounds which do not rely on any $A$ estimates) and from \eqref{e:bounds_on_garbage} (which only uses \eqref{eq:speed:bound} and  \eqref{e:space_time_conv}), we deduce that 
\begin{align}
  \snorm{ D^2((U \cdot \Ncal) \Ncal)}_{L^q(\XXX(s))} \les M e^{-\frac s2} \snorm{\eta^{-\frac 16}}_{L^q(\XXX(s))}  + M e^{-s} \abs{\XXX(s)}^{\frac 1q} \les M  e^{-\frac s2}  \, .
  \label{eq:D2:UdotN}
\end{align}
Moreover, by \eqref{normal}, \eqref{eq:gA}, the fact that $D^k$ annihilates $\dot f$ and $\Jcal \Ncal \cdot V$, we have that 
$$D^kg_U 
= \beta_1 \beta_\tau e^{\frac s2} D^k \left( \Jcal (\kappa + e^{-\frac s2} W +  Z)  \right)    
= 2 \beta_1 \beta_\tau D^k \left( \Jcal  U\cdot \Ncal \right)    
= 2 \beta_\tau \beta_1 e^{\frac{s}{2}}D^k (U_1 - e^{-\frac{s}{2}} \phi_{\nu\gamma} y_\gamma U_\nu) \, ,
$$
so that from \eqref{eq:speed:bound} and \eqref{eq:support} we obtain
\begin{align}
  \snorm{ D^k g_U}_{L^2} \les e^{\frac s2} \norm{U}_{\dot{H}^k}  \, .
  \label{eq:Dk:gU}
\end{align}
By combining \eqref{eq:D2:gU}--\eqref{eq:Dk:gU} we  obtain that the right side of \eqref{decomp1} is bounded from above as
\begin{align*} 
 & \snorm{D^kg_U }_{L^2}^\acal  \snorm{D^k U}_{L^2}^\bcal
   \snorm{ D^2 g_U}_{L^q}^{1- \acal } \snorm{ D^2 (U \cdot \Ncal \Ncal)}_{L^q}^{1- \bcal }  \snorm{D^k U}_{L^2}\\
&\qquad
  \les (e^{\frac{  s}{2}} \snorm{U}_{\dot{H}^k})^\acal        \snorm{U}_{\dot{H}^k}^{ \bcal }
   M^{1- \acal } (M e^{-\frac s2})^{1- \bcal }    \snorm{U}_{\dot{H}^k} \\
 &\qquad
  \les M^{2-\acal - \bcal} e^{\frac{(\acal+\bcal-1) s}{2}} \snorm{U}_{\dot{H}^k}^{1+\acal +\bcal} \,.
\end{align*} 
Recalling from Lemma \ref{lem:tailored:interpolation}  that $1- \acal - \bcal =  \frac{1}{2k-4} \in (0,1)$, the and using the norm equivalence  \eqref{norm_compare}, by 
Young's inequality with a small parameter $ \delta >0$, we have that the left side of \eqref{decomp1} is bounded as
\begin{align} 
2 \sum_{\abs{\gamma}=k}
 \lambda^\modckg \int_{ \mathbb{R}^3  }\abs{ \mathcal{I} _1 \,  \gui  } 
&\le C_k M^{2-\acal - \bcal}
 e^{\frac{(\acal+\bcal-1) s}{2}} \lambda ^\frac{-k(1+\acal +\bcal )}{2}    E_k^{1+\acal +\bcal }     \notag\\
&\le  \delta  E_k^2  + e^{-s} M^{4k-3} \,.
 \label{decomp2}
\end{align} 
In the last inequality we have used that by definition $\lambda = \lambda(k,\delta)$, $\delta \in (0,\frac{1}{32}]$ is a fixed universal constant, and $C_k$ is a constant that only depends on $k$; thus, we may use a power of $M$ (which is taken to be sufficiently large) to absorb all the $k$ and $\delta$ dependent constants.

Next, we estimate the   $\mathcal{I} _2$ term in \eqref{eq:basic:shitstorm}.  First, we note that by \eqref{eq:special1} we have
\begin{align*} 
\norm{ \mathcal{I} _2}_{L^2} 
&\les \sum_{j=1}^{k-2} \norm{D^{k-1-j} Dg_U}_{L^ {\frac{2(k-1)}{k-1-j}} }  \norm{ D^j (\p_1 A_\nu \Tcal^\nu )}_{L^ {\frac{2(k-1)}{j}} } \\
& \les \sum_{j=1}^{k-2} \norm{g_U}_{\dot{H}^{k} }^{\frac{k-1-j}{k-1}}  \norm{D g_U}_{ L^ \infty }^{\frac{j}{k-1}}
\norm{ \p_1 A_\nu \Tcal^\nu }_{ \dot{H}^{k-1} }^{\frac{j}{k-1}} \norm{ \p_1 A_\nu \Tcal^\nu  }_{L^\infty}^{ \frac{k-1-j}{k-1}}  \, .
\end{align*} 
Then, by appealing to \eqref{eq:gA}, \eqref{eq:W_decay}, \eqref{eq:A_bootstrap}, \eqref{e:bounds_on_garbage}, \eqref{e:G_ZA_estimates}, \eqref{norm_compare}, \eqref{eq:Dk:gU}, and \eqref{eq:Moser:inequality}, we deduce  
\begin{align*} 
\norm{ \mathcal{I} _2}_{L^2} 
& \les \sum_{j=1}^{k-2} \left(e^{\frac s2} \norm{U}_{\dot{H}^{k} }\right)^{\frac{k-1-j}{k-1}} \left(\norm{A}_{\dot{H}^k} + M \eps e^{-\frac{k+2}{2}s} \right)^{\frac{j}{k-1}} \left(M e^{-\frac{3s}{2}}\right)^{ \frac{k-1-j}{k-1}} \notag\\
& \les \sum_{j=1}^{k-2} \left( \lambda^{-\frac k2} E_k \right)^{\frac{k-1-j}{k-1}} \left( \lambda^{-\frac k2} E_k + M \eps e^{-\frac{k+2}{2}s} \right)^{\frac{j}{k-1}} \left(M e^{-s}\right)^{ \frac{k-1-j}{k-1}} \notag\\
&\les (M \eps)^{\frac{1}{k-1}} \lambda^{-\frac k2} E_k + M e^{-s}
\end{align*} 
since $\norm{D g_U}_{L^\infty} \les 1$. By taking $\eps$ sufficiently small, in terms of $M$, $\lambda = \lambda(k,\delta)$, and $k$, we obtain from the above estimate that 
\begin{align} 
2 \sum_{\abs{\gamma}=k}
 \lambda^\modckg \int_{ \mathbb{R}^3  }\abs{ \mathcal{I} _2 \,  \gui  } 
 &
  \le \eps^ {\frac{1}{k}} E_k^2 +  e^{-s} 
   \label{decomp2.5}
 \end{align} 
 for all $s\geq -\log \eps$.
 
At last, we estimate the ${\mathcal I}_3$ term in \eqref{eq:basic:shitstorm}, which is estimated similarly to the ${\mathcal I}_2$ term as
\begin{align*}
\norm{{\mathcal I}_3}_{L^2}
&\les \sum_{j=1}^{k-2} \norm{h_U}_{\dot{H}^{k} }^{\frac{k-1-j}{k-1}}  \norm{D h_U}_{ L^ \infty }^{\frac{j}{k-1}}
\norm{ \p_\nu U_i}_{ \dot{H}^{k-1} }^{\frac{j}{k-1}} \norm{ \p_\nu U_i}_{L^\infty}^{ \frac{k-1-j}{k-1}}  \, .
\end{align*}
From \eqref{eq:hA}, \eqref{eq:W_decay},  \eqref{eq:Z_bootstrap}, \eqref{eq:A_bootstrap}, \eqref{e:bounds_on_garbage}, and the Moser inequality \eqref{eq:Moser:inequality}, we have  
\begin{align*}
\norm{ h_U}_{\dot H^k} 
\les e^{-\frac{s}{2}} \norm{\Ncal U\cdot \Ncal}_{\dot H^k} + \kappa e^{-\frac s2} \norm{A_\gamma \Tcal^\gamma}_{\dot{H}^k}  
\les M e^{-\frac{s}{2}} \norm{ U}_{\dot H^k} + M \eps e^{-\frac{k+1}{2} s} \,.
\end{align*}
On the other hand, by \eqref{e:h_estimates} we have $\norm{D h_U}_{L^\infty} \les e^{-s}$, while from  \eqref{eq:W_decay}, \eqref{eq:Z_bootstrap}, \eqref{eq:A_bootstrap},  and \eqref{Udef00} we obtain $\norm{\check\nabla  U}_{L^\infty} \les e^{-\frac s2}$. Combining the above three estimates, we deduce that 
\begin{align*}
\norm{{\mathcal I}_3}_{L^2}
&\les \sum_{j=1}^{k-2} \left( M e^{-\frac s2} \norm{U}_{\dot H^k} + e^{-2s} \right)^{\frac{k-1-j}{k-1}}  e^{-\frac{j}{k-1}s}
\norm{ U}_{ \dot{H}^{k} }^{\frac{j}{k-1}} e^{- \frac{k-1-j}{2(k-1)}s} 
\les Me^{-s} \norm{ U}_{ \dot{H}^{k} } + e^{-s}
\end{align*}
from which we deduce
\begin{align} 
2 \sum_{\abs{\gamma}=k}
 \lambda^\modckg \int_{ \mathbb{R}^3  }\abs{ \mathcal{I} _3 \,  \gui  } 
 &
  \le \eps^{\frac 12} E_k^2 +  e^{-s} 
   \label{decomp3}
 \end{align} 
upon taking $M$ to be sufficiently large in terms of $k$, and $\eps$ sufficiently large in terms of $M$.
Combining \eqref{decomp2}, \eqref{decomp2.5}, and \eqref{decomp3}, we have thus shown that
\begin{align} 
2 \sum_{\abs{\gamma}=k}
 \lambda^\modckg \int_{ \mathbb{R}^3  }\abs{ F_{U_i,(1)}^{(\gamma-1,U)}  \,  \gui  } 
 \le ( \delta + \eps^{\frac 1k} + \eps^ {\frac{1}{2}} ) E_k^2 + M^{4k-2} e^{-s}  \,.
\label{FU5}
\end{align}

To estimate the integral with the forcing function $ F_{U_i,(2)}^{(\gamma-1,U)} $ defined in \eqref{FU-gamma-1-U}, we first note that due to the Leibniz rule and the fact that $D^2 (\Jcal \Ncal) = 0$, we have
$$  \comm{\p^\gamma}{\Jcal \Ncal_j} U_j
= \sum_{ \abs{\beta} = k-1, \beta\leq \gamma} {\gamma \choose \beta}    \p^{\gamma-\beta} (\Jcal \Ncal_j) \p^\beta U_j  
$$
for $ \modckg =k$.   Hence, by \eqref{e:bounds_on_garbage} we obtain
\begin{align*} 
2 \sum_{\abs{\gamma}=k}
 \lambda^\modckg \int_{ \mathbb{R}^3  }\abs{ F_{U_i,(2)}^{(\gamma-1,U)}  \,  \gui  } 
 \les \eps \snorm{\p_1 U}_{L^ \infty } \snorm{D^{k-1} U}_{L^2}\snorm{D^{k} U}_{L^2} \les \eps  e^{-\frac{s}{2}}  \snorm{D^{k-1} U}_{L^2}\snorm{D^{k} U}_{L^2} \,,
\end{align*} 
where we have used \eqref{Udef00},  together with the bounds
 \eqref{eq:W_decay}, \eqref{eq:Z_bootstrap},  \eqref{eq:A_bootstrap}.
By \eqref{eq:special3} applied with $\varphi = DU$, which thus obeys $\norm{\varphi}_{L^\infty} \les e^{-\frac s2}$, and Young's inequality with $ \delta >0$, 
\begin{align} 
2 \sum_{\abs{\gamma}=k}
 \lambda^\modckg \int_{ \mathbb{R}^3  }\abs{ F_{U_i,(2)}^{(\gamma-1,U)}  \,  \gui  } 
\leq \eps^{\frac 12} ( e^{-\frac{s}{2}} )^{1 + \frac{2}{2k-5}} \snorm{ U}_{\dot H^k }^{2 - \frac{2}{2k-5}}  
 \leq  \delta E_k^2 + e^{-s}
  \,,
\label{FU6}
\end{align} 
where we have used $\eps$ to absorb all $k$ and $\delta$ dependent constants. 
Hence, 
\eqref{FU5} and \eqref{FU6} together yield
\begin{align} 
2 \sum_{\abs{\gamma}=k}
 \lambda^\modckg \int_{ \mathbb{R}^3  }\abs{ F_{U_i}^{(\gamma-1,U)}  \,  \gui  } 
 \le  2(\delta + \eps^ {\frac{1}{k}} ) E_k^2 + 2 e^{-s} M^{4k-2}\,.
\label{FU7}
\end{align} 

Now, we turn to the forcing function  $F_{U_i}^{(\gamma,S)} $ in \eqref{FU-gamma-S} which we have decomposed as
 $F_{U_i}^{(\gamma,S)} =F_{U_i(1)}^{(\gamma,S)} +F_{U_i(2)}^{(\gamma,S)} +F_{U_i(3)}^{(\gamma,S)} +F_{U_i(4)}^{(\gamma,S)} $, and bound each of these contributions individually. 
We first note that the bounds for the integrals with $F_{U_i(1)}^{(\gamma,S)}$ and $F_{U_i(3)}^{(\gamma,S)}$ are obtained directly from the $\nabla \check S$ estimate in \eqref{eq:US_est} and the $\p_1 Z$ estimate in \eqref{eq:Z_bootstrap}, yielding
 \begin{align}   
2 \sum_{\abs{\gamma}=k}
 \lambda^\modckg \int_{ \mathbb{R}^3  }\abs{ \left( F _{U_i,(1)}^{(\gamma,S)} + F_{U_i,(3)}^{(\gamma,S)} \right) \,  \gui }
 \le M^2 e^{-s}  E_k^2 \le \eps^{\frac 12} E_k^2 \,.  \label{FU8}
 \end{align} 
 The bound for the integral with $F_{U_i(2)}^{(\gamma,S)}$  is obtained in the same way as the bound for 
 $\mathcal{F} _{U_i,(3)}^{(\gamma,U)}$ in \eqref{FU3}. Indeed, as far as our bounds are concerned $\p^\beta \p_\nu S$ behaves in the same exact way as $\p^\beta \p_\nu S$, and by \eqref{eq:US_est} we have $\norm{\nabla S}_{L^\infty} \les M \eps^{\frac 12}$, which is similar to the bound   $\norm{\nabla h_U} \les \eps$ which was used in \eqref{FU3}. In order to avoid redundancy we omit these details and simply claim
 \begin{align}   
2 \sum_{\abs{\gamma}=k}
 \lambda^\modckg \int_{ \mathbb{R}^3  }\abs{ F _{U_i,(2)}^{(\gamma,S)}  \,  \gui  } 
 \le \eps^ {\frac{1}{4}}  E_k^2 \,. \label{FU9}
\end{align} 
Similarly, the bound for the integral with $F_{U_i(4)}^{(\gamma,S)}$  is obtained in the same way as the bound for $\mathcal{F} _{U_i,(2)}^{(\gamma,U)}$ in \eqref{FU2}: $\p^\beta \p_1 S$ plays the same role as $\p^\beta \p_1 U$, whereas by \eqref{eq:US_est} we have $\norm{\check \nabla S}_{L^\infty} \les \eps^{\frac 12}$, which is better than the bound $\norm{\check \nabla g_U}_{L^\infty} \leq 1$ that was used in  \eqref{FU2}, reason for which we do not even need to appeal to our specific $\lambda$ choice for this estimate. In order to avoid redundancy we omit these details and state the resulting bound
 \begin{align} 
2 \sum_{\abs{\gamma}=k}
 \lambda^\modckg \int_{ \mathbb{R}^3  }\abs{ F_{U_i,(4)}^{(\gamma,S)}  \,  \gui  } 
 \le \eps^{\frac 14} E_k^2 \,.
\label{FU10}
\end{align} 
The estimates \eqref{FU8}--\eqref{FU10} together yield
 \begin{align} 
2 \sum_{\abs{\gamma}=k}
 \lambda^\modckg \int_{ \mathbb{R}^3  }\abs{ F_{U_i}^{(\gamma,S)}  \,  \gui  } 
 \le 3 \eps^ {\frac{1}{4}} E_k^2   \,.
\label{FU11}
\end{align} 

The last forcing term in the $U$ equation is $ F_{U_i}^{(\gamma-1,S)} $ defined by \eqref{FU-gamma-1-S}. We first note that the commutator term may be bounded identically to the commutator term in $F_{U_i(2)}^{(\gamma-1,U)}$ since $S \p_1 S$ may be used interchangeably with $U_j \p_1 U_i$ in terms of our estimates. Similarly, the summation term in $ F_{U_i}^{(\gamma-1,S)} $ is treated in the same way as $F_{U_i(1)}^{(\gamma-1,U)}$ for the same reasons which we invoked earlier in the $ F_{U_i}^{(\gamma,S)} $  discussion.  In summary, the integral with the forcing term  $ F_{U_i}^{(\gamma-1,S)} $
is estimated in the identical manner as \eqref{FU7}, and we obtain that
\begin{align} 
2 \sum_{\abs{\gamma}=k}
 \lambda^\modckg \int_{ \mathbb{R}^3  }\abs{ F_{U_i}^{(\gamma-1,S)}  \,  \gui  } 
 \le  2(\delta + \eps^ {\frac{1}{k}} ) E_k^2 + 2 e^{-s} M^{4k-2} \,.
\label{FU12}
\end{align} 

Combining the estimates \eqref{FU4}, \eqref{FU7}, \eqref{FU11}, and \eqref{FU12},  and choosing  $\eps$ to be sufficiently small in terms of $k$ and $\delta$, we obtain
we obtain that
\begin{align*} 
2 \sum_{\abs{\gamma}=k}
 \lambda^\modckg \int_{ \mathbb{R}^3  } \abs{F_{ U^i}^{(\gamma)} \,  \gui } & 
\le (2+ 8\delta)  E_k^2  + e^{-s} M^{4k-1} \,,
\end{align*}
which proves the inequality \eqref{Hk_est_FU}.   

Upon comparing the $S$-forcing terms in \eqref{S-gamma-forcing} with the $U$-forcing terms in \eqref{U-gamma-forcing}, we observe that they only differ by exchanging the letters $U$ and $S$ in several places; hence,  inequality \eqref{Hk_est_FS} is proved mutatis mutandi to \eqref{Hk_est_FU}. To avoid redundancy we omit these details.
\end{proof} 

\subsection{The $\dot H^k$ energy estimate}

We now turn to the main energy estimate.
\begin{proposition}[$\dot H^k$  estimate for $U$ and $S$]\label{prop:L2}
For any integer $k$ satisfying
\begin{align}
 k \geq 18 \,,
 \label{eq:k-cond}
\end{align}
with $\delta$ and $\lambda= \lambda(k,\delta)$ as specified in Lemma~\ref{lem:forcing:1}, we have the estimate
\begin{align}
&E_k^2(s)
  \leq  e^{-2(s-s_0)} E_k^2(s_0) + 2 e^{-s} M^{4k-1} \left(1 - e^{-(s-s_0)}\right) 
 \label{eq:L2}
\end{align}
for all $s\geq  s_0 \geq -\log \eps$.
\end{proposition}

\begin{proof}[Proof of Proposition~\ref{prop:L2}]
We fix a   multi-index $\gamma \in {\mathbb N}_0^3$ with $\abs{\gamma} = k$,  and consider the sum of the $L^2$ inner-product of \eqref{US-L2-U} with 
$\lambda^\modckg \p^\gamma U^i$ and the $L^2$ inner-product of  \eqref{US-L2-S} with $\lambda^\modckg  \p^\gamma S$.   
With the damping function $ \mathcal{D} _\gamma$ defined in \eqref{Dgamma} and the transport velocity
$ \mathcal{V} _U$ defined in \eqref{V_U}, using the fact that $\dot Q$ is skew-symmetric  we find that
\begin{align}
&\tfrac{d}{ds}  \int_{ \mathbb{R}^3  } \lambda^\modckg \left(   \abs{\p^\gamma U }^2    +   \abs{ \gs}^2  \right) 
+ \lambda^\modckg \int_{ \mathbb{R}^3  }  (2 \mathcal{D} _\gamma - \operatorname{div} \mathcal{V} _U) \,  \left( \abs{\p^\gamma U }^2 +
 \abs{ \gs}^2\right) \notag  \\
 & \qquad\qquad\qquad
 + 2 \beta_\tau \lambda^\modckg \int_{ \mathbb{R}^3  }  ( \beta_1 + \beta_3 + 2 \beta_3 \gamma_1)  \Jcal \p_1 W  \gs \p^\gamma U \cdot \Ncal \notag \\
& = 2 \lambda^\modckg  \int_{ \mathbb{R}^3  } (\mathcal{F} _{ U^i}^{(\gamma)} \,  \gui + \mathcal{F} _{S}^{(\gamma)}\, \gs) 
+ 4 \beta_\tau \beta_3\lambda^\modckg  \int_{ \mathbb{R}^3  } \left( e^{\frac{s}{2}}  \Jcal \Ncal_j   \p^\gamma U_j \p_1 S
+ e^{-\frac{s}{2}}  \p^\gamma U_\nu \p_\nu S  \right) \, \gs \,.
\label{energy0}
\end{align}
We note that the last integral on the right-hand side of the identity  \eqref{energy0} arises via integration by parts as follows:
\begin{align*} 
&4 \beta_\tau \beta_3  \int_{ \mathbb{R}^3  } \left( \Jcal \Ncal_i e^{\frac{s}{2}}  \p_1 (\gs) + e^{-\frac{s}{2}} \delta^{i\nu} \p_\nu ( \gs) \right) \, S \gui 
\\
&\qquad \qquad + 4 \beta_\tau \beta_3  \int_{ \mathbb{R}^3  } \left(  e^{\frac{s}{2}}\Jcal \Ncal_j  \p_1 (\p^\gamma U_j) + e^{-\frac{s}{2}}  \p_\nu (\p^\gamma U_\nu) \right) \, S \p^\gamma S
\\
&=  4 \beta_\tau \beta_3 \int_{ \mathbb{R}^3  } \left(e^{\frac{s}{2}}  \p_1 \left(  \Jcal \Ncal \cdot \p^\gamma U     \gs\right)  + e^{-\frac{s}{2}}  \p_\nu \left(\p^\gamma U_\nu \gs\right) \right) \, S  \\
&=  - 4 \beta_\tau \beta_3 \int_{ \mathbb{R}^3  } \left(e^{\frac{s}{2}}  \left(  \Jcal \Ncal \cdot \p^\gamma U     \gs\right)  \p_1 S + e^{-\frac{s}{2}}   \left(\p^\gamma U_\nu \gs\right) \p_\nu S \right)  \\
&=  - 4 \beta_\tau \beta_3 \int_{ \mathbb{R}^3  } \left(e^{\frac{s}{2}}   \Jcal \Ncal \cdot \p^\gamma U       \p_1 S + e^{-\frac{s}{2}}    \p^\gamma U_\nu  \p_\nu S \right) \gs \,.
\end{align*} 
The second and third integrals  on the left-hand side of the identity  \eqref{energy0} can be combined.   
Using \eqref{V_U}, given  the bounds \eqref{Jp1W},  \eqref{e:G_ZA_estimates} and  \eqref{e:h_estimates}, 
the second integral on the left-hand side of  \eqref{energy0} has an integrand with the  lower bound
\begin{align*} 
&
\left( 2 \mathcal{D} _\gamma - \operatorname{div} \mathcal{V} _U\right)\left( \abs{\gs}^2 +  \abs{\p^\gamma U}^2\right) \\
& \qquad\qquad
=   \left(\abs{\gamma}  - \tfrac{5}{2}+2 \gamma_1 +  (2 \gamma_1 -1)(\beta_\tau \beta_1 \Jcal \p_1 W + \p_1 G_U ) - \p_\nu h^\nu\right)
\left( \abs{\gs}^2 +  \abs{\p^\gamma U}^2\right) \\
&\qquad\qquad
 \ge    \left(\abs{\gamma}  - \tfrac{5}{2}+2 \gamma_1 -  \beta_\tau \beta_1 (2 \gamma_1 -1)_+ - \eps^ {\frac{1}{4}}  \right) \left( \abs{\gs}^2 +  \abs{\p^\gamma U}^2\right)\,,
\end{align*} 
while the third integral on the left-hand side of  \eqref{energy0} has an integrand with the lower bound 
\begin{align*} 
2\beta_\tau ( \beta_1 + \beta_3 + 2 \beta_3 \gamma_1)  \Jcal \p_1 W \ \gs \p^\gamma U \cdot \Ncal 
& \ge -
\beta_\tau ( \beta_1 + \beta_3 + 2 \beta_3 \gamma_1)  \Jcal \abs{\p_1 W} \left( \abs{\gs}^2 +  \abs{\p^\gamma U}^2\right) \\
& \ge -\beta_\tau (1 + 2 \beta_3 \gamma_1) \left( \abs{\gs}^2 +  \abs{\p^\gamma U}^2\right) \,,
\end{align*} 
where we have again used \eqref{Jp1W}, and the fact that by \eqref{eq:various:beta} we have $ \beta_1 + \beta_3 = 1$.
Hence these two integrals have the lower bound given by
$$
 \lambda^\modckg \int_{\mathbb{R}^3  } \left( \abs{\gamma} - \tfrac{5}{2}  + 2(1- \beta_\tau)\gamma_1 - \beta_\tau -\eps^ {\frac{1}{4}} \right)
\left( \abs{\gs}^2 +  \abs{\p^\gamma U}^2\right)  \,.
$$
Since by \eqref{eq:beta:tau},  $\abs{ \beta_\tau -1} \le \eps^ {\frac{1}{2}} $, it follows that for $\eps$ taken sufficiently small, 
by summing \eqref{energy0}  over all $\abs{\gamma}=k$, we obtain that
\begin{align} 
& \tfrac{d}{ds} E_k^2(s) + (k- \tfrac{15}{4}) E_k^2(s) \notag \\
& \le 
\sum_{\abs{\gamma}=k}\left(
 2\lambda^\modckg \int_{ \mathbb{R}^3  } (\mathcal{F} _{ U^i}^{(\gamma)} \,  \gui + \mathcal{F} _{S}^{(\gamma)}\, \gs) 
+ 4 \beta_\tau \beta_3 \lambda^\modckg \int_{ \mathbb{R}^3  } \left( e^{\frac{s}{2}}  \Jcal \Ncal_j   \p^\gamma U_j \p_1 S
+ e^{-\frac{s}{2}}  \p^\gamma U_\nu \p_\nu S  \right)\gs  \right)   \,. \label{energy1}
\end{align} 
Recalling that $S= \tfrac{1}{2} ( e^{-\frac{s}{2}} W +\kappa-Z)$,  that
$\abs{\Jcal} \le1 + M \eps$ from  \eqref{e:bounds_on_garbage},  and  that $\beta_\tau \beta_3 \le (1+ \eps^{\frac 14}) \left( \tfrac{\alpha }{1+\alpha}\right) \le 1$ for $\eps$ taken
sufficiently small,  we find that
\begin{align*} 
&4 \beta_\tau \beta_3\lambda^\modckg  \int_{ \mathbb{R}^3  } \left| e^{\frac{s}{2}}  \Jcal \Ncal_j   \p^\gamma U_j \p_1 S
+ e^{-\frac{s}{2}}  \p^\gamma U_\nu \p_\nu S  \right|  \abs{\gs} \notag\\
&\qquad \le 2 (1+M\eps) \lambda^{\modckg}  
\left ( \snorm{ \p_1 W}_{L^ \infty } + e^{\frac{s}{2}} \snorm{ \p_1 Z}_{L^ \infty } + e^{-s} \snorm{ \check\nabla   W}_{L^ \infty } 
+ e^{-\frac{s}{2}} \snorm{ \check\nabla   Z}_{L^ \infty } \right)
 \snorm{\p^\gamma U}_{L^2}  \snorm{\p^\gamma S}_{L^2}  \, .
\end{align*} 
Hence, using  \eqref{eq:W_decay},  \eqref{eq:Z_bootstrap}, and \eqref{eq:A_bootstrap} we obtain that the second term in \eqref{energy1} is estimated as
\begin{align*} 
4 \beta_\tau \beta_3 
\sum_{\abs{\gamma}=k} \lambda^\modckg  
\int_{ \mathbb{R}^3  } 
\left| e^{\frac{s}{2}}  \Jcal \Ncal_j   \p^\gamma U_j \p_1 S + e^{-\frac{s}{2}}  \p^\gamma U_\nu \p_\nu S  \right| 
\abs{\p^\gamma S}
\le ( 2+ \eps^ {\frac{1}{4}} ) E_k \,.
\end{align*} 
It follows from \eqref{energy1}, that 
\begin{align} 
& \tfrac{d}{ds} E_k^2(s) + (k-6) E_k^2(s) 
 \le 
2 \sum_{\abs{\gamma}=k}
 \lambda^\modckg \int_{ \mathbb{R}^3  } (\mathcal{F} _{ U^i}^{(\gamma)} \,  \gui + \mathcal{F} _{S}^{(\gamma)}\, \gs) \,.
  \label{energy2}
\end{align} 

By Lemma \ref{lem:forcing:1}, for $0< \delta \le \tfrac{1}{32} $,
\begin{align*} 
& \tfrac{d}{ds} E_k^2(s) + (k-6) E_k^2(s) 
 \le 
 2(2+ 8\delta)  E_k^2  + 2 e^{-s} M^{4k-1}  \,,
\end{align*} 
and hence, by  \eqref{eq:k-cond} we have that  
\begin{align*}
\tfrac{d}{ds} E_k^2 + 2 E_k^2  
&\leq  2 e^{-s} M^{4k-1}  \,,
\end{align*}
and so 
we obtain that 
\begin{align*}
E_k^2(s) \leq  e^{-2(s-s_0)} E_k^2(s_0) +  2 e^{-s} M^{4k-1} \left(1 - e^{-(s-s_0)}\right) \,,
\end{align*}
for all $s\geq s_0 \geq -\log \eps$. This concludes the proof of Proposition~\ref{prop:L2}.
\end{proof}

In conclusion of this section,  we mention that Proposition~\ref{prop:L2} applied with $s_0=-\log \eps$ yields the proof  of Proposition \ref{cor:L2}.

\begin{proof}[Proof of Proposition \ref{cor:L2}]

We recall the identities $D^k W = e^{\frac{s}{2}} D^k( U \cdot \Ncal +S)$,   $D^kZ = D^k( U \cdot \Ncal -S)$, and $A_\nu  = U \cdot \Tcal_\nu$ 
Therefore, by  \eqref{e:bounds_on_garbage}, \eqref{eq:special1}, using the  Poincar\'e inequality in the $\check y$ direction, and the fact that the diameter of $\XXX(s)$ in the $\check e$ directions is  $4 \eps^{\frac 16} e^{\frac{s}{2}}$, for any $\gamma$ with $\abs{\gamma} = k$, we obtain 
\begin{align*}
&\norm{e^{-\frac s2} \p^\gamma W -  \Ncal \cdot \p^\gamma U - \p^\gamma S}_{L^2}
+ \norm{\p^\gamma Z -  \Ncal \cdot \p^\gamma U + \p^\gamma S}_{L^2}
+\norm{\p^\gamma A_\nu -  \Tcal^\nu \cdot \p^\gamma U }_{L^2}
\notag\\
&\leq 2 \norm{\comm{\p^\gamma}{\Ncal}\cdot U}_{L^2} + \norm{\comm{\p^\gamma}{\Tcal^\nu}\cdot U}_{L^2} 
\notag\\
&\les  \sum_{j=1}^{k} \left(\norm{D^j \Ncal}_{L^\infty} + \norm{D^j \Tcal^\nu}_{L^\infty}\right) \norm{D^{k-j}U}_{L^2(\XXX(s))}
\notag\\
&\les \eps \sum_{j=1}^{k} e^{-\frac{j s}{2}} (4\eps^{\frac 16} e^{\frac s2})^{j} \norm{U}_{\dot{H}^k}
\notag\\
&\les \eps \norm{U}_{\dot{H}^k} \,.
\end{align*}
Summing over all $\gamma$ with $\abs{\gamma} =k$ relates the $\dot H^k$ norm of $W$, $Z$, $A$ with  the  $\dot H^k$ norm of $U$ and $S$.

The initial datum assumption \eqref{eq:data:Hk}  together with \eqref{norm_compare} thus imply that
\[
E_k^2(-\log \eps) \leq \eps\,.
\]
Thus, from \eqref{eq:L2} and \eqref{norm_compare}  we obtain
\[
\lambda^k \left( \norm{U(\cdot,s)}_{\dot{H}^k}^2 + \norm{S(\cdot,s)}_{\dot{H}^k}^2 \right) \leq E_k^2(s) \leq \eps^{-1} e^{-2s} + 2 e^{-s} M^{4k-1} (1- \eps^{-1} e^{-s})
\]
and  the inequalities \eqref{eq:AZ-L2}--\eqref{eq:W-L2} immediately follow by  combining the above inequalities.
\end{proof}

\section{Conclusion of the proof of the main theorems}
\label{sec:conclusion}

\subsection{The blow up time and location}
The blow up time $T_*$ is defined uniquely by the condition $\tau(T_*) = T_*$ which in view of \eqref{eq:modulation:IC} is equivalent to
\begin{equation}\label{good-vlad}
\int_{-\eps}^{T_*} (1-\dot \tau(t)) dt = \eps \, .
\end{equation} 
The estimate for $\dot \tau$  in \eqref{eq:acceleration:bound} shows that  for $\eps$ taken sufficiently small,
\begin{align} 
\abs{T_*} \leq 2M^2 \eps^{2} \,. \label{T*-bound}
\end{align} 

We also note here that the bootstrap assumption  \eqref{eq:acceleration:bound} and the definition of $T_*$ ensures that  $\tau(t) > t$ for all $t\in [-\eps, T_*)$.
Indeed, when $t = - \eps$, we have  that $\tau(-\eps) = 0 > -\eps$, and the function $t\mapsto \int_{-\eps}^t (1-\dot \tau) dt' - \eps = t -\tau(t)$ is strictly increasing.

The blow up location is determined by $\xi_* = \xi(T_*)$, which by \eqref{eq:modulation:IC} is the same as 
$$
\xi_* = \int_{-\eps}^{T_*} \dot \xi(t) dt \, .
$$
In view of \eqref{eq:acceleration:bound}, for $\eps$ small enough,
find  that 
\begin{align} 
\abs{\xi_*} \leq M \eps \,,\label{xi*-bound}
\end{align} 
so that the blow up location is $\OO(\eps)$ close to the origin. 

\subsection{H\"{o}lder bound for $w$}
\begin{proposition}\label{prop:Holder}  $w \in L^\infty([-\eps,T_*);C^{\sfrac 13})$.
\end{proposition} 
\begin{proof}[Proof of Proposition \ref{prop:Holder}]
We choose two points $y$ and $y'$ in $ \mathcal{X} $ such that $ y\neq y'$ and define $x$ and $x'$ via the relations
\begin{equation}\label{cov1}
y_1 = e^{\frac{3}{2}s }x_1 \,, \ \ \check y = e^{\frac{s}{2}} \check x\,, \ \ \text{ and } \ \ 
y_1' = e^{\frac{3}{2}s }x_1' \,, \ \ \check y' = e^{\frac{s}{2}} \check x'\,.
\end{equation} 
Using the identity \eqref{w_ansatz} and the change of variables \eqref{cov1}, we see that
\begin{align*}
\frac{\abs{w(x_1,\check x,t) - w(x_1', \check x',t)}}{ ( \abs{x_1-x_1'}^2 + \abs{\check x- \check x'}^2)^{\sfrac 16}} 
&= 
\frac{ e^{-\frac{s}{2}} \abs{W(y_1,\check y,s) - W(y_1',\check y',s)}}{ ( e^{-3s}\abs{y_1-y_1'}^2 + e^{-s}\abs{\check y- \check y'}^2)^{\sfrac 16}} 
= 
\frac{ \abs{W(y_1,\check y,s) - W(y_1',\check y',s)}}{ ( \abs{y_1-y_1'}^2 + e^{2s}\abs{\check y- \check y'}^2)^{\sfrac 16}}  \,,
\end{align*}
so that
\begin{align}
\frac{\abs{w(x_1,\check x,t) - w(x_1', \check x',t)}}{ ( \abs{x_1-x_1'}^2 + \abs{\check x- \check x'}^2)^{\sfrac 16}} 
& \le
\frac{ \abs{W(y_1,\check y,s) - W(y_1',\check y,s)}}{  \abs{y_1-y_1'}^{\sfrac 13}}
+\frac{ \abs{W(y_1',\check y,s) - W(y_1',\check y',s)}}{ e^{\frac{s}{3}}  \abs{\check y-\check y'}^{\sfrac 13}}   \,.
\label{holder1}
\end{align}
By the fundamental theorem of calculus and 
estimate \eqref{eq:W_decay}, we have that
\begin{align} 
\sup_{y_1 \neq y_1'}\frac{\abs{W(y_1,\check y ,s) - W(y_1',\check y,s)}}{\abs{y_1-y_1'}^{\sfrac 13}} 
\le 
\sup_{y_1 \neq y_1'}\frac{ \int_{y_1'}^{y_1} (1 + z_1^ {\sfrac{2}{3}}) ^{-1}  dz_1    }{\abs{y_1-y_1'}^{\sfrac 13}}  \le 3 \,, \label{holder2}
\end{align} 
and similarly for $\nu=2,3$,
\begin{align*} 
\sup_{\check y \neq \check y'}\frac{\abs{W(y_1', y_\nu ,s) - W(y_1', y_\nu',s)}}{e^{\frac{s}{3}} \abs{y_\nu-y_\nu'}^{\sfrac 13}} 
\le 
\sup_{y_1 \neq y_1'}\frac{  \int_{y_\nu'}^{y_\nu} \abs{\p_\nu W  } dz_\nu    }{e^{\frac{s}{3}} \abs{y_\nu-y_\nu'}^{\sfrac 13}} 
\le \sup_{y_1 \neq y_1'} e^{-\frac{s}{3}} \abs{y_\nu-y_\nu'}^{\sfrac 23}   \,,
\end{align*} 
where we have again used \eqref{eq:W_decay} which gives the bound $\abs{\p_\nu W} \le 1$.  Since both $y_\nu$ and $y_\nu '$ are in
$\mathcal{X}(s) $, by \eqref{e:space_time_conv}
$$
 \abs{ y_\nu - y_\nu '}^{\sfrac 23} \leq  \eps^{\frac{1}{10}}  e^{\frac{s}{3}} 
$$
and hence
\begin{align} 
\sup_{\check y \neq \check y'}\frac{\abs{W(y_1', y_\nu ,s) - W(y_1', y_\nu',s)}}{e^{\frac{s}{3}} \abs{y_\nu-y_\nu'}^{\sfrac 13}} 
\les 1 \,. \label{holder3}
\end{align} 
Combining \eqref{holder1}--\eqref{holder3}, we see that
\begin{align*} 
\sup_{ x \neq  x'}\frac{\abs{w(x_1,\check x,t) - w(x_1', \check x',t)}}{  \abs{x-x'}^{\sfrac 13}}  \les 1
\end{align*} 
where the implicit constant is universal, and is in particular independent of $s$ (and thus $t$).  This concludes the proof of the uniform-in-time
 H\"older $\sfrac 13$ estimate for $w$. 
 
The fact that $\tilde w$ has the same  H\"older $\sfrac 13$ regularity follows from the transformation $x$ to $\tilde x$ given in \eqref{x-sheep},  the
transformation from $w$ to $\tilde w$ given in \eqref{wza-sheep}, together with the bound for $\phi(t)$ given in \eqref{eq:speed:bound}.
\end{proof} 

\begin{remark} 
A straightforward computation shows that the $C^\alpha$ H\"older norms of $ w$, with $\alpha>\sfrac 13$,  blow up as $t\to T_*$ with a rate proportional to $(T_*-t)^{\sfrac{(1-3\alpha)}{2}}$. 
\end{remark} 

\subsection{Bounds for vorticity and sound speed}
\begin{corollary}[Bounds on  density and  vorticity]\label{cor:vort} 
The density remains bounded and non-trivial and satisfies
\begin{align} 
\snorm{ \tilde \rho^ \alpha ( \cdot , t) -  \alpha \tfrac{\kappa_0}{2} }_{ L^ \infty } \le \alpha  \eps^ {\frac{1}{8}}  \ \ \text{ for all } \ t\in[-\eps, T_*] \,. \label{density-bound-final}
\end{align} 
The vorticity has the bound
\begin{align} 
&\norm{\omega(\cdot, t) }_{L^\infty}  \le C_0 \kappa_0^{\frac 1\alpha}  \ \ \text{ for all } \ \ 
 \ t\in[-\eps, T_*] \,, \label{vorticity-bound-final}
\end{align} 
where $C_0$ is a universal constant. 
In addition, if we assume that 
\begin{align}
\abs{\omega(\cdot, -\eps) }\geq c_0 \qquad \mbox{on the set} \qquad B(0,2 \eps^{\sfrac 34}) \,, \label{vort-assume}
\end{align}
for some $c_0>0$, then at the location of the shock we have a nontrivial vorticity, and moreover
\begin{align}
\abs{\omega(\cdot, T_*)} \geq \tfrac{c_0}{C_0} \qquad \mbox{on the set} \qquad    B(0, \eps^{\sfrac 34})  \,. \label{vort-nonzero}
\end{align}
\end{corollary} 
\begin{proof}[Proof of Corollary \ref{cor:vort}] 
Using the the identities  \eqref{vort00}, \eqref{sv-sheep}, and  \eqref{svort-trammy},  we have that
$$
\Omega(y,s) = \tilde \zeta (\tilde x, t) =  \frac{\tilde \omega( \tilde x,t) }{ \tilde \rho(\tilde x, t) } \,,
$$
and hence
 from 
Proposition \ref{prop:vorticity}, it follows that
\begin{equation}\label{svort-final}
\norm{ \frac{\tilde \omega( \cdot ,t) }{ \tilde \rho(\cdot , t) }  }_{L^\infty} 
\leq 2 \,. 
\end{equation} 
for $t\in[-\eps, T_*)$.
Next, using the identities \eqref{eq:tilde:u:def}, \eqref{sigma-sheep}, and \eqref{S-trammy}, we find that
$$
 (\alpha S(y,s))^ {\frac{1}{\alpha }}  = (\alpha \tilde \sigma(x,t))^ {\frac{1}{\alpha }} =\tilde \rho( \tilde x,t)  \,,
$$
so that by Proposition \ref{prop:sound}, the estimate \eqref{density-bound-final} immediately follows.   Then, with the definition of the
transformation \eqref{eq:tilde:u:def}, we have that
\begin{align} \label{density-final}
 \left( \alpha (\tfrac{\kappa_0}{2}-  \eps^ {\sfrac{1}{8}} )\right)^ {\sfrac{1}{\alpha }}   \le  \rho ( \xcal  , t)  \le   \left( \alpha (\tfrac{\kappa_0}{2} +  \eps^ {\sfrac{1}{8}} )\right)^ {\sfrac{1}{\alpha }} \ \ \text{ for all } \ t\in[-\eps, T_*)\,,  \xcal  \in \mathbb{R}^3   \,. 
\end{align} 
The bounds \eqref{svort-final} and \eqref{density-final} together show that \eqref{vorticity-bound-final} holds for $\eps$ taken sufficiently small with respect to $\kappa_0$.

From  \eqref{a_ansatz} and \eqref{eq:UdotN:S}, $U= \tfrac 12 \left( \kappa + e^{-\frac s2} W + Z\right)\Ncal + A_\nu \Tcal^\nu$.
By \eqref{e:space_time_conv}, \eqref{eq:W_decay}, \eqref{eq:Z_bootstrap}, \eqref{eq:A_bootstrap}, and \eqref{kappa-kappa0},
\begin{align*} 
\snorm{U}_{L^ \infty } & \le \tfrac{1}{2} e^{-\frac{s}{2}} \snorm{W}_{L^ \infty }+  \tfrac{1}{2}  \snorm{Z}_{L^ \infty }+\snorm{A}_{L^ \infty }    + {\tfrac{1}{2}} \abs{ \kappa - \kappa _0} + \tfrac{1}{2} \abs{ \kappa _0} \\
& \le 2 \eps^ {\frac{1}{6}}  + \tfrac{1}{2} M^2 \eps^ {\frac{3}{2}} + \tfrac{3}{2} M \eps + \tfrac{\kappa_0}{2}  \le \tfrac{\kappa_0}{2}  + \eps^ {\frac{1}{8}} \,.
\end{align*} 

Let $X(\xcal, t)$  denote the Lagrangian flow of $u$:   $\p_t X( \xcal, t) =  u(  \xcal, X( \xcal, t))$ for $t \in (-\eps, T_*)$ such that $X( \xcal, -\eps) = \xcal$.   
Then, 
\begin{align} 
\frac{d}{dt} \p_{\xcal_j} X^i  & = (\p_{\xcal_k}  u^i \circ X) \p_{\xcal_j}X^k \label{eq:DX0}   \,.
\end{align}

We shall make use of the transformations \eqref{eq:tilde:x:def} and \eqref{eq:tilde:u:def} to relate $\p_{\tilde x}$ derivates of $\tilde u(\tilde x, t)$  with $\p_{\xcal}$
derivatives of $u(\xcal, t)$.   It is convenient to define the normal and tangent vectors that are function of $\xcal$, so we set
$$
\mathcal{N} (\check \xcal,t) = R(t) \Ncal (\check {\tilde x},t) \,, \ \ \mathcal{T} ^\nu (\check \xcal,t) = R(t) \Tcal ^\nu(\check {\tilde x},t) \,.
$$
We then have that $ u \cdot \mathcal{N} = \tilde u \cdot \Ncal$ and
\begin{align} 
\p_{\xcal_k} (u \cdot \mathcal{N} )  \mathcal{N} _k = \p_{\tilde x_j} (\tilde u \cdot \Ncal) R^T_{jk} R_{km} \Ncal_m = \p_{\tilde x_j} (\tilde u \cdot \Ncal)\Ncal_j \,.
 \label{early-buzz}
\end{align} 
By \eqref{early-buzz} and Lemma~\ref{lem:div} we obtain
\begin{align} 
\p_{\xcal_k} (u \cdot \mathcal{N} )  \mathcal{N} _k
&=\operatorname{div}_{\tilde x} \tu  -\Tcal^\nu_j  \p_{\tilde x_j} \tilde a_\nu
- (\tu\cdot \Ncal)  \p_{\tilde x_\mu}\Ncal_\mu 
-  \tilde a_\nu \p_{\tilde x_\mu} \Tcal^\nu_\mu \,. \label{buzz-buzz}
\end{align} 
We then write \eqref{eq:DX0} as
\begin{align*} 
\frac{d}{dt} \p_{\xcal_j} X^i 
& = \left(  \p_{\xcal_k}(u\cdot {\mathcal{N}})  {\mathcal{N}}_i + (u\cdot {\mathcal{N}}) \p_{\xcal_k}{\mathcal{N}}_i    
+ \p_{\xcal_k}a_\nu {\mathcal{T}}^\nu_i + a_\nu \p_{\xcal_k}{\mathcal{T}}^\nu_i    \right)  \circ X \ \p_{\xcal_j}X^k  \,,
\end{align*} 
and expand
$$
\p_{\xcal_j}X^k = \p_{\xcal_j}X^m \mathcal{N} _m \mathcal{N} _k + \p_{\xcal_j}X^m \mathcal{T} ^\mu_m \mathcal{T} ^\mu_k  \,.
$$
We then have that
\begin{subequations} 
\label{vomit}
\begin{align} 
& \frac{d}{dt}\left( \p_{\xcal_j} X^i  \ {\mathcal{T}}^\nu_i \circ X  \right)  \notag\\
&=
\left(   ({\mathcal{N}} \cdot \nabla_{\!\xcal})a_\nu  + (u \cdot {\mathcal{N}})({\mathcal{N}} \cdot \nabla_{\!\xcal}){\mathcal{N}}_i  {\mathcal{T}}^\nu_i 
+( \dot{\mathcal{T}}^\nu + \mathcal{T}^\nu,_\gamma u_\gamma  ) \cdot  {\mathcal{N}}  \right)  \circ X \ \left( \p_{\xcal_j}X^k \ {\mathcal{N}}_k \circ X\right) \notag \\
&  
+ \Bigl(   ({\mathcal{T}}^\mu  \cdot \nabla_{\!\xcal}) a_\nu  + (u \cdot {\mathcal{N}}) ({\mathcal{T}}^\mu \cdot \nabla_{\!\xcal}){\mathcal{N}}_i {\mathcal{T}}^\nu_i   
+ ( \dot{\mathcal{T}}^\nu  +  \mathcal{T} ^\nu,_\gamma u_\gamma )\cdot  {\mathcal{T}}^\mu \Bigr) 
 \circ X \ \left( \p_{\xcal_j}X^k \ {\mathcal{T}}^\mu_k \circ X\right) \,, \label{tan-comp} \\
& \frac{d}{dt}\left( \p_{\xcal_j} X^i  \ {\mathcal{N}}_i \circ X  \right) 
\notag\\
&=
\Bigl(   ({\mathcal{N}} \cdot \nabla_{\!\xcal})(u \cdot {\mathcal{N}}) + a_\nu({\mathcal{N}} \cdot \nabla_{\!\xcal}){\mathcal{T}}^\nu_i  {\mathcal{N}}_i   \Bigr)  \circ X \ \left( \p_{\xcal_j}X^k \ {\mathcal{N}}_k \circ X\right) \notag\\
&  
+ \left(   ({\mathcal{T}}^\mu  \cdot \nabla_{\!\xcal})(u \cdot {\mathcal{N}}) + a_\nu({\mathcal{T}}^\mu \cdot \nabla_{\!\xcal}){\mathcal{T}}^\nu_i {\mathcal{N}}_i+ (\dot{\mathcal{N}} + \mathcal{N},_\nu u_\nu ) \cdot  {\mathcal{T}}^\mu   \right) 
 \circ X \ \left( \p_{\xcal_j}X^k \ {\mathcal{T}}^\mu_k \circ X\right) \,.  \label{nor-comp}
\end{align}  
\end{subequations}

In Lagrangian coordinates, conservation
of mass can be written as $\rho \circ X = (\det \nabla_{\!\xcal} X) ^{-1} \rho_0$. Hence, by  \eqref{density-final}, there exists $C_X >0$ such that
\begin{align}
\tfrac{1}{C_X} \le  \det ( \nabla_{\!\xcal} X(\xcal, t) )  \le C_X \ \ \text{ for all } \ t\in[-\eps, T_*)\,, x \in \mathbb{R}^3   \,.  \label{light-buzz}
\end{align} 
The kinematic identity
\begin{align*} 
\frac{d}{dt}\det \nabla_{\!\xcal} X = \det \nabla_{\!\xcal} X  \operatorname{div_{\xcal}} u \circ X
\end{align*} 
leads to
\begin{align} 
 \det \nabla_{\!\xcal} X(\xcal, t) =\exp \int_{-\eps}^t (\operatorname{div}_{\xcal}  u \circ X)(\xcal, t') dt' \,, \label{nice-buzz}
\end{align} 
and hence from \eqref{damn-thats-nice}, \eqref{light-buzz} and \eqref{nice-buzz}, 
\begin{align} 
\tfrac{1}{C_X} \le \exp \int_{-\eps}^{T_*} (\operatorname{div}_{\xcal}  u \circ X)(\xcal, t') dt' \le C_X \,. \label{feeling-good}
\end{align} 
It is clear from the transformations \eqref{eq:tilde:x:def} and \eqref{eq:tilde:u:def} that 
\begin{align} 
\tfrac{1}{C_X} \le \exp \int_{-\eps}^{T_*} (\operatorname{div}_{\tilde x}  \tilde u \circ X)(\tilde x, t') dt' \le C_X  \label{even-better-mofo}
\end{align} 
and from   \eqref{damn-thats-nice}, \eqref{page48}, \eqref{even-better-mofo}, and  \eqref{buzz-buzz},
\begin{align} 
 \exp \int_{-\eps}^{T_*}  (\mathcal{N} _j  \p_{\xcal_j} (u\cdot   \mathcal{N} )) \circ X  dt' \le C \,. \label{buzzed-good-bro}
\end{align} 
By possibly enlarging the constant $C$ in \eqref{buzzed-good-bro},  by  \eqref{def_f}, \eqref{tangent}, \eqref{normal},  \eqref{damn-thats-nice}, and \eqref{page48}, we obtain
\begin{align} 
 \exp \int_{-\eps}^{T_*} \abs{ \diamondsuit } dt' \le C \,, \label{buzzed-good-bro2}
\end{align} 
where $\diamondsuit $ denotes one of the $10$ remaining exponential stretchers in \eqref{vomit}.   Consequently, 
taking the inner-product of \eqref{tan-comp} with $\p_{\xcal_j}X^k \ {\mathcal{T}}^\nu_k \circ X $ and summing this with the inner-product of 
\eqref{nor-comp} and $\p_{\xcal_j}X^k \ {\mathcal{N}}_k \circ X $ and applying Gronwall, we find that 
$$
\abs{\p_{\xcal_j}X^k \ {\mathcal{N}}_k \circ X}^2 + \abs{\p_{\xcal_j}X^k \ {\mathcal{T}}^\nu_k \circ X}^2 = \abs{ \nabla_{\!\xcal} X}^2  \le C  \,,
$$
since $X$ is the identity map at time $t=-\eps$.  
This implies that the eigenvalues of $ \nabla X$ are uniformly bounded from above on the time interval
$[-\eps, T_*)$, and therefore by \eqref{light-buzz}, the eigenvalues are bounded in absolute value from below by $\lambda_{\operatorname{min}} >0$.
Using the Lagrangian version of \eqref{svorticity}, which is given by,
$$
\zeta( X(\xcal, t), t) = \nabla_{\!\xcal} X(\xcal, t)  \cdot \zeta_0(\xcal) \,,
$$
we see that on the set that $\zeta_0(\xcal) \ge c_0$, we have that
\begin{align} 
\abs{ \zeta( X(\xcal, t), t)} \ge \lambda_{\operatorname{min}}  c_0 \,, \label{wish_I_was_drunk}
\end{align} 

 Since
$X(\xcal, T_*) - X(\xcal, -\eps) = \int_{-\eps}^{T_*} u(X(\xcal, s)) ds$,  and  $\| u \|_{L^ \infty } = \| U \|_{L^ \infty }$ we have from \eqref{T*-bound} that
\begin{align} 
\snorm{X(\cdot , T_*) - X(\cdot , -\eps) }_{L^ \infty }  & \le (T_* + \eps) \| u \|_{L^ \infty }   \le (2M^2 \eps^2 + \eps) ( \tfrac{\kappa_0}{2}  + \eps^ {\frac{1}{8}}) 
\le \eps \kappa_0 \,. \label{X-bound1}
\end{align} 
It follow from \eqref{density-final} and \eqref{wish_I_was_drunk} that if the condition  \eqref{vort-assume} on the initial vorticity holds, then 
\eqref{vort-nonzero} and this concludes the proof.
\end{proof} 

\subsection{Convergence to stationary solution}\label{s:CSS}
\begin{theorem}[Convergence to stationary solution]\label{thm-2week-delay}
There exists a $10$-dimensional symmetric $3$-tensor $\mathcal{A}$ such that, with $\bar W_ \mathcal{A} $ defined in Appendix~\ref{sec:delay}, we have that the solution $W(\cdot,s)$ of \eqref{eq:euler:ss:a} satisfies
$$
\lim_{s\to \infty } W(y,s)  = \bar W_ \mathcal{A} (y) 
$$
for any fixed $y \in \RR^3$.
\end{theorem} 
\begin{proof}[Proof of Theorem \ref{thm-2week-delay}]
We will first show that as $s\rightarrow\infty$, that the equation \eqref{eq:euler:ss:a}, converges pointwise to the self-similar Burgers equation
\begin{align*} 
\partial_s W-\tfrac12  W+\left(  W+\tfrac32 y_1\right)\partial_1  W+ \tfrac12 \check y \cdot\check\nabla  W=0\,.
\end{align*} 
To do this, we write \eqref{eq:euler:ss:a} as
\begin{align*} 
\partial_s W-\tfrac12  W+\left(  W+\tfrac32 y_1\right)\partial_1  W+ \tfrac12 \check y \cdot\check\nabla  W=F\,.
\end{align*} 
where
\begin{align*}
F:=F_W-e^{-\frac s2}\beta_{\tau} \dot \kappa+(W-g_W)\partial_1 W+h_W\cdot\check\nabla W\,. 
\end{align*}
The aim is to show uniform decay of $F$.

From \eqref{eq:gW}, \eqref{eq:acceleration:bound}, \eqref{eq:beta:tau}, \eqref{eq:W_decay}, \eqref{e:h_estimates}, and \eqref{eq:F_WZ_deriv_est}, we have that
\begin{align}\label{eq:FFF}
\abs{F}\les e^{-\frac s2}+\abs{G_W}
\end{align}
Thus we must show uniform decay of $G_W$.
 Recalling the definition of $G_W$ in \eqref{eq:gW}, and applying \eqref{eq:speed:bound}, \eqref{eq:dot:Q},\eqref{eq:beta:tau}, \eqref{eq:GW:hW:0}, \eqref{e:bounds_on_garbage},  \eqref{eq:V:bnd}, together with the fact that we are taking $\kappa\leq M$, we find that
\begin{align}
\abs{G_W}
&\les M e^{-\frac s2}\abs{\check y}^2 +e^{\frac s2}\abs{ \kappa  + \beta_2 Z + 2\beta_1V\cdot \nn  }
\notag \\
&\les  M e^{-\frac s2} \abs{\check y}+
e^{\frac s2}\sabs{\kappa  + \beta_2 Z^0 - 2\beta_1 (R^T \dot \xi)_1  }+ \abs{V}\abs{\Ncal-e_1}  \notag \\
&\quad \qquad +\abs{\dot Q_{11} \left( e^{-s}y_1 +  {\tfrac 12 e^{-\frac s2} \phi_{\nu\mu} y_\nu y_\mu} \right)}+
\abs{\beta_2e^{\frac s2}(Z-Z^0)+ 2\beta_1\dot Q_{1\nu}y_\nu} \notag \\
&\les  e^{-\frac s3}(\abs{y}+1) + \abs{y}\norm{\nabla V}_{L^{\infty}}+e^{\frac s2}\abs{Z-Z^0-\check\nabla Z^0 \cdot \check y}+\abs{\p_1 Z^0 y_1} \notag \\
&\quad \qquad +
\abs{\beta_2e^{\frac s2}\check\nabla Z^0 \cdot\check y+ 2\beta_1\dot Q_{1\nu}y_\nu} \notag \\
&\les  e^{-\frac s3}(1+\abs{y}^2)+
\abs{\beta_2e^{\frac s2}\check\nabla Z^0 \cdot\check y+ 2\beta_1\dot Q_{1\nu}y_\nu}
\,. \label{eq:GWWW}
\end{align}
The identity \eqref{eq:dot:Q:1}, together with the bounds  \eqref{mod-boot}, \eqref{eq:dot:Q}, \eqref{eq:beta:tau},  \eqref{eq:Z_bootstrap}, \eqref{eq:A_bootstrap}, and \eqref{eq:GW:hW:0}, shows that
\begin{align}\label{eq:GWWW2}
\abs{\beta_2e^{\frac s2}\p_\nu Z^0+2\beta_1\dot Q_{1\nu}}\les e^{-\frac s3}\,,
\end{align}
and thus,  using \eqref{eq:FFF}, \eqref{eq:GWWW} and \eqref{eq:GWWW2}, we conclude that
\begin{align}\label{e:force:decay}
\abs{F}\les  e^{-\frac s3}(1+\abs{y}^2)\,.
\end{align}

With $\bar W_{\mathcal A}$ denoting the stationary solution constructed in Appendix~\ref{sec:delay} whose Taylor coefficients about $y=0$ match those of  
$\lim_{s \to \infty }W(y,s)$ up to third order, we define
$$
\tilde W_{\mathcal A}= W- \bar  W_{\mathcal A}\,,
$$
which satisfies the equation
\begin{align}\label{eq:tilde:WA}
( \p_s+\p_1 \bar W_{\mathcal A}- \tfrac{1}{2} )  \tilde W _{\mathcal A}
+  \left(  W + \tfrac{3}{2}   y_1  \right)  \p_{1}  \tilde W_{\mathcal A}
+  \tfrac{1}{2} y_\mu  \p_{\mu} \tilde W _{\mathcal A}
&=  F \,.
\end{align}
In particular, since $\lim_{s \to \infty } D^3 W(0,s) = D^3 W_ \mathcal{A} (0)$, for $\delta >0$, there exists $s_\delta \ge -\log \eps$ such that
\begin{align} 
\abs{D^3 \tilde W(0, s_\delta)} \le \delta \,.  \label{ignoramus1}
\end{align} 

An application of Lemma~\ref{lem:GN} to the function $D^2 W$ and the estimate \eqref{eq:W_decay} yields
\begin{align} 
\snorm{ D^4 W}_{L^ \infty } \les \| W\|_{\dot H^m}^{\frac{4}{2m-7}}  \| D^2W\|_{L^ \infty }^{\frac{2m-11}{2m-7}} \les M^{\frac{10m-11}{2m-7}} \les M^6 \,,
\label{dull-bound}
\end{align} 
for $m \ge 18$.   
Now fix $\delta>0$ and $s_0\geq s_\delta$.   We also fix a point $y_0$.
Using \eqref{ignoramus1}, \eqref{dull-bound}, and the fundamental theorem of calculus, we obtain that
\begin{align} 
\abs{ \tilde W_ \mathcal{A}(y_0, s_0 )} \les \delta + \abs{y_0}^4 M^6 \,.
\label{eq:tilde:WA:initial}
\end{align} 
Here, we have made use of the fact that $\p^\gamma\tilde W_ \mathcal{A} (0, s_0) = 0$ for $\abs{\gamma} \le 2$.

Next, consider the Burgers trajectory $\Phi^{y_0}(s)$, defined by
\begin{subequations} 
\label{burgers-flow}
\begin{alignat}{2}
\p_s \Phi^{y_0}&=\left(W\circ \Phi^{y_0}+\tfrac32 \Phi^{y_0}_1,\tfrac12 \Phi^{y_0}_2,\tfrac 12 \Phi^{y_0}_3\right)  \qquad && s> s_0 \,, \\
 \Phi^{y_0}(s_0) & = y_0 \,.
\end{alignat} 
\end{subequations} 
From the bootstrap $\abs{\partial_1 W}\leq 1$ for $\abs{y}\le \mathcal{L} $,  the explicit formula for $\bar W$ which yields $\bar W(0,\check y)=0$, 
the fundamental theorem of calculus, and the bounds \eqref{eq:W_decay} and \eqref{eq:bootstrap:Wtilde2} , we obtain that
\begin{align*}
\abs{W(y)}\leq\abs{W(y_1,\check y)-W(0,\check y)} + \abs{\tilde W(0,\check y)}&\leq \abs{y_1} + \eps^{\frac{1}{13}}  \abs{\check y} \text{ for } \abs{y}\leq \mathcal L  \,,
\end{align*}
and therefore $y\cdot \left(W+\tfrac32 y_1,\tfrac12 y_2,\tfrac 12 y_3\right)\ge \frac 25 \abs{y}^2$ whenever  $\abs{y}\leq \mathcal L$. 
 It follows from \eqref{burgers-flow}, that
$$
\p_s\abs{\Phi^{y_0}(s)}^2 \ge \tfrac 45 \abs{\Phi^{y_0}}^2\,,
$$
and that
\begin{align} 
\abs{\Phi^{y_0}(s)} \ge \abs{y_0} e^ {\frac{2}{5}(s-s_0)}  \,.  \label{ignoramus2}
\end{align} 
Notice, then, that this trajectory will move at least a distance of length one in the time increment $s-s_0=-\frac{5}{2}\log\abs{y_0} \to \infty $ as $\abs{y_0} \to 0$.
Moreover,
 from \eqref{ignoramus2}, we have that 
\begin{align} 
\abs{\Phi^{y_0}(s_0-\tfrac{5}{2}\log \abs{y_0}+\tfrac{5}{2}\log \mathcal L)}\geq  \mathcal L\,.  \label{ignoramus3}
\end{align} 

Returning now to the evolution equation  \eqref{eq:tilde:WA}, we shall first consider the case that $\abs{y}\le  \mathcal{L} $.
We use the fact that the anti-damping term $(\p_1 \bar W_ \mathcal{A} - \frac{1}{2} )
\tilde W_ \mathcal{A} \ge -\frac{3}{2} \tilde W_ \mathcal{A}$ since  $\abs{\p_1\bar W_{\mathcal A}}\leq 1$.
As a consequence of the forcing estimate \eqref{e:force:decay} and  the initial condition bound \eqref{eq:tilde:WA:initial}, 
we apply the  Gr\"onwall inequality on the time interval $s\in [s_0\,,  s_0-\frac{5}{2}\log\abs{y_0}+\frac{5}{2}\log \mathcal L]$ to obtain that
\begin{equation}\label{womby}
\abs{\tilde W_{\mathcal A}\circ \Phi^{y_0} (s)}\les e^{\frac 32 (s-s_0)} M^6( \abs{y_0}^4+\delta)\les \abs{y_0}^ {-\frac{15}{4}} \mathcal{L} ^{\frac32} M^6
(\abs{y_0}+ \delta ) \les
M^6\mathcal L^{\frac32}\abs{y_0}^{\frac1{4}} \,,
\end{equation}
where we have assumed that $s_0\ge s_\delta$ is taken sufficiently large so that $\delta\leq \abs{y_0}^4$.

By continuity of $\Phi^{y_0}(s)$, we see from \eqref{ignoramus3} that   for any $y_*$ such that $\abs{y_*} \in [\abs{y_0}\,,  \mathcal L]$, there exists 
$ s_* \in [s_0\,, s_0-\frac{5}{2}\log \abs{y_0}+\frac{5}{2}\log \mathcal L]$ such that
\[\Phi^{y_0}(s_*)=y_*,\]
and hence by \eqref{womby},  we obtain that
\begin{equation}
\abs{\tilde W_{\mathcal A}(y_*,s_*)}\les M^6\mathcal L^{\frac32}\abs{y_0}^{\frac1{4}}\,.  \label{ignoramus4}
\end{equation}
By letting $\abs{y_0} \to 0$, any point $y_* \in (0, \mathcal{L} ]$ is equal to $\Phi^{y_0}(s_*)$ for some $y_0$ approaching the origin. Hence,
by continuity, taking $s\rightarrow \infty$ and  letting $\abs{y_0}\rightarrow 0$ in \eqref{ignoramus4}, we have that for any fixed $\abs{y}\leq \mathcal L$, 
\begin{equation}
\lim_{s\rightarrow \infty}\abs{\tilde W_{\mathcal A}(y,s)}=0\,. \label{ignoramus5}
\end{equation}
Furthermore the convergence in uniform on the interval $[0, \mathcal L]$.

It remains to establish the convergence as $s \to \infty $ for the case that   $\abs{y}\geq \mathcal L$.   We 
fix $\delta>0$. From  \eqref{ignoramus5}, there exists an $s_0\geq -\log\eps$ sufficiently large,  such that
\begin{align} 
\abs{\tilde W_{\mathcal A}(y_0,s_0)}\leq \delta \text{ for } \abs{y_0}=\mathcal L\,.   \label{ignoramus6}
\end{align} 
We again apply the Gronwall inequality to  \eqref{eq:tilde:WA}, but now on the time interval $s \in [s_0 , s_0-\tfrac13\log \delta]$.
We find that
\begin{align}
\abs{\tilde W_{\mathcal A}\circ \Phi^{y_0}(s)}\les e^{\frac 32 (s-s_0)} \delta\les \delta^{\frac12} \,.   \label{ignoramus7}
\end{align}

For all $\abs{y}\geq \LLL= \eps^{-\frac{1}{10}}$,
$$
\abs{W(y)} \leq  (1+\eps^{\frac{1}{20}})  \eta^{\frac 16}(y) \leq  (1+\eps^{\frac{1}{20}})^2 \abs{y}   
$$
and so, it follows that
\begin{align*}
y\cdot \left(W+\tfrac32 y_1,\tfrac12 y_2,\tfrac 12 y_3\right)
&\geq y_1^2 + \tfrac 12 \abs{y}^2 - \abs{y_1}  (1+\eps^{\frac{1}{20}})^2 \abs{y}    
\geq \tfrac{1}{2}\abs{y}^2 - \tfrac 14  (1+\eps^{\frac{1}{20}})^4 \abs{y}^2   
\geq \tfrac 15 \abs{y}^2 \,,
\end{align*} 
and hence for $\abs{y_0} \ge \mathcal{L} $, 
\begin{align} 
\abs{\Phi^{y_0}(s)} \ge \abs{y_0} e^ {\frac{1}{5}(s-s_0)}  \,.  \label{ignoramus8}
\end{align} 
Thus, for $s_0 \leq s\leq s_0-\frac34\log \delta$, \eqref{ignoramus8} shows that
\begin{align} 
\Phi^{y_0}(s) \geq \delta^{-\frac1{15} } \mathcal{L} \,.  \label{ignoramus9}
\end{align} 
By continuity, we see from \eqref{ignoramus9} that   for any $y$ such that $\abs{y} \in [ \mathcal L, \delta^{-\frac1{15} } \mathcal{L} ]$, there exists 
$s \in [s_0 , s_0-\tfrac13\log \delta]$ such that
\[\Phi^{y_0}(s)=y,\]
and hence by \eqref{ignoramus7}, 
\begin{equation}
\abs{\tilde W_{\mathcal A}(y,s)}\les  \delta ^ {\frac{1}{2}}  \,.  \notag
\end{equation}
Thus, for any fixed $y$, taking $ \delta \to 0$ and $s \to \infty $ shows that $\tilde W(y,s) \to 0$. This completes the proof.
\end{proof}

\subsection{Proof of Theorem~\ref{thm:SS}}

The system of equations \eqref{euler-ss} for $(W,Z,A)$, with initial data $(W_0, Z_0, Z_0)$ satisfying the conditions of the theorem, is locally well-posed.   In particular,
because the transformations from \eqref{eq:Euler2} to \eqref{euler-ss} are smooth for sufficiently short time, we use the fact that  \eqref{eq:Euler2} is
 locally well-posed in Sobolev spaces and has a well-known continuation principle  (see, for example,  \cite{Ma1984}): 
Letting $\text{U}=(u,\sigma): \mathbb{R}^3  \times \mathbb{R}  \to \mathbb{R}^3  \times \mathbb{R}  _+$ with
initial data $\text{U}_0 = \text{U}( \cdot , -\eps) \in H^k$ for some $k \ge 3$,  there exists a unique local-in-time  solution to the Euler equations \eqref{eq:Euler} 
satisfying $\text{U} \in C([-\eps,T), H^k)$.  Moreover, if $\norm{\text{U}( \cdot , t)}_{C^1} \le K < \infty $ for all $t\in [-\eps, T)$, then there exists $ T_1 > T$, such that $\text{U}$ extends  to a
solution of \eqref{eq:Euler2} satisfying  $\text{U} \in C([-\eps,T_1), H^k)$.  This implies that
$(W,Z,A)$ are continuous-in-time with values in  $ H^k$ and define a local unique solution to \eqref{euler-ss} with initial data $(W_0, Z_0, Z_0)$.   Moreover, the evolution of the modulation functions is described by the  system of ten nonlinear ODEs  \eqref{eq:dot:phi:dot:n} and \eqref{eq:dot:kappa:dot:tau}.
This system also has local-in-time existence and uniqueness as discussed in Remark \ref{rem:local:constraints}.
In Sections \ref{sec:dynamic:closure}--\ref{sec:energy} we close the bootstrap stipulated in Section \ref{sec:bootstrap}, and thus obtain global-in-time solutions with bounds given by the bootstrap.

In particular, the closure of the bootstrap shows that solutions $(W,Z,A)$ to  \eqref{euler-ss} exist  globally in self-similar time,  that
$(W,Z,A) \in C([-\log\eps,+ \infty ); H^k) \cap C^1([-\log\eps,+ \infty ); H^{k-1})$, and that the estimates stated in Theorem \ref{thm:SS} are verified.  
Theorem \ref{thm-2week-delay} shows that $ \lim_{s\to +\infty } W(y,s)= \bar W_ \mathcal{A}  $, where $\bar W_ \mathcal{A} $ is a $C^ \infty $ stationary solution of
the 3D self-similar Burgers equation described in Appendix  \ref{sec:delay}.   Moreover, $\bar W_ \mathcal{A}$ satisfies the conditions stated in
Theorem \ref{thm:SS}.  The  bootstrap estimates
\eqref{mod-boot} then show that the modulation functions are in $C^1[-\eps, T_*)$.  This completes the proof. 

Let us now provide a brief summary of the closure of the bootstrap given  in  Sections \ref{sec:dynamic:closure}--\ref{sec:energy}, which
consisted of the following five steps:
\begin{itemize}[itemsep=2pt,parsep=2pt,leftmargin=.3in]
 \item[(A)] $L^ \infty $ bounds for $\p^\gamma W$ in different
spatial regions for $\abs{\gamma} \le 4$; 
 \item[(B)]  $L^ \infty $ bounds for $\Omega$;
 \item[(C)]  $L^ \infty $ bounds for $\p^\gamma Z$, and $\p^\gamma A$ for $\abs{\gamma}\le 2$;   
 \item[(D)]  $L^2 $ bounds for $\p^\gamma W$, $\p^\gamma Z$, and $\p^\gamma A$ for $\abs{\gamma}=k$, $k\ge 18$;  and
 \item[(E)]  bounds for the modulation functions.
\end{itemize}

(A) We   split the analysis for $W$ into three spatial regions in the support $\XXX(s)$, required to close the bootstrap assumptions
\eqref{eq:W_decay}--\eqref{eq:bootstrap:Wtilde3:at:0}.  The first region ($\abs{y}\leq \ell$) was a small  neighborhood of $y=0$ where the Taylor series of the solution was used.
 The second (large) intermediate region ($\ell \leq \abs{y}\leq \LLL$) was chosen so that $W(y,s)$ and some of its derivatives  remained close to $\bar W$, while the 
third spatial region ($\abs{y}\geq \LLL$) allowed $W$ to decrease to zero at the boundary of $\XXX(s)$, while maintaining important bounds on derivatives.

%
We began our study in the first region $\abs{y} \le \ell$.  Our analysis relied on the structure of the equations satisfied by the  perturbation function $\tilde W(y,s)= W(y,s)-\bar W(y)$ and its
derivatives,  given by $\p^\gamma \tilde W$ by \eqref{eq:tilde:W:evo} and  \eqref{eq:p:gamma:tilde:W:evo}.    As we  showed in 
\eqref{eq:p:gamma:tilde:W:damping},  for $\abs{\gamma}=4$ the damping term satisfies
$ D_{\tilde W}^{(\gamma)} \ge {\sfrac{1}{3}} $ and hence using the bootstrap assumptions, we obtained the $L^ \infty $
bound  \eqref{eq:end:of:november} for all $s\ge -\log\eps$, which closed the
bootstrap \eqref{eq:bootstrap:Wtilde4}.

The ten  time-dependent modulation functions $\kappa, \tau, n_\nu, \xi_i, \phi_{\nu \mu}$, solving the coupled system of ODE given by
\eqref{eq:dot:phi:dot:n} and \eqref{eq:dot:kappa:dot:tau}, were used to enforce the dynamic constraints  $\p^\gamma\tilde W(0,s)=0$ for $\abs{\gamma}=2$.
Using these conditions at $y=0$, and the $L^ \infty $ bound on $\p^\gamma \tilde W$ for $\gamma=4$, we obtained the bound \eqref{eq:Schweinsteiger:2}
for $\abs{\p^\gamma \tilde W(0,s)}$ for $\abs{\gamma} \le 3$, and this closed the bootstrap  \eqref{eq:bootstrap:Wtilde3:at:0}.   The fundamental theorem of 
calculus then closed the remaining bootstrap assumption  \eqref{eq:bootstrap:Wtilde:near:0} for $\abs{y}\le \ell$.

We next obtained $L^ \infty $ estimates for $\p^\gamma \tilde W$ in the region $\ell \le \abs{y} \le \mathcal{L} $.   We  relied on our estimates for trajectories defined
in \eqref{flows}--\eqref{traj}.  In particular, we  proved in Lemma  \ref{lem:escape} that for any
 $y_0 \in \RR^3$ such that $\abs{y_0} \geq \ell$ and $s_0 \geq - \log \eps$, 
$\pw^{y_0}(s)\geq \abs{y_0}e^{\frac{s-s_0}{5}} $ for all  $s \geq s_0$.   Thanks to \eqref{e:space_time_conv}, we were able to convert  temporal decay to spatial decay 
so that
the exponential {\it escape to infinity} of trajectories $\pwy$ provided the essential time-integrability of forcing and damping functions in
\eqref{eq:tilde:W:evo} and  \eqref{eq:p:gamma:tilde:W:evo}, when composed with $\pwy$.   Specifically, these equations were rewritten in weighted form as
\eqref{eq:tildeq}--\eqref{eq:Dq:def}, and then composed with $\pwy$, to which we applied Gr\"onwall's inequality.    We thus obtained
 the weighted estimate \eqref{tildeWfinal}
for $\tilde W$ as well as the weighted estimates for $ \nabla \tilde W$ in \eqref{p1tildeWfinal} and \eqref{p2tildeWfinal}, which closed the bootstrap assumptions
\eqref{eq:tildeW_decay}, which in turn,  as stated in Remark \ref{goodremark1}, closed the first three bootstrap assumption on $W$ in \eqref{eq:W_decay} for the region
$\abs{y} \le \mathcal{L} $.

It remained to close the $L^ \infty $ bootstrap assumptions for $\p^\gamma W$ for $\abs{\gamma}=2$ in the region
$\abs{y} \ge \ell$. We employed
the same type of weighted estimates along trajectories $\pwy$ as for the study of $ \nabla \tilde W$ above, and  thus established the bound \eqref{D2Wfinal} which, in 
conjunction with our choice of $\ell =(\log M)^{-5}$ satisfying \eqref{eq:kingfisher:3}, closed the bootstrap assumption in \eqref{eq:W_decay}.   
Finally, in the third spatial region  $\abs{y} \ge  \mathcal{L} $, using the same type of weighed estimates along trajectories $\pwy$, we obtained weighted estimates \eqref{Wlargey-final} for $W$ and  \eqref{p1Wlargey-final}--\eqref{p2Wlargey-final}  for  $ \nabla W$    
which closed the first three bootstrap assumptions in \eqref{eq:W_decay} for $\abs{y} \ge  \mathcal{L} $.   This completed the $L^ \infty $ estimates for $\p^\gamma W$.

(B) The specific vorticity estimates required a decomposition of the vector $\mrg$ into the normal component $\mrg \cdot \Ncal $ and the tangential components
$\mrg \cdot \Tcal^\nu$ as was done in \eqref{sv-ss}.  We observed  that these geometric components of specific vorticity have forcing functions \eqref{ssvort-force}
which are bounded;  therefore,  in  Proposition \ref{prop:vorticity}, we established the upper  bound \eqref{svort-bound}.    For the
self-similar sound speed $S$, we also  established the upper and lower bounds \eqref{S-bound} in  Proposition \ref{prop:sound}.

(C) We then closed the bootstrap assumptions \eqref{eq:Z_bootstrap} and  \eqref{eq:A_bootstrap} for $\p^\gamma Z$ and $\p^\gamma A$ with $\abs{\gamma} \le 2$.
To do so, we relied upon Lemma \ref{lem:phiZ}, wherein we proved that trajectories $\pzy(s)$ and $\pay(s)$ escape to infinity exponentially fast for all $y_0 \in \XXX_0$,
and also upon Corollary  \ref{cor:p1W} which  established the integrability (for all time) of both  $\p_1 W$ and $\p_1 \tilde W$ along these trajectories.  This then allowed
us to use weighted estimates for $\p^\gamma Z$ to obtain  the bounds \eqref{zztop1}--\eqref{zztop4} which closed the bootstrap assumptions \eqref{eq:Z_bootstrap}. 
The same type of weighted estimates for $A$ then yielded the bounds \eqref{A-estimates-1} which closed the bootstrap assumptions  \eqref{eq:A_bootstrap} for all
$\abs{\gamma} \le 2$ with $\gamma_1=0$.  For the latter case, we relied crucially on the previously obtained  specific vorticity estimates.
In particular, Lemma \ref{lem:remarkable:sheep:structure} proved that bounds on geometric components of specific vorticity give the desired $L^ \infty $ bounds on $\p_1 A$.

(D) In order to complete the bootstrap argument, we obtained $\dot H^k$-type energy estimates for the $(U,S)$-system of equations \eqref{US-euler-ss}.  The evolution for 
the differentiated system $(\p^\gamma U, \p^\gamma S)$ was computed in  \eqref{US-L2}--\eqref{S-gamma-forcing}.  
The main idea for closing the energy estimate
was to make use of the $L^ \infty$ bounds for $\p^\gamma W$ and $\p^\gamma Z$ with $\abs{\gamma}\le 2$ and for $\p^\gamma A$ with $\abs{\gamma}=1$.  Together
with the damping obtained when $k$ is chosen large enough, the lower-order $L^ \infty$ bounds effectively linearized the resulting damped differential inequalities which 
then lead to global-in-time bounds.   Instead of obtaining bounds for the $\dot H^k$-norm directly, we  instead obtained bounds for the weighted norm
$E_k^2(s) =
\sum_{|\gamma|=k} \lambda^\modckg  ( \norm{\p^\gamma U(\cdot,s)}_{L^2}^2 +\norm{\p^\gamma S(\cdot,s)}_{L^2}^2)$, where 
 $\lambda= \tfrac{\delta ^2}{12k^2} $,  $0 < \delta\le \tfrac{1}{32}$, and $k \ge 18$. 
 The energy method proceeded in the following manner:
 we considered the sum of the $L^2$ inner-product of \eqref{US-L2-U} with 
$\lambda^\modckg \p^\gamma U^i$ and the $L^2$ inner-product of  \eqref{US-L2-S} with $\lambda^\modckg  \p^\gamma S$.   We made use of a fundamental cancellation 
of terms containing $k$$+$$1$ derivatives that lead to the identity \eqref{energy0}, obtained the lower-bound on the damping, and employed the error bounds from Lemma \ref{lem:forcing:1}.  This lead us to the differential inequality $\tfrac{d}{ds} E_k^2 + 2 E_k^2  \leq  2 e^{-s} M^{4k-1} $ which then yielded the  desired $\dot H^k$ bound. 

(E) Closing the bootstrap assumptions for the modulation variables used the precise form of the ODE system  \eqref{eq:dot:phi:dot:n} and \eqref{eq:dot:kappa:dot:tau} and
relied on the bounds $W$, $Z$, $A$, and some of their partial derivatives at $y=0$.   The bounds \eqref{tau-final}--\eqref{n-final} closed the bootstrap assumptions
\eqref{mod-boot}.

\subsection{Proof of Theorem~\ref{thm:main}} 
The blow up time $T_*$ is uniquely determined by the formula \eqref{good-vlad}; the blow up location is defined by $\xi_* = \xi(T_*)$.  
The bounds \eqref{T*-bound} and \eqref{xi*-bound} shows that $\abs{T_*} = \OO(\eps^2)$ and $\abs{\xi_*} = \OO(\eps)$, respectively.  Moreover,
$\kappa(t)$ satisfies \eqref{kappa-kappa0}, and from \eqref{eq:phi:0:def} and  \eqref{eq:speed:bound}, 
for each $t \in [-\eps, T_*)$, we have that 
$\sabs{\Ncal(\check {\tilde x},t) - \Ncal_0(\check \xcal)} +\sabs{\Tcal^\nu(\check {\tilde x},t) - \Tcal_0^\nu(\check \xcal)}  = \OO(\eps) \,.$

By Theorem \ref{thm:SS}, $(W,Z,A) \in C([-\log\eps,+ \infty ); H^k)$ and since $U =  \tfrac{1}{2} ( e^{-\frac{s}{2}} W + \kappa+ Z )\Ncal + A_\nu \Tcal^\nu$ and
$S =  \tfrac{1}{2} ( e^{-\frac{s}{2}} W + \kappa- Z )\Ncal + A_\nu \Tcal^\nu$, then $(U,S) \in C([-\log\eps,+ \infty ); H^k)$.  The identities \eqref{trannies} together with
the change of variables \eqref{eq:y:s:def} show that 
$(\mathring u,\mathring \sigma) \in C([-\eps,T_* ); H^k)$.  It then follows from  the sheep shear coordinate and function transformation, \eqref{x-sheep} and \eqref{usigma-sheep},
together with the fact that $\abs{\phi} = \OO(\eps)$ from \eqref{eq:speed:bound}
 that $(\tilde u,\tilde \sigma) \in C([-\eps,T_* ); H^k)$.  Finally, the transformations \eqref{eq:tilde:x:def} and  \eqref{eq:tilde:u:def}  show that $(u, \sigma) \in C([-\eps,T_* ); H^k)$.  Clearly $\rho \in C([-\eps,T_* ); H^k)$ as well.

From the change of variables \eqref{x-sheep}, we have that
\begin{align*} 
\p_{\tilde x_1} \tilde w (\tilde x,t)  &= \p_{x_1} w(x,t) \,, \ \
\p_{\tilde x_\nu} \tilde w (\tilde x,t) 
= \p_{x_\nu} w(x,t) - \p_{x_1} w(x,t) \p_{x_\nu} f(\check x,t)   \,,
\end{align*}
so that by \eqref{normal}, this identity is written as
\begin{align*} 
\p_{\tilde x_j} \tilde w (\tilde x,t)  = \p_{x_1} w(x,t) \Jcal  \Ncal_j   + \delta _{j \mu} \p_{x_\mu} w(x,t) \,.
\end{align*} 
Hence, we see that
\begin{subequations} 
\begin{align}
(\Ncal \cdot \nabla_{\!\tilde x})\tilde w (\tilde x,t)  & = \p_{x_1} w(x,t) \Jcal + \Ncal_\mu\p_{x_\mu} w(x,t) = e^{s} \p_1 W(y,s) \Jcal + \Ncal_\mu\p_{\mu}  W(y,s) \,, 
\label{laphroaig10}\\
(\Tcal^\nu \cdot \nabla_{\!\tilde x})\tilde w (\tilde x,t)  & =  \Tcal^\nu_\mu\p_{x_\mu} w(x,t)  =   \Tcal^\nu_\mu\p_{\mu} W(y,s) \,. \label{cask-strength}
\end{align} 
\end{subequations} 
Using the definitions of the transformation \eqref{vort00}, \eqref{x-sheep}, \eqref{eq:y:s:def}, \eqref{w_ansatz}, the fact that $f(0,t)=0$, and the constraints \eqref{eq:constraints}, we see
from \eqref{laphroaig10} that
$$
(\Ncal \cdot \nabla_{\!\tilde x})\tilde w (\xi(t),t)  =  e^{s} \p_1 W(0,s) \Jcal + \Ncal_\mu\p_{\mu}  W(0,s) = -e^{s} = \tfrac{-1}{\tau(t) -t} \,,
$$
and hence $\lim_{t\to T_*} (\Ncal \cdot \nabla_{\!\tilde x})\tilde w (\xi(t),t) = - \infty $.  Moreover, from \eqref{ic-kappa0-phi0} and \eqref{e:bounds_on_garbage}, we have that
$\abs{\Jcal} \les 1+ \eps$ and $\abs{\Ncal_\nu} \les \eps^ {\frac{3}{2}} $, and so from \eqref{laphroaig10}, it  follows  that
$$
\tfrac{1}{2(T_*-t)} \leq  \norm{\Ncal \cdot \nabla_{\! \tilde x} \tilde w(\cdot,t)}_{L^\infty} \leq \tfrac{2}{T_*-t}  \text{ as } t \to T_* \,.
$$

By Theorem \ref{thm:SS}, we have that 
\begin{align*} 
\snorm{ e^ {\frac{3s}{2}} \p_1 Z }_{ L^ \infty}+\snorm{ e^ {\frac{3s}{2}} \p_1 A }_{ L^ \infty} +\snorm{ e^ {\frac{s}{2}} \check\nabla   Z }_{ L^ \infty}  \le M \eps^ {\frac{1}{2}}    \le M^ {\frac{1}{2}}\,, \ \ 
\snorm{ e^ {\frac{s}{2}} \check\nabla   Z }_{ L^ \infty}  \le M \eps^ {\frac{1}{2}}  \,, 
\end{align*} 
and hence by the transformation \eqref{eq:y:s:def}, \eqref{z_ansatz},  and  \eqref{a_ansatz}, 
\begin{align*} 
\snorm{  \nabla _{\!  x}  z }_{ L^ \infty} + \snorm{  \nabla _{\!  x}  a }_{ L^ \infty}  \les M  \,.
\end{align*} 
Since
\begin{align*} 
\p_{\tilde x_\nu} \tilde z (\tilde x,t) = \p_{x_\nu} z(x,t) - \p_{x_1} z(x,t) \p_{x_\nu} f(\check x,t) \text{ and } 
\p_{\tilde x_\nu} \tilde a (\tilde x,t) = \p_{x_\nu} a(x,t) - \p_{x_1} a(x,t) \p_{x_\nu} f(\check x,t) \,,
\end{align*} 
and hence
\begin{align*} 
\snorm{  \nabla _{\!  \tilde x}  \tilde z }_{ L^ \infty} + \snorm{  \nabla _{\!  \tilde x}  \tilde a }_{ L^ \infty}  \les M  \,.
\end{align*} 
By Corollary \ref{cor:vort}, 
$\snorm{ \tilde \rho^ \alpha ( \cdot , t) -  \alpha \tfrac{\kappa_0}{2} }_{ L^ \infty } \le \alpha  \eps^ {\frac{1}{8}}$  for all $t\in[-\eps, T_*] $, and hence $\rho$ is strictly positive
and bounded.  Now
$$\tilde u \cdot \Ncal  = \tfrac{1}{2} (\tilde w + \tilde z)\,, \ \ \   \rho = \left(\tfrac{ \alpha }{2} (\tilde w + \tilde z)\right)^ {\sfrac{1}{\alpha }} \,,$$
and hence \eqref{lagavulin16} immediately follows.   Finally,  Corollary \ref{cor:vort} establishes the claimed vorticity bounds.

\begin{remark}\label{rem:idiot}
Note that the $(\tilde w, \tilde z, \tilde a)$ as defined by \eqref{laphroaig18} are solutions to the system \eqref{eq:Euler-riemann}.   Thus,
one may obtain $(u, \rho)$ as a solution of \eqref{eq:Euler} and define $(\tilde w, \tilde z, \tilde a)$  by \eqref{laphroaig18} or  equivalently, one may 
directly solve  \eqref{eq:Euler-riemann} with the corresponding initial conditions.
\end{remark} 

\subsection{Open set of initial data, the proof of Theorem~\ref{thm:open:set:IC}}
\begin{proof}[Proof of Theorem~\ref{thm:open:set:IC}] 

Let us denote by $\tilde{\mathcal F}$ the set of initial data $(u_0, \sigma_0)(\xcal)$, or equivalently $(\tilde w_0, \tilde z_0, \tilde a_0)(\xcal)$, which are related via the identity \eqref{eq:tilde:wza:0}, which satisfy the hypothesis of Theorem~\ref{thm:main}: the support property \eqref{eq:support-init-x}, the $\tilde w_0(\xcal)$ bounds \eqref{IC-setup}--\eqref{check-two-w0}, the $\tilde z_0(\xcal)$ estimates in \eqref{eq:z0:ic}, the $\tilde a_0(\xcal)$ bounds in \eqref{eq:a0:ic}, the specific vorticity upper bound \eqref{eq:svort:IC}, and the Sobolev estimate \eqref{Hk-xland}.  We will let  $\mathcal F$ be a sufficiently small neighborhood of  $\tilde{\mathcal F}$ in the $H^{k}$ topology. The specific smallness  will be implicit in the arguments given below.

A first comment is in order regarding all the initial datum assumptions which are {\em inequalities}, namely \eqref{love-fosters}--\eqref{Hk-xland}. These initial datum bounds are technically not open conditions, since for convenience we have written ``$\leq$'' instead of ``$<$''. However, we note that all of these bounds can be slightly weakened by introducing a pre-factor that is close to $1$ without affecting any of the conclusions of the theorem. Therefore, we view \eqref{love-fosters}--\eqref{Hk-xland} as stable with respect to small perturbations. 

This leaves us to treat the assumption that $(\tilde w_0-\kappa_0, \tilde z_0,\tilde a_0)$ are supported in the set $\XX_0$ defined by \eqref{eq:support-init-x}, and the pointwise conditions on $\tilde w_0$ at $\xcal=0$ given in \eqref{IC-setup}--\eqref{eq:w0:power:series}. We first deal with the support issue, where we use the finite speed of propagation of the Euler system. After that, we explain why the invariances of the Euler equation allow us to relax the pointwise constraints at the origin. Due to finite speed of propagation, these two matters are completely unrelated: the second issue is around $\xcal=0$, while the first one is for $|\xcal|$ large. Thus, in the proof we completely disconnect these two matters.
 
Let $(u_0,\sigma_0)  \in \tilde{\mathcal F}$ and consider a small $H^k$ perturbation $(\bar u_0, \bar \sigma_0) $ which decays rapidly at infinity, but need not have compact support in $\XX_0$. By the local existence theory in $H^k$, from this perturbed initial datum 
\[
(u_0 + \bar u_0 ,\sigma_0 + \bar \sigma_0) =: (u_{0, \rm total} ,  \sigma_{0, \rm total} )
\] we have a maximal local in time $C^0_t H^k_x$ smooth solution of the 3D Euler system \eqref{eq:Euler2}. Let us denote this solution  as $(u_{\rm total} ,  \sigma_{\rm total})$, and let its maximal time of existence be $T_{\rm total}$. The standard continuation criterion implies that if $\int_{-\eps}^T \norm{u_{\rm total}}_{C^1}  < \infty$, then solution may be continued past $T$.

In addition to the set $\XX_0$ defined in \eqref{eq:support-init-x}, for $n\in \{1,2\}$ we introduce the nested cylinders 
\[
\XX_n = \left\{\abs{\xcal_1}\leq  \tfrac{1}{2^{n+1}}  \eps^ {\frac{1}{2}}  ,\abs{\check \xcal}\leq \tfrac{1}{2^{n}}  \eps^ {\frac{1}{6}}  \right\}\,.
\]
Clearly $\XX_3\subset  \XX_2 \subset \XX_1 \subset \XX_0$, and we have 
\begin{align}
\operatorname{dist}(\XX_{n+1},\XX_n^c) \geq \eps^{\frac 34} \, ,\qquad \mbox{for all} \qquad n\in \{0,1\}\,.
\label{eq:nest}
\end{align}
Let $\psi$ be a $C^\infty$ smooth non-negative cutoff function, with $\psi  \equiv 1$ on $\XX_1$ and $\psi^\sharp \equiv 0$ on $\XX_0^c$. Then, we define 
\begin{align*}
(u_0^\sharp,  \sigma_0^\sharp)(\xcal) 
&=(u_0 + \sigma_0) + \psi(\xcal) (\bar u_{0} , \bar  \sigma_{0})(\xcal) \,,
\\
(u_0^\flat,   \sigma_0^\flat)(\xcal) 
&= (1 - \psi(\xcal)) (\bar u_{0} , \bar  \sigma_{0})(\xcal) \,.
\end{align*}
By construction, the {\em inner} initial value $(u_0^\sharp,   \sigma_0^\sharp)$ is compactly supported in $\XX_0$ and  is a small $H^k$ disturbance of the  data $(u_0, \sigma_0)$ on $\XX_0$. Therefore, we can apply Theorem~\ref{thm:main} to this initial datum, and the resulting inner solution $(u^\sharp,   \sigma^\sharp)$ of the Euler system \eqref{eq:Euler2} satisfies all the conclusions of Theorem~\ref{thm:main} (with a suitably defined $(w^\sharp,z^\sharp,a^\sharp)$ defined as in \eqref{eq:tilde:wza:0}). In particular, we have a bound on the maximum  wave speed due to the bound
\begin{align} 
\norm{u^\sharp}_{L^\infty} +   {\norm{\sigma^\sharp }_{L^{\infty}}} \les \kappa_0\,, \label{this:is:trivial}
\end{align} 
and $(u^\sharp, \sigma^\sharp) \in C([-\eps,T_*);H^k)$  with $T_* = \OO(\eps^2)$\,. 
The key observation is that because  $(u_0^\sharp,  \sigma_0^\sharp)$ is identical to our perturbed initial datum $(u_{0, \rm total} , \sigma_{0, \rm total})$ on $\XX_1$ (the cutoff is identically equal to $1$ there), by using the finite speed of propagation and the uniqueness of smooth solutions to the compressible Euler system, from the bounds \eqref{eq:nest} and \eqref{this:is:trivial} we deduce that   
\begin{align}
(u^\sharp,   \sigma^\sharp)(\xcal,t) = (u_{\rm total} , \sigma_{\rm total})(\xcal,t) \qquad \mbox{on} \qquad \XX_2 \times [-\eps,T_*) \,.
\label{eq:inner:part:done}
\end{align}
In particular, because Theorem~\ref{thm:main} guarantees that the only singularity in $(u^\sharp,   \sigma^\sharp)$ occurs at $\xi_* = \OO(\eps)$ at time $T_*$, we know that 
\begin{align}
\label{eq:mushrooms}
\sup_{[-\eps,T_*)}  \norm{(u^\sharp,  \sigma^\sharp)}_{H^k(\XX_2^c)} \leq {\mathcal M}_{k,\eps}
\end{align}
for some constant ${\mathcal M}_{k,\eps}$, which depends  polynomially on $\eps^{-k}$ in view of \eqref{Hk-xland}.

It remains to analyze the total solution on the set $\XX_2^c$. For this purpose, write 
\begin{align}
(u_{\rm total} ,   \sigma_{\rm total} )(\xcal,t)  = (u^\sharp,   \sigma^\sharp)(\xcal,t) + (u^\flat, \sigma^\flat)(\xcal,t) 
\label{eq:total:part:done}
\end{align}
and note that $(u^\flat, \sigma^\flat)$ solves  a version of \eqref{eq:Euler2} where we also add linear terms due to $(u^\sharp,\sigma^\sharp)$:
\begin{subequations}
\label{eq:Euler22}
\begin{align}
\tfrac{1+\alpha}{2}\p_t u^\flat + ((u^\flat  + u^\sharp) \cdot \nabla_{\!\xcal}) u^\flat + \alpha \sigma^\flat \nabla_{\!\xcal} \sigma^\flat
&=(u^\flat \cdot \nabla_{\!\xcal}) u^\sharp  + \alpha \sigma^\flat \nabla_{\!\xcal} \sigma^\sharp + \alpha \sigma^\sharp \nabla_{\!\xcal} \sigma^\flat \,,  \label{eq:momentum3} \\
\tfrac{1+\alpha}{2}\partial_t \sigma^\flat + ((u^\flat  + u^\sharp)  \cdot \nabla_{\!\xcal}) \sigma^\flat  + \alpha \sigma^\flat \operatorname{div}_\xcal u^\flat
&= (u^\flat \cdot \nabla_{\!\xcal}) \sigma^\sharp + \alpha \sigma^\flat \operatorname{div}_\xcal u^\sharp + \alpha \sigma^\sharp \operatorname{div}_\xcal u^\flat \,,  \label{eq:mass3}\\
(u^\flat,\sigma^\flat)|_{t= -\eps} &= (u_0^\flat,\sigma_0^\flat)(\xcal) = (1-\psi(\xcal))(\bar u_{0} , \bar  \sigma_{0})(\xcal) \,.
\label{eq:IC3}
\end{align}
\end{subequations}
In particular, the initial condition in \eqref{eq:IC3} has small Sobolev norm, and is compactly supported in $\XX_1^c$, by the definition of the cutoff function $\psi$. Additionally, every  term in \eqref{eq:momentum3} and \eqref{eq:mass3} contains either a $u^\flat$ or a $\sigma^\flat$ term. Combined with \eqref{this:is:trivial}, the implication is that as long as the maximal wave speed due to $(u^\flat,\sigma^\flat)$ is bounded, e.g.~$\OO(1)$, then on the time interval $[-\eps,T_*)$ the support of the solution $(u^\flat,\sigma^\flat)$ cannot travel a distance larger than $\OO(\eps)$. Hence, due to \eqref{eq:nest}, we have that the support of $(u^\flat,\sigma^\flat)$ remains confined to $\XX_2^c$, again, conditional on an $\OO(1)$ bound for $\snorm{u^\flat}_{L^\infty} + \snorm{\sigma^\flat}_{L^\infty}$ (we have such a bound for short time, but it may not be clear that it holds uniformly on $[-\eps,T_*)$). Next, we recall that using a standard $H^3$ energy estimate for the system \eqref{eq:Euler22}, we may prove that
\begin{align*}
\tfrac{d}{dt} \snorm{(u^\flat,\sigma^\flat)}_{H^{k-1}}^2 \les  \snorm{(u^\flat,\sigma^\flat)}_{H^{k-1}}^3 + \snorm{(u^\flat,\sigma^\flat)}_{H^{k-1}}^2 \snorm{(u^\sharp,\sigma^\sharp)}_{H^k(\XX_2^c)}
\end{align*}
where the implicit constant only depends on $\alpha$ and $k\geq 18$, and we have used the aforementioned support property of $(u^\flat,\sigma^\flat)$. Since we have previously established in \eqref{eq:mushrooms} that $\snorm{(u^\sharp,\sigma^\sharp)}_{H^k(\XX_2^c)} \leq {\mathcal M}_{k,\eps}$ uniformly on $[-\eps,T_*)$, we deduce that if $T_*$ obeys
\begin{align}
\snorm{(u^\flat_0,\sigma^\flat_0)}_{H^{k-1}}^2 \exp\left(2 (T_*+\eps) {\mathcal M}_{k,\eps}\right) \leq 1
\label{eq:how:small:is:small?}
\end{align}
then uniformly on $[-\eps,T_*)$ we have $\snorm{(u^\flat,\sigma^\flat)}_{H^{k-1}} \les 1$; this bound also implies the desired $\OO(1)$ wave speed. To conclude the argument, all we have to do is to choose our initial disturbance $(\bar u_0,\bar \sigma_0)$ to have a small enough $H^{k-1}$ norm (in terms of $\eps$) so that \eqref{eq:how:small:is:small?} holds. We combined this $\OO(1)$ bound on the $H^{k-1}$ norm of the outer solution with \eqref{eq:inner:part:done} and \eqref{eq:total:part:done} to deduce that the total solution $(u_{\rm total},\sigma_{\rm total})$ behaves extremely tame on $\XX_2^c$, and its behavior is given by the bounds in Theorem~\ref{thm:main} on $\XX_2$. We have thus proven that one may indeed remove the strict support condition from the assumptions of Theorem~\ref{thm:main}, as desired.

It remains to show that the pointwise constraints \eqref{IC-setup}--\eqref{eq:w0:power:series} on $\tilde w_0$  can be turned into open conditions. First, we note cf.~\eqref{ic-kappa0-phi0} that Theorem~\ref{thm:main} allows for $\kappa_0$ to be taken in an open set, and by definition  $\eps$ is taken to be sufficiently small, thus also in an open set. As a consequence the conditions on $\tilde w_0(0)$ and $\partial_1 \tilde w_0(0)$  in \eqref{eq:w0:power:series} are open conditions.
It remains to show that by applying an affine coordinate change, we may replace  the assumptions \eqref{IC-setup}, \eqref{grad-d1-w0}, and the last equation in \eqref{eq:w0:power:series} by open conditions.
 
We start with the last condition in \eqref{eq:w0:power:series}.
We aim to show that if $\mathcal F$ is a sufficiently small neighborhood of $\tilde{\mathcal F}$, and $B\subset \mathbb R^3$ is a sufficiently small ball around the origin (with radius depending solely on $\eps$), then  there exists functions ${\mcal}_2,\mcal_3:B\times \mathcal F\rightarrow (-\sfrac12,\sfrac12)$ such that if we define the vector
\begin{equation}\label{eq:Fosters}
\mcal (\xcal, u_0,\sigma_0):=(\mcal_1,\mcal_2,\mcal_3) := \left((1- \mcal_2^2-\mcal_3^2)^{\frac 12}, \mcal_2,\mcal_3\right),
\end{equation}
then for any $\xcal\in B$ and $(u_0,\sigma_0)\in \mathcal F$  
\begin{equation}\label{eq:bud_light}
\mcal_j(\xcal,u_0,\sigma_0)  (\mcal(\xcal,u_0,\sigma_0)\times \nabla_{\! \xcal})  u_{0j}  +  (\mcal(\xcal,u_0,\sigma_0)\times \nabla_{\! \xcal})  \sigma_0  =0\,.
\end{equation}
We denote by $(m_2,m_3)$ two {\em free variables}, i.e. they do not depend on $(\xcal,u_0,\sigma_0)$, and are not to be confused with the pair $(\mcal_2,\mcal_3)$. In terms of $(m_2,m_3)$ we define the vector
\begin{align}
m := (m_1,m_2,m_3):=  \left( (1- m_2^2-m_3^2)^{\frac 12}, m_2,m_3\right) \,.
\label{eq:Fosters:steroids}
\end{align}
in analogy to \eqref{eq:Fosters}. Also in terms of $(m_2,m_3)$ we define the rotation matrix $R= R(m_2,m_3)$ using the definition \eqref{eq:R:def} with $m$ replacing $n$; more explicitly, replace $(n_2,n_3)$ with $(m_2,m_3)$ in \eqref{eq:R:def:detail}. Then, using $R$ we define two vectors which are orthogonal to the vector $m$ defined in \eqref{eq:Fosters:steroids}, as  
$$
\nu_\beta := \nu_\beta(m_2, m_3) := R(m_2,m_3) e_\beta \qquad \mbox{for} \qquad \beta \in \{2,3\}\,.
$$
By construction, $(m,\nu_2,\nu_3)$ form an orthonormal basis. Then, for each $\beta \in \{2,3\}$ define functions
$$
G_\beta(\xcal,u_0,\sigma_0, m_2,m_3):={m}_j  \nu_\beta \cdot \nabla_{\! \xcal} u_{0j}(\xcal) + \nu_\beta \cdot \nabla_{\! \xcal} \sigma_0(\xcal)
$$
where the summation is over $j \in \{1,2,3\}$. 
Thus one can rewrite \eqref{eq:bud_light} as
\begin{equation}\label{eq:tuborg}
G(\xcal,u_0,\sigma_0,\mcal_2(\xcal,u_0,\sigma_0),\mcal_3(\xcal,u_0,\sigma_0))=0
\end{equation}
with $G = (G_2,G_3)$.
By  \eqref{eq:w0:power:series} we have for $(u_0,\sigma_0)\in \tilde{\mathcal F}$ that
\begin{equation}\label{eq:VB1}
G(0,u_0,\sigma_0,0,0)=0\,.
\end{equation}
Moreover, employing the notation $\nabla_{m}f=(\partial_{m_2}f,\partial_{m_3}f)$, for $(u_0,\sigma_0)\in\tilde{\mathcal F}$ we have by \eqref{eq:w0:power:series} that
\begin{align}
\nabla_{\!  m}G(0,u_0,\sigma_0,0,0)&
=\begin{bmatrix}
-\p_1 (u_{01}+\sigma_0) & 0 \\
0 & -\p_1 (u_{01}+\sigma_0)
\end{bmatrix}
=  \eps^{-1}\begin{bmatrix}
1 & 0 \\
0& 1
\end{bmatrix}
\,.\label{eq:VB2}
\end{align}
By \eqref{grad-d1-w0} and \eqref{check-two-w0}, we have
\begin{align}
\nabla_{\! \xcal}G(0,u_0,\sigma_0,0,0)&= (\nabla_{\!\xcal} \p_{\xcal_2} (u_{01}+\sigma_0), \nabla_{\!\xcal} \p_{\xcal_2} (u_{01}+\sigma_0))|_{\xcal=0} \notag\\
&=
\begin{bmatrix}
0 & 0 \\
\p_{\xcal_2}\p_{\xcal_2} \tilde w_0(0) & \p_{\xcal_2}\p_{\xcal_3} \tilde w_0(0) \\
\p_{\xcal_2}\p_{\xcal_3} \tilde w_0(0) & \p_{\xcal_3}\p_{\xcal_3} \tilde w_0(0)
\end{bmatrix}
=\OO(1)\,.
\label{eq:VB3}
\end{align}
Using \eqref{eq:w0:pipi}, \eqref{Hk-xland},   and the interpolation Lemma~\ref{lem:GN} we also have 
\begin{equation}\label{eq:VB4}
\abs{\nabla_{\! m}^2G(\xcal,u_0,\sigma_0,m_2,m_3)}+\abs{\nabla_{\! \xcal}\nabla_{\! m}G(\xcal,u_0,\sigma_0,m_2,m_3)}+\abs{\nabla_{\! \xcal}^2G(\xcal,u_0,\sigma_0,m_2,m_3)}\les \eps^{-7}\,.
\end{equation}

For every $\delta>0$, if we assume $\mathcal F$ is a sufficiently small neighborhood of $\tilde{\mathcal F}$, then for $(u_0,\sigma_0)\in\mathcal F$, we can replace \eqref{eq:VB1}-\eqref{eq:VB4} with
\begin{subequations}
\label{eq:tsingtao}
\begin{align}\label{eq:tsingtao1}
G(0,u_0,\sigma_0,0,0)&=\OO(\delta)\,,\\\label{eq:tsingtao2}
\nabla_{\! m}G(0,u_0,\sigma_0,0,0)&
=\eps^{-1}
\Id
+\OO(\delta)\,,\\\label{eq:tsingtao3}
\nabla_{\! \xcal}G(0,u_0,\sigma_0,0,0)&=\OO(1)\,,
\\
\label{eq:tsingtao4}
\abs{\nabla_{\! \xcal,m}^2G(\xcal,u_0,\sigma_0,m_2,m_3)}&\les \eps^{-7}\,.
\end{align}
\end{subequations}
For {\em a  fixed $(u_0,\sigma_0)\in\mathcal F$}, now consider the map  $\Psi_{u_{0},\sigma_0}\colon \mathbb R^3\times (-\sfrac12,\sfrac12)^2 \to \mathbb R^3\times \mathbb R^2$ given by
\begin{equation}\label{eq:map:to:invert}
\Psi_{u_{0},\sigma_0} (\xcal,m_2,m_3) = (\xcal, G(\xcal,u_0,\sigma_0,m_2,m_3))
\end{equation}
with gradient with respect to $\xcal$ and $m$ given by in block form as
\[D\Psi_{u_{0},\sigma_0}:=\begin{bmatrix}
\Id & 0 \\
\nabla_{\xcal} G& \nabla_m G
\end{bmatrix}\,.\]
From \eqref{eq:tsingtao2} and \eqref{eq:tsingtao3}, we have $\operatorname{det} (D\Psi_{u_0,\sigma_0}) \geq \frac 12\eps^{-2}$, for $\delta \les 1$. Thus, by the inverse function theorem, for each $(u_0,\sigma_0)\in\mathcal F$, there exists an inverse map $\Psi_{u_{0},\sigma_0}^{-1}$ defined in a neighborhood of $(0,G(0,u_0,\sigma_0,0,0))$. Moreover, using \eqref{eq:tsingtao2}-\eqref{eq:tsingtao4}, we can infer that the domain of this inverse function $\Psi_{u_{0},\sigma_0}^{-1}$ contains a ball around $(0,G(0,u_0,\sigma_0,0,0))$ whose radius can be bounded  from below  in terms of $\eps$, independently of $\delta \les 1$. In particular, by assuming $\delta$ to be sufficiently small  in terms of $\eps$, as a consequence of \eqref{eq:VB1} and \eqref{eq:tsingtao1}, we can ensure that the domain of $\Psi_{u_{0},\sigma_0}^{-1}$ contains a ball $B$ centered at the origin with radius depending solely on $\eps$. In other words, assuming $\mathcal F$ is a sufficiently small neighborhood of $\tilde{\mathcal F}$, then $\Psi_{u_{0},\sigma_0}^{-1}$ is well defined on $B$, where $B$ is independent of $(u_0,\sigma_0) \in {\mathcal F}$. The key step is to define
\[
(\mcal_2,\mcal_3) := (\mcal_2,\mcal_3)(\xcal,u_0,\sigma_0):=\mathbb P_{\mcal} \Psi_{u_{0},\sigma_0}^{-1}(\xcal,0)
\, ,
\]
where $\mathbb P_{\mcal}$ is the projection of the vector $\Psi_{u_{0},\sigma_0}^{-1}(\xcal,0)$ onto its last two components. Note that as a consequence of \eqref{eq:tsingtao2}-\eqref{eq:tsingtao4}, we obtain
\begin{align}\label{eq:m_bnd}
\abs{\nabla_{\! \xcal} (\mcal_2,\mcal_3)}\les \abs{(D\Psi_{u_{0},\sigma_0})^{-1}} \les 1 \quad \mbox{and}\quad \abs{\nabla_{\! \xcal}^2 (\mcal_2,\mcal_3)} \les \abs{(D\Psi_{u_{0},\sigma_0})^{-1}} \abs{\nabla_{\! \xcal} (D\Psi_{u_{0},\sigma_0})}\les \eps^{-7}
\end{align}
for all $\xcal\in B$, where we reduce the radius of $B$  if required (dependent only on $\eps$). In order to see the first bound we note that $D\Psi_{u_{0},\sigma_0}$ is a lower triangular matrix. Then using \eqref{eq:tsingtao2} we obtain that $\operatorname{det}(D\Psi_{u_{0},\sigma_0})\geq \frac{1}{2\eps^2}$. Moreover, applying \eqref{eq:tsingtao2} and \eqref{eq:tsingtao3}, we can bound the entries of the cofactor matrix by a constant multiple of $\eps^{-2}$, from which we conclude
\[\abs{(D\Psi_{u_{0},\sigma_0})^{-1}}= \abs{{\rm Cof} D\Psi_{u_{0},\sigma_0}} \abs{\operatorname{det}(D\Psi_{u_{0},\sigma_0})}^{-1} \les 1\,.\]
Thus, we have identified the desired functions $(\mcal_2,\mcal_3)(\xcal,u_0,\sigma_0)$ such that \eqref{eq:tuborg}, and thus \eqref{eq:bud_light} holds for all $\xcal \in B$  and all $(u_0,\sigma_0) \in {\mathcal F}$.

Next, we turn to relaxing the constraint \eqref{grad-d1-w0}. For each  $(u_0,\sigma_0) \in\mathcal F$, we wish for find $\xcal\in B$ such that
\begin{align}\label{eq:H:goal}
H(\xcal, u_0,\sigma_0):= \big( (\nabla \partial_k u_{0j})(\xcal) \mcal_j(\xcal,u_0,\sigma_0)  + (\nabla \partial_k \sigma_0)(\xcal) \big) \mcal_k(\xcal,u_0,\sigma_0) =0\,.
\end{align}
Using \eqref{grad-d1-w0}, for $(u_0,\sigma_0)\in\tilde{\mathcal F}$ we have 
\begin{align}\label{eq:H0}
H(0,u_0,\sigma_0) =0\,,
\end{align}
where we used the identity $ \mcal(0,u_0,\sigma_0)=e_1$.
Moreover, we have 
\begin{align}
\nabla_{\! \xcal} H
&=(\nabla^2 \partial_k u_{0j})(\xcal) \mcal_j(x,u_0,\sigma_0) \mcal_k(\xcal,u_0,\sigma_0)+(\nabla \partial_k u_{0j})(\xcal) \otimes\nabla_{\! \xcal} (\mcal_j(x,u_0,\sigma_0) \mcal_k(\xcal,u_0,\sigma_0)) 
\notag \\
&\qquad + (\nabla^2 \partial_k \sigma_0)(\xcal)   \mcal_k(\xcal,u_0,\sigma_0)+(\nabla \partial_k \sigma_{0})(\xcal) \otimes\nabla_{\! \xcal}   \mcal_k(\xcal,u_0,\sigma_0) \, .
\label{shine:on:you:crazy:diamond}
\end{align}
For $(u_0,\sigma_0)\in\tilde{\mathcal F}$ and $\xcal=0$, by the definition of  $\bar w_\eps$ in \eqref{eq:bar:w:eps} and the property \eqref{eq:Hess_py1_barW} of $\bar W$, we have 
\begin{align}
&(\nabla^2 \partial_k u_{0j}) (0) \mcal_j(0,u_0,\sigma_0) \mcal_k(0,u_0,\sigma_0) + 
(\nabla^2 \partial_k  \sigma_0) (0)   \mcal_k(0,u_0,\sigma_0)   \notag\\
&\qquad =
(\nabla^2 \partial_1 (u_{01}+\sigma_0))(0)=\begin{bmatrix} 
6\eps^{-4} & 0 & 0\\
0 & 2\eps^{-2} & 0 \\
0 & 0 & 2\eps^{-2} \\
\end{bmatrix}+ \RSZ \label{eq:coors1}
\end{align}
where by  \eqref{eq:tilde:w0:3:derivative} and the fact that $k\geq 18$, the remainder $\RSZ$ is bounded as 
\begin{align}
\abs{\RSZ_{11}}\leq \eps^{- \frac 72 - \frac{1}{7}}\,, \qquad  \abs{\RSZ_{1\mu}} + \abs{\RSZ_{\mu1}} \leq \eps^{-\frac 52 - \frac{1}{7}}\,, \qquad  \abs{\RSZ_{\mu\nu}}\leq \eps^{- \frac 32 - \frac{1}{7}}\,.\label{eq:coors2}
\end{align}
By \eqref{grad-d1-w0}, \eqref{eq:w0:pipi}, \eqref{eq:z0:ic}, \eqref{eq:a0:ic}, \eqref{Hk-xland} (which implies by Sobolev embedding an estimate on $\p_{\xcal_1}^2 \tilde a_0$ and $\nabla_{\! \xcal} \p_{ \xcal_1}  \tilde a_0$, where we also use that $k\geq 18$) and \eqref{eq:m_bnd}
\begin{align}
&\abs{ (\partial_i \partial_k u_{0j})(0)  \, \nabla_{\! \xcal_\ell} (\mcal_j(\xcal,u_0,\sigma_0) \mcal_k(0,u_0,\sigma_0)) +  (\partial_i \partial_k \sigma_{0})(0) \, \nabla_{\! \xcal_\ell}  \mcal_k(0,u_0,\sigma_0)} \notag\\
&\qquad \les\begin{cases}
\abs{\p_1 \nabla \tilde z_0(0)}+\abs{\p_1 \nabla \tilde a_0(0)},&\mbox{if }i=1\\
\abs{\check\nabla^2 \tilde w_0(0)}+\abs{\check\nabla \nabla \tilde z_0(0)}+\abs{\check\nabla \nabla \tilde a_0(0)},&\mbox{otherwise} 
\end{cases} \notag\\
&\qquad \les\begin{cases}
\eps^{-\frac{3}{2} - \frac{1}{10}},&\mbox{if }i=1\\
\eps^{-\frac{1}{2} - \frac{1}{10}},&\mbox{otherwise} 
\end{cases}\label{eq:coors3}  \, .
\end{align}
Inserting the bounds \eqref{eq:coors1}--\eqref{eq:coors3} into identity \eqref{shine:on:you:crazy:diamond} we deduce that 
$$ 
\operatorname{det}(\nabla_{\! \xcal} H)(0,u_0,\sigma_0) \geq \eps^{-8}\,,
$$ for all $(u_0,\sigma_0)\in \tilde{\mathcal{F}}$.

Using a similar computation, whose details we omit to avoid redundancy, for $\xcal\in B$ and all $(u_0,\sigma_0)\in\tilde{\mathcal F}$, we may use \eqref{eq:m_bnd}, \eqref{eq:tilde:w0:4:derivative}, \eqref{Hk-xland}, and Gagliardo-Nirenberg-Sobolev to show
\begin{equation}\label{eq:coors4}
\abs{\nabla_{\! \xcal}^2 H}\leq \eps^{-9} \,.
\end{equation}
Therefore, we have established bounds for $H$ similar to those we have established earlier in \eqref{eq:tsingtao} for $G$, which will allow us to again apply the inverse function theorem. More precisely,  let us fix $(u_0,\sigma_0) \in \mathcal F$ and assuming again that $\mathcal F$ is a sufficiently small neighborhood of $\tilde{\mathcal F}$, the map $\Phi_{u_0,\sigma_0} \colon \mathbb R^3 \to \mathbb R^3$  given by
\begin{equation}\label{eq:map:to:invert2}
\Phi_{u_0,\sigma_0}(\xcal) = H(\xcal,u_0,\sigma_0)
\end{equation}
is invertible in a ball centered at $H(\xcal,u_0,\sigma_0)$, with a radius depending solely on $\eps$.  Due to \eqref{eq:H0} we may ensure that this ball contains the origin, and by appealing to \eqref{eq:H0}-\eqref{eq:coors4} and a similar argument to that used on to invert the map in \eqref{eq:map:to:invert}, by assuming $\mathcal F$ is a sufficiently small neighborhood of $\tilde{\mathcal F}$, the map $\Phi_{u_0,\sigma_0}$ defined in \eqref{eq:map:to:invert2} is shown to be invertible in a ball containing the origin, whose radius depends solely on $\eps$ and so is independent of $(u_0,\sigma_0) \in \mathcal F$.
This shows that for each $(u_0,\sigma_0) \in \mathcal F$ there exists $\xcal_0$ in a ball centered around the origin, such that \eqref{eq:H:goal} holds.

To conclude, for a given $(u_0,\sigma_0)\in\mathcal F$ we construct $\xcal_0,\mcal_2(\xcal_0,u_0,\sigma_0)$ and $\mcal_3(\xcal_0,u_0,\sigma_0)$ such that   \eqref{eq:bud_light} and \eqref{eq:H:goal} hold. That is, we have
\[\mcal \times \nabla_{\! \xcal} (\mcal\cdot u_0(\xcal_0)+\sigma_0(\xcal_0)) =0\qquad\mbox{and}\qquad
\nabla_{\! \xcal} (\mcal \cdot \nabla (\mcal\cdot u_0(\xcal_0)+\sigma_0(\xcal_0)) )=0\,.
\]
By the arguments above, we can ensure $\xcal_0,\mcal_2,\mcal_3$ are uniquely defined in a small ball around the origin and they can be made arbitrarily small by assuming that $\mathcal F$ is a sufficiently small neighborhood of $\tilde{\mathcal F}$. Then replacing $(u_0,\sigma_0)$ by
\[(\bar u_0,\bar \sigma_0)(\xcal)=(R^Tu_0(R(\xcal + \xcal_0),\sigma_0(R(\xcal + \xcal_0))\]
where $R$ is the rotation matrix defined in \eqref{eq:R:def} with $(\mcal_2,\mcal_3)$ replacing $(n_2,n_3)$;
then,  we have that $(\tilde u_0,\tilde \sigma_0)$ satisfy the conditions
\[
\check \nabla_{\!\xcal} (\bar u_{01} + \bar \sigma_0)(0) =0\quad\mbox{and}\quad
\nabla_{\! \xcal} \p_{\xcal_1} (\bar u_{01}+\bar \sigma_0)(0) =0\,.
\]
i.e.~the constraint \eqref{grad-d1-w0}   and the last equation in \eqref{eq:w0:power:series}, which was our goal.  To complete the proof, we note that by construction we have that  $\xcal_0$,  $\mcal_2$, and $\mcal_3$ are small and $\mathcal F$ is a sufficiently small neighborhood of $\tilde{\mathcal F}$;  thus, the global minimum of $\p_{\xcal_1}  (\bar u_{01}+\bar \sigma_0)$ must be attained very close to $0$. By the above formula, $\xcal =0$ is indeed a critical point of $\p_{\xcal_1} (\bar u_{01}+\bar \sigma_0)$, and using that the non-degeneracy condition \eqref{eq:tilde:w0:3:derivative} is stable under small perturbations, the minimality condition \eqref{IC-setup} also holds for $(\bar u_0,\bar \sigma_0)$ at $\xcal =0$.  This completes the proof of  Theorem~\ref{thm:open:set:IC}. 
\end{proof}

\appendix

\section{Appendices}\label{sec:toolshed}

\subsection{A family of self-similar solutions to the 3D Burgers equation}
\label{sec:delay}

\begin{proposition}[Stationary solutions for self-similar 3D Burgers]\label{prop-stationary-burgers}
Let $ \mathcal{A} $ be  a symmetric $3$-tensor   such that $ \mathcal{A}_{1jk} = \mathcal{M} _{jk}$ with $ \mathcal{M} $  a positive definite
symmetric matrix.   Then, there exists a $C^ \infty $ solution $\bar W_ \mathcal{A} $ to 
\begin{align}
- \tfrac 12 \bar W_{\mathcal A} + \left( \tfrac{3y_1}{2} + \bar W_{\mathcal A} \right) \p_1 \bar W_{\mathcal A} + \tfrac{\check y}{2} \cdot\check \nabla \bar W_{\mathcal A} = 0 \,,
\label{eq:EQUATION}
\end{align}
which has the following properties:
\begin{itemize} 
\item  $\bar W_{\mathcal A}(0) = 0$, $\p_1 \bar W_{\mathcal A}(0) = -1$, $\p_2 \bar W_{\mathcal A}(0) = 0$, 
\item $\p^\alpha \bar W_{\mathcal A}(0) = 0$ for  $|\alpha|$  even, 
\item  $\p^\alpha \bar W_{\mathcal A}(0) = \mathcal{A}_ \alpha  $  for $|\alpha| = 3$.   
 \end{itemize} 
\end{proposition} 
\begin{proof}[Proof of Proposition \ref{prop-stationary-burgers}]
We first construct an analytic  solution $W = W(y_1,\check y)$ of the 3D self-similar Burgers equation \eqref{eq:EQUATION} for 
$\abs{y} \le r_0$ with $r_0>0$ small, and to be specified below.
To constrict such a solution, we make the following power series ansatz:
\begin{align}
\bar W_{\mathcal A}(y) = - y_1 + \sum_{|\alpha|=3} \frac{\mathcal A_\alpha}{\alpha!} y^\alpha + \sum_{|\alpha|\geq 5, {\rm odd}} a_\alpha y^\alpha := \sum_{\alpha} a_\alpha y^\alpha
\label{eq:ANSATZ}
\end{align}
where $y^\alpha = y_1^{\alpha_1}  y_2^{\alpha_2} y_3^{\alpha_3}$. 
We note that the  properties listed in the statement of the Proposition are satisfied by any function with a convergent power series expansion as above.

Inserting \eqref{eq:ANSATZ} into \eqref{eq:EQUATION}, we deduce that for $\abs{\alpha}\geq 3$
\begin{align}
-\tfrac{1}{2} a_{\alpha}+\sum_{\beta+\gamma=\alpha+e_1}\gamma_1 a_{\beta}a_{\gamma}+\tfrac32 \alpha_1 a_{\alpha}+\tfrac12(\alpha_2+\alpha_3  )a_{\alpha}=0
\,.
\label{vodka:vodka}
\end{align}
Using that $a_{e_1}=-1$ we obtain the recursive expression for $\abs{\alpha}\geq 3$
\begin{align}\label{e:sum:sum:summy}
 a_{\alpha}=\tfrac{2}{3-\abs{\alpha}}\sum_{\substack{\beta+\gamma=\alpha+e_1,\\\beta,\gamma \neq \alpha}}\gamma_1 a_{\beta}a_{\gamma}\,.
\end{align}
To see that the formula provides a recursive definition, we note that since $a_0=0$, note that no term of the type $a_{\nu}$ for $\abs{\nu}> \abs{\alpha}$ appears  on the right hand side. Also note that the only terms of the type $a_{\nu}$ for $\abs{\nu}=\abs{\alpha}$ that appear on the right hand side have the property that $\abs{\check \nu}>\abs{\check\alpha}$.

We seek a bound of the type
\begin{equation}\label{eq:abnd}
a_{\alpha}\leq C_{\alpha_1}C_{\alpha_2}C_{\alpha_3} D^{\abs{\alpha}-2}
\end{equation}
for $\abs{\alpha}\geq 2$, where $C_n$ are Catalan numbers. The inequality \eqref{eq:abnd} is trivial for the case $\abs{\alpha}=2$ since in that case we have $a_{\alpha}=0$. Note that by choosing $D$ sufficiently large, dependent on $\mathcal A$, we obtain \eqref{eq:abnd} for all $\abs{\alpha=3}$. Finally, for $\abs{\alpha}\geq 4$, we may use that $a_{e_1}$ does not appear in the sum \eqref{e:sum:sum:summy} to conclude that
\begin{align*}
\abs{a_{\alpha}}&\leq \tfrac{2}{\abs{\alpha}} \sum_{\substack{\beta+\gamma=\alpha+e_1,\\\beta,\gamma \neq \alpha}}
\beta_1 C_{\beta_1}C_{\alpha_1+1-\beta_1}C_{\beta_2}C_{\alpha_2-\beta_2}C_{\beta_3}C_{\alpha_3-\beta_3}D^{\abs{\alpha}-3}
\\
&\leq  \tfrac{2\alpha_1}{\abs{\alpha}} C_{\alpha_1+2} C_{\alpha_2+1}C_{\alpha_3+1}D^{\abs{\alpha}-3}\\
&\leq  C_{\alpha_1} C_{\alpha_2}C_{\alpha_3}D^{\abs{\alpha}-2}
\end{align*}
where in the second line we used the identity $C_{n+1}=\sum_{j=0}^n C_j C_{n-j}$ and in the third line we used that $C_{n+1}\leq 4 C_n$ and assumed that $D\geq 512$.

From \eqref{eq:abnd} and the bound $C_n\leq 4^n$, we conclude that
\begin{equation}
a_{\alpha}\leq (4D)^{\abs{\alpha}}\,.  \label{Melbourne-breeds-tards}
\end{equation}
from which it immediately follows that the Taylor series \eqref{eq:ANSATZ} converges absolutely, with radius of convergence bounded from below by $r_0:=(8D)^{-1}$.  

Next, we substitute  the partial sum $P_n(y):=\sum_{\abs{ \alpha }=1}^n a_ \alpha y^\alpha  $ of the Taylor series in \eqref{eq:ANSATZ} into \eqref{eq:EQUATION}.  
We consider the expression for the nonlinear term, which by appealing to \eqref{vodka:vodka} becomes
\begin{align*} 
P_n \p_1 P_n 
&= \left( \sum_{\abs{\beta}=1}^n a_ \beta y^ \beta \right) \left( \sum_{ \abs{\gamma}=1}^n  \gamma _1 a_\gamma y_1^{ \gamma_1-1} y_2^{ \gamma_2} y_3^{ \gamma_3}\right) 
\notag\\
&= \sum_{\abs{\alpha}=1}^n y^\alpha \sum_{\substack{1\leq |\beta|,|\gamma|\leq n \\ \beta+\gamma=\alpha+e_1}} \gamma _1 a_\beta  a_\gamma + \sum_{\abs{\alpha} = n+1}^{2n} y^\alpha \sum_{\substack{1\leq |\beta|,|\gamma|\leq n \\ \beta+\gamma=\alpha+e_1}} \gamma_1 a_\beta a_\gamma
\notag\\
&= \left( \tfrac 12 P_n - \tfrac{3y_1}{2} \p_1 P_n - \tfrac{\check y}{2} \cdot \check \nabla P_n \right) +  
\underbrace{\sum_{\abs{\alpha} = n+1}^{2n} y^\alpha \sum_{\substack{1\leq |\beta|,|\gamma|\leq n \\ \beta+\gamma=\alpha+e_1}} \gamma_1 a_\beta a_\gamma}_{=: {\mathcal R}}
\,.
\end{align*} 
For the remainder term $ \mathcal{R} $,  using that $\abs{y}\leq r_0 = (8D)^{-1}$ and \eqref{Melbourne-breeds-tards},  we have that 
\begin{align*}
\abs{\mathcal R}
&\leq \sum_{\abs{\alpha} = n+1}^{2n} \abs{y}^{\abs{\alpha}} \sum_{\substack{1\leq |\beta|,|\gamma|\leq n \\ \beta+\gamma=\alpha+e_1}} \gamma_1 \abs{a_\beta} \abs{a_\gamma}
\leq \sum_{\abs{\alpha} = n+1}^{2n}  {\abs{\alpha}\!+\!2 \choose \abs{\alpha} }r_0^{\abs{\alpha}}  (4 D)^{|\alpha+1|} \\
& 
\leq  4D  \sum_{j = n+1}^{\infty} {j\!+\!2 \choose j}  2^{-j} \les   n^2  2^{-n}
\end{align*}
which vanishes exponentially fast as $n \to \infty$.  This shows that $ \bar W_ \mathcal{A} $ defined by \eqref{eq:ANSATZ} is an analytic solution of
\eqref{eq:EQUATION}
for all $\abs{y} \le r_0$.

We next extend this solution to the entire domain, and we do so via trajectories. Let $\Phi^{y_0}$ be the trajectory
\begin{equation}\label{eq:dingo}
\p_s \Phi^{y_0}=\left( \tfrac{3y_1}{2} + \bar W_{\mathcal A},\tfrac12 x_2,\tfrac12x_3\right)\circ\Phi^{y_0} ,\quad\Phi^{y_0}(0)=y_0\,.
\end{equation}
Let us choose $0<\delta<\frac{r_0}{2}$ sufficiently small such that
\begin{gather}\label{eq:platypus1}
-1<\p_1 \bar W_{\mathcal A}(y)<-\tfrac12 \,,
\\
 y\cdot \left( \tfrac{3y_1}{2} + \bar W_{\mathcal A},\tfrac12 x_2,\tfrac12x_3\right)\geq \tfrac{\abs{y}}{3}\,,\label{eq:platypus2}
\end{gather}
for all $\abs{y}\leq \delta$.

For any $\frac{\delta}{2} \leq \abs{y_0}\leq \delta$ and $s\geq 0$, we define
\begin{equation}\label{eq:green:turtle}
\bar W_{\mathcal A}\circ \Phi^{y_0}=e^{\frac s2} \bar W_{\mathcal A}(y_0)\,.
\end{equation}
Let $\mathcal D$ be the domain of $\bar W_{\mathcal A}$. The aim is to prove that $\bar W_{\mathcal A}=\mathbb R^3$. First we show that the definition \eqref{eq:green:turtle} assigns a unique value for every $y\in\mathcal D$. In particular, suppose for a given $y_*\in \mathcal D$, there exists $y_0,\tilde y_0$ such that $\abs{y_0},\abs{\tilde y_0}\leq \delta$ such that
\[
 \Phi^{y_0}(s_0)=\Phi^{\tilde y_0}(\tilde s_0)=y_*
\]
for some $s_0,\tilde s_0\geq 0$. Without loss of generality, assume $s_0\geq \tilde s_0$. Let us denote $\bar y:=\Phi^{\tilde y_0}( \tilde s_0-s_0)$ which satisfies $\abs{\bar y}<\delta$ by \eqref{eq:platypus2} and we have
\begin{equation}\label{eq:echidna}
 \Phi^{y_0}(s_0)=\Phi^{\bar y_0}( s_0)=y_*\,.
 \end{equation}
From \eqref{eq:dingo} and \eqref{eq:green:turtle},we have
\begin{align*}
 \Phi_1^{y_0}(s)&=e^{\frac32 s}(y_{0_1}+\bar W_{\mathcal A}(y_0)(1-e^{-s}))\,,\\ 
 \Phi_1^{\bar y_0}(s)&=e^{\frac32 s}(\bar y_{0_1}+\bar W_{\mathcal A}(\bar y_0)(1-e^{-s}))\, ,\\
 \check \Phi^{y_0}(s)&=\check \Phi^{\bar y_0}( s)=
e^{\frac{s-s_0}{2}}\check y_*\,.
\end{align*}
In particular substituting $s=s_0$ into the first two equations and $s=0$ in the second equation we obtain
\[y_{0_1}+\bar W_{\mathcal A}(y_0)(1-e^{-s_0})=\bar y_{0_1}+\bar W_{\mathcal A}(\bar y_0)(1-e^{-s_0})\quad\mbox{and}\quad  y_{0_\nu}= \bar y_{0_\nu}\]
Rearranging the first equation, we have
\begin{align*}
y_{0_1}-\bar y_{0_1}=(W_{\mathcal A}(\bar y_0)-\bar W_{\mathcal A}(y_0))(1-e^{-s_0})
\end{align*}
which is impossible by \eqref{eq:platypus1} and the  Fundamental Theorem of Calculus.
Thus we must have $y_0=\bar y_0$, and thus we obtain a unique value for $\bar W_{\mathcal A}(y_*)$.

Now consider trajectories beginning at a point $y_0$ on the ball $\abs{y_0}=\delta$. Then differentiating \eqref{eq:dingo} in $y_1$ and solving explicitly along trajectories $\Phi^{y_0}$, we obtain
\begin{align*}
\p_1 \bar W_{\mathcal A}\circ\Phi^{y_0}(s) =\frac{\p_1\bar W_{\mathcal A}(y_0)}{(\p_1\bar W_{\mathcal A}(y_0)+1)e^s-\p_1\bar W_{\mathcal A}(y_0)}\geq -1 \,.
\end{align*}
Here we have used that that the hessian of $\p_1 \bar W_{\mathcal A}$ at $0$ given by $\nabla^2 \p_1 \bar W_{\mathcal A}(0)$  is positive definite, and we 
have assumed that $\delta$ is taken sufficiently small. Indeed,  from the above calculation,  we further have  that
\begin{align*}
\abs{\p_1 \bar W_{\mathcal A}\circ\Phi^{y_0}(s) }\leq C_1e^{-s}
\end{align*}
for some $C_1$ depending on $\mathcal A$ and $\delta$.
Then,  by Gr\"onwall's inequality,  we can bound $\check \nabla W_{\mathcal A}$ along trajectories by
\begin{align*}
\abs{\check \nabla \bar W_{\mathcal A}\circ\Phi^{y_0}(s)}
\leq \exp(C_1 (1-e^{-s}))\abs{\check \nabla \bar W_{\mathcal A}(y_0)}\leq C_2
\,,
\end{align*}
where again $C_2$ depends on $\mathcal A$ and $\delta$.

Let us now observe that by the fundamental theorem of calculus,
\begin{align*}
(y_1,2C_2y_2,2C_2y_3 )\cdot\left( \tfrac{3y_1}{2} + \bar W_{\mathcal A},C_2 y_2,C_2y_3\right)&\geq \tfrac32 y_1^2 +2C_2^2\abs{\check y}-\abs{y_1\bar W_{\mathcal A}}\\
&\geq \tfrac12 y_1^2 +2C_2^2\abs{\check y}-C_2\abs{y_1}\abs{\check y}\\&
\geq \tfrac14 y_1^2+C_2^2\abs{\check y}\geq \tfrac14 \abs{(y_1,2C_2y_2,2C_2y_3 )}^2 \,.
\end{align*}
This, in turn,  implies that
\begin{equation}\label{eq:quoll}
\abs{(\Phi^{y_0}_1,2C_2\Phi^{y_0}_2,2C_2\Phi^{y_0}_3 )}\geq 
\abs{(y_{0_1},2C_2y_{0_2},2C_2y_{0_3} )}e^{\frac s4}
\,.
\end{equation}
By a simple continuity argument, this implies that $\mathcal D=\mathbb R^3$.\footnote{Suppose $y_*\in\partial \mathcal D$, then there exists a sequences $y_j,\tilde y_j\in \mathbb R^3$, $s_j\geq 0$ such that we have the following:  $y_j\rightarrow y_*$, $\abs{\tilde y_j}= \delta$ and $y_j=\Phi^{\tilde y_j}(s_j)$. The bound \eqref{eq:quoll} implies that the sequence $s_j$ is uniformly bounded. Then taking a subsequence if necessary, by continuity, there exists $\tilde y$ satisfying $\abs{\tilde y}=\delta$ and $s_*$ such that $\Phi^{\tilde y}(s_*)=y_*$. Thus $y_*\in\mathcal D$ and we conclude $\mathcal D$ is closed. Note that if $y_*\in\mathcal D$, then there exists $\tilde y$ satisfying $\abs{\tilde y}=\delta$ and $s_*$ such that $\Phi^{\tilde y}(s_*)=y_*$. Furthermore, by flowing a small ball around $\tilde y$ by the vector field $\left( \tfrac{3y_1}{2} + \bar W_{\mathcal A},\tfrac12 x_2,\tfrac12x_3\right)$ one can verify that $\mathcal D$ contains a small ball around $y_*$. Thus $\mathcal D$ is open. Since $\mathcal D$ is open, closed and non-empty, $\mathcal D=\mathbb R^3$.}
\end{proof}

\subsection{The derivation of the self-similar equation}
\label{sec:derivation}
The goal of this appendix is to provide details concerning the derivation of the self-similar equations \eqref{euler-ss}, starting from the standard form of the equations in \eqref{eq:Euler}. This derivation was described in Subsections~\ref{sec:time:dependent:coord}--\ref{sec:s:s:equations}, and in this Appendix we include the details that were omitted earlier. 

\subsubsection{The time-dependent coordinate system}
\label{app:time:dependent:coord}
The first step is to go from the spatial coordinate $\xcal$ to the rotated coordinate $\tilde x$. For this purpose, the rotation matrix $R$ defined in \eqref{eq:R:def} may be written out explicitly as
\begin{align}
R =   R(t)  
&=
 \begin{bmatrix}
n_1  & -n_2  & - n_3 \\
n_2    & 1 - \frac{n_2^2 }{1+n_1 } & - \frac{n_2  n_3 }{1+ n_1 } \\
n_3  &  - \frac{n_2  n_3 }{1+ n_1 } & 1 - \frac{n_3^2 }{1+n_1 }
\end{bmatrix} 
=  \begin{bmatrix}
\sqrt{1 -  |\check n|^2} & -n_2 & - n_3\\
n_2  & 1 - \frac{n_2^2}{1+\sqrt{1 -  |\check n|^2}} & - \frac{n_2 n_3}{1+ \sqrt{1 - |\check n|^2}} \\
n_3  &  - \frac{n_2 n_3}{1+ \sqrt{1 -  |\check n|^2}} & 1 - \frac{n_3^2}{1+\sqrt{1 -  |\check n|^2}}
\end{bmatrix}  
\label{eq:R:def:detail}
\end{align}
The new basis of $\RR^3$ given by $\tilde e_i = R e_i$ is thus given explicitly as 
\[
\te_1 = (n_1,n_2,n_3), \quad \te_2 = \left(-n_2,1-\tfrac{n_2^2}{1+n_1}, - \tfrac{n_2 n_3}{1+n_1} \right), \quad \mbox{and} \quad \te_3 = \left(-n_3, -\tfrac{n_2 n_3}{1+n_1}, 1 - \tfrac{n_3^2}{1+n_1}\right). 
\]
The time derivative of the matrix $R$ is given cf.~\eqref{eq:R:dot:def} in terms of $\dot{n}_2$, $\dot{n}_3$ and the matrices
\begin{align}
R^{(2)} =
\begin{bmatrix}
-\frac{n_2}{n_1} & -1 & 0\\
1  & -\frac{n_2(2 + 2 n_1 -n_2^2 -2 n_3^2)}{n_1(1+n_1)^2}  & - \frac{n_3(1 -n_3^2 + n_1)}{n_1(1+n_1)^2}  \\
0  & - \frac{n_3(1 -n_3^2 + n_1)}{n_1(1+n_1)^2}  & - \frac{n_2 n_3^2}{n_1 (1+n_1)^2}
\end{bmatrix}  
= \begin{bmatrix}
-n_2 & -1 & 0\\
1  &  -n_2 & -\frac{n_3}{2} \\
0  &-\frac{n_3}{2} & 0
\end{bmatrix}  + \OO\left(  |\check n|^2 \right)
\label{eq:R:2:def}
\end{align}
and 
\begin{align}
R^{(3)} =
\begin{bmatrix}
-\frac{n_3}{n_1} & 0 & -1 \\
0  & -\frac{n_2^2 n_3}{n_1(1+n_1)^2}  & - \frac{n_2(1 -n_2^2 + n_1)}{n_1(1+n_1)^2}  \\
1  & - \frac{n_2(1 -n_2^2 + n_1)}{n_1(1+n_1)^2}  & - \frac{n_3(2 + 2 n_1 -2 n_2^2 - n_3^2)}{n_1 (1+n_1)^2}
\end{bmatrix}  
= \begin{bmatrix}
-n_3 & 0 & -1\\
0  &  0 & -\frac{n_2}{2} \\
1  &-\frac{n_2}{2} & -n_3
\end{bmatrix}  + \OO\left(  |\check n|^2 \right)
\label{eq:R:3:def}
\end{align}
where we recall that by definition $n_1 = \sqrt{1 -  |\check n|^2}$. With this notation, the matrices $Q^{(2)} = (R^{(2)})^T R$ and $Q^{(3)} = (R^{(3)})^T R$ appearing in \eqref{eq:Q:def} may be spelled out as
\begin{align}
Q^{(2)} =  
\begin{bmatrix}
0 & 1 + \frac{n_2^2}{n_1(1+n_1)} & \frac{n_2 n_3}{n_1(1+n_1)} \\
-1 - \frac{n_2^2}{n_1(1+n_1)} & 0  &  \frac{n_3}{1+n_1}  \\
- \frac{n_2 n_3}{n_1(1+n_1)} & - \frac{n_3}{1+n_1}  & 0
\end{bmatrix}
=
\begin{bmatrix}
0 & 1 &  0 \\
-1  & 0  & \frac{n_3}{2}  \\
0  &- \frac{n_3}{2}  & 0
\end{bmatrix}  + \OO(|\check n|^2)
\, ,
\label{eq:Q2:def}
\end{align}
and
\begin{align}
Q^{(3)} =  
\begin{bmatrix}
0 & \frac{n_2 n_3}{n_1(1+n_1)} & 1 + \frac{n_3^2}{n_1(1+n_1)} \\
-\frac{n_2 n_3}{n_1(1+n_1)} & 0 & -  \frac{ n_2}{1+n_1} \\
- 1 - \frac{n_3^2}{n_1(1+n_1)}   &  \frac{ n_2}{1+n_1}  & 0
\end{bmatrix}
=
\begin{bmatrix}
0 & 0 & 1 \\
0 & 0 & - \frac{n_2}{2}  \\
-1 &  \frac{ n_2}{2}  &0
\end{bmatrix}  + \OO(|\check n|^2) \, .
\label{eq:Q3:def}
\end{align}
Note that both matrices $Q^{(2)}$ and $Q^{(3)}$ are skew-symmetric, and thus so is $\dot Q$.

Next we turn to the definitions of $\tilde u$ and $\tilde \rho$ in \eqref{eq:tilde:u:def}, which may be rewritten as
\[
u(\xcal,t) = R(t) \tu (  R^T(t)  (\xcal-\xi(t)),t)  
\qquad \mbox{and} \qquad
\rho(\xcal,t) =  \trho (  R^T(t) (\xcal-\xi(t)),t)  \, .
\]
From the definitions of $\tilde x$, $\tilde u$ and $\tilde \rho$ in \eqref{eq:tilde:x:def}--\eqref{eq:tilde:u:def} we obtain that
\begin{align*} 
\p_t \tilde x_k
&= \dot R_{\ell k} R_{\ell m} \tilde x_m - R_{\ell k} \dot{\xi}_\ell  \\
\p_{\xcal_\ell} \tilde x_k 
&= R_{\ell k}\\
 \tfrac{1+ \alpha }{2} \p_t u_i 
&= \dot R_{ij} \tu_j +  \tfrac{1+ \alpha }{2} R_{ij} \p_t \tu_j  + R_{ij} \p_{\tilde x_k} \tu_j (\dot R_{\ell k} R_{\ell m} \tilde x_m - R_{\ell k} \dot{\xi}_\ell) \\
\p_{\xcal_\ell} u_i
&= R_{ij} \p_{\tilde x_k} \tu_j R_{\ell k}\\
 \tfrac{1+ \alpha }{2} \p_t \rho
&=  \tfrac{1+ \alpha }{2} \p_t \trho + \p_{\tilde x_k} \trho (\dot R_{ \ell k} R_{\ell m} \tilde x_m - R_{\ell k} \dot{\xi}_\ell)\\
\p_{\xcal_\ell} \rho 
&= \p_{\tilde x_k} \trho R_{\ell k}
\, .
\end{align*} 
Using the above identities and the fact that  $R R^T = \Id$ implies $R_{ki} \dot R_{kj} = - \dot{R}_{ki} R_{kj}$, 
we may write the Euler equations in the basis $(\te_1, \te_2,\te_3)$ as 
\begin{subequations}
\label{eq:Euler-coordinates}
\begin{align}
 \tfrac{1+ \alpha }{2} \partial_t  \tu_i - \dot R_{ki} R_{kj} \tu_j +  (\dot{R}_{\ell j} R_{\ell m}  \tilde x_m - R_{\ell j} \dot \xi_\ell) \p_{\tilde x_j} \tu_i  + \tu_j  \p_{\tilde x_j} \tu_i  +  \tfrac{1}{2 \alpha } \p_{\tilde x_i} \trho^{2 \alpha} &= 0 \,,  \\
 \tfrac{1+ \alpha }{2} \partial_t \left( \tfrac{\trho^\alpha }{ \alpha }\right) + (\dot{R}_{\ell j} R_{\ell m}  \tilde x_m - R_{\ell j} \dot \xi_\ell)  \p_{\tilde x_j} \left( \tfrac{\trho^\alpha }{ \alpha }\right)+ \tu_j \p_{\tilde x_j} \left( \tfrac{\trho^\alpha }{ \alpha }\right) + \alpha  \left( \tfrac{\trho^\alpha }{ \alpha }\right) \p_{\tilde x_j} \tu_j&=0 \,.
\end{align}
\end{subequations}
The perturbations \eqref{eq:Euler-coordinates} presents over the usual Euler system are only due to the $\dot R(t)$ and $\dot{\xi}(t)$ terms, arising from our time dependent change of coordinates. The first term is a linear rotation term, while the second term alters the transport velocity, to take into account rotation.  
Using the definitions of $\dot Q$ in \eqref{eq:Q:def} and $\tilde \sigma$ in \eqref{eq:tilde:u:def}, the system \eqref{eq:new:Euler} now directly follows from \eqref{eq:Euler-coordinates}\,.

\subsubsection{The adapted coordinates}
We first collect a number of properties of the function  $f(\check{\tilde x},t)$ defined in \eqref{def_f}.
Due to symmetry with respect to $\nu\gamma$, we clearly have that 
$$
f,_\nu = \phi_{\nu \gamma}(t) \tilde x_\gamma
$$
so that $f(0,t) =  f,_\nu (0,t) = 0$, and for the Hessian,  we have that
\begin{align*}
f,_{\nu\gamma}(\tilde x,t) =   \phi_{\nu \gamma}(t) \,.
\end{align*}
For the derivative with respect to space and time we have
\begin{align*}
 \dot f,_\nu = \dot\phi_{\nu \gamma}(t) \tilde x_\gamma\,.
\end{align*}

The following Lemma is useful in deriving the equations satisfied by $\mru, \mrs, w,z$, and $a_\nu$.
\begin{lemma}[The divergence operator in the $(\nn, \tt^2,\tt^3)$ basis]\label{lem:div}
\begin{align} 
\operatorname{div}_{\tilde x} \tu 
& = \Ncal_j  \p_{\tilde x_j} (\tu \cdot   \Ncal) +\Tcal^\nu_j  \p_{\tilde x_j} (\tu \cdot   \Tcal^\nu) 
+ (\tu\cdot \Ncal)  \p_{\tilde x_\mu}\Ncal_\mu 
+  (\tu \cdot \Tcal^\nu)  \p_{\tilde x_\mu} \Tcal^\nu_\mu
   \,. \label{div-identity}
\end{align} 
\end{lemma} 
\begin{proof}[Proof of Lemma \ref{lem:div}]
With respect to the orthonormal basis vectors $(\Ncal, \Tcal^2, \Tcal^3)$, we have
\begin{align*} 
\operatorname{div}_{\!\tilde x} \tu  & = \p_{\Ncal} \tu \cdot \Ncal +  \p_{\Tcal^\nu} \tu \cdot \Tcal^\nu   \\
& = \Ncal_j  \p_{\tilde x_j}\tu_i  \Ncal_i+\Tcal^\nu_j  \p_{\tilde x_j}\tu_i  \Tcal^\nu_i  \\
& = \Ncal_j  \p_{\tilde x_j} (\tu \cdot   \Ncal) +\Tcal^\nu_j  \p_{\tilde x_j} (\tu \cdot   \Tcal^\nu) 
-  \tu_i  \Ncal_\beta \Ncal_i ,_\beta - \tu_i  \Tcal^\nu_\beta  \Tcal^\nu_{i,\beta}  \\
& = \Ncal_j  \p_{\tilde x_j} (\tu \cdot   \Ncal) +\Tcal^\nu_j  \p_{\tilde x_j} (\tu \cdot   \Tcal^\nu) 
- (\tu\cdot \Ncal) \Ncal_i  \Tcal^\nu_\beta \Tcal^\nu_{i,\beta}
-  (\tu \cdot \Tcal^\nu)\Tcal^\nu_i  \left( \Ncal_\beta \Ncal_{i,\beta}  +
 \Tcal^\gamma_\beta  \Tcal^\gamma_{i,\beta} \right)  \,.
\end{align*} 

The equation \eqref{div-identity} then follows from the
following identities:
\begin{align*} 
 \Tcal^\nu_i \left(\Ncal_\beta  \Ncal^i,_\beta +  \Tcal^\gamma_\beta  \Tcal^\gamma_{i,\beta} \right)= - \Tcal^\nu_{\mu,\mu} \text{ for } \nu=2,3,
 \ \text{ and } \  \Ncal_i  \Tcal^\nu_\beta  \Tcal^\nu_{i,\beta} & = - N_{\mu,\mu}
 \,.
\end{align*} 
For the first identity, we first consider the case that $\nu=2$, in which case
\begin{align*} 
 \Tcal^2_i \left(\Ncal_\beta  \Ncal^i,_\beta +  \Tcal^\gamma_\beta  \Tcal^\gamma_{i,\beta} \right) 
 & = \Ncal_\beta \Ncal_{i,\beta} \Tcal^2_i  + \Tcal^3_\beta\Tcal^3_{i,\beta} \Tcal^2_i  \\
 & = -\Ncal_\beta\Ncal_i  \Tcal^2_{i,\beta} -     \Tcal^3_\beta\Tcal^3_i \Tcal^2_{i,\beta}  \\
  & = -\left( \Ncal_\beta \Ncal_i  +  \Tcal^2_\beta\Tcal^2_i+    \Tcal^3_\beta \Tcal^3_i \right) \Tcal^2_{i,\beta}  \\
   & = -\left( \Ncal_j\Ncal_i  +  \Tcal^2_j\Tcal^2_i+    \Tcal^3_j\Tcal^3_i \right) \Tcal^2_{i,j} \\
  &= -  \Tcal^2_{j,j}  =  -  \Tcal^2_{\mu,\mu}   \,.
\end{align*}
and clearly, the same holds for $\nu=3$.   For the second identity, note that
\begin{align*} 
 \Ncal_i  \Tcal^\nu_\beta \Tcal^\nu_{i,\beta} = -  \Ncal_{i,\beta} \Tcal^\nu_\beta \Tcal^\nu_i = - \Ncal_{i,\beta} (\Tcal^\nu_\beta \Tcal^\nu_i+ N_\beta \Ncal_i)
 = - \Ncal_{i,j} (\Tcal^\nu_j \Tcal^\nu_i+ \Ncal_j \Ncal_i)= -\Ncal_{j,j} = -\Ncal_{\mu,\mu}  
\end{align*} 
which concludes the proof of the Lemma.
\end{proof} 

Besides the above Lemma, it is useful to note that  under the sheep shear transform \eqref{x-sheep}--\eqref{usigma-sheep} a term of the type $  b \cdot \nabla \tilde g$ becomes $\check b \cdot \check \nabla g + \Jcal b \cdot \nn \, \p_1 g$. In particular, for $  b = T^\nu$, the term involving $\p_1 g$ disappears and we are left with $\check b \cdot \check \nabla g$. This is a key identity used in the following computations. 

Proving that the Euler system in the $\tilde x$ variable \eqref{eq:new:Euler} becomes \eqref{usigma-sheep}--\eqref{eq:time:rescale} in the $x$ variable, is a matter of applying the above observation, identity \eqref{div-identity}, and the chain rule. It is also not difficult to prove that  \eqref{tgvorticity} becomes \eqref{svorticity_sheep} under this change of variables.

\subsubsection{The adapted Riemann variables}
We give the details concerning the derivation of the system \eqref{eq:Euler-riemann-sheep} directly from \eqref{eq:new:Euler}.

We start from \eqref{eq:new:Euler}, in which the space variable is $\tilde x$, and the time is the original time $t$, i.e., prior to \eqref{eq:time:rescale}.  We  define the intermediate Riemann variables 
\begin{align} 
\label{eq:tilde:Riemann}
\tilde w = \tilde u \cdot \Ncal + \tilde \sigma  \,, \qquad \tilde z =  \tilde u \cdot \Ncal - \tilde \sigma   \,, \qquad \tilde a_\nu = \tilde u \cdot \tt ^\nu \,,
\end{align} 
which are still functions of $(\tilde x,t)$, so that
\begin{align*} 
\tilde u \cdot \Ncal  = \tfrac{1}{2} (\tilde w+\tilde z) \,, \qquad  \tilde \sigma  = \tfrac{1}{2} (\tilde w- \tilde z) \,.  
\end{align*}

The Euler sytem \eqref{eq:new:Euler} can be written in terms of the new variables $(\tilde w, \tilde z,\tilde a_2, \tilde a_3)$ as
\begin{subequations}
\label{eq:Euler-riemann}
\begin{align}
&
 \tfrac{1+ \alpha }{2} \p_t \tilde w 
+ \left(\tilde v_j + {\tfrac{1}{2}} (\tilde w+\tilde z) \Ncal _j 
+ {\tfrac{\alpha }{2}} (\tilde w-\tilde z) \Ncal _j 
+ \tilde a_\nu {\Tcal}^\nu_j\right) \p_j \tilde w \notag\\
&\quad = - \alpha \tilde \sigma {\Tcal}^\nu_j \p_j \tilde a_\nu
+ \tilde a_\nu {\Tcal}^\nu_i   \dot{\Ncal_i} 
+ \dot Q_{ij} \tilde a_\nu {\Tcal}^\nu_j \Ncal _i  \notag\\
&\quad\qquad  
+ \left( \tilde v_\mu + \tilde u \cdot \Ncal \Ncal _\mu 
+ \tilde  a_\nu {\Tcal}^\nu_\mu\right) \tilde a_\nu {\Tcal}^\nu_i   
\Ncal_{i,\mu}
- \alpha\tilde \sigma (\tilde a_\nu {\Tcal}^\nu_{\mu,\mu} 
+\tilde u\cdot \Ncal \Ncal_{\mu,\mu})   \,, \\
&
 \tfrac{1+ \alpha }{2} \p_t \tilde z 
+ \left(\tilde v_j + {\tfrac{1}{2}} (\tilde w+\tilde z) \Ncal _j 
- {\tfrac{\alpha }{2}} (\tilde w- \tilde z) \Ncal _j 
+ \tilde a_\nu {\Tcal}^\nu_j\right) \p_j \tilde z \notag\\
&\quad= \alpha \tilde \sigma {\Tcal}^\nu_j \p_j \tilde a_\nu
+ \tilde a_\nu {\Tcal}^\nu_i   \dot{\Ncal_i} 
+ \dot Q_{ij} \tilde a_\nu {\Tcal}^\nu_j \Ncal _i  \notag\\
&\quad \qquad  
+ \left( \tilde v_\mu + \tilde u \cdot \Ncal \Ncal _\mu + \tilde a_\nu {\Tcal}^\nu_\mu\right) \tilde a_\nu {\Tcal}^\nu_i \Ncal_{i,\mu} 
+ \alpha\tilde \sigma (\tilde a_\nu {\Tcal}^\nu_{\mu,\mu} 
+\tilde u\cdot \Ncal \Ncal_{\mu,\mu}) \,, \\
&
 \tfrac{1+ \alpha }{2} \p_t \tilde a_\nu 
+ \left(\tilde v_j + {\tfrac{1}{2}} (\tilde w+\tilde z) \Ncal _j 
+\tilde a_\gamma {\Tcal}^\gamma_j \right) \p_j \tilde a_\nu \notag\\
&\quad= - \alpha \tilde \sigma  {\Tcal}^\nu_i\p_i\tilde \sigma
+ ( \tilde u\cdot\Ncal  \Ncal_i  + \tilde a_\gamma {\Tcal}^\gamma_i) \dot{\Tcal}^\nu_i  \notag\\
&\quad\qquad  
+  \dot Q_{ij} (\tilde u\cdot\Ncal   \Ncal _j   + \tilde a_\gamma  {\Tcal}^\gamma_j ) {\Tcal}^\nu _i 
+  ( \tilde v_\mu + \tilde u \cdot \Ncal \Ncal _\mu + \tilde a_\gamma \Tcal^\gamma_\mu )  
( \tilde u\cdot\Ncal  \Ncal_i + \tilde a_\gamma  {\Tcal}^\gamma_i)   \Tcal^\nu_{i,\mu}   \,.
 \end{align}
\end{subequations}

Next, using the sheep change of coordinates $\tilde x \mapsto x$ defined in \eqref{x-sheep}, we have that the Riemann variables defined earlier in \eqref{tildeu-dot-T} may be written as
\begin{subequations} 
\label{wza-sheep}
\begin{align} 
w(x_1,x_2,x_3,t) & = \tilde w (x_1+ f(x_2,x_3,t),x_2,x_3,t) = \tilde w (\tilde x, t) \,, \\
z(x_1,x_2,x_3,t) & = \tilde z (x_1+ f(x_2,x_3,t),x_2,x_3,t) = \tilde z (\tilde x, t)\,, \\
a_\nu (x_1,x_2,x_3,t) & = \tilde a_\nu  (x_1+ f(x_2,x_3,t),x_2,x_3,t) = \tilde a (\tilde x, t)\,, 
\end{align} 
\end{subequations} 
in analogy to \eqref{usigma-sheep}. Using  the new $x$ variable and unknowns $(w,z,a_2,a_3)$, the system \eqref{eq:Euler-riemann} takes the form
\begin{subequations}
\label{eq:Euler-riemann-sheep0}
\begin{align}
& \tfrac{1+ \alpha }{2} \p_t  w +  \left(-\dot f+\Jcal v \cdot \nn   + {\tfrac{\Jcal}{2}} ( w+ z) 
+ {\tfrac{\alpha \Jcal }{2}} ( w- z)  \right) \p_1  w \notag\\
&\qquad + \left(v_\mu + {\tfrac{1}{2}} ( w+ z) \Ncal _\mu
+ {\tfrac{\alpha }{2}} ( w- z) \Ncal _\mu
+  a_\nu \Tcal^\nu_\mu\right) \p_\mu w \notag \\
&
= - \alpha\mrs \Tcal^\nu_\mu \p_\mu a_\nu +  a_\nu \Tcal^\nu_i   \dot{\Ncal_i} + \dot Q_{ij}  a_\nu \Tcal^\nu_j \Ncal _i  
+ \left( v_\mu + \mru\cdot \Ncal \Ncal _\mu +  a_\nu \Tcal^\nu_\mu\right)  a_\nu \Tcal^\nu_i 
\Ncal_{i,_\mu}  \notag\\
&\qquad 
- \alpha\mrs(  a_\nu \Tcal^\nu_{\mu,\mu}
+\mru\cdot \Ncal \Ncal_{\mu,\mu})  \,, \\
& \tfrac{1+ \alpha }{2} \p_t  z +  \left(-\dot f + \Jcal v \cdot \nn   + {\tfrac{\Jcal}{2}} ( w+ z) 
- {\tfrac{\alpha \Jcal}{2}} ( w- z)  \right) \p_1  z \notag\\
&\qquad + \left(v_\mu + {\tfrac{1}{2}} ( w+ z) \Ncal _\mu
- {\tfrac{\alpha }{2}} ( w- z) \Ncal _\mu
+  a_\nu \Tcal^\nu_\mu\right) \p_\mu z \notag \\
&
=  \alpha\mrs \Tcal^\nu_\mu \p_\mu a_\nu
+  a_\nu \Tcal^\nu_i   \dot{\Ncal_i} + \dot Q_{ij}  a_\nu \Tcal^\nu_j \Ncal _i  
+ \left( v_\mu + \mru\cdot \Ncal \Ncal _\mu +  a_\nu \Tcal^\nu_\mu\right)  a_\nu \Tcal^\nu_i \Ncal_{i,\mu}  \notag\\
& \qquad 
+ \alpha \mrs a_\nu( \Tcal^\nu_{\mu,\mu}
+ \mru\cdot \Ncal  \Ncal_{\mu,\mu}) \,, 
\\
& \tfrac{1+ \alpha }{2} \p_t  a_\nu +  \left(-\dot f+\Jcal v\cdot \nn  + {\tfrac{\Jcal}{2}} ( w+ z)  \right) \p_1  a_\nu 
 + \left(v_\mu + {\tfrac{1}{2}} ( w+ z) \Ncal _\mu
+ a_\gamma \Tcal^\gamma_\mu \right) 
\p_\mu a_\nu \notag\\
&= - \mrs \Tcal^\nu_\mu  \p_\mu \mrs 
+ ( \mru\cdot \Ncal  \Ncal_i + a_\gamma \Tcal^\gamma_i) \dot\Tcal^\nu_i  
 + \dot Q_{ij}  ( \mru\cdot \Ncal  \Ncal _j  + a_\gamma  \Tcal^\gamma_j) \Tcal^\nu _i   \notag \\
 &\qquad 
  +   \left( v_\mu + \mru\cdot \Ncal \Ncal _\mu 
 +  a_\gamma \Tcal^\gamma_\mu\right)  ( \mru\cdot \Ncal \Ncal_i + a_\gamma \Tcal^\gamma_i) \Tcal^\nu_{i,\mu}\,.
 \end{align}
\end{subequations}
The system \eqref{eq:Euler-riemann-sheep} now directly follows from \eqref{eq:Euler-riemann-sheep0},  and by appealing to the notation in 
\eqref{eq:various:beta}.

\subsection{Interpolation}
\label{sec:interpolation}
In this appendix we summarize a few interpolation inequalities that are used throughout the manuscript.

\begin{lemma}[Gagliardo-Nirenberg-Sobolev]\label{lem:GN} 
Let $u: \mathbb{R}  ^d \to \mathbb{R}  $.  Fix $1\le q , r \le \infty $ and $j,m \in \mathbb{N}  $,   and $ \tfrac{j}{m}  \le 
\alpha \le 1$.  Then,  if
$$
\tfrac{1}{p}  = \tfrac{j}{d} + \alpha \left( \tfrac{1}{r} - \tfrac{m}{d} \right)  + \tfrac{1- \alpha }{q} \,,
$$
then
\begin{align} 
\| D^ju\|_{L^p} \le C \| D^m u\|_{L^r}^ \alpha \|u\|^{1 - \alpha }_{L^q} \,. 
\label{eq:special0}
\end{align} 
\end{lemma} 
We shall make use of \eqref{eq:special0} for the case that $p= \tfrac{2m}{j} $, $r=2$, $q= \infty $, which yields
\begin{align} 
\norm{D^j \varphi}_{L^{\frac{2m}{j}}} \les  \norm{\varphi}_{\dot{H}^m}^{\frac{j}{m}}\norm{\varphi}_{L^\infty}^{1 - \frac{j}{m}} \,,
\label{eq:special1}
\end{align} 
whenever $\varphi \in H^m(\RR^3)$ has compact support. The above estimate and the Leibniz rule  classically imply the Moser inequality
\begin{align}
\norm{ \phi \,  \varphi}_{\dot{H}^m}  \les  \norm{\phi}_{L^\infty} \norm{\varphi}_{\dot H^{m}} + 
 \norm{\phi}_{\dot H^{m}} \norm{\varphi}_{L^\infty}\,.
 \label{eq:Moser:inequality}
\end{align}
for all $\phi, \varphi \in H^m( \mathbb{R}^3  ) $ with compact support.
At various stages in the proof we also appeal to the following special case  of \eqref{eq:special0}
\begin{align} 
\snorm{\varphi}_{ \dot H^{k-2}} &\les \snorm{ \varphi}_{\dot H^{k-1} }^\frac{2k-7}{2k-5}  \snorm{  \varphi}_{L^\infty }^\frac{2}{2k-5}  \,, \label{eq:special3} 
\end{align} 
for $\varphi \in H^{k-1}( \mathbb{R}^3  ) $ with compact support.
Lastly, in Section~\ref{sec:energy} we make use of:
\begin{lemma}
\label{lem:tailored:interpolation}
Let $k\geq 4$ and $0 \le l \le k-3$.   Then for $\acal + \bcal = 1 - \frac{1}{2k-4} \in (0,1)$,  and  $q=\frac{6(2k-3)}{2k-1}$,
\begin{align} 
\snorm{D^{2+l} \phi \, D^{k-1-l} \varphi}_{L^2}  \les \snorm{D^k \phi}_{L^2}^\acal \snorm{D^k \varphi}_{L^2}^\bcal 
\snorm{D^2 \phi}_{L^q}^{1-\acal}  \snorm{D^2 \varphi}_{L^q}^{1-\bcal}  \,. \label{cor1}
\end{align} 
\end{lemma}
\begin{proof}[Proof of Lemma~\ref{lem:tailored:interpolation}]
For $0 \le l \le k-3$, define $q=q(k) = \frac{6(2k-3)}{2k-1}$ and $p = p(k,l) = \frac{2 q (k-3)}{2(k-3) + (q-4)l}$. This is the only exponent $p$ such that $\frac{1}{p}$ is an affine function of $l$, and for $l=0$ we have $p=q$, while for $l = k-3$ we have that $p=\frac{2q}{q-2}$. By H\"older's inequality, 
we have
\begin{align*}
\norm{D^{2+l} \phi \, D^{k-1-l} \varphi}_{L^2} \leq  \norm{ D^{2+l} \phi}_{L^p}\norm{ D^{k-1-l} \varphi}_{L^{\frac{2p}{p-2}}  }\, .
\end{align*}
By the Gagliardo-Nirenberg-Sobolev interpolation inequality,
\begin{subequations} 
\begin{align} 
\norm{ D^{2+l} \phi}_{L^p} & \les \norm{D^k \phi}_{L^2}^ \acal \norm{ D^2 \phi}_{L^q}^{1- \acal } \,, \label{i1} \\
\norm{ D^{k-1-l} \varphi}_{L^{\frac{2p}{p-2}}  }&  \les \norm{D^k \varphi}_{L^2}^ \bcal \norm{ D^2 \varphi}_{L^q}^{1- \bcal } \,, \label{i2}
\end{align} 
\end{subequations} 
where the exponents $\acal$ and $ \bcal $ are given by
\begin{align} 
\acal = \frac{ \frac{1}{q} - \frac{1}{p}+ \frac{l}{3}  }{\frac{1}{q} - \frac{1}{2} + \frac{k-2}{3}} \,, \ \ \ 
\bcal = \frac{ \frac{1}{q} - \frac{p-2}{2p}+ \frac{k-3-l}{3}  }{\frac{1}{q} - \frac{1}{2} + \frac{k-2}{3}} \,.  \label{alpha_beta}
\end{align} 
Then,  $\acal + \bcal = 1 - \frac{1}{2k-4} \in (0,1)$, and \eqref{cor1} is established.
\end{proof}

\subsection*{Acknowledgments} 
T.B.\ was supported by the NSF grant DMS-1900149 and a Simons Foundation Mathematical and Physical Sciences Collaborative Grant.   S.S.\ was supported by the Department of Energy Advanced Simulation and Computing (ASC) Program. V.V.\ was supported by the NSF grant DMS-1911413.


\end{document}